\documentclass[10pt]{article}         

\usepackage{amsmath,amsfonts,amssymb}   
\usepackage{amsthm}
\usepackage{mathrsfs}
\usepackage{dsfont}
\usepackage{url}

\newtheorem{thm}{Theorem}[section]
\newtheorem{defn}{Definition}[section]
\newtheorem{lem}[thm]{Lemma}
\newtheorem{prop}[thm]{Proposition}
\newtheorem{cor}[thm]{Corollary}
\newtheorem{rem}{Remark}
\title{Noether Theorems for Lagrangians involving fractional Laplacians}
\author{Filippo Gaia}

\begin{document}
\allowdisplaybreaks
%
\maketitle
\begin{abstract}
In this work we derive Noether Theorems for energies of the form
\begin{equation*}
E(u)=\int_\Omega L\left(x,u(x),(-\Delta)^\frac{1}{4}u(x)\right)dx
\end{equation*}
for Lagrangians exhibiting invariance under a group of transformations acting either on the target or on the domain of the admissible functions $u$, in terms of fractional gradients and fractional divergences. Here $\Omega$ stays either for an Euclidean space $\mathbb{R}^n$ or for the circle $\mathbb{S}^1$. We then discuss some applications of these results and related techniques to the study of nonlocal geometric equations and to the study of stationary points of the half Dirichlet energy on $\mathbb{S}^1$. In particular we introduce the $\frac{1}{2}$-fractional Hopf differential as a simple tool to characterize stationary point of the half Dirichlet energy in $H^\frac{1}{2}(\mathbb{S}^1,\mathbb{R}^m)$ and study their properties.
Finally we show how the invariance properties of the half Dirichlet energy on $\mathbb{R}$ can be used to obtain Pohozaev identities. 
\end{abstract}
\tableofcontents


\section{Introduction}
In the study of problems involving nonlocal operators, a common strategy consists in looking for analogies with familiar problems for local operators, and trying to adapt the concepts, the results and the strategies known for the local problems to the nonlocal framework. When working with problems involving fractional Laplacians, in particular, it may be useful to think about similar problems for the classical Laplacian.
For instance, the ellipticity of $(-\Delta)^s$ (where $s\in (0,1)$) is an essential element in many arguments, whose local analogous exploit the ellipticity of $-\Delta$;
as a specific example, one can think about the regularity theory for Cauchy problems for the fractional Laplacian (see \cite{SilvestreHolder}).
In this spirit, studying a minimization problem for energies of the kind
\begin{equation}\label{eq_intro_eqmodel_1}
\int_{\Omega} L\left(x, u(x), (-\Delta)^su(x)\right)dx
\end{equation}
(where $u$ is a function in an appropriate function space defined on a manifold $\Omega$), one might be interested in parallels to the classical tools used in the Calculus of Variations.
The goal of this thesis is to give an analogous to Noether's Theorem for some Lagrangians of the form (\ref{eq_intro_eqmodel_1}) exhibiting invariance under continuous symmetries and discuss applications of these results and related methods.\\
The interest for this kind of results is strongly related to the field of integro-differential equations; indeed, as we will also see later, many integro-differential equations have a variational formulation in terms of an energy of the form (\ref{eq_intro_eqmodel_1}). Integro-differential equations arise naturally in the study of L\'{e}vy processes, i.e. a stochastic process with independent, stationary increments, of which Brownian motion constitues a particular case; such processes have interesting applications in modelling market fluctuations. Other applications of integro-differential equations can be found in Ecology, Fluid Mechanics, Elasticity and in various other fields where nonlocal interactions are studied\footnote{For an overview of the applications of integro-differential equations see the Introduction to Part I in \cite{RosOtonPhD}, the Introduction to \cite{Garofalo} and the references therein.}.\\

Noether's Theorem is a powerful tool in the study of Lagrangians of the form
\begin{equation}\label{eq_intro_eqmodel_2}
E(u)=\int_{\Omega} L\left(x, u(x), Du(x),..., D^ku(x)\right)dx,
\end{equation}
where $k\in \mathbb{N}$ and $u$ is a function in an appropriate function space defined on a domain $\Omega\subset \mathbb{R}^n$ and taking values in $\mathbb{R}^m$, for some $n, m\in\mathbb{N}$.
In its original formulation, published by E. Noether in 1918 (see \cite{Noether}), the Theorem states that if the energy (\ref{eq_intro_eqmodel_2}) is invariant under the action of a continuous group of transformations of the domain and the target of the admissible functions, generated by $r$ independent parameters, then there exist $r$ linearly independent combinations of the Lagrange expressions of $E$ which are equal to divergences. Here by Lagrange expressions are meant the variations of the energy induced by variations $\delta u_i$, for $i\in \lbrace 1,...,m\rbrace$, i.e. the terms appearing in the Euler-Lagrange equation of the corresponding variational problem $\delta E=0$. In particular, if $u$ is a critical point of $E$, the Theorem yields $r$ divergence-free quantities involving $u$.
For the sake of simplicity, in this document we consider separately transformations acting on the target and transformations acting on the domain, and we neglect transformations acting non-trivially on both.\\

\paragraph{Symmetries in the target}
We first considered the case of Lagrangians exhibiting invariance under a group of transformations acting on the target of the admissible functions. In this case we adapted to the nonlocal framework the proof of Noether Theorem as reported in Theorem 1.3.1 in \cite{Helein} (Theorem \ref{thm_noether_thm} in this document), substituting gradients and divergences by the following fractional counterparts, introduced by K. Mazowiecka and A. Schikorra in \cite{divcurl}.
\begin{defn}\label{def_introduction_fractional_gradients_divergences}\noindent
\begin{enumerate}
\item Let $s\in [0,1]$. Let $f:\mathbb{R}^n\to\mathbb{R}^n$. For any $(x,y)\in \mathbb{R}^n\times\mathbb{R}^n$ let
\begin{equation}
\text{d}_sf(x,y):=\frac{f(x)-f(y)}{\lvert x-y\rvert^s};
\end{equation}
$\text{d}_s f$ is called the \textbf{$s$-gradient of $f$}.
\item Let $s\in (0,1)$. Let $F: \mathbb{R}^n\times\mathbb{R}^n\to\mathbb{R}$ so that that for any $\phi\in C^\infty_c(\mathbb{R}^n)$,
\begin{equation}
\int_{\mathbb{R}^n}\int_{\mathbb{R}^n}\left\lvert F(x,y)\text{d}_s\phi(x,y)\right\rvert\frac{dxdy}{\lvert x-y\rvert^n}<\infty.
\end{equation}
The \textbf{$s$-divergence} of $F$ is the distribution defined by
\begin{equation}
\text{div}_sF[\phi]:=\int_{\mathbb{R}^n}\int_{\mathbb{R}^n}F(x,y)\text{d}_s\phi(x,y)\frac{dxdy}{\lvert x-y\rvert^n}\quad\text{for any }\phi\in C^\infty_c(\mathbb{R}^n)
\end{equation}
for any $\phi\in C^\infty_c(\mathbb{R}^n)$.
\end{enumerate}
\end{defn}
For a discussion of the previous definitions and some of their properties, see Section \ref{ssec: Fractional gradients and fractional divergences}.
One of the main difficulties arising in the process of adapting the arguments known for the local problems to the present setting is given, of course, by the nonlocal nature of the operators. By means of the previous definitions, as well as the pointwise characterization of the fractional Laplacians in terms of integrals (Equation (\ref{eq_prel_pointwise_definition_fraclap})), we were often able to reformulate nonlocal expressions as local ones for an additional variable.\newline 
We obtained the following result (Theorem \ref{prop_noether_thm_variaton_target_cor}):
\begin{thm}\label{prop_1}
Let $L:\mathbb{R}^n\times\mathbb{R}^m\times \mathbb{R}^m\to \mathbb{R}$ be a function in $C^2_{loc}\cap\dot{W}^{2,\infty}$ and assume that there exist constants $C$, $a$, $b$, $c$ such that $\frac{n}{n+\frac{1}{2}}<a\leq b\leq 2\leq c$ and
\begin{equation}\label{eq_prop_noeth_thm_vartar_condition_on_L_introduction}
\left\lvert L(x,y,p)-L(x,y,p')\right\rvert\leq C\left( \lvert y-y'\rvert^2+\lvert y-y'\rvert^c+\lvert p-p'\rvert^a+\lvert p-p'\rvert^b \right)
\end{equation}
for any $x\in\mathbb{R}^n$, $y,y', p,p'\in \mathbb{R}^m$. Let $\mathscr{M}$ be a smooth manifold embedded in $\mathbb{R}^m$.
For $u\in H^\frac{1}{2}(\mathbb{R}^n,\mathscr{M})$ let
\begin{equation}
E(u):=\int_{\mathbb{R}^n}L\left(x,u(x),(-\Delta)^\frac{1}{4}u(x)\right)dx
\end{equation}
whenever the integral is well defined.
Let $X$ be a $C^4$ vector field on $\mathscr{M}$ which is an infinitesimal symmetry for $L$, in the sense that for its flow $\phi_t$, for $t$ in a neighbourhood of $0$
\begin{equation}
L\left(x,\phi_t\circ u(x), (-\Delta)^\frac{1}{4}(\phi_t\circ u)(x)\right)=L\left(x,u(x),(-\Delta)^\frac{1}{4}u(x)\right)
\end{equation}
for a.e. $x\in \mathbb{R}^n$.
Assume that the function
\begin{equation}
L_u^i:\mathbb{R}^n\to \mathbb{R}, \quad x\mapsto D_{p_i}L\left(x,u(x),(-\Delta)^\frac{1}{4}u(x)\right)
\end{equation}
(where $p_i$ denotes the $(n+i)^{th}$ argument of $L$) belongs to $L^2(\mathbb{R}^n)$ for any $i\in \lbrace 1,...,n\rbrace$. Then for any critical point $u$ of $E$ in $H^\frac{1}{2}(\mathbb{R}^n, \mathscr{M})$ there holds
\begin{equation}
\text{div}_\frac{1}{2}\left(L_u^i(x)X^i(u(y))-L_u^i(y)X^i(u(x))\right)=0.
\end{equation}
(we use implicit summation over repeated indices).
\end{thm}

In this framework, a model problem is given by the half Dirichlet energy
\begin{equation}\label{eq_intro_half_Dirichlet_energy}
E(u)=\int_{\mathbb{R}}\left\lvert (-\Delta)^\frac{1}{4}u(x)\right\rvert^2 dx
\end{equation}
for functions $u\in H^\frac{1}{2}(\mathbb{R}, \mathbb{S}^m)$ taking values in the $m$-dimensional sphere $\mathbb{S}^m$ for some $m\in\mathbb{N}$. In fact the energy $E$ is invariant under rotations in the target, thus Theorem \ref{prop_1} yields the following result.
\begin{lem}\label{lem_intro_div12Omega}
Let $u\in H^\frac{1}{2}(\mathbb{R}^n, \mathbb{S}^{m})$ be a critical point of the energy $E$ in $H^\frac{1}{2}(\mathbb{R}^n, \mathbb{S}^{m})$. For any $i, k\in \lbrace 1,..., m+1\rbrace$, $i\neq k$ let
\begin{equation}
\begin{split}
\Omega_{ik}(x,y):=& u^k(x)\text{d}_\frac{1}{2}u^i(x,y)-u^i(x)\text{d}_\frac{1}{2}u^k(x,y)
\end{split}
\end{equation}
for any $x,y\in \mathbb{R}^n$. Then
\begin{equation}\label{eq_intro_div_1_2_of_Omega_vanishes2}
\text{div}_\frac{1}{2} \left(\Omega_{ik}\right)=0.
\end{equation}
\end{lem}
The proof of Lemma \ref{lem_intro_div12Omega} follows indirectly from Theorem \ref{prop_1} (see Lemmas \ref{lem_1_2_unexpected_divergence_free_quantity} and \ref{lem_1_2_div_of_Omega_vanishes} or \ref{lem_alternative_argument_div12_Omega_ik=0}), or directly from the same argument used to prove it (see Lemma \ref{lem_direct_proof_divOmega=0}).\\
Equation (\ref{eq_intro_div_1_2_of_Omega_vanishes2}) has been used by K. Mazowiecka and A. Schikorra (who derived it by direct calculations) in combination with an analogous of Wente's Lemma for fractional quantities (Corollary \ref{cor_fractional_Wente_Lemma}) to obtain an alternative proof of the following regularity result (Theorem 3.2 in \cite{divcurl}), similar to the one developed by F. H\'{e}lein in \cite{Heleinarticle} for the corresponding local problem.
\begin{thm}\label{thm_intro_continuity_critical_pts} (F. Da Lio, T. Rivi\`{e}re, \cite{3Termscomm})
Let $m\in \mathbb{N}$, let $u\in \dot{H}^\frac{1}{2}(\mathbb{R},\mathbb{S}^m)$ be a critical point of the energy (\ref{eq_definition_fractional_Dirichlet_energy_2}), then $u$ is continuous.
\end{thm}
This example illustrates how, also in the context of nonlocal operators, Noether Theorem finds applications in the study of geometric equations.

\paragraph{Symmetries in the domain}
Next we considered the case of Lagrangians exhibiting invariance under a group of transformations acting on the domain. In this case we adapted to the nonlocal framework the proof of Theorem 2.1 in \cite{DaLio} (Theorem \ref{thm_Noether_thm_vardom} in this document). We chose to work with this article as it draws interesting connections between Noether's Thoerem and Pohozaev-type identities\footnote{By this are meant "identities produced by the conformal invariance of the highest order derivative term in the Lagrangian from which the Euler Lagrange is issued".}. In this article, Noether Theorem is formulated for functions defined on the disc $D^2$, thus we tried to find a nonlocal parallel for functions defined on the boundary $\partial D^2$ of the disc, with the idea of regarding the nonlocal problem as a sort of "trace" of the local problem for functions on $D^2$.
In this context we used the following definition for $\frac{1}{4}$-fractional divergence:
\begin{defn}
Let $a\in L^{2,1}(\partial D^2)$, $b\in L^2(\partial D^2)$. Let
\begin{equation}\label{eq_definition_introduction_fracdiv_circle_mapF}
F: \partial D^2\times\partial D^2\to \mathbb{R}, \quad (x,y)\mapsto a(x)\cdot b(y),
\end{equation}
then the \textbf{$\frac{1}{4}$-fractional divergence} of $F$ is the distribution defined by
\begin{equation}\label{eq_definition_introduction_fractional_divergence_circle}
\text{div}_\frac{1}{4}[a(x)\cdot b(y)][\phi]=\int_{\partial D^2}b(y)\cdot\int_{\partial D^2}a(x)(\phi(x)-\phi(y))K^\frac{1}{4}(x-y)dydx,
\end{equation}
where $K^\frac{1}{4}$ denotes the Kernel of the $\frac{1}{4}$-fractional Laplacian on $\partial D^2$, defined in (\ref{def_Kernel_Ks_0_12}).
\end{defn}
By computation (\ref{eq_proof_lem_l2_norm_bounded_to_show_integrable_ae}), for any $\phi\in C^\infty(\partial D^2)$ the integral in (\ref{eq_definition_introduction_fractional_divergence_circle}) converges absolutely, and by Lemma (\ref{lemma_approximation_identity_14}) (and the discussion thereafter) there holds
\begin{equation}
\begin{split}
\left\lvert \text{div}_\frac{1}{4}\left( a(x)\cdot b(y)\right)[\phi]\right\rvert\leq  C[\phi]_{\mathbb{A}^1}\lVert a\rVert_{H^{-\frac{1}{2}}} \lVert b\rVert_{L^2}
\end{split}
\end{equation}
for some constant $C$, where
\begin{equation}
[\phi]_{\mathbb{A}^1}:=\sum_{\substack{k\in\mathbb{Z}\\ k\neq 0}}\lvert k\rvert\lvert \widehat{\phi}(k)\rvert.
\end{equation}
Therefore the definition of $\frac{1}{4}$-divergence can be extended by continuity to functions of the form (\ref{eq_definition_introduction_fracdiv_circle_mapF}), where $a\in H^{-\frac{1}{2}}(\partial D^2)$, $b\in L^2(\partial D^2)$\footnote{It is possible to define the $s$-divergence of functions $F$ of different forms and for different values of $s$, see for instance (\ref{eq_definition_frac_div_14}) and (\ref{eq_discussion_extension_definition_fracdiv_12}).}.\newline
We can now proceed to state the Noether theorem for variations in the domain (see Theorem \ref{prop_final_version_Noether_thm_vardom} and Lemma \ref{lem_condition_for_prop_def_Noeth_vardom}):
\begin{thm}\label{prop_2}
Let $L\in C^2(\mathbb{R}^m\times\mathbb{R}^m, \mathbb{R})$ so that for any $x,p\in\mathbb{R}^m$
\begin{equation}
\lvert L(x,p)\rvert\leq C(1+\lvert p\rvert^2),\quad\lvert D_pL(x,p)\rvert\leq C(1+\lvert p\rvert)
\end{equation}
for some constant $C>0$. For any $u\in H^\frac{1}{2}(\partial D^2,\mathbb{R}^m)$ we set
\begin{equation}
E(u):=\int_{\partial D^2}L\left(u(x),(-\Delta)^\frac{1}{4}u(x)\right)dx.
\end{equation}
Let $X$ be a smooth vector field on $\partial D^2$ and assume that its flow $\phi_t$ satisfies
\begin{equation}\label{eq_prop_2_in_intro_invariance_flow}
L\left(u\circ\phi_t(x),(-\Delta)^\frac{1}{4}(u\circ\phi_t)(x)\right)=L\left(u(\phi_t(x)),(-\Delta)^\frac{1}{4}u(\phi_t(x))\right)\phi_t'(x)
\end{equation}
for a.e. $x\in \partial D^2$, for $t$ in a neighbourhood of the origin.
Assume that $u\in H^\frac{1}{2}(\partial D^2, \mathbb{R}^m)$ satisfies the following stationarity equation
\begin{equation}
\frac{d}{dx}L\left(u, (-\Delta)^\frac{1}{4}u\right)X+\text{div}_\frac{1}{4}\left( D_pL\left(u(x), (-\Delta)^\frac{1}{4}u(x) \right)u'(y)\right)X=0
\end{equation}
as distributions.
Then, in the sense of distributions,
\begin{equation}
\frac{d}{dx}\left[L\left(u, (-\Delta)^\frac{1}{4}u\right)X\right](x)+\text{div}_\frac{1}{4}\left[D_pL\left(u(x), (-\Delta)^\frac{1}{4}u(x)\right)u'(y)X(y)\right]=0.
\end{equation}
\end{thm}
Here again a model problem is given by the half Dirichlet energy
\begin{equation}\label{eq_intro_half_Dirichlet_energy_for_the_circle}
E(u)=\int_{\partial D^2}\left\lvert(-\Delta)^\frac{1}{4}u(x)\right\rvert^2 dx
\end{equation}
for functions $u\in H^\frac{1}{2}(\partial D^2, \mathbb{R}^m)$, where the flow $\phi_t$ in (\ref{eq_prop_2_in_intro_invariance_flow}) consists of traces of conformal diffeomorphisms of $D^2$.\\
As an application of the idea of Noether theorem we will see how the invariance properties of the half Dirichlet energy can be exploited to derive Pohozaev identities (Proposition \ref{prop_Pohozaev_identities_on_the_circle_intro}):
\begin{prop}\label{prop_Pohozaev_identities_on_the_circle_intro}
Let $u\in H^1(\partial D^2, \mathbb{R})$. Then, for any vector field $X$ on $\partial D^2$ whose flow is a family of traces of automorphisms of $D^2$, there holds
\begin{equation}
u\frac{d}{dx}\left[(-\Delta)^\frac{1}{2}uX\right]=u(-\Delta)^\frac{1}{2}\left(u'X\right)
\end{equation}
in the sense of distributions. In particular
\begin{equation}
\int_{\partial D^2}u'(x)(-\Delta)^\frac{1}{2}u(x)dx=0,
\end{equation}
and for any $\delta\in \partial D^2$
\begin{equation}
\int_{\partial D^2}u'(x)(-\Delta)^\frac{1}{2}u(x)\sin(x-\delta)dx=0.
\end{equation}
\end{prop}
After that, in Section \ref{ssec:Functions on the real line} we will give a brief overview of corresponding results for functions defined on the real line, in particular of Noether's Theorem for variations in the domain (Proposition \ref{prop_first_noeth_vardom_14_whole_R}) and of Pohozaev identities (Proposition \ref{prop_Pohozaev_identity_for_the_whole_R}).\\
In Section \ref{ssec: Stationary points of the energy} will then return to the study of the half Dirichlet energy for functions defined on the circle.
We will see that for sufficiently regular functions $u$ on $\partial D^2$, $u$ is a stationary point of the half Dirichlet energy if and only if 
\begin{equation}\label{eq_introduction_discussion_first_criterion_stationarity}
(-\Delta)^\frac{1}{2}u(x)\cdot \partial_\theta u(x)=0
\end{equation}
for any $x\in\partial D^2$ (see Equations (\ref{eq_discussion_Noether_formula_for_rotations_from_14}) (\ref{eq_discussion_Noether_formula_for_rotations_from_12})) or equivalently, if we set $v:=(-\Delta)^\frac{1}{2}u$ and $v_+:=\sum_{k\in \mathbb{Z}_{>0}} \widehat{v}(k)e^{ik\cdot}$, if and only if
\begin{equation}\label{eq_intro_definition_frac_Hopf_for_regular_fcts}
v_+(x)\cdot_{\mathbb{C}^m}v_+(x)=0
\end{equation}
for any $x\in\partial D^2$, where $\cdot_{\mathbb{C}^m}$ denotes the $\mathbb{C}$-scalar product in $\mathbb{C}^m$ (see Lemma \ref{lem_indentity_Fourier_coefficients}).
For functions $u$ in $H^\frac{1}{2}(\partial D^2,\mathbb{R}^m)$, the product in (\ref{eq_introduction_discussion_first_criterion_stationarity}) might not even be defined as a distribution. Nevertheless, by means of the techniques developed to adapt Noether Theorem to the nonlocal context, we will see that the product in (\ref{eq_intro_definition_frac_Hopf_for_regular_fcts}) is a well defined distribution, providing simple characterization of stationary points of the half Dirichlet energy in $H^\frac{1}{2}(\partial D^2,\mathbb{R}^m)$ (see Lemma \ref{lem_fractional_Hopf_differential_well_def} and Theorem \ref{lem_equivalent_conditions_vanishing_H(u)}):
\begin{prop}
The map
\begin{equation}
H^\frac{1}{2}(\partial D^2)\to H^{-3}(\partial D^2),\quad u\mapsto \mathscr{H}_\frac{1}{2}(u):=v_+\cdot_{\mathbb{C}^m}v_+
\end{equation}
is well defined. Here $v_+\cdot_{\mathbb{C}^m}v_+$ denotes the distribution whose Fourier coefficients are defined by\footnote{Notice that if $u$ is sufficiently regular, this definition coincides with the actual $\mathbb{C}$-scalar product of $v_+$, where $v_+$ is defined as above.}
\begin{equation}
\mathscr{F}\left(v_+\cdot_{\mathbb{C}^m}v_+\right)(k):=\begin{cases}
\sum_{\substack{a,b\in\mathbb{N}_{>0},\\ a+b=k}} a\widehat{u}(a)\cdot_{\mathbb{C}^m} b \widehat{u}(b)& \text{ if }k>0\\
0 &\text{ if }k\leq 0
\end{cases}
\end{equation}
Moreover, $u\in H^\frac{1}{2}(\partial D^2)$ is a stationary point of the half Dirichlet energy if and only if $\mathscr{H}_\frac{1}{2}(u)=0$ as distribution.
\end{prop}
For any $u\in H^\frac{1}{2}(\partial D^2)$ we call the distribution $\mathscr{H}_\frac{1}{2}(u)$ the \textbf{$\frac{1}{2}$-fractional Hopf differential} of $u$.\\
The characterization in terms of the $\frac{1}{2}$-fractional Hopf differential allows to deduce in a simple way several properties of stationary points of the half Dirichlet energy in $H^\frac{1}{2}(\partial D^2,\mathbb{R}^m)$; in particular we will obtain a characterization of stationary points in terms of their Fourier coefficients and one in terms of their harmonic extensions in $D^2$ (following the idea of regarding the present problem as a "trace" of the local problem in $D^2$). We summarize these results in the following Theorem (see Lemma \ref{lem_indentity_Fourier_coefficients} and Theorem \ref{lem_equivalent_conditions_vanishing_H(u)}):

\begin{thm}
Let $u\in H^\frac{1}{2}(\partial D^2,\mathbb{R}^m)$ and let $\tilde{u}$ be its harmonic extension in $D^2$.
Then the following statements are equivalent:
\begin{enumerate}
\item $u$ is a stationary point of the half Dirichlet energy,
\item for any $k\in \mathbb{N}_{>0}$
\begin{equation}
\sum_{\substack{n,m\in \mathbb{N}_{>0},\\ n+m=k}} n\widehat{u}(n)\cdot_{\mathbb{C}^m} m \widehat{u}(m)=0,
\end{equation}
\item $\frac{1}{r}\lvert\partial_\theta \tilde{u}\rvert=\lvert \partial_r\tilde{u}\rvert$ in $D^2$,
\item $\partial_\theta \tilde{u}\cdot \partial_r \tilde{u}=0$ in $D^2$.
\end{enumerate}
In particular, $u$ is a stationary point of the half Dirichlet energy if and only if $\tilde{u}$ is conformal in $D^2$.
\end{thm}
We remark that the the implication $1.\Rightarrow 2.$ already appeared in \cite{DaLio} and was also obtained in the work of preparation of \cite{Cabre} (see Remark \ref{rem_DaLio_Cabre}), while the fact that a function $u\in H^\frac{1}{2}(\partial D^2,\mathbb{R}^m)$ is a stationary point of $E$ if and only if its harmonic extension in $D^2$ is conformal was first obtained in \cite{CompBubble} and \cite{MillotSire}.\\

In the last section of this document, we will show how the invariance properties of the half Dirichlet energy on $\mathbb{R}$ can be used to prove identities analogous to the classical Pohozaev identity (Equation (\ref{eq_Pohozaev_local})). From the invariance of the half Dirichlet energy on $\mathbb{R}$ under translations and dilations in the domain we will deduce that if a function $u:[a,b]\to \mathbb{R}$ is sufficiently regular (where $[a,b]$ is a bounded interval of $\mathbb{R}$), then the integrals
\begin{equation}\label{eq_discussion_introduction_integral_to_evaluate_for_Pohozaev}
\int_a^b u'(x)(-\Delta)^\frac{1}{2}u(x)dx,\quad \int_a^b xu'(x)(-\Delta)^\frac{1}{2}u(x)dx
\end{equation}
only depend on the values of $u$ and $(-\Delta)^\frac{1}{2}u$ in an arbitrary small neighbourhood of $a$ and $b$. Therefore we will first study the asymptotic behaviour of $u(x)$ and $(-\Delta)^\frac{1}{2}u(x)$ as $x$ approaches $a$ and $b$ (Section \ref{ssec:Asymptotics}), and we will then use it to obtain an explicit expression for the integrals in (\ref{eq_discussion_introduction_integral_to_evaluate_for_Pohozaev}).
We will formulate the results for solutions $u\in H^\frac{1}{2}(\mathbb{R}, \mathbb{R})$ of the problem
\begin{equation}\label{eq_discussion_introduction_Cauchy_problem_Pohozaev_domains}
\begin{cases}
(-\Delta)^\frac{1}{2} u=f(u)&\text{ weakly in }(a,b)\\
u=0&\text{ in }(a,b)^c,
\end{cases}
\end{equation}
where  $f:\mathbb{R}\to \mathbb{R}$ is a Lipschitz function. This will ensure that $u$ has all the necessary properties to carry out the argument outlined above in a rigorous way, and will also allow to compare directly the present results to the classical Pohozaev identity and the existing literature for nonlocal results. Nevertheless we remark that these argument do not use directly the fact that $u$ is a solution of (\ref{eq_discussion_introduction_Cauchy_problem_Pohozaev_domains}), therefore the results are valid for more general functions (see Remark \ref{rem_regularity_Pohozaev}).
We obtained the following explicit expressions for the integrals in (\ref{eq_discussion_introduction_integral_to_evaluate_for_Pohozaev}) (Theorem \ref{prop_Pohozaev_identity_for_open_bounded_intervals_final}).
\begin{thm}\label{prop_3}
Assume that $u\in H^\frac{1}{2}(\mathbb{R}, \mathbb{R})$ is a solution of (\ref{eq_discussion_introduction_Cauchy_problem_Pohozaev_domains}). Moreover assume that $u^2\in C^2([a,b])$.
Then
\begin{equation}\label{eq_prop_statement_Pohozaev_proof_asymptotic_translations_into}
\int_a^bu'(x)(-\Delta)^\frac{1}{2}u(x)dx=\frac{\pi}{8}\left[ \lim_{x\to a^+}\frac{u^2(x)}{x-a}-\lim_{x\to b^-}\frac{u^2(x)}{b-x}\right]
\end{equation}
and
\begin{equation}\label{eq_prop_statement_Pohozaev_proof_asymptotic_intro}
\int_a^b u'(x)x(-\Delta)^\frac{1}{2}u(x)dx=\frac{\pi}{8}\left[\lim_{x\to a^+}\frac{u^2(x)}{x-a}a-\lim_{x\to b^-}\frac{u^2(x)}{b-x}b\right].
\end{equation}
\end{thm}
We remark that the second part of Theorem (\ref{prop_3}) was proved first by X. Ros-Oton and J. Serra in \cite{RosOtonSerra} (Theorem 1.1 in the reference, Theorem \ref{thm_Pohozaev_identity_RosOton_Serra} in this document) with a different argument, in the much more general setting of domains $\Omega$ in $\mathbb{R}^n$, with $(-\Delta)^s$ for $s\in (0,1)$ instead of $(-\Delta)^\frac{1}{2}$, for $f$ locally Lipschitz and without assumption on $u^2$.\\

This document is the result of a Master's thesis written at the University of Zurich under the supervision of prof. Tristan Rivi\`{e}re (ETH Zurich) and prof. Xavier Ros-Oton (UZH). In some parts, this document doesn't follow the form of an expository paper but rather the form of a journal where I registered the progresses I made, with many remarks and many arguments one might find trivial, but that I wanted to mark down. Sometimes I gave more than one proof for the same statement, maybe repeating some of the arguments, sometimes I wrote a Lemma that I didn't use later, but I think this form better reflects the work I did in the past months.\\

I would like to sincerely thank prof. Tristan Rivi\`{e}re for introducing me to this subject, for all the time he dedicated to me, for his brilliant ideas and his enthusiasm.
I would also like to thank prof. Xavier Ros-Oton for accepting to be internal advisor for this thesis, and for his valuable advices, and prof. Francesca Da Lio, for her corrections and useful remarks.
Finally I would like to thank Alessandro Pigati for all the time he spent discussing with me the problems arising during this work and many others, I'm really grateful for his help and his patience towards me.

\section{Preliminaries}
In this section, we recall some important definitions and we fix some conventions that will be used throughout this document.

\subsection{On Euclidean spaces}
For the following, let $n$ be a positive natural number. We let $\mathscr{S}(\mathbb{R}^n)$ denote the \textbf{Schwartz space} on $\mathbb{R}^n$ (for either scalar valued function or vector valued functions) (see \cite{Grafakos1}, Section 2.2.1) and we let $\mathscr{S}'(\mathbb{R}^n)$ denote its topological dual, i.e. the \textbf{space of tempered distributions}.
For any $\phi\in \mathscr{S}(\mathbb{R}^n)$ we define the \textbf{Fourier transform} of $\phi$ to be
\begin{equation}
\widehat{\phi}(\xi)=\frac{1}{\left(\sqrt{2\pi}\right)^n}\int_{\mathbb{R}^n}e^{-ix\cdot \xi}\phi(x)dx.
\end{equation}
for any $\xi\in\mathbb{R}$,
and the \textbf{inverse Fourier transform} of $\phi$ to be
\begin{equation}
\check{\phi}(x)=\frac{1}{\left(\sqrt{2\pi}\right)^n}\int_{\mathbb{R}^n}e^{ix\cdot \xi}\phi(\xi)d\xi.
\end{equation}
for any $x\in\mathbb{R}$. We will sometimes denote the Fourier transform of $\phi$ and the inverse transform of $\phi$ respectively by $\mathscr{F}[\phi]$ and $\mathscr{F}^{-1}[\phi]$.\\
As usual, we extend the definitions of Fourier transform and inverse Fourier transform to temperate distribution as follows:
for any temperate distribution $f$ and any Schwartz function $\phi$ we define
\begin{equation}
\langle \widehat{f},\phi\rangle=\langle f,\widehat{\phi}\rangle
\end{equation}
and
\begin{equation}
\langle \check{f},\phi\rangle=\langle f,\check{\phi}\rangle,
\end{equation}
where $\langle\cdot,\cdot\rangle$ denote the action of a temperate distribution on a Schwartz function.\\
With these conventions, \textbf{Parseval's Relation} takes the following form:
\begin{equation}
\int_{\mathbb{R}^n} f(x)\overline{h}(x)dx=\int_{\mathbb{R}^n}\widehat{f}(\xi)\overline{\check{h}}(\xi)d\xi
\end{equation}
for any $f, h\in L^2(\mathbb{R}^n)$, while \textbf{Plancherel's Identity} takes the form
\begin{equation}
\lVert f\rVert_{L^2}=\lVert \widehat{f}\rVert_{L^2}=\lVert \check{f}\rVert_{L^2}
\end{equation}
for any $f\in L^2(\mathbb{R}^n)$ (for the proofs, see \cite{Grafakos1}, Theorem 2.2.14).\\
We also recall the definition of Hilbert transform for functions defined on $\mathbb{R}$: let $p\in [1,\infty)$ and let $p\in L^p$.
Then the \textbf{Hilbert transform} of $f$ is given by
\begin{equation}
Hf(x)=\frac{1}{\pi} PV\int_\mathbb{R}\frac{f(y)}{x-y}dy
\end{equation}
for a.e. $x\in \mathbb{R}$ (in fact, the limit in the definition of principal value exists for a.e. $x\in\mathbb{R}$, see \cite{Grafakos1}, Theorem 5.1.3).
For a Schwartz function $\phi$, the Hilbert transform can also be defined in terms of its Fourier transform: for any $\xi\in \mathbb{R}$
\begin{equation}
\widehat{H\phi}(\xi)=-i\text{sgn}(\xi)\widehat{\phi}(\xi)
\end{equation}
(see \cite{Grafakos1}, pp. 316-317). This characterization can be extended to any $f\in L^2$ by Plancherel identity.
The Hilbert transform defines an operator $(p,p)$ for any $p\in (1,\infty)$ (see \cite{Grafakos1}, Theorem 5.1.7).\\
We now proceed to the definitions of Sobolev spaces and the fractional Laplacian.
Let $n\in\mathbb{N}_{>0}$, $s\in (0,1)$, $p\in [1,\infty)$, then 
\begin{equation}
W^{s,p}(\mathbb{R}^n):=\left\lbrace u\in L^p(\mathbb{R}^n)\bigg\vert [u]_{\dot{W}^{s,p}(\mathbb{R}^n)}:=\left(\int_{\mathbb{R}^n}\int_{\mathbb{R}^n}\frac{\lvert u(x)-u(y)\rvert^p}{\lvert x-y\rvert^{n+sp}} dxdy\right)^\frac{1}{p}<\infty \right\rbrace
\end{equation}
The $W^{s,p}(\mathbb{R}^n)$ are called \textbf{fractional Sobolev spaces}, they are Banach spaces with respect to the norm
\begin{equation}
\lVert u\rVert_{W^{s,p}(\mathbb{R}^n)}=\lVert u\rVert_{L^p(\mathbb{R}^n)}+[u]_{\dot{W}^{s,p}(\mathbb{R}^n)},
\end{equation}
while $[\cdot]_{W^{s,p}(\mathbb{R}^n)}$ is called \textbf{Gagliardo seminorm} (see \cite{Hitchhiker}, Section 2).
If in the definition of $W^{s,p}(\mathbb{R}^n)$ we drop the requirement that $u$ belongs to $L^p$, we obtain the \textbf{homogeneous Sobolev space}
\begin{equation}\label{def_hom_sobolev_w}
\dot{W}^{s,p}(\mathbb{R}^n)=\left\lbrace u\in L^1_{loc}\bigg\vert [u]_{\dot{W}^{s,p}}<\infty\right\rbrace.
\end{equation}
This is a vector space endowed with the seminorm $[\cdot]_{W^{s,p}(\mathbb{R}^n)}$. Clearly, $W^{s,p}(\mathbb{R}^n)\subset\dot{W^{s,p}}(\mathbb{R}^n)$. As usual we can define an equivalence relation on $\dot{W}^{s,p}(\mathbb{R}^n)$ by
\begin{equation}
u\sim v\Longleftrightarrow [u-v]_{\dot{W}^{s,p}(\mathbb{R}^n)}=0.
\end{equation}
for any $u,v\in \dot{W}^{s,p}(\mathbb{R}^n)$.
The set of the $\sim$-equivalence classes has then an induced normed space structure, where the norm is given by $\lVert [u]\rVert=[u]_{\dot{H}^s(\mathbb{R}^n)}$ for any equivalence class $[u]$. By an abuse of notation, we will denote also the normed space of equivalence classes by $\dot{H}^s(\mathbb{R}^n)$.\\
When $s\in (0,1)$, $p=2$, we can give an equivalent definition of Sobolev spaces by means of the Fourier transform.
Indeed, for $n\in\mathbb{N}$, $s\in \mathbb{R}$ and $p\in [1,\infty)$ we define
\begin{equation}
H^s(\mathbb{R}^n)=\left\lbrace u\in \mathscr{S}'(\mathbb{R}^n)\bigg\vert \lVert u\rVert_{H^s(\mathbb{R}^n)}:=\left\lVert \mathscr{F}^{-1}\left((1+\lvert \xi\rvert^2)^\frac{s}{2}\widehat{u}(\xi)\right) \right\rVert_{L^2(\mathbb{R}^n)}<\infty\right\rbrace.
\end{equation}
Then $H^s$ is a normed space with norm $\lVert\cdot\rVert_{H^s(\mathbb{R})}$.
By Plancherel's identity, in order to prove that the previous definitions of non-homogeneous Sobolev spaces coincide when $s\in (0,1)$, $p=2$, it is enough to show that the seminorms are equivalent. In fact they only differ by a multiplicative constant (see \cite{Hitchhiker}, Proposition 3.4).
Similarly, we can define the homogeneous Sobolev spaces as follows.
Let $\mathscr{S}'_{(0)}$ denote the set of tempered distributions of order $0$. Then
\begin{equation}\label{def_hom_sobolev_h}
\dot{H}^s(\mathbb{R}^n)=\left\lbrace u\in \mathscr{S}'_{(0)}(\mathbb{R}^n)\bigg\vert \widehat{u}\in L^2_{loc}(\mathbb{R}\smallsetminus\lbrace0\rbrace), [u]_{H^s(\mathbb{R}^n)}:=\left\lVert \lvert \xi\rvert^s\widehat{u}(\xi) \right\rVert_{L^2(\mathbb{R}^n)}<\infty\right\rbrace.
\end{equation}
As before, $\dot{H^s}$ is a vector space endowed with the seminorm $[\cdot]_{H^s(\mathbb{R}^n)}$, and if $0\leq s$, $H^s(\mathbb{R}^n)\subset \dot{H}^s(\mathbb{R}^n)$, and we can make $\dot{H}^s$ into a normed vector space by identifying elements whose difference has seminorm equal to $0$.
To prove that the previous definitions of non-homogeneous Sobolev spaces coincide when $s\in (0,1)$, $p=2$, one can use an argument similar to the one used for the non-homogeneous case: one can check that the seminorms are equivalents, therefore $\dot{H^s}(\mathbb{R}^n)$ and $\dot{W}^{s,2}(\mathbb{R}^n)$ coincide as sets, and the induced normed vector spaces are isomorphic (see Lemma \ref{lem_equivalent_definitions_homogeneous_sobolev}).
From definition (\ref{def_hom_sobolev_h}), we see that two elements of $\dot{H}^s(\mathbb{R}^n)$ (or $\dot{W}^{s,2}(\mathbb{R}^n)$) lie in the same equivalence classes if and only if their difference is a constant. Indeed, $u,v\in \dot{H}^s(\mathbb{R}^n)$ lie in the same equivalence class if and only if
\begin{equation}\label{eq_intro_difference_is_constant}
\int_{\mathbb{R}^n} \lvert \xi\rvert^{2s}\lvert \mathscr{F}[u-v](\xi)\rvert^2d\xi=0.
\end{equation}
Now by definition of $\dot{H}^s(\mathbb{R}^n)$, $\mathscr{F}[u-v]$ can be written as the sum of a measurable function $f\in L^2_{loc}(\mathbb{R}^n\smallsetminus\lbrace0\rbrace)$ and a tempered distribution $g$ of order $0$ supported in $\lbrace 0\rbrace$, i.e. a multiple of a Dirac delta in $0$. Equation (\ref{eq_intro_difference_is_constant}) implies that $f\equiv 0$, therefore $\mathscr{F}[u-v]$ is the multiple of a Dirac delta in $0$. It follows that $u-v$ is a constant function.

We also remark that for any $s\in \mathbb{R}$, we may identify the topological dual space $H^s(\mathbb{R}^n)^\ast$ with $H^{-s}(\mathbb{R}^n)$. Indeed, by definition of $H^s(\mathbb{R}^n)$, the function
\begin{equation}
\mathscr{I}_s: H^s(\mathbb{R}^n)\to L^2(\mathbb{R}^n), \quad u\mapsto \mathscr{F}^{-1}\left[\left(1+\lvert \xi\rvert^2\right)^{\frac{s}{2}}\widehat{u}\right].
\end{equation}
is an isometric isomorphism between vector spaces.
Now, for any functional $F\in H^s(\mathbb{R}^n)^\ast$, any $\phi\in\mathscr{S}(\mathbb{R}^n)$,
\begin{equation}
\left\langle F, \phi\right\rangle=\left\langle \check{F}, \widehat{\phi}\right\rangle=\left\langle \left(1+\lvert \xi\rvert^2\right)^{-\frac{s}{2}}\check{F}, \left(1+\lvert \xi\rvert^2\right)^{\frac{s}{2}}\widehat{\phi} \right\rangle.
\end{equation}
Therefore a temperate distribution $F$ belongs to $H^s(\mathbb{R}^n)^\ast$ if and only if \newline $\left(1+\lvert \xi\rvert^2\right)^{-\frac{s}{2}}\widehat{F}$ belongs to $L^2(\mathbb{R}^n)$, and in this case
\begin{equation}
\lVert F\rVert_{{H^s}^\ast}=\left\lVert\left(1+\lvert \xi\rvert^2\right)^{-\frac{s}{2}}\widehat{F}  \right\rVert_{L^2}=\lVert F\rVert_{H^{-s}}.
\end{equation}
Finally, for any smooth manifold $\mathscr{M}$ embedded in $\mathbb{R}^m$ for some $m\in \mathbb{N}_{>0}$, for any $s\in \mathbb{R}$ we define
\begin{equation}
H^s(\mathbb{R}^n,\mathscr{M}):=\left\lbrace u: \mathbb{R}^n\to \mathbb{R}^m\bigg\vert u\in H^s(\mathbb{R}^n)\text{ and }u(x)\in \mathscr{M}\text{ for a.e. }x\in\mathbb{R}^n \right\rbrace.
\end{equation}
In a similar way we define $\dot{H}^s(\mathbb{R}^n,\mathscr{M})$, $W^{s,p}(\mathbb{R}^n,\mathscr{M})$ and $\dot{W}^{s,p}(\mathbb{R}^n,\mathscr{M})$.\\
Having defined Sobolev spaces, we can pass to the definition of fractional Laplacians: let $s\in (0,1)$, $\sigma\in \mathbb{R}$. Then for any $u\in \dot{H^\sigma}(\mathbb{R}^n)$, the \textbf{$s$-fractional Laplacian} of $u$ is defined as
\begin{equation}\label{eq_prel_Fourier_definition_fraclap}
(-\Delta)^s u:=\mathscr{F}^{-1}\left[\lvert \xi\rvert^{2s} \widehat{u}(\xi)\right].
\end{equation}
Then, by definition, $u\in \dot{H}^{\sigma-2s}(\mathbb{R}^n)$, and $(-\Delta)^s$ defines a continuous operator from $\dot{H}^\sigma(\mathbb{R}^n)$ to $\dot{H}^{\sigma-2s}(\mathbb{R}^n)$ (the continuity is to be understood with respect to the topologies induced by the norms on the quotient spaces described above).\\
If $s\in (0,1)$, $\phi\in\mathscr{S}$, $(-\Delta)^su$ can be also defined by means of the pointwise formula
\begin{equation}\label{eq_prel_pointwise_definition_fraclap}
(-\Delta)^s\phi(x):=C_{n,s} PV\int_{\mathbb{R}^n}\frac{\phi(x)-\phi(y)}{\lvert x-y\rvert^{n+2s}}dy,
\end{equation}
where $C_{n,s}$ is a constant depending on $n$ and $s$ (see \cite{Hitchhiker}, Proposition 3.3).
In fact, this definition can be extended to functions in $H^{2s}(\mathbb{R}^n)$ (see \cite{ten_eq_def}, Lemmas 4.2 and 5.2) (a.e. in $\mathbb{R}^n$), and therefore also to functions in $\dot{H}^{2s}\cap L^\infty(\mathbb{R}^n)$ (a.e. in $\mathbb{R}^n$). Indeed, if $u\in \dot{H}^{2s}\cap L^\infty(\mathbb{R}^n)$ and $z\in \mathbb{R}^n$, let $\eta$ be a cut-off function equal to $1$ in $B_1(z)$ and equal to $0$ in $B_2(z)^c$. Then $u\eta\in H^{2s}(\mathbb{R}^n)$ (by Lemma \ref{lem_Hs_cap_Linfty_is_an_algebra}) and Equation (\ref{eq_prel_pointwise_definition_fraclap}) holds for $u\eta$. On the other hand, by Lemma \ref{lem_approx_sobolev} there exists a sequence $(\phi_k)_k$ in $\mathscr{S}(\mathbb{R}^n)$ such that $\phi_n\to u(1-\eta)$ in $\dot{H}^{2s}(\mathbb{R}^n)$ as $k\to \infty$, and therefore $(-\Delta)^s\phi_n\to (-\Delta)^s[u(1-\eta)]$ in $L^2(\mathbb{R}^n)$. Thus, by substituting the original sequence with a subsequence if necessary, we have that for a.e. $x\in\mathbb{R}^n$
\begin{equation}
\frac{\phi_k(x)-\phi_k(y)}{\lvert x-y\rvert^{n+2s}}\to \frac{u(1-\eta)(x)-u(1-\eta)(y)}{\lvert x-y\rvert^{n+2s}}\quad \text{for a.e. $y\in\mathbb{R}^n$}
\end{equation}
as $k\to \infty$. In particular for a.e. $x$ in $B_\frac{1}{2}(z)$ by Lebesgue's Dominated convergence Theorem one obtains that 
\begin{equation}\label{eq_discussion_introduction_pointwise_formula_fraclap_for_homogeneous_Sobolev_second_part}
(-\Delta)^\frac{s}{2}[u(1-\eta)](x)=C_{n,s}\int_{\mathbb{R}^n}\frac{u(1-\eta)(x)-u(1-\eta)(y)}{\lvert x-y\rvert^{n+2s}}dy,
\end{equation}
in fact, for our choice of $x$, $(1-\eta)(x)=0$ and the denominator of the integrand in (\ref{eq_discussion_introduction_pointwise_formula_fraclap_for_homogeneous_Sobolev_second_part}) is bounded from below on the support of $(1-\eta)$.
By the linearity of the $s$-fractional Laplacian and of formula on the right hand side of (\ref{eq_prel_pointwise_definition_fraclap}), we conclude that (\ref{eq_prel_pointwise_definition_fraclap}) holds for a.e. $x\in B_\frac{1}{2}(z)$, and thus for a.e. in $\mathbb{R}^n$, even if $u\in\dot{H}^{2s}\cap L^\infty(\mathbb{R}^n)$.\\
We remark that if $n=1$,
\begin{equation}\label{eq_introduction_definition_constant_C1s}
C_1,s=-2^{2s}\pi^{-\frac{1}{2}}\frac{\Gamma(\frac{1}{2}+2s)}{\Gamma(-s)}
\end{equation}
(see \cite{Grafakos2}, p. 9)\footnote{The constant in the reference differs by the multiplicative constant $(2\pi)^{2s}$ due to a different choice of convention for the Fourier transform.}.
In particular, if $s=\frac{1}{2}$, $C_{1,\frac{1}{2}}=\frac{1}{\pi}$.\\
We remark that if $s\in (0, \frac{1}{2})$, the principal value in (\ref{eq_prel_pointwise_definition_fraclap}) can be substituted by a converging integral.\\
Both from expression (\ref{eq_prel_Fourier_definition_fraclap}) and expression (\ref{eq_prel_pointwise_definition_fraclap}) follows that for any $s\in (0,1)$, the $s$-fractional Laplacian is homogeneous of degree $2s$: for any $u\in \dot{H}^{2s}(\mathbb{R}^n)$, for any $\lambda>0$ let $u_\lambda$ be the function defined by $u_\lambda(x)=u(\lambda x)$ for any $x\in\mathbb{R}^n$. Then $u_\lambda\in \dot{H}^{2s}(\mathbb{R}^n)$ and
\begin{equation}
(-\Delta)^s u_\lambda(x)=\lvert\lambda\rvert^{2s}u(\lambda x)
\end{equation}
for a.e. $x\in \mathbb{R}^n$.
We also observe that, by definition, for any $s\in (0,1)$, $(-\Delta)^s$ is self-adjoint, in the sense that for any $u, v\in \mathscr{S}(\mathbb{R}^n, \mathbb{R}^m)$
\begin{equation}\label{eq_rem_eq_def_fraclap_for_Schartz}
\int_{\mathbb{R}^n}(-\Delta)^s u(x)\cdot v(x)dx=\int_{\mathbb{R}^n}u(x)\cdot (-\Delta)^s v(x)dx.
\end{equation}
Next we remark that for any $s\in (0,1)$, for any $u,v\in \dot{H}^s\cap L^\infty(\mathbb{R}^n)$, formula (\ref{eq_prel_pointwise_definition_fraclap}) yields the following "Leibnitz rule": for a.e. $x\in \mathbb{R}^n$
\begin{equation}\label{eq_discussion_introduction_Leibnitz_rule_fraclaps}
\begin{split}
(-\Delta)^s(uv)(x)=&PV\int_{\mathbb{R}^n}\frac{uv(x)-uv(y)}{\lvert x-y\rvert^{n+2s}}dy\\
=&u(x)PV\int_{\mathbb{R}^n}\frac{v(x)-v(y)}{\lvert x-y\rvert^{n+2s}}dy+v(x)\int_{\mathbb{R}^n}\frac{u(x)-u(y)}{\lvert x-y\rvert^{n+2s}}dy\\&-PV\int_{\mathbb{R}^n}\frac{(u(x)-u(y))(v(x)-v(y)}{\lvert x-y\rvert^{n+2s}}dy\\
=&u(x)(-\Delta)^sv(x)+v(x)(-\Delta)^su(x)\\&-\int_{\mathbb{R}^n}\frac{(u(x)-u(y))(v(x)-v(y)}{\lvert x-y\rvert^{n+2s}}dy,
\end{split}
\end{equation}
where in the last step we used the fact that since $u,v\in \dot{H}^{2s}(\mathbb{R}^n)$, the last integral converges absolutely for a.e. $x\in\mathbb{R}^n$.

We conclude this section with some result about the action of fractional Laplacians on functions spaces. First we fix the following convention, which will be used throughout this document: for $s\in \mathbb{R}_{+}\smallsetminus \mathbb{N}$, let
\begin{equation}
C^{s}(\mathbb{R}^n):=C^{\lfloor s\rfloor, s-\lfloor s\rfloor}(\mathbb{R}^n).
\end{equation}
\begin{prop} (Propositions 2.5 and 2.6 in \cite{Silvestre})
Let $u\in C^{\alpha}$ for $\alpha\in (0,1]$ and $\alpha>2\sigma>0$; then $(-\Delta)^\sigma u\in C^{\alpha-2\sigma}$ and
\begin{equation}
\left[ (-\Delta)^\sigma u\right]_{C^{\alpha-2\sigma}}\leq C[u]_{C^{\alpha}},
\end{equation}
where $C$ depends only on $\alpha$, $\sigma$ and $n$.
\end{prop}
The following result describes fractional Laplacian of Schwartz function:
\begin{prop}\label{prop_fraclap_of_Schwartz}
(Proposition 2.9 in \cite{Garofalo})\footnote{The cited Proposition gives a bounded, integrable bound for $(-\Delta)^s u$ outside $B_1(0)$. To obtain the Proposition we stated, it is enough to observe that $(-\Delta)^s u$ is smooth: in fact, by means of a cut-off function, one can write the Fourier transform of $(-\Delta)^s u$ as the sum of a Schwartz function and a function supported on a neighbourhood of the origin; the inverse Fourier transform of both summands are smooth.}
Let $s\in (0,1)$ $p\in [1,\infty]$, let $u\in \mathscr{S}(\mathbb{R}^n)$; then $(-\Delta)^s u \in C^\infty_{loc}\cap L^p(\mathbb{R}^n)$.
\end{prop}

\subsection{On the circle}
Let $\mathbb{S}^1:= {\mathbb{R}}/ {2\pi \mathbb{Z}}$, with the metric induced by $\mathbb{R}$. Its points can be thought of as equivalence classes of points of $\mathbb{R}$, which we will sometimes denote as $[x]$ for $x\in \mathbb{R}$. Mostly, by an abuse of notation, we will simply write $x$ for $[x]$. Sometimes it will be useful to identify $\mathbb{S}^1$ with $\partial D^2\subset \mathbb{C}$ through the embedding
\begin{equation}\label{eq_definition_identification_map_i}
i: \mathbb{S}^1\to \partial D^2,\quad [x]\mapsto e^{ix}.
\end{equation}
Let $\mathscr{D}(\mathbb{S}^1):=C^\infty(\mathbb{S}^1)$ (for either scalar or vector valued functions) and let $\mathscr{D}'(\mathbb{S}^1)$ denote its topological dual.
For any element $f$ of $\mathscr{D}'(\mathbb{S}^1)$, for any $k\in\mathbb{Z}$, we define the \textbf{$k^{th}$-Fourier coefficient} of $f$ by
\begin{equation}
\widehat{f}(k):=\frac{1}{2\pi}\langle f, e^{-ik\cdot}\rangle,
\end{equation}
where the brackets denote the action of an element of $\mathscr{D}'(\mathbb{S}^1,\mathbb{R}^m)$ on an element of $\mathscr{D}(\mathbb{S}^1,\mathbb{R}^m)$.
We observe that with these conventions,
\textbf{Parseval's Relation} takes the form
\begin{equation}
\int_{\mathbb{S}^1}f(x)\overline{g(x)}dx=2\pi\sum_{n\in\mathbb{Z}}\widehat{f}\overline{\widehat{g}(n)}
\end{equation}
for any $f,g\in L^2(\mathbb{S}^1)$, while
\textbf{Plancherel's Identity} takes the form
\begin{equation}
\lVert u\rVert_{L^2(\mathbb{S}^1)}^2=2\pi\sum_{n\in\mathbb{Z}}\lvert \widehat{u}(n)\vert^2.
\end{equation}
for any $u\in L^2(\mathbb{S}^1)$.
We now proceed with definition of Sobolev spaces on the circle.
For $s\in \mathbb{R}$ we define the \textbf{$s$-inhomogeneous Sobolev space} by
\begin{equation}
H^s(\mathbb{S}^1)=\left\lbrace u\in \mathscr{D}'(\mathbb{S}^1)\bigg\vert \lVert u\rVert_{H^s}^2:=\sum_{n\in\mathbb{Z}}(1+\lvert n\rvert^{2})^s \hat{u}(n)^2<\infty\right\rbrace.
\end{equation}
By Plancherel's identity, $H^0(\mathbb{S}^1)=L^2(\mathbb{S}^1)$.
We also introduce the homogeneous Sobolev seminorm
\begin{equation}
[v]_{\dot{H}^s}:=\left(\sum_{n\in\mathbb{Z}\smallsetminus\lbrace0\rbrace}\lvert n\rvert^{2s} \widehat{u}(n)^2\right)^\frac{1}{2}
\end{equation}
for any $v\in H^s(\mathbb{S}^1)$.
When $s=\frac{1}{2}$ the homogeneous Sobolev seminorm can be characterized by a double integral (see Lemma \ref{lem_appendix_circle_eqivalent_Sobolev_seminorm}).
Since $n\leq 1+\lvert n\rvert^2\leq 2\lvert n\rvert^2$, for any $n\in\mathbb{Z}\smallsetminus\lbrace0\rbrace$,
\begin{equation}
[u]_{\dot{H}^s}\leq C_s \lVert u\rVert_{H^s},
\end{equation}
for any $u\in H^s(\mathbb{S}^1)$ for some constant $C_s$ depending only on $s$.
We observe that for $s>0$, $u\in H^s(\mathbb{S}^1)$,
\begin{equation}\label{eq_discussion_preliminaries_circle_bound_Sobolev_norm_through_seminorm}
\lVert u\rVert_{H^s}^2=\sum_{n\in\mathbb{Z}}(1+\lvert n\rvert^2)^s\widehat{u}(n)^2\leq \sum_{n\in\mathbb{Z}}\left(2^\frac{s}{2}\lvert n\rvert+1\right)\widehat{u}(n)^2=2^\frac{s}{2}[u]_{\dot{H}^s}^2+\lVert u\rVert_{L^2}^2.
\end{equation}
Moreover we recall that $\mathscr{D}'\left(\mathbb{S}^1\right)=\bigcup_{s\in\mathbb{R}}H^s(\mathbb{S}^1)$.
We also remark that for any $s\in \mathbb{R}$, we may identify the topological dual space $H^s(\mathbb{S}^1)^\ast$ with $H^{-s}(\mathbb{S}^1)$. Indeed, by definition of $H^s(\mathbb{S}^1)$, the function
\begin{equation}
\mathscr{I}_s: H^s(\mathbb{S}^1)\to l^2, \quad u\mapsto \left((1+\lvert n\rvert^2)^\frac{s}{2}\widehat{u}(n) \right)_n.
\end{equation}
is an isometric isomorphism between vector spaces.
Now, for any functional $F\in H^s(\mathbb{S}^1)^\ast$, any $\phi\in C^\infty(\mathbb{S}^1)$,
\begin{equation}
\left\langle F, \phi\right\rangle=\sum_{k\in\mathbb{Z}}\widehat{F}(k)\overline{\widehat{\phi}(k)}=\sum_{k\in\mathbb{Z}}\left[(1+\lvert k\rvert^2)^{-\frac{s}{2}}\widehat{F}(k)\right]\left[(1+\lvert k\rvert^2)^\frac{s}{2}\overline{\widehat{\phi}(k)}\right]
\end{equation}
Therefore a distribution $F$ in $\mathscr{D}'(\mathbb{S}^1)$ belongs to $H^s(\mathbb{R}^n)^\ast$ if and only if \newline $\left((1+\lvert n\rvert^2)^\frac{s}{2}\widehat{u}(n) \right)_n$ belongs to $l^2$, and in this case
\begin{equation}
\lVert F\rVert_{{H^s}^\ast}=\left\lVert\left((1+\lvert n\rvert^2)^\frac{s}{2}\widehat{u}(n) \right)_n  \right\rVert_{l^2}=\lVert F\rVert_{H^{-s}}.
\end{equation}
Next we define fractional Laplacians for functions on the circle. Let $\sigma, s\in\mathbb{R}$ and let $u\in H^{\sigma}(\mathbb{S}^1)$. Then the \textbf{$s$-fractional Laplacian} of $u$ is 
the element of $\mathscr{D}'(\mathbb{S}^1)$ characterized by
\begin{equation}
\widehat{(-\Delta)^su}(n)=\lvert n\rvert^{2s}\hat{u}(n)
\end{equation}
for any $n\in \mathbb{Z}$. Thus $(-\Delta)^\frac{1}{2}u\in H^{\sigma-2s}(\mathbb{S}^1)$.\\
We also recall the definition of Hilbert transform for functions defined on $\mathbb{S}^1$:
for $u\in L^1(\mathbb{S}^1)$, let
\begin{equation}
Hu(x):=\frac{1}{2\pi}PV\int_{\mathbb{S}^1}\cot\left(\frac{1}{2}(x-y)\right)u(y)dy.
\end{equation}
for all $x\in \mathbb{S}^1$. $Hu$ is called the Hilbert transform of $u$. Equivalently, $Hu$ is the function on $\mathbb{S}^1$ characterized by
\begin{equation}
\widehat{Hu}(n)=-i\text{sgn}(n)\hat{f}(n)
\end{equation}
for all $n\in\mathbb{Z}$ (see \cite{Schlag}, proof of Proposition 3.18).
Therefore, for any $s\in\mathbb{R}^+$, if $u\in H^s(\mathbb{S}^1)$ $Hu\in H^s(\mathbb{S}^1)$. Moreover, if we denote by $'$ the derivative of a distribution,
\begin{equation}\label{eq_disc_intro_comparison_H_and_fraclap}
Hu'=(Hu)'=(-\Delta)^\frac{1}{2}u\text{ in }H^{s-1}(\mathbb{S}^1),
\end{equation}
as the corresponding Fourier coefficients coincide.
This yields a pointwise formula for $(-\Delta)^\frac{1}{2}u$ when $u\in H^1(\mathbb{S}^1)$: for a.e. $x\in \mathbb{S}^1$,
\begin{equation}\label{eq_proof_pointwise_formula_ibp}
\begin{split}
(-\Delta)^\frac{1}{2}u(x)=&\frac{1}{2\pi}\lim_{\varepsilon\to 0}\int_{\lvert x-y\rvert>\varepsilon}u'(y)\cot\left(\frac{1}{2}(x-y)\right)dy\\
=&\frac{1}{2\pi}\lim_{\varepsilon\to 0}\int_{\lvert x-y\rvert>\varepsilon}(u(x)-u(y))\frac{\pi}{\sin^2\left(\frac{1}{2}(x-y)\right)}dy\\
&+\frac{1}{2\pi}(u(x+\varepsilon)-u(x))\cot\left(\frac{1}{2}\varepsilon\right)\\
&-\frac{1}{2\pi}(u(x-\varepsilon)-u(x))\cot\left(-\frac{1}{2}\varepsilon\right),
\end{split}
\end{equation}
Now we observe that by Plancherel's identity

\begin{equation}
\begin{split}
&\left\lVert \left(u(\cdot+\varepsilon)-u\right)\cot\left(\frac{1}{2}\varepsilon\right)-(u(\cdot-\varepsilon)-u)\cot\left(-\frac{1}{2}\varepsilon\right)\right\rVert_{L^2}^2\\
=&\left\lVert[(u(\cdot+\varepsilon)-u)-(u-u(\cdot-\varepsilon))]\cot\left(\frac{1}{2}\varepsilon\right)\right\rVert_{L^2}^2\\
=&\cot^2\left(\frac{1}{2}\varepsilon\right)2\pi\sum_{n\in\mathbb{Z}}\left\lvert[(e^{i\varepsilon n}-1)-(1-e^{-i\varepsilon n})]\widehat{u}(n)\right\rvert^2\\
\leq&\frac{4}{\varepsilon^2}2\pi\sum_{n\in\mathbb{Z}}\left\lvert2(\cos(n\varepsilon)-1)\widehat{u}(n)\right\rvert^2\leq 16\pi\sum_{n\in\mathbb{Z}}\left\lvert\frac{(\cos(n\varepsilon)-1)}{\varepsilon n}n\widehat{u}(n)\right\rvert^2
\end{split}
\end{equation}
Since the map $x\mapsto \frac{\cos(x)-1}{x}$ is bounded on $\mathbb{R}^\ast$, and since $u\in H^1(\mathbb{S}^1)$ we conclude by dominated convergence that
\begin{equation}\label{eq_proof_pointwise_formula_convergence}
\left(u(\cdot+\varepsilon)-u\right)\cot\left(\frac{1}{2}\varepsilon\right)-(u(\cdot-\varepsilon)-u)\cot\left(-\frac{1}{2}\varepsilon\right)\to 0
\end{equation} 
in $L^2(\mathbb{S}^1)$ as $\varepsilon\to 0$. Therefore there exists a subsequence of $(\varepsilon_n)_n$ along which convergence (\ref{eq_proof_pointwise_formula_convergence}) takes place a.e. in $\mathbb{S}^1$.
As the limits in (\ref{eq_proof_pointwise_formula_ibp}) are well defined (by \cite{Schlag}, Proposition 3.18), convergence (\ref{eq_proof_pointwise_formula_convergence}) takes place a.e. in $\mathbb{S}^1$ along the original sequence. Therefore we conclude that for a.e. $x\in \mathbb{S}^1$
\begin{equation}\label{eq_pointwise_formula_fraclap_circle}
(-\Delta)^\frac{1}{2}u(x)=\frac{1}{2\pi} PV\int_{\mathbb{S}^1}\frac{u(x)-u(y)}{\sin^2\left(\frac{1}{2}(x-y)\right)}dy.
\end{equation}
\begin{rem}
Several authors define $\mathbb{S}^1$ as the manifold $\partial D^2\subset \mathbb{C}$, endowed with the distance induced by $\mathbb{C}$. We observe that for $x, y\in [0, 2\pi)$
\begin{equation}
\lvert e^{ix}-e^{iy}\rvert^2=2(1-\cos(x-y))=4\sin^2\left(\frac{1}{2}(x-y)\right)
\end{equation}
Therefore $\left\lvert \sin\left(\frac{1}{2}(x-y)\right)\right\rvert$ actually corresponds (up to a multiplicative constant) to the $\mathbb{C}$-distance between $i([x])$ and $i([y])$ (where i is the identification map defined in (\ref{eq_definition_identification_map_i})). Authors defining  $\mathbb{S}^1$ as the manifold $\partial D^2\subset \mathbb{C}$ endowed with the $\mathbb{C}$-distance usually define the $\frac{1}{2}$-fractional Laplacian of a function $u\in C^\infty(\partial D^2)$ (up to multiplicative constants) as
\begin{equation}
(-\Delta)^\frac{1}{2}u(x)= PV\int_{\partial D^2}\frac{u(x)-u(y)}{\lvert x-y\rvert^2}dy
\end{equation}
for any $x\in \partial D^2$ . By the argument above, this definition coincides (up to multiplicative constants) with the one given in (\ref{eq_pointwise_formula_fraclap_circle}), upon precomposing functions $u$ on $\partial D^2$ with the identification map $i$.
\end{rem}
Following \cite{RoncalStinga}, Theorem 3.6, we set
\begin{equation}\label{eq_explit_formula_Kernel_12}
K^\frac{1}{2}(x):=\frac{1}{2\pi}\frac{1}{\sin^2\left(\frac{1}{2}x\right)}
\end{equation}
for any $x\in \mathbb{S}^1\smallsetminus\lbrace 0\rbrace$.\\
On the other hand, according to \cite{RoncalStinga}, Theorem 3.4, if $0<s<\frac{1}{2}$, $u\in C^{0, 2s+\varepsilon}(\mathbb{S}^1)$ for some $\varepsilon>0$, then $(-\Delta)^su$ is continuous and, for any $x\in \mathbb{S}^1$
\begin{equation}\label{eq_pointwise_formula_fraclap_circle_s}
(-\Delta)^su(x)=\int_{\mathbb{S}^1}(u(x)-u(y))K^s(x-y)dy,
\end{equation}
where
\begin{equation}\label{def_Kernel_Ks_0_12}
K^s(x)=c_s\int_0^\infty P_t(x)\frac{dt}{t^{1+2s}},
\end{equation}
and $c_s$ is a constant depending on $s$ and $P_t$ is the Poisson kernel for $\mathbb{S}^1$, defined as
\begin{equation}\label{eq_intro_circle_definition_Poisson_kernel}
P_r(x)=\sum_{n\in \mathbb{Z}}r^{\lvert n\rvert}e^{inx}=\frac{1-r^2}{1-2r\cos(x)+r^2}
\end{equation}
for $r\in (0,1)$, $x\in\mathbb{S}^1$ (for more informations abot the Poisson's Kernel for $\mathbb{S}^1$, see \cite{Schlag}, Ch. 2).
In particular we observe that $K^s(x)=K^s(-x)$ for any $x\in\mathbb{S}^1$, $s\in(1, \frac{1}{2})$. Moreover there exist positive constants $b_s$ and $B_s$ such that
\begin{equation}\label{eq_estimates_kernel_fraclap_s}
\frac{b_s}{\left\lvert \sin\left(\frac{1}{2}x\right)\right\rvert^{1+2s}}\leq K^s(x)\leq\frac{B_s}{\left\lvert \sin\left(\frac{1}{2}x\right)\right\rvert^{1+2s}}
\end{equation}
for any $x\in\mathbb{S}^1\smallsetminus\lbrace 0\rbrace$.\newline
We observe that by definition, for any $s\in (0,1)$ the fractional Laplacian is self-adjoint, in the sense that for any $u, v\in C^\infty(\mathbb{S}^1, \mathbb{R}^m)$
\begin{equation}
\int_{S^1}(-\Delta)^s u(x)\cdot v(x) dx=\int_{S^1}u(x)\cdot (-\Delta)^s v(x) dx.
\end{equation}
Finally we remark that also in the case of the circle, for any $s\in (0,\frac{1}{2}]$ the pointwise formulas (\ref{eq_pointwise_formula_fraclap_circle_s}) and (\ref{eq_pointwise_formula_fraclap_circle}) imply the following "Leibnitz rule": for any $u,v\in C^1(\mathbb{S}^1)$ there holds
\begin{equation}\label{eq_discussion_intro_Leibnitz_rule_fraclap_circle}
\begin{split}
(-\Delta)^s(uv)(x)=&PV\int_{\mathbb{S}^1}(uv(x)-uv(y))K^s(x-y)dy\\
=&u(x)PV\int_{\mathbb{R}^n}(v(x)-v(y))K^s(x-y)dy\\
&+v(x)\int_{\mathbb{R}^n}(u(x)-u(y))K^s(x-y)dy\\
&-PV\int_{\mathbb{R}^n}(u(x)-u(y))(v(x)-v(y)K^s(x-y)dy\\
=&u(x)(-\Delta)^sv(x)+v(x)(-\Delta)^su(x)\\&-\int_{\mathbb{R}^n}(u(x)-u(y))(v(x)-v(y)K^s(x-y)dy,
\end{split}
\end{equation}
where in the last step we used the fact that since $u,v\in H^{2s}(\mathbb{S}^1)$, the last integral converges absolutely for a.e. $x\in\mathbb{S}^1$.

\section{Variations in the target}
In this section we consider variational problems whose Lagrangians involving a $\frac{1}{4}$-fractional Laplacian and exhibiting symmetries under variations in the target. First we derive a Noether theorem for such Lagrangian (Theorem \ref{prop_noether_thm_variaton_target_cor}), then we will consider the particular case of harmonic map into sphere, and we try to recast a regularity result by K. Mazowiecka and A. Schikorra in the framework of Noether theorems (Section \ref{ssec: harmonic maps into spheres}).

\subsection{Fractional gradients and fractional divergences}\label{ssec: Fractional gradients and fractional divergences}
The Noether theorem involving a $\frac{1}{4}$-fractional Laplacian will be formulated in terms of fractional gradients and divergences. This were introduced by by K. Mazowiecka and A. Schikorra in \cite{divcurl}. In this section we recall their definition and some of their basic properties.

\begin{defn}
Let $s\in [0,1]$. Let $f:\mathbb{R}^n\to\mathbb{R}^n$. For any $(x,y)\in \mathbb{R}^n\times\mathbb{R}^n$ let
\begin{equation}
\text{d}_sf(x,y):=\frac{f(x)-f(y)}{\lvert x-y\rvert^s}
\end{equation}
$\text{d}_s f$ is called the \textbf{$s$-gradient of $f$}.
\end{defn}
\begin{defn}\label{def_fractional_divergence_Rn}
Let $s\in (0,1)$. Let $F: \mathbb{R}^n\times\mathbb{R}^n\to\mathbb{R}$ and assume that for any $\phi\in C^\infty_c(\mathbb{R}^n)$
\begin{equation}\label{eq_def_condition_def_div}
\int_{\mathbb{R}^n}\int_{\mathbb{R}^n} \left\lvert F(x,y)\text{d}_s\phi(x,y)\right\rvert\frac{dxdy}{\lvert x-y\rvert^n}<\infty.
\end{equation}
The \textbf{$s$-divergence} of $F$ is the distribution defined by
\begin{equation}
\text{div}_sF[\phi]:=\int_{\mathbb{R}^n}\int_{\mathbb{R}^n}F(x,y)\text{d}_s\phi(x,y)\frac{dxdy}{\lvert x-y\rvert^n}\quad\text{for any }\phi\in C^\infty_c(\mathbb{R}^n).
\end{equation}
\end{defn}

We also recall the definition of some function spaces which interact naturally with fractional gradients and derivatives (see Lemma \ref{lem_cont_grad_div}).\\
Let's denote by $\mathcal{M}$ the set of of measurable functions on $\mathbb{R}^n\times\mathbb{R}^n$ (either scalar valued or vector valued).

\begin{defn}\noindent
\begin{enumerate}
\item Let $p\in [1,\infty)$. We define the \textbf{off-diagonal $L^p$-space} as
\begin{equation}
\begin{split}
L^p\left(\bigwedge^1_{od}\mathbb{R}^n\right)=&\left\lbrace F\in \mathcal{M}\bigg\vert\right.\\
&\left.\phantom{\lbrace} \lVert F\rVert_{L^p(\bigwedge^1_{od}\mathbb{R}^n)}:=\left(\int_{\mathbb{R}^n}\int_{\mathbb{R}^n}\left\lvert F(x,y)\right\rvert^p\frac{dxdy}{\lvert x-y\rvert^n}\right)^\frac{1}{p}<\infty\right\rbrace.
\end{split}
\end{equation}
\item Let $s\in [0,1)$. We define the homogeneous, off-diagonal $s$-Sobolev space by
\begin{equation}
\dot{H}^s\left(\bigwedge^1_{od}\mathbb{R}^n\right):=\left\lbrace F\in \mathcal{M} \bigg\vert  [F]_{\dot{H}^s(\bigwedge^1_{od}\mathbb{R}^n)}:= \left\lVert \frac{F(x,y)}{\lvert x-y\rvert^s}\right\rVert_{L^2(\bigwedge^1_{od}\mathbb{R}^n)}<   \infty \right\rbrace.
\end{equation}
In particular, $\dot{H}^0\left(\bigwedge^1_{od}\mathbb{R}^n\right)=L^2\left(\bigwedge^1_{od}\mathbb{R}^n\right)$.
\end{enumerate}
\end{defn}
Then, for any $p\in [1,\infty)$, $L^p\left(\bigwedge^1_{od}\mathbb{R}^n\right)$ is a normed space with norm \newline $\lVert \cdot\rVert_{L^p(\bigwedge^1_{od}\mathbb{R}^n)}$, and for any $s\in (0,1)$, $\dot{H}^s\left(\bigwedge^1_{od}\mathbb{R}^n\right)$ is a vector space endowed with the seminorm $[\cdot]_{\dot{H}(\bigwedge^1_{od}\mathbb{R}^n)}$.\\
For $F$, $G$ we also introduce the notation
\begin{equation}\label{eq_definition_product_in_one_var}
 F\cdot_{\bigwedge_{od}^1\mathbb{R}^n} G(x)=PV\int_{\mathbb{R}^n}F(x,y)G(x,y)\frac{dy}{\lvert x-y\rvert^n}
\end{equation}
whenever the principal value is well defined.
Next we recall some basic results about fractional gradients and fractional divergences.

\begin{lem}\label{lem_cont_grad_div}
Let $\sigma\in(0,1)$, $s\in(0,\sigma]$,
\begin{enumerate}
\item if $s\in(0,\sigma]$, $\text{d}_s$ is a bounded linear operator from $\dot{H}^{\sigma,2}(\mathbb{R}^n)$ to $\dot{H}^{\sigma-s}(\bigwedge^1_{od}\mathbb{R}^n)$ and for any $u\in \dot{H}^s(\mathbb{R}^n)$, 
\begin{equation}\label{eq_lem_eq_norm_and_continuity}
[\text{d}_s u]_{\dot{H}^{\sigma-s}(\bigwedge^1_{od}\mathbb{R}^n)}=[u]_{\dot{W}^{\sigma,2}(\mathbb{R}^n)}.
\end{equation}
In particular, if $u\in \dot{H}^s(\mathbb{R}^n)$, $\text{d}_su\in L^2\left(\bigwedge^1_{od}\mathbb{R}^n\right)$.
\item if $s\in(\sigma,1)$, $\text{div}_s$ is a bounded linear operator from $\dot{H}^{\sigma}(\bigwedge^1_{od}\mathbb{R}^n)$ to $\dot{H}^{\sigma-s}(\mathbb{R}^n)^\ast$, the topological dual of $\dot{H}^{\sigma-s}(\mathbb{R}^n)$.
\end{enumerate}
\end{lem}

\begin{proof}
For $u\in \dot{H}^\sigma(\mathbb{R}^n)$,
\begin{equation}
[\text{d}_s u]_{\dot{H}^{\sigma-s}(\bigwedge^1_{od}\mathbb{R}^n)}=\left(\int_{\mathbb{R}^n}\int_{\mathbb{R}^n}\frac{\lvert u(x)-u(y)\rvert^2}{\lvert x-y\rvert^{2s}}\frac{dxdy}{\lvert x-y\rvert^{2\sigma-2s+n}}\right)^\frac{1}{2}=[u]_{\dot{W}^{\sigma,2}}.
\end{equation}
For $F\in\dot{H}^{\sigma}(\bigwedge^1_{od}\mathbb{R}^n)$, H\"{o}lder's inequality and (\ref{eq_lem_eq_norm_and_continuity}) implies that for any $\phi\in \dot{H}^s(\mathbb{R}^n)$
\begin{equation}
\begin{split}
\lvert\text{div}_s F [\phi]\rvert&=\left\lvert\int_{\mathbb{R}^n}\int_{\mathbb{R}^n}F(x,y)\text{d}_s \phi(x,y)\frac{dxdy}{\lvert x-y\rvert^n}\right\rvert\\
&\leq \left( \int_{\mathbb{R}^n}\int_{\mathbb{R}^n}\frac{\lvert F(x,y)\rvert^2}{\lvert x-y\rvert^{2\sigma}}\frac{dxdy}{\lvert x-y\rvert^n}\right)^\frac{1}{2}\left(\int_{\mathbb{R}^n}\int_{\mathbb{R}^n}\frac{\lvert \phi(x)-\phi(y)\rvert^2}{\lvert x-y\rvert^{2s-2\sigma}}\frac{dxdy}{\lvert x-y\rvert^n}\right)^\frac{1}{2}\\
&=[F]_{\dot{H}^{\sigma}(\bigwedge^1_{od}\mathbb{R}^n)}[\phi]_{W^{s-\sigma,2}}.
\end{split},
\end{equation}
therefore $F\in \dot{W}^{s-\sigma,2}(\mathbb{R}^n)^\ast\simeq \dot{H}^{s-\sigma}(\mathbb{R}^n)$ with $\lVert F\rVert_{\dot{W}^{s-\sigma,2}(\mathbb{R}^n)^\ast}\leq [F]_{\dot{H}^{\sigma}(\bigwedge^1_{od}\mathbb{R}^n)}$.
\end{proof}

\begin{lem}\label{lem_div_d_=_fraclap}
Let $s\in (0,1)$, let $u, v\in \dot{H}^{s}(\mathbb{R}^n)$. Then
\begin{equation}\label{eq_rem_duality_gradient_fraclap}
\int_{\mathbb{R}^n}(-\Delta)^\frac{s}{2} u(x) (-\Delta)^\frac{s}{2}v(x)dx=\frac{1}{2}C_{n,s}\left\langle \text{d}_s u, \text{d}_s v\right\rangle,
\end{equation}
where $C_{n,s}$ is the constant introduced in (\ref{eq_prel_pointwise_definition_fraclap}).
In particular, if $u\in \dot{H}^s(\mathbb{R}^n)$,
\begin{equation}\label{eq_lem_composition_fracdiv_fracgrad_is_fraclap}
(-\Delta)^s u=\frac{C_{n,s}}{2}\text{div}_s \text{d}_s u
\end{equation}
as distributions.
\end{lem}

\begin{proof}
We first show (\ref{eq_rem_duality_gradient_fraclap}) for functions $\phi$, $\psi$ in $\mathscr{S}(\mathbb{R}^n)$. We observe that
\begin{equation}
\begin{split}
\int_{\mathbb{R}^n}(-\Delta)^\frac{s}{2}\phi(x)(-\Delta)^\frac{s}{2}\psi(x)dx=&\int_{\mathbb{R}^n}(-\Delta)^s\phi(x)\psi(x)dx\\ =&\frac{C_{n,s}}{2}\left(\int_{\mathbb{R}^n} \psi(x)\cdot PV\int_{\mathbb{R}^n}\frac{\phi(x)-\phi(y)}{\lvert x-y\rvert^{2s+n}}dydx \right.\\
&- \left. \int_{\mathbb{R}^n}\psi(y)\cdot PV\int_{\mathbb{R}^n}\frac{\phi(x)-\phi(y)}{\lvert x-y\rvert^{2s+n}}dxdy \right).
\end{split}
\end{equation}
Now, by  Taylor's theorem, for any $\varepsilon>0$, $x\in\mathbb{R}^n$,
\begin{equation}
\begin{split}
\left\lvert\int_{B_\varepsilon(x)^c}\frac{\phi(x)-\phi(y)}{\lvert x-y\rvert^{n+2s}}dy\right\rvert=&\left\lvert \frac{1}{2}\int_{B_\varepsilon(0)^c}\frac{2\phi(x)-\phi(x-z)-\phi(x+z)}{\lvert z\rvert^{n+2s}}dz\right\rvert\\
\leq&\int_{\mathbb{R}^n} \frac{\lVert D^2 \phi\rVert_{L^\infty}\lvert x-y\rvert^2}{\lvert z\rvert^{2s+n}}dz,
\end{split}
\end{equation}
and the right side multiplied by $\psi$ is integrable in $x$. Therefore, by Lebesgue's Dominated convergence Theorem,
\begin{equation}
\begin{split}
&\int_{\mathbb{R}^n}\psi(x)PV\int_{\mathbb{R}^n}\frac{\phi(x)-\phi(y)}{\lvert x-y\rvert^{n+2s}}dydx\\
=&\lim_{\varepsilon\to 0}\int_{\mathbb{R}^n}\psi(x)\int_{\mathbb{R}^n}\mathds{1}_{B_\varepsilon(x)^c}(y) \frac{\phi(x)-\phi(y)}{\lvert x-y\rvert^{n+2s}}dydx.
\end{split}
\end{equation}
The same holds if we switch $\phi$ and $\psi$, therefore
\begin{equation}\label{eq_proof_lem_remove_the_limit}
\begin{split}
&\int_{\mathbb{R}^n}(-\Delta)^s\phi(x)\psi(x)dx\\
=&\lim_{\varepsilon\to 0}\frac{C_{n,s}}{2}\int_{\mathbb{R}^n}\int_{\mathbb{R}^n}\mathds{1}_{\lvert x-y\rvert>\varepsilon}\frac{\lvert(\phi(x)-\phi(y))\cdot (\psi(x)-\psi(y))\rvert}{\lvert x-y\rvert^{n+2s}}dydx.
\end{split}
\end{equation}
Finally, by H\"{o}lder's inequality,
\begin{equation}\label{eq_rem_conv_int_hold}
\begin{split}
&\int_{\mathbb{R}^n}\int_{\mathbb{R}^n}\frac{\lvert(\phi(x)-\phi(y))\cdot (\psi(x)-\psi(y))\rvert}{\lvert x-y\rvert^{n+2s}}dxdy\\
\leq &\lVert \text{d}_s \phi \rVert_{L^2(\bigwedge^1_{od}\mathbb{R}^n)}\lVert \text{d}_s\psi\rVert_{L^2(\bigwedge^1_{od}\mathbb{R}^n)}.
\end{split}
\end{equation}
Therefore the limit on the right side of (\ref{eq_proof_lem_remove_the_limit}) exists and is equal to
\begin{equation}
\begin{split}
\frac{C_{n,s}}{2}\int_{\mathbb{R}^n}\int_{\mathbb{R}^n}\frac{(\phi(x)-\phi(y))\cdot (\psi(x)-\psi(y))}{\lvert x-y\rvert^{n+2s}}dydx=&\frac{C_{n,s}}{2}\langle\text{d}_s \phi, \text{d}_{s}\psi\rangle\\
=&\frac{C_{n,s}}{2}\text{div}_s \text{d}_{s}\psi[\phi].
\end{split}
\end{equation}
This concludes the proof of (\ref{eq_rem_duality_gradient_fraclap}) when $\phi,\psi\in \mathscr{S}(\mathbb{R}^n)$.
Now if $u, v\in \dot{H}^s(\mathbb{R}^n)$, by Lemma \ref{lem_approx_sobolev} there exist sequences $(u_n)_n$, $(v_n)_n$ in $\mathscr{S}(\mathbb{R}^n)$ such that $u_n\to u$ and $v_n\to v$ in $\dot{H}^s(\mathbb{R}^n)$ as $n\to \infty$. Then
\begin{equation}
\begin{split}
&(-\Delta)^\frac{s}{2}u_n\to (-\Delta)^\frac{s}{2}u\text{ in }L^2(\mathbb{R}^n)\\
&(-\Delta)^\frac{s}{2}v_n\to (-\Delta)^\frac{s}{2}v\text{ in }L^2(\mathbb{R}^n)
\end{split}
\end{equation}
as $n\to \infty$, and therefore
\begin{equation}
\int_{\mathbb{R}^n}(-\Delta)^\frac{s}{2}u_n(x)(-\Delta)^\frac{s}{2}v_n(x)dx\to\int_{\mathbb{R}^n}(-\Delta)^\frac{s}{2}u(x)(-\Delta)^\frac{s}{2}v(x)dx
\end{equation}
as $n\to \infty$.
On the other hand, by Lemma \ref{lem_cont_grad_div},
\begin{equation}
\begin{split}
&\text{d}_su_n\to \text{d}_su  \text{ in } L^2\left(\bigwedge^1_{od}\mathbb{R}^n\right),\\
&\text{d}_sv_n\to \text{d}_sv  \text{ in } L^2\left(\bigwedge^1_{od}\mathbb{R}^n\right)
\end{split}
\end{equation}
as $n\to\infty$. Therefore
\begin{equation}
\begin{split}
&\int_{\mathbb{R}^n}\int_{\mathbb{R}^n}\frac{(u_n(x)-u_n(y))\cdot (v_n(x)-v_n(y))}{\lvert x-y\rvert^{n+2s}}dydx\\
\to &\int_{\mathbb{R}^n}\int_{\mathbb{R}^n}\frac{(u(x)-u(y))\cdot (v(x)-v(y))}{\lvert x-y\rvert^{n+2s}}dydx
\end{split}
\end{equation}
as $n\to \infty$, or equivalently
\begin{equation}
\langle \text{d}_s u_n, \text{d}_s v_n\rangle\to\langle \text{d}_s u, \text{d}_s v\rangle
\end{equation}
as $n\to \infty$.
Thus, as (\ref{eq_rem_duality_gradient_fraclap}) holds for Schwartz functions, it also holds for functions in $\dot{H}^s(\mathbb{R}^n)$.

Similarly, if $\phi\in \mathscr{S}(\mathbb{R}^n)$ and $\psi\in C^\infty_c(\mathbb{R}^n)$, (\ref{eq_rem_duality_gradient_fraclap}) implies
\begin{equation}\label{eq_proof_lem_computation_fraclap_for_schwartz}
\begin{split}
\langle (-\Delta)^s\phi, \psi\rangle=&\int_{\mathbb{R}^n}(-\Delta)^\frac{s}{2}\phi(x)(-\Delta)^\frac{s}{2}\psi(x)dx=C_{n,s}\langle \text{d}_s \phi,\text{d}_s \psi\rangle\\
=&C_{n,s}\langle\text{div}_s\text{d}_s\phi, \psi\rangle.
\end{split}
\end{equation}
In order to show (\ref{eq_lem_composition_fracdiv_fracgrad_is_fraclap}), for any $u\in \dot{H}^s(\mathbb{R}^n)$, $\phi\in C^\infty_c(\mathbb{R}^n)$ we compute
\begin{equation}
\left\langle (-\Delta)^s u, \phi \right\rangle=\int_{\mathbb{R}^n}(-\Delta)^\frac{s}{2}u(x)(-\Delta)^\frac{s}{2}\phi(x) dx=\frac{C_{n,s}}{2} \langle\text{d}_s u, \text{d}_s v\rangle,
\end{equation}
where in the second step we used identity (\ref{eq_rem_duality_gradient_fraclap}).
As this holds for any $\psi\in C^\infty_c(\mathbb{R}^n)$, we conclude that
\begin{equation}
(-\Delta)^s u=\frac{C_{n,s}}{2}\text{div}_s \text{d}_s u
\end{equation}
as distributions.
\end{proof}
Finally we recall the following two results from \cite{divcurl}. Let $n\in \mathbb{N}_{>0}$.
\begin{thm}\label{thm_div_curl_and_Hardy}
(Theorem 2.1 in \cite{divcurl}) Let $s\in (0,1)$, $p\in (1, \infty)$. Let $F\in L^p\left(\bigwedge^1_{od}\mathbb{R}^n\right)$ and $g\in \dot{W}^{s, p'}(\mathbb{R}^n)$. Assume that $\text{div}_sF=0$. Then $F\cdot \text{d}_sg$ belongs to the Hardy space $\mathcal{H}^1(\mathbb{R}^n)$ and we have the estimate
\begin{equation}
\lVert F\cdot \text{d}_sg\rVert_{\mathcal{H}^1(\mathbb{R}^n)}\leq C \lVert F\rVert_{L^p\left(\bigwedge^1_{od}\mathbb{R}^n\right)}\lVert \text{d}_s g\rVert_{L^{p'}\left(\bigwedge^1_{od}\mathbb{R}^n\right)}.
\end{equation}
\end{thm}
From Theorem \ref{thm_div_curl_and_Hardy} follows the following fractional analogous to Wente's Lemma.
\begin{cor}\label{cor_fractional_Wente_Lemma}
(Corollary 2.3 in \cite{divcurl}) Let $s\in (0,1)$, $p\in (1, \infty)$. Let $F\in L^p\left(\bigwedge^1_{od}\mathbb{R}^n\right)$ and $g\in \dot{W}^{s, p'}(\mathbb{R}^n)$. Assume that $\text{div}_s F=0$. Moreover, let $T$ be a linear operator such that for some $\Lambda>0$
\begin{equation}
\left\lvert T[\phi]\right\rvert\leq \Lambda \left\lVert (-\Delta)^\frac{n}{4}\phi\right\rVert_{L^{2,\infty}(\mathbb{R}^n)}\text{ for all } \phi\in C^\infty_c(\mathbb{R}^n).
\end{equation}
Then any distributional solution $u\in \dot{W}^{\frac{n}{2},2}(\mathbb{R}^n)$ to
\begin{equation}
(-\Delta)^\frac{n}{2}u=F\cdot \text{d}_s g+T\text{ in }\mathbb{R}^n
\end{equation}
is continuous.
\end{cor}

\subsection{A Noether Theorem for Lagrangians involving a $\frac{1}{4}$-fractional Laplacian}
For this section, let $\mathscr{M}$ be a smooth closed manifold embedded in $\mathbb{R}^m$.
Let's first consider the energy
\begin{equation}\label{eq_intro_general_energy}
E(u)=\int_{\mathbb{R}^m}L\left(x,u(x), \nabla u(x)\right)dx.
\end{equation}
Here $L$ is a $C^1$ function with at most quadratic growth in the second variable, and $u$ is an element of $H^1\cap L^\infty(\mathbb{R}^n, \mathscr{M})$.
A possible formulation for Noether's theorem for variations in the target reads as follows:
\begin{thm}\label{thm_noether_thm}
(Theorem 1.3.1 in \cite{Helein} ) Let $X$ be a smooth vector field on $\mathscr{M}$, which is an infinitesimal symmetry for $L$, i.e.
\begin{equation}\label{eq_thm_infinitesimal_symmetry_condition}
\frac{\partial L}{\partial y^i}(x, y, A)X^i(y)+\frac{\partial L}{\partial A_\alpha^i}(x, y, A)\frac{\partial X^i}{\partial y^j}(y)A^j_\alpha=0.
\end{equation}

Let $u$ be a critical point in $H^2(\mathbb{R}^n,\mathscr{M})$ of the energy $E$ defined in (\ref{eq_intro_general_energy}). Then
\begin{equation}\label{eq_thm_noether_current}
\text{div}\left(\frac{\partial L}{\partial A}\cdot X(u)\right)=0
\end{equation}
as distributions.
\end{thm}

For the proof, one takes a test function $\phi\in C^\infty_c(\mathbb{R}^n)$ and consider the solution $u^\phi_t$ of the ODE
\begin{equation}
\begin{cases}
u^\phi_0(x)=u(x)\\
\partial_t u^\phi_t(x)=\phi(x)X(u(x)).
\end{cases}
\end{equation}
for any $x\in\mathbb{R}^n$. Plugging $u^\phi_t$ in $L$ and exploiting both equation (\ref{eq_thm_infinitesimal_symmetry_condition}) and the fact that $u$ is a critical point of $L$, one obtain (\ref{eq_thm_noether_current}) in the weak form .\\

An important example of energy of the form (\ref{eq_intro_general_energy}) is given by the Dirichlet energy
\begin{equation}\label{eq_Dirichlet_energy}
E_D(u)=\int_{\mathbb{R}^n}\lvert\nabla u(x)\rvert^2dx,
\end{equation}
for functions $u\in H^1(\mathbb{R}^n,\mathscr{M})$. In particular, if $n=2$ and $\mathscr{M}$ is a sphere (say of dimension $m-1$), F. H\'{e}lein showed by means of Wente's lemma that critical points of $E_D$ in $H^1(\mathbb{R}^n, \mathbb{S}^{m-1})$ are continuous (see \cite{Heleinarticle}).
Motivated by this examples, one could look for analogous for $n=1$. A natural energy to consider in this case would then be
\begin{equation}\label{eq_1_4_Dirichlet_energy}
E(u)=\int_{\mathbb{R}}\lvert (-\Delta)^\frac{1}{4}u(x)\rvert^2dx,
\end{equation}
for functions $u\in H^\frac{1}{2}(\mathbb{R}, \mathbb{S}^{m-1})$. Critical points of (\ref{eq_1_4_Dirichlet_energy}) are called \textbf{half-harmonic maps}. We will sometimes refer to $E$ as half Dirichlet energy.\\
In this case, the argument of Noether theorem for invariance under rotations in the target yields   quantities $\Omega_{ik}$, for $i,k\in \lbrace 1,...m\rbrace$, whose $\frac{1}{2}$-divergence is equal to $0$ (Lemma \ref{lem_direct_proof_divOmega=0}). K. Mazowiecka and A. Schikorra showed in \cite{divcurl} that this fact can be use to obtain a new proof of the continuity of half harmonic maps into spheres (first proved by F. Da Lio and T. Rivière in \cite{3Termscomm}), analogous to the argument used by F. H\'{e}lein for the local case.
In the following, we try to repeat the proof of Noether theorem for energies involving the $\frac{1}{4}$-Laplacian. It turns out, that for critical points of such energies, the $\frac{1}{2}$-divergence of a characteristic quantity vanishes(Theorem \ref{prop_noether_thm_variaton_target_cor}). In the next section we will return to the energy (\ref{eq_1_4_Dirichlet_energy}) and to the problem of the regularity of its critical points.\\
First we introduce the following definitions.

\begin{defn}\noindent
\begin{enumerate}
\item Let $u_t$ be a family of functions in $\dot{H}^\frac{1}{2}(\mathbb{R}^n, \mathscr{M})$, for $t$ in a neighbourhood of $0$. We say that $u_t$ is \textbf{differentiable} in $0$ if there exists $u'\in \dot{H}^\frac{1}{2}(\mathbb{R}^n, \mathbb{R}^m)$ such that
\begin{equation}
u_t=u+tu'+o_{H^\frac{1}{2}}(t)\text{ as }t\to 0.
\end{equation}
We call $u'$ the derivative of $u_t$ in $0$.
\item A function $u\in \dot{H}^\frac{1}{2}(\mathbb{R}^n, \mathscr{M})$ is said to be a \textbf{critical point} of $E$ in $\dot{H}^\frac{1}{2}(\mathbb{R}^m,\mathscr{M})$ with respect to variations in $\mathscr{M}$
if for any family of functions $u_t$ in $\dot{H}^\frac{1}{2}(\mathbb{R}^n, \mathscr{M})$, for $t$ in a neighbourhood of $0$, differentiable in $0$ and such that $u_0=u$, the function $t\mapsto E(u_t)$ defined in a neighbourhood of $0$ is differentiable in $0$ and
\begin{equation}
\frac{d}{dt}\bigg\vert_{t=0} E(u_t)=0.
\end{equation}
\end{enumerate}
\end{defn}
Adapting the argument of Theorem \ref{thm_noether_thm} to Lagrangians involving fractional Laplacians we obtain the following result:
\begin{thm}\label{prop_noether_thm_variaton_target_cor}
Let $L:\mathbb{R}^n\times\mathbb{R}^m\times \mathbb{R}^m\to \mathbb{R}$ be a function in $C^2_{loc}$ such that the second derivatives of $L$ involving the second and third variables are bounded. And assume that there exist constants $C$, $a$, $b$, $c$ such that $\frac{n}{n+\frac{1}{2}}<a\leq b\leq 2\leq c$ and
\begin{equation}\label{eq_prop_noeth_thm_vartar_condition_on_L}
\left\lvert L(x,y,p)-L(x,y,p')\right\rvert\leq C\left( \lvert y-y'\rvert^2+\lvert y-y'\rvert^c+\lvert p-p'\rvert^a+\lvert p-p'\rvert^b \right)
\end{equation}
for any $x\in\mathbb{R}^n$, $y,y', p,p'\in \mathbb{R}^m$. Let $u\in \dot{H}^\frac{1}{2}(\mathbb{R}^n,\mathscr{M})$ and assume that for any $v$ in a $H^\frac{1}{2}(\mathbb{R}^n)$-neighbourhood of $u$ the integral
\begin{equation}\label{eq_prop_definition_of_energy}
E(v):=\int_{\mathbb{R}^n}L\left(x,v(x),(-\Delta)^\frac{1}{4}v(x)\right)dx
\end{equation}
converges absolutely.
Let $X$ be a $C^4$ vector field on $\mathscr{M}$ which is an infinitesimal symmetry for $L$, in the sense that for for its flow $\phi_t$, for $t$ in a neighbourhood of $0$ there holds
\begin{equation}\label{eq_prop_condition_infinitesimal_symmetry}
L\left(x,\phi_t\circ u(x), (-\Delta)^\frac{1}{4}(\phi_t\circ u)(x)\right)=L\left(x,u(x),(-\Delta)^\frac{1}{4}u(x)\right)
\end{equation}
for a.e. $x\in \mathbb{R}^n$.
Assume that the function
\begin{equation}\label{eq_prop_assumption_integrability_of_p_derivative}
L^i_u: \mathbb{R}^n\to \mathbb{R}, \quad x\mapsto D_{p_i}L\left(x,u(x),(-\Delta)^\frac{1}{4}u(x)\right)
\end{equation}
(where $p_i$ denotes the $(n+i)^{th}$ argument of $L$) belongs to $L^2(\mathbb{R}^n)$ for any $i\in \lbrace 1,...,n\rbrace$. Then for any critical point $u$ of the energy $E$ in $H^\frac{1}{2}(\mathbb{R}^n, \mathscr{M})$ there holds
\begin{equation}\label{eq_thm_1_result_Noeththm_vartar_div_free_quantity}
\begin{split}
&(-\Delta)^\frac{1}{4}\left(L_u^i (X^i\circ u)  \right)
-\text{div}_\frac{1}{4}\left(L_u^i(x)\text{d}_\frac{1}{4}(X^i\circ u)(x,y) \right)=0
\end{split}
\end{equation}
(we use implicit summation over repeated indices).
Equivalently
\begin{equation}\label{eq_thm_Noether_thm_vartar_alternative_form}
\text{div}_\frac{1}{2}\left(L_u^i(x)X^i(u(y))-L_u^i(y)X^i(u(x))\right)=0.
\end{equation}
\end{thm}


\begin{rem}\label{rem_technical_assumtions_noether_thm}
The assumption (\ref{eq_prop_assumption_integrability_of_p_derivative}) is satisfied in most cases of interest. For example, if $L$ satisfies the estimate
\begin{equation}\label{eq_rem_condition_for_integrability_derivative_p}
\lvert D_{p_i}(x, Y, Z)\rvert\leq C \lvert Z\rvert\text{ for any }(x, Y, Z)\in \mathbb{R}\times\mathbb{R}^n\times\mathbb{R}^n,
\end{equation}
for some constant $C$, for any $i\in \lbrace1,...,n\rbrace$, then $\lvert L_u(x)\rvert\leq (-\Delta)^\frac{1}{4}u(x)$ for any $x\in\mathbb{R}^n$ and therefore $L_u\in L^2(\mathbb{R}^n)$.
For instance, condition (\ref{eq_rem_condition_for_integrability_derivative_p}) is satisfied by the Lagrangian $L(x, Y, Z)=Z^2$, which corresponds to the half Dirichlet energy.
\end{rem}

\begin{proof}
Let $X$ be as in the lemma and let $\psi\in C^\infty_c(\mathbb{R}^n)$.
For any $x\in\mathbb{R}^n$ we now consider the Cauchy problem
\begin{equation}\label{eq_proof_prop_Cauchy_prop_PicLind}
\begin{cases}
\gamma (0)&=u(x)\\
\dot{\gamma}(t)&=\psi(x)X(\gamma(t)),
\end{cases}
\end{equation}
we set $g(x,y)=\psi(x)X(y)$ and we observe that for any $x\in\mathbb{R}^n$, $g$ is Lipschitz as a function of the second variable.
Therefore, given a closed interval $[a,b]$ containing $0$ in its interior, by Picard-Lindel\"{o}f theorem there exist a unique solution $\gamma$ of (\ref{eq_proof_prop_Cauchy_prop_PicLind}), and $\gamma\in C^1([a,b])$. We denote such solution $u_t(x)$.
We claim that
\begin{equation}\label{eq_proof_prop_first_claim_asymptotic_ut}
u_t=u+t\psi (X\circ u)+b_t,
\end{equation}
where $b_t=o_{H^\frac{1}{2}\cap L^\infty}(t)$ as $t\to 0$, and $b_t$ has compact support (in $x$) uniformly for $t$ in a neighbourhood of $0$. This, in particular, will imply that $u_t$ is differentiable in $0$ with derivative $\psi X(u)$.
To prove the claim, we observe that for any $x\in \mathbb{R}^n$
\begin{equation}
\partial_t^2u_t(x)=\psi(x)dX(u_t(x))\partial_t u_t(x),
\end{equation}
and this expression is uniformly bounded (since $\partial_t u$ satisfies (\ref{eq_proof_prop_Cauchy_prop_PicLind})).
This implies that for any $x,y\in\mathbb{R}^n$
\begin{equation}\label{eq_proof_prop_estimate_to_bound_Gagliardo_for_asymptotic}
\begin{split}
\left\lvert \partial_t^2(u_t(x)-u_t(y))\right\rvert\leq & \lVert \psi\rVert_{L^\infty}\lVert dX\rVert_{L^\infty}\lvert \partial_t u(x)-\partial_t u_t(y)\rvert\\
&+\lVert \psi\rVert_{L^\infty}\left\lVert \partial_t u\right\rVert_{L^\infty}[dX]_{Lip}\lvert u_t(x)-u_t(y)\rvert\\
&+\lVert dX\rVert_{L^\infty}\left\lVert \partial_t u\right\rVert_{L^\infty}\lvert \psi(x)-\psi(y)\rvert,
\end{split}
\end{equation}
and by (\ref{eq_proof_prop_Cauchy_prop_PicLind})
\begin{equation}
\lvert \partial_t u(x)-\partial_t u(y)\rvert\leq \lVert \psi\rVert_{L^\infty}[X]_{Lip}\lvert u_t(x)-u_t(y)\rvert+\lVert X\rVert_{L^\infty}\lvert \psi(x)-\psi(y)\rvert.
\end{equation}
Now by Taylor's theorem we have that for $\varepsilon>0$, for any $x,y\in \mathbb{R}^n$
\begin{equation}\label{eq_proof_prop_estimate_to_bound_Gagliardo_seminorm}
\begin{split}
&\left\lvert [u_{t}(x)-u(x)-t\psi(x)X(u(x))]-[u_{t}(y)-u(y)-t\psi(y)X(u(y))]\right\rvert\\
\leq &\frac{t^2}{2}\max_{\lvert s\rvert\leq \varepsilon} \left\lvert \partial_s^2(u_s(x)-u_s(y))\right\rvert
\end{split}
\end{equation}
whenever $\lvert t\rvert\leq \varepsilon$. Therefore, applying estimate (\ref{eq_proof_prop_estimate_to_bound_Gagliardo_seminorm}) to the Gagliardo seminorm of $u_{t}-u-t\psi X(u)$ we obtain for any $t$ with $\lvert t\rvert\leq \varepsilon$
\begin{equation}
\begin{split}
&\left[u_{t}-u-t\psi X(u)\right]_{W^{\frac{1}{2},2}}^2\\
=&\int_{\mathbb{R}^n}\int_{\mathbb{R}^n}\frac{\left\lvert [u_{t}(x)-u(x)-t\psi(x)X(u(x))]-[u_{t}(y)-u(y)-t\psi(y)X(u(y))]\right\rvert^2}{\lvert x-y\rvert^{n+1}}dxdy\\
\leq& C t^4\left[ \int_{\mathbb{R}^n}\int_{\mathbb{R}^n}\frac{\max_{\lvert s\rvert\leq \varepsilon}\lvert u_s(x)-u_s(y)\rvert^2}{\lvert x-y\rvert^{n+1}}dxdy+ \int_{\mathbb{R}^n}\int_{\mathbb{R}^n}\frac{\lvert \psi(x)-\psi(y)\rvert^2}{\lvert x-y\rvert^{n+1}}dxdy\right]\\
\leq&Ct^4\left[ \sup_{\lvert s\rvert\leq \varepsilon}[u_s]_{W^{\frac{1}{2},2}}^2+[\psi]_{W^{\frac{1}{2},2}}^2\right]
\end{split}
\end{equation}
for some constant $C$ depending only on $\psi$ and $X$.
We claim that $[u_t]_{W^{\frac{1}{2},2}}^2$ is bounded for $t$ in a neighbourhood of $0$.
For this, it is enough to show that there exist constants $C_1$, $C_2$ depending only on $\psi$, $X$ such that
\begin{equation}\label{eq_proof_prop_estimate_Gagliardo_norm}
\lvert u_t(x)-u_t(y)\rvert\leq C_1\left(\lvert \psi(x)-\psi(y)\rvert+\lvert u(x)-u(y)\rvert\right) e^{C_2t}.
\end{equation}
Applying estimate (\ref{eq_proof_prop_estimate_Gagliardo_norm}) to the Gagliardo seminorm of $u_t$ one see that it is in fact bounded for $t$ in a neighbourhood of $0$.
To show (\ref{eq_proof_prop_estimate_Gagliardo_norm}) we argue as follows.
For $x,y\in\mathbb{R}^n$, $t\in\mathbb{R}$ define $\delta_{x,y}(t):=u_t(x)-u_t(y)$. Then
\begin{equation}
\delta_{x,y}'(t)=\partial_t u_t(x)- \partial_t u_t(y)=g(x, u_t(x))-g(y, u_t(y)).
\end{equation}
Then, for $t\in [a,b]$
\begin{equation}
\delta_{x,y}(t)=\delta_{x,y}(0)+\int_0^tg(x, u_s(x))-g(y, u_s(y))ds
\end{equation}
and therefore
\begin{equation}
\lvert \delta_{x,y}(t)\rvert \leq \lvert \delta_{x,y}(0)\rvert+\lvert\psi(x)-\psi(y)\rvert \int_0^t\lvert X(u_s(y))\rvert ds+\lVert \psi\rVert_{L^\infty}[X]_{Lip}\int_0^t\lvert \delta_{x,y}(s)\rvert ds.
\end{equation}
Thus, by Gronwall Lemma
\begin{equation}
\lvert \delta_{x,y}\rvert\leq C_1\left(\lvert \delta_{x,y}(0)\rvert+\lvert \psi(x)-\psi(y)\rvert \right)e^{C_2t}
\end{equation}
for some constants $C_1$, $C_2$ depending only on $\psi$ and $X$. This can be rewritten as
\begin{equation}
\lvert u_t(x)-u_t(y)\rvert \leq C_1\left(\lvert u(x)-u(y)\rvert+\lvert \psi(x)-\psi(y)\rvert\right) e^{C_2t}
\end{equation}
as desired. We conclude that $[b_t]_{\dot{W}^{\frac{1}{2},2}}= o_{\dot{H}^\frac{1}{2}}(t)$ as $t\to 0$.\\
On the other hand as we remarked above, $\partial_t^2 u_t(x)$ is bounded uniformly in $x$ and $t$. Therefore, by Taylor's theorem, for $x\in \mathbb{R}^n$
\begin{equation}\label{eq_proof_prop_Linf_estimate}
\lvert u_t(x)-u(x)-t\psi X(u(x))\rvert\leq \frac{t^2}{2}\left\lVert\partial_t^2 u_t\right\rVert_{L^{\infty}(\mathbb{R}^n\times\mathbb{R})}.
\end{equation}
This shows that $b_t=o_{L^\infty}(t)$ as $t\to 0$.
Finally we remark that $b_t=u_t(x)-u(x)-t\psi X(u(x))$ is supported in $\text{supp}(\psi)$ for any $t\in \mathbb{R}$ for which $u_t$ is defined, therefore(\ref{eq_proof_prop_Linf_estimate}) also implies that $b_t=o_{L^2}(t)$ as $t\to 0$. This concludes the proof of (\ref{eq_proof_prop_first_claim_asymptotic_ut}).\\
Next, let $\Omega:=\text{supp}(\psi)$ and let $\Gamma$ be a compact subset of $\mathbb{R}^n$ such that $\Omega\subset \Gamma$ and $\text{dist}(\Omega, \partial \Gamma)>1$.
We claim that for any $t\in (-\varepsilon, \varepsilon)$ , for any $x\in \Gamma^c$ there holds
\begin{equation}\label{eq_proof_thm_estimate_fraclap_outside_support_claim}
\left\lvert (-\Delta)^\frac{1}{4}b_t(x)\right\rvert\leq \lvert \Omega\rvert\frac{\lVert b_t\rVert_{L^\infty}}{\text{dist}(x, \Omega)^{n+\frac{1}{2}}}.
\end{equation}
In fact, for any $t\in (-\varepsilon, \varepsilon)$ , for any $x\in \Gamma^c$
\begin{equation}
(-\Delta)^\frac{1}{4}b_t(x)=\int_{\mathbb{R}^n}\frac{-b_t(y)}{\lvert x-y\rvert^{n+\frac{1}{2}}}dy
\end{equation}
and since, for any $x\in \Omega\supset\text{supp}(b_t)$, there holds $\lvert x-y\rvert\geq \text{dist}(x, \Omega)$, estimate (\ref{eq_proof_thm_estimate_fraclap_outside_support_claim}) follows.
Since $b_t=o_{L^\infty}(t)$ as $t\to 0$ estimate (\ref{eq_proof_thm_estimate_fraclap_outside_support_claim}) implies that for any $s> \frac{n}{n+\frac{1}{2}}$ there holds
\begin{equation}\label{eq_proof_thm_1_estimate_fraclap_rest_outside_Gamma}
\mathds{1}_{\Gamma^c}\left\lvert (-\Delta)^\frac{1}{4}b_t\right\rvert^s=o_{L^1}(t)\quad \text{as }t\to 0.
\end{equation}
We also observe that since $b_t=o_{H^\frac{1}{2}}(t)$, there holds $(-\Delta)^\frac{1}{4}b_t=o_{L^2}(t)$ as $t\to 0$.
Now we compute
\begin{equation}\label{eq_proof_prop_expansion_lagrangian}
\begin{split}
&E\left(u+t\psi X(u)+b_t\right)\\
=&\int_{\mathbb{R}^n}L\left(x,u(x)+t\psi(x)X(u(x))+b_t,\right.\\
 &\qquad\quad\left.(-\Delta)^\frac{1}{4}u(x)+t\psi(x)(-\Delta)^\frac{1}{4}(X\circ u)(x)+(-\Delta)^\frac{1}{4}b_t\right)dx\\
=&\int_{\mathbb{R}^n}L\left(x, u(x)+t\psi(x) X(u(x)), (-\Delta)^\frac{1}{4}u(x)+ t(-\Delta)^\frac{1}{4}(\psi (X\circ u))(x)\right) dx+o(t).
\end{split}
\end{equation}
For the last step, one can use the following argument:
By estimate (\ref{eq_prop_noeth_thm_vartar_condition_on_L}), the distance between the integrands in the second and third expressions can be bounded by
\begin{equation}
C\left(\lvert b_t\rvert^2+\lvert b_t\rvert^c +\left\lvert(-\Delta)^\frac{1}{4}b_t\right\rvert^a+\left\lvert (-\Delta)^\frac{1}{4}b_t\right\rvert^b\right)
\end{equation}
for any $x\in \mathbb{R}^n$. Since $b_t=o_{L^2\cap L^\infty}(t)$, $\lvert b_t\rvert^2+\lvert b_t\rvert^c=o_{L^1}$ as $t\to 0$. Moreover, by estimate (\ref{eq_proof_thm_1_estimate_fraclap_rest_outside_Gamma}) and the fact that $(-\Delta)^\frac{1}{4}b_t=o_{L^2}(t)$, $\left\lvert(-\Delta)^\frac{1}{4}b_t\right\rvert^a+\left\lvert (-\Delta)^\frac{1}{4}b_t\right\rvert^b=o_{L^1}(t)$ as $t\to 0$.\\
Now we observe that by the "Leibnitz rule" (\ref{eq_discussion_introduction_Leibnitz_rule_fraclaps}), for a.e. $x\in\mathbb{R}^n$
\begin{equation}\label{eq_proof_prop_product_fraclap_l2}
\begin{split}
&(-\Delta)^\frac{1}{4}\left(\psi (X\circ u)\right)(x)-(-\Delta)^\frac{1}{4}\psi(x)X(u(x))\\
=&(-\Delta)^\frac{1}{4}(X\circ u)(x)\psi(x)-\int_{\mathbb{R}^n}\text{d}_\frac{1}{4}\psi(x,y)\text{d}_\frac{1}{4}(X\circ u)(x,y)\frac{dy}{\lvert x-y\rvert^n}.
\end{split}
\end{equation}
Therefore, by Taylor's Theorem, we can rewrite (\ref{eq_proof_prop_expansion_lagrangian}) as
\begin{equation}\label{eq_proof_prop_expansion_lagrangian_2}
\begin{split}
&\int_{\mathbb{R}^n}L\left(x,u(x)+t\psi(x)X(u(x)), (-\Delta)^\frac{1}{4}u(x)+t\psi(x)(\Delta)^\frac{1}{4}(X\circ u)(x)\right)dx\\
&+t\int_{\mathbb{R}^n} D_{p}L \left(x, u(x)+t\psi(x)X(u(x)), (-\Delta)^\frac{1}{4}u(x)+t\psi(x)(-\Delta)^\frac{1}{4}(X\circ u)(x)\right)\\
&\phantom{+t\int_{\mathbb{R}^n}}\cdot \left( (-\Delta)^\frac{1}{4}\psi(x)X(u(x))-\int_{\mathbb{R}^n}\text{d}_\frac{1}{4}\psi(x,y)\text{d}_\frac{1}{4}(X\circ u)(x,y)\frac{dy}{\lvert x-y\rvert^n}\right)dx\\
&+\frac{t^2}{2}\int_{\mathbb{R}^n} D_{p}^2L\left(x, u(x)+t\psi(x)X(u(x)),\phantom{(-\Delta)^\frac{1}{4}u(x)}\right.\\
&\phantom{+\frac{t^2}{2}\int_{\mathbb{R}^n} D_{p_i}^2L\bigg(}\left. (-\Delta)^\frac{1}{4}u(x)+t\psi(x)(-\Delta)^\frac{1}{4}(X\circ u)(x)+\xi_x\right)\\
&\phantom{+t\int_{\mathbb{R}^n}}\cdot \left( (-\Delta)^\frac{1}{4}\psi(x)X(u(x))-\int_{\mathbb{R}^n}\text{d}_\frac{1}{4}\psi(x,y)\text{d}_\frac{1}{4}(X\circ u)(x,y)\frac{dy}{\lvert x-y\rvert^n}\right)^2dx\\&+o(t)
\end{split}
\end{equation}
for some vectors $\xi_x$ depending on $x$.
Here, by an abuse of notation, we denoted with a dot the action of the Hessian $D^2_pL$ on two (equal) vectors.
As $D_p^2L$ is assumed to be bounded, and since both $(-\Delta)^\frac{1}{4}\psi (X\circ u)$ and
\begin{equation}
\int_{\mathbb{R}^n}\text{d}_\frac{1}{4}\psi(x,y)\text{d}_\frac{1}{4}(X\circ u)(x,y)\frac{dy}{\lvert x-y\rvert^n}
\end{equation}
lie in $L^2(\mathbb{R}^n)$, the last term in (\ref{eq_proof_prop_expansion_lagrangian_2}) can be absorbed in $o(t)$.\\
Finally we observe that since the second derivatives of $L$ involving the second and the third variables are bounded, for any $x\in \mathbb{R}^n$, $t$ in a neighbourhood of $0$ we obtain
\begin{equation}\label{eq_proof_prop_last_estimate_distance}
\begin{split}
&\left\lvert D_pL\left(x,u(x)+t\psi(x)X(u(x)), (-\Delta)^\frac{1}{4}u(x)+t\psi(x)(-\Delta)^\frac{1}{4}(X\circ u)(x)\right)\right.\\
&\phantom{\vert}\left.-D_pL\left(x,u(x),(-\Delta)^\frac{1}{4}u(x)\right)\right\rvert\\
\leq & t [D_pL]_{Lip}\left(\psi(x)^2\left\lvert X(u(x))\right\rvert^2+\psi(x)^2\left\lvert(-\Delta)^\frac{1}{4}X(u(x))\right\rvert^2\right)^\frac{1}{2}.
\end{split}
\end{equation}
Since the right hand side of (\ref{eq_proof_prop_last_estimate_distance}) lies in $L^2(\mathbb{R}^n)$, we conclude that
\begin{equation}\label{eq_proof_prop_first_term_final_comparison}
\begin{split}
&E\left(u+t\psi (X\circ u)+b_t\right)\\
=&\int_{\mathbb{R}^n}L\left(x,u(x)+t\psi(x)X(u(x)), (-\Delta)^\frac{1}{4}u(x)+t\psi(x)(\Delta)^\frac{1}{4}(X\circ u)(x)\right)dx\\
&+t\int_{\mathbb{R}^n} D_{p_i}L \left(x, u(x), (-\Delta)^\frac{1}{4}u(x)\right)\cdot \left( (-\Delta)^\frac{1}{4}\psi(x)X^i(u(x))\right.\\
&\left.\phantom{+t\int_{\mathbb{R}^n}}-\int_{\mathbb{R}^n}\text{d}_\frac{1}{4}\psi(x,y)\text{d}_\frac{1}{4}(X^i\circ u)(x,y)\frac{dy}{\lvert x-y\rvert^n}\right)dx+o(t).
\end{split}
\end{equation}
On the other hand, let $\phi_t$ be the flow of $X$ and let $\overline{u_t}(x):=\phi_t(u(x))$ for any $x\in \mathbb{R}^n$.
By Lemma \ref{lem_derivative_of_composition_fraclap_14} the map $t\mapsto (-\Delta)^\frac{1}{4}\overline{u_t}(x)$ is of class $C^1$ for a.e. $x\in \mathbb{R}^n$, and
\begin{equation}
\partial_t\bigg\vert_{t=0}\left[(-\Delta)^\frac{1}{4}\overline{u_t}(x)\right]=(-\Delta)^\frac{1}{4}(X\circ u)(x).
\end{equation}

Then the derivative in $t$ at $0$ of
\begin{equation}\label{eq_proof_prop_difference_for_infinitesimal_symmetry}
\begin{split}
L\left(x, u(x)+tX(u)(x), (-\Delta)^\frac{1}{4}(u+tX(u))(x)\right)-L\left(x, \overline{u_t}(x), (-\Delta)^\frac{1}{4}\overline{u_t}(x)\right)\\
\end{split}
\end{equation}
exists and is equal to $0$ for a.e. $x\in\mathbb{R}^n$.
By Taylor's theorem, for a.e. $x\in\mathbb{R}^n$ (\ref{eq_proof_prop_difference_for_infinitesimal_symmetry}) can be rewritten as
\begin{equation}\label{eq_proof_prop_difference_for_infinitesimal_symmetry_rest}
\begin{split}
&\frac{t^2}{2}D^2L \left(x, u(x)+s_{x,t}X(u(x)), (-\Delta)^\frac{1}{4}u(x)+s_{x,t}(-\Delta)^\frac{1}{4}(X\circ u)(x)\right)\\
&\cdot \left(X(u(x)), (-\Delta)^\frac{1}{4}(X\circ u)(x)\right)^2-\frac{t^2}{2}\partial^2_t\bigg\vert_{t=s_{x,t}}L\left(x, \overline{u_t}(x), (-\Delta)^\frac{1}{4}\overline{u_t}(x)\right)
\end{split}
\end{equation}
for some $s_{x,t}$ between $0$ and $t$ depending on $x$ and $t$, where again, by abuse of notation, we denoted with a do the action of the Hessian $D^2L$ on two (equal) vectors. By assumption (\ref{eq_prop_condition_infinitesimal_symmetry}), the last term in (\ref{eq_proof_prop_difference_for_infinitesimal_symmetry_rest}) vanishes. Therefore (\ref{eq_proof_prop_difference_for_infinitesimal_symmetry_rest}) can be bounded by
\begin{equation}\label{eq_proof_prop_difference_for_infinitesimal_symmetry_rest_bound}
\begin{split}
&\frac{t^2}{2}\left\lVert D^2 L\right\rVert_{L^\infty}\left\lvert \left(X(u(x)), (-\Delta)^\frac{1}{4}(X\circ u)(x)\right)\right\rvert^2.
\end{split}
\end{equation}
for any $x\in\mathbb{R}^n$. Since $u$ and $(-\Delta)^\frac{1}{4}(X\circ u)$ are locally square integrable, we can rewrite (\ref{eq_proof_prop_difference_for_infinitesimal_symmetry_rest_bound}) as $t^2f_u(x)$ for any $x\in \mathbb{R}^n$, where $f_u$ is a locally integrable function on $\mathbb{R}^n$.
Therefore, for any $x\in \mathbb{R}^n$
\begin{equation}
\begin{split}
&L\left(x, u(x)+tX(u(x)), (-\Delta)^\frac{1}{4}u(x)+t(-\Delta)^\frac{1}{4}(X\circ u)(x)\right)\\
=&L\left(x, \overline{u_t}(x), (-\Delta)^\frac{1}{4}\overline{u_t}(x)\right)+t^2f_u(x)\\
=&L\left(x, u(x), (-\Delta)^\frac{1}{4}u(x)\right)+t^2f_u(x),
\end{split}
\end{equation}
for $t$ in a neighbourhood of $0$. The last equality follow from the assumption (\ref{eq_prop_condition_infinitesimal_symmetry}).
Thus, substituting $t$ with $t\psi(x)$ for any $x\in\mathbb{R}^n$ we obtain (for $t$ in a possibly smaller neighbourhood of $0$)
\begin{equation}\label{eq_proof_prop_second_term_final_comparison}
\begin{split}
&\int_{\mathbb{R}^n}L(x, u+t\psi(x)X(u(x)), (-\Delta)^\frac{1}{4}u(x)+t\psi(x)(-\Delta)^\frac{1}{4}(X\circ u)(x))dx\\
=&\int_{\mathbb{R}^n}L(x,u(x), (-\Delta)^\frac{1}{4}u(x))dx+t^2\int_{\text{supp}(\psi)}\psi(x)^2 f_u(x)dx\\
=&E(u)+o(t).
\end{split}
\end{equation}
Comparing (\ref{eq_proof_prop_first_term_final_comparison}) and (\ref{eq_proof_prop_second_term_final_comparison}) we obtain
\begin{equation}
\begin{split}
E(u_t)=&E(u+t\psi (X\circ u)+o(t))\\
=&E(u)+t\int_{\mathbb{R}^n} D_{p_i}L(u(x),(-\Delta)^\frac{1}{4}u(x))\left((-\Delta)^\frac{1}{4}\psi(x)X^i(u(x))\right.\\
&\left.-\int_{\mathbb{R}^n}\text{d}_\frac{1}{4}\psi(x,y)\text{d}_\frac{1}{4}(X^i\circ u)(x,y)\frac{dy}{\lvert x-y\rvert^n}\right)dx+o(t)
\end{split}
\end{equation}
Finally, since $u$ is a critical point of $E$ in $H^\frac{1}{2}(\mathbb{R}^n, \mathscr{M})$, and since, by (\ref{eq_proof_prop_first_claim_asymptotic_ut}), $u_t$ is differentiable in $0$, $E(u_t)=E(u)+o(t)$ as $t\to 0$.
Therefore
\begin{equation}\label{eq_proof_prop_statement_without_principal_value}
\begin{split}
0=&\int_{\mathbb{R}^n} D_{p_i}L(x,u(x),(-\Delta)^\frac{1}{4}u(x))\cdot\left((-\Delta)^\frac{1}{4}\psi(x)X^i(u(x))\right.\\
&\left.-\int_{\mathbb{R}^n}\text{d}_\frac{1}{4}\psi(x,y)\text{d}_\frac{1}{4}(X^i\circ u)(x,y)\frac{dy}{\lvert x-y\rvert^n}\right)dx\\
=&\int_{\mathbb{R}^n} L^i_u(x)(-\Delta)^\frac{1}{4}\psi(x)X^i(u(x))dx\\
&-\int_{\mathbb{R}^n}\int_{\mathbb{R}^n}L^i_u(x)\text{d}_\frac{1}{4}\psi(x,y)\text{d}_\frac{1}{4}(X^i\circ u)(x,y)\frac{dxdy}{\lvert x-y\rvert^n}.
\end{split}
\end{equation}
We remark that since $L_u^i\in L^2(\mathbb{R}^n)$, both integrals in (\ref{eq_proof_prop_statement_without_principal_value}) converge absolutely.
As this holds for any $\psi\in C^\infty_c(\mathbb{R}^n)$, in the sense of distributions
\begin{equation}
\begin{split}
&(-\Delta)^\frac{1}{4}\left(D_{p_i}L\left(x, u(x), (-\Delta)^\frac{1}{4}u(x)\right)\cdot X^i(u(x))  \right)\\
&-\text{div}_\frac{1}{4}\left( D_{p_i}L\left(x, u(x), (-\Delta)^\frac{1}{4}u(x)\right)\text{d}_\frac{1}{4}(X^i\circ u)(x,y) \right)=0.
\end{split}
\end{equation}
Finally we observe that rewriting the term $(-\Delta)^\frac{1}{4}\psi$ in (\ref{eq_proof_prop_statement_without_principal_value}) by means of the pointwise formula (\ref{eq_prel_pointwise_definition_fraclap}) we obtain for any $\phi\in C^\infty_c(\mathbb{R}^n, \mathbb{R})$
\begin{equation}
\begin{split}
0=&\int_{\mathbb{R}^n}\int_{\mathbb{R}^n}L_u^i(x)X^i(u(x))\frac{\phi(x)-\phi(y)}{\lvert x-y\rvert^{n+\frac{1}{2}}}\\
&\phantom{\int_{\mathbb{R}^n}\int_{\mathbb{R}^n}}-\frac{L_u^i(x)(X^i(u(x))-X^i(u(y))(\phi(x)-\phi(y))}{\lvert x-y\rvert^{n+\frac{1}{2}}}(\phi(x)-\phi(y))dxdy\\
=&\frac{1}{2}\int_{\mathbb{R}^n}\int_{\mathbb{R}^n}\left(L_u^i(x)X^i(u(y))-L_u^i(y)X^i(u(x))\right)(\phi(x)-\phi(y))dxdy
\end{split}
\end{equation}
Here we observe that also the first integral converges absolutely on $\mathbb{R}^n\times\mathbb{R}^n$. In fact, by H\"{o}lder's inequality, for $i\in \lbrace 1,...n\rbrace$
\begin{equation}\label{eq_proof_prop_principal_value_converges_absolutely}
\begin{split}
&\int_{\mathbb{R}^n}\int_{\mathbb{R}^n}\frac{L_u^i(x)^2(\psi(x)-\psi(y))^2}{\lvert x-y\rvert^{n+\frac{1}{2}}}dxdy\\
\leq &\int_{\mathbb{R}^n} L^i_u(x)^2\left(\int_{B_1(x)}\frac{[\psi]_{Lip}^2}{\lvert x-y\rvert^{n-\frac{3}{2}}}dy +\int_{B_1(x)^c}\frac{4\lVert \psi\rVert_{L^\infty}^2}{\lvert x-y\rvert^{n+\frac{1}{2}}}dy\right)dx.
\end{split}
\end{equation}
Since this is true for any $\psi\in C^\infty_c(\mathbb{R}^n,\mathbb{R})$, we conclude that
\begin{equation}
\text{div}_\frac{1}{2}\left(L_u^i(x)X^i(u(y))-L_u^i(y)X^i(u(x))\right)=0.
\end{equation}
\end{proof}

\begin{lem}\label{lem_derivative_of_composition_fraclap_14}
Let $i\in \mathbb{N}$, let $u\in \dot{H}^\frac{1}{2}(\mathbb{R}^n, \mathscr{M})$, let $X$ be a vector field of class $C^{i+3}$ on $\mathscr{M}$ and let $\phi_t$ denote its flow. Then for a.e. $x\in \mathbb{R}^n$ the function
\begin{equation}\label{eq_lem_derivative_of_composition_fraclap_flow_map}
\mathbb{R}\to \mathbb{R}^m, \quad t\mapsto (-\Delta)^\frac{1}{4}[\phi_t\circ u](x)
\end{equation}
is of class $C^i$, and for any $k\in \lbrace 0,...,i\rbrace$ there holds
\begin{equation}\label{eq_lem_explicit_formula_derivative_fraclap_composed_with_flow}
\partial_t^{k} \left[(-\Delta)^\frac{1}{4}\left( \phi_t\circ u\right) \right](x)=(-\Delta)^\frac{1}{4}\left[ \partial_t^{k}(\phi_t\circ u) \right](x).
\end{equation}
Moreover, given any bounded interval $I$ of $\mathbb{R}$, for any $k\in \lbrace 0,...,i\rbrace$,  $\partial_t^{k} \left[(-\Delta)^\frac{1}{4}\left( \phi_t\circ u\right) \right]$ is bounded a.e., uniformly in $t$ for $t\in I$, by a function of $x$ in $L^2(\mathbb{R}^n)$.\\
In particular, if we assume $X$ to be of class $C^4$, the map (\ref{eq_lem_derivative_of_composition_fraclap_flow_map}) is of class $C^1$ and Equation (\ref{eq_lem_explicit_formula_derivative_fraclap_composed_with_flow}) holds for $k=1$.
\end{lem}
\begin{proof}
Let $I$ be a bounded interval of $\mathbb{R}$.
First we observe that since $X$ is of class $C^{i+3}$, $\phi_t$ is of class $C^{i+2}$ jointly in $x$ and $t$. Therefore, for any $k\in \lbrace 0,...,i\rbrace$, we can write for a.e. $x\in\mathbb{R}^n$
\begin{equation}\label{eq_lem_derivative_of_fractional_Laplacian_composed_with_flow_i}
\begin{split}
&(-\Delta)^\frac{1}{4}\left[\partial_t^{k}(\phi_t\circ u)\right]=PV\int_{\mathbb{R}^n}\frac{\partial_t^{k}\phi_t(u(x))-\partial_t^{k}\phi_t(u(y))}{\lvert x-y\rvert^{n+\frac{1}{2}}}dy\\
=&\partial_t^{k}D\phi_t(u(x)) PV\int_{\mathbb{R}^n}\frac{u(x)-u(y)}{\lvert x-y\rvert^{n+\frac{1}{2}}}dy+\int_{\mathbb{R}^n}\partial_t^{k}f(x,y,t)dy,
\end{split}
\end{equation}
where for any $x,y\in \mathbb{R}^n$, $t\in \mathbb{R}$,
\begin{equation}
f(x,y,t):=\frac{\phi_t(u(x))-\phi_t(u(y))}{\lvert x-y\rvert^{n+\frac{1}{2}}}-D\phi_t(u(x)) \frac{u(x)-u(y)}{\lvert x-y\rvert^{n+\frac{1}{2}}}.
\end{equation}
Since $\phi_t$ is of class $C^{i+2}$, for any $x,y\in \mathbb{R}^n$ with $x\neq y$, $f(x,y,t)$ defines a function of $t$ of class $C^{i+1}$. Now we claim that for a.e. $x$, $\partial_t^{i}f(x,y,t)$ is uniformly bounded a.e. by a function of $y$ in $L^1(\mathbb{R}^n)$ for $t\in I$ (the bound might depend on $x$). On the one hand, this will imply, by Lebesgue's Dominated convergence Theorem, that the expression in (\ref{eq_lem_derivative_of_fractional_Laplacian_composed_with_flow_i}) is continuous in $t$, and since this will be true for any $k\in \lbrace 0,...,i\rbrace$, that the function defined in (\ref{eq_lem_derivative_of_composition_fraclap_flow_map}) is of class $C^i$. On the other hand, it will imply that the the $k^{th}$ derivatives and the integral in the second expression of the second line of (\ref{eq_lem_derivative_of_fractional_Laplacian_composed_with_flow_i}) can be inverted. Therefore, for any $k\in \lbrace 1,...,i\rbrace$
\begin{equation}\label{eq_lem_derivative_of_fractional_Laplacian_composed_with_flow_i_second_part}
\begin{split}
(-\Delta)^\frac{1}{4}\left[\partial_t^{k}(\phi_t\circ u)\right]=&\partial_t^{k}D\phi_t(u(x)) PV\int_{\mathbb{R}^n}\frac{u(x)-u(y)}{\lvert x-y\rvert^{n+\frac{1}{2}}}dy\\
&+\partial_t\int_{\mathbb{R}^n}\partial_t^{(k-1)}f(x,y,t)dy\\
=&\partial_t(-\Delta)^\frac{1}{4}\left[\partial_t^{(k-1)}(\phi_t\circ u)\right].
\end{split}
\end{equation}
Since Equation (\ref{eq_lem_derivative_of_fractional_Laplacian_composed_with_flow_i_second_part}) will be true for any $k\in \lbrace 1,...,i\rbrace$, we will obtain Equation (\ref{eq_lem_explicit_formula_derivative_fraclap_composed_with_flow}).\\
In order to construct the $L^1$-bound for the integrand of the second term of the second line of (\ref{eq_lem_derivative_of_fractional_Laplacian_composed_with_flow_i}), we first observe that by Taylor's Theorem, for any $x,y\in \mathbb{R}^n$, $t\in \mathbb{R}$
\begin{equation}\label{eq_proof_lemma_estimate_integrand_Linf_Sobolev}
\left\lvert \partial_t^{k}f(x,y,t)\right\rvert\leq \partial_t^{k}D^2\phi_t(\xi_{x,y,t})\cdot\frac{(u(x)-u(y))^2}{\lvert x-y\rvert^{n+\frac{1}{2}}}
\end{equation}
for some $\xi_{x,y,t}$ between $u(x)$ and $u(y)$, depending on $x$, $y$ and $t$. Here, by an abuse of notation, we denote with a dot the action of the quadratic form $\partial_t^{k}D^2\phi_t(\xi_{x,y,t})$ on two (equal) vectors.
Therefore, for any $t\in I$
\begin{equation}\label{eq_proof_lemma_bound_derivative_fraclap_composed_with_flow}
\begin{split}
&\left\lvert\partial_t^{k}D^2\phi_t(\xi_{x,y,t})\cdot\frac{(u(x)-u(y))^2}{\lvert x-y\rvert^{n+\frac{1}{2}}} \right\rvert
\leq \left\lVert  \partial_t^{k}D^2\phi_t\right\rVert_{L^\infty(\mathscr{M}\times I)}\frac{\lvert u(x)-u(x)\rvert^2}{\lvert x-y\rvert^{n+\frac{1}{2}}}.
\end{split}
\end{equation}
Now we observe that since $u\in \dot{H}^\frac{1}{2}(\mathbb{R}^n)$, the term on the right hand side of (\ref{eq_proof_lemma_bound_derivative_fraclap_composed_with_flow}) lies in $L^1(\mathbb{R}^n\times \mathbb{R}^n)$, and thus for a.e. $x\in \mathbb{R}^n$, it defines an integrable function of $y$. This shows the existence of the desired $L^1$-bound.\\
Finally we show that for any $k\in \lbrace 0,...,i\rbrace$, $\partial_t^{k} \left[(-\Delta)^\frac{1}{4}\left( \phi_t\circ u\right) \right]$ is bounded a.e., uniformly in $t$ for $t\in I$, by a function of $x$ in $L^2(\mathbb{R}^n)$. To this end we first observe that since $(-\Delta)^\frac{1}{4}u\in L^2(\mathbb{R}^n)$, the first term in the second line of (\ref{eq_lem_derivative_of_fractional_Laplacian_composed_with_flow_i}) is bounded for any $t\in I$ by the square integrable function $\left\lVert \partial_t^{i}D\phi_t(u(x))\right\rVert_{L^\infty(\mathscr{M}\times I)}\left\lvert (-\Delta)^\frac{1}{4}u(x)\right\rvert$. In order to estimate the second term in the second line of (\ref{eq_lem_derivative_of_fractional_Laplacian_composed_with_flow_i}) we first observe that, by estimate (\ref{eq_proof_lemma_estimate_integrand_Linf_Sobolev}), for a.e. $x\in \mathbb{R}^n$, for any $t\in I$
\begin{equation}\label{eq_proof_lemma_before_applying_Leibnitz}
\left\lvert\int_{\mathbb{R}^n}\partial_t^{k}f(x,y,t)dy\right\rvert\leq \left\lVert \partial_t^{k}D^2\phi_t\right\rVert_{L^\infty(\mathscr{M}\times I)}\int_{\mathbb{R}^n}\frac{\lvert u(x)-u(y)\rvert^2}{\lvert x-y\rvert^{n+\frac{1}{2}}}dy.
\end{equation}
Now by the "Leibnitz rule" (\ref{eq_discussion_introduction_Leibnitz_rule_fraclaps}), for a.e. $x\in\mathbb{R}^n$ we can rewrite the integral on the right hand side of (\ref{eq_proof_lemma_before_applying_Leibnitz}) as
\begin{equation}
\int_{\mathbb{R}^n}\frac{\lvert u(x)-u(x)\rvert^2}{\lvert x-y\rvert^{n+\frac{1}{2}}}dy=2u(x)\cdot (-\Delta)^\frac{1}{4}u(x)-(-\Delta)^\frac{1}{4}u^2(x).
\end{equation}
Since $u\in \dot{H}^\frac{1}{2}\cap L^\infty(\mathbb{R}^n)$, also $u^2\in \dot{H}^\frac{1}{2}\cap L^\infty(\mathbb{R}^n)$ by Lemma \ref{lem_Hs_cap_Linfty_is_an_algebra}, therefore the integral on the right hand side of (\ref{eq_proof_lemma_bound_derivative_fraclap_composed_with_flow}) lies in $L^2(\mathbb{R}^n)$. Therefore we conclude that $\partial_t^{k} \left[(-\Delta)^\frac{1}{4}\left( \phi_t\circ u\right) \right]$ is bounded a.e., uniformly in $t$ for $t\in I$, by a function of $x$ in $L^2(\mathbb{R}^n)$.\\
In the following we give an alternative proof of the fact that the function $t\mapsto (-\Delta)^\frac{1}{4}[\phi_t\circ u](x)$ is differentiable for a.e. $x\in \mathbb{R}^n$, and that Equation (\ref{eq_lem_explicit_formula_derivative_fraclap_composed_with_flow}) holds for $n=1$ whenever $X$ is of class $C^5$.\\
First we observe that by the semigroup properties of the flow, it is enough to show the statement for $t=0$.
Since the vector field $X$ is of class $C^5$, its flow $\phi_t$ is of class $C^4$ jointly in $t$ and $x$.
Applying Taylor's theorem in $t$, we obtain
\begin{equation}
\begin{split}
(-\Delta)^\frac{1}{4}[\phi_t\circ u](x)=&(-\Delta)^\frac{1}{4}u(x)+t(-\Delta)^\frac{1}{4}\left[X\circ u\right](x)\\
&+t^2PV\int_{\mathbb{R}^n}\frac{\int_0^1(1-r)\partial_s^2\vert_{s=rt}(\phi_s(u(x))
-\phi_s(u(y)))dr}{\lvert x-y\rvert^{n+\frac{1}{2}}}dy.
\end{split}
\end{equation}
To estimate the last term, we can apply Taylor's theorem to $\partial_t^2\phi_t $ (in the $x$-argument). We obtain
\begin{equation}
\begin{split}
\partial_t^2(\phi_t(u(x))-\phi_t(u(y)))=&D_x[\partial_t^2\phi_t](u(x))(u(x)-u(y))\\
&+\sum_{\lvert \beta\rvert=2}R_{\beta,t,x}(y)(u(x)-u(y))^\beta,
\end{split}
\end{equation}
where the sum is taken over all $n$-multiindices $\beta$ of degree 2 and for any $\beta$, $t\in \mathbb{R}$, $x\in \mathbb{R}^n$
\begin{equation}
R_{\beta,t,x}(y)=\int_0^1(1-s)D^\beta\partial_t^2\phi_t(u(x)+s(u(y)-u(x)))ds,
\end{equation}
so that
\begin{equation}
\left\lvert R_{\beta,x,t}(y)\right\rvert\leq \lVert D^2\partial_t^2\phi_t\rVert_{L^\infty}.
\end{equation}
Therefore we obtain
\begin{equation}
\begin{split}
&PV\int_{\mathbb{R}^n}\frac{\int_0^1(1-r)\partial_t^2\vert_{t=r}(\phi_t(u(x))-\phi_t(u(y)))dr}{\lvert x-y\rvert^{n+\frac{1}{2}}}dy\\
=&\int_0^1(1-r)D_x[\partial_s^2\vert_{s=rt}\phi_s](u(x))dr(-\Delta)^\frac{1}{4}u(x)+A(t,x),
\end{split}
\end{equation}
where 
\begin{equation}
\begin{split}
\lvert A(t,x)\rvert= &\left\lvert PV\int_{\mathbb{R}^n}\frac{\int_0^1(1-r)\sum_{\lvert \beta\rvert=2}R_{\beta, rt, x}(y)(u(x)-u(y))^\beta}{\lvert x-y\rvert^{n+\frac{1}{2}}} \right\rvert\\
\leq& n^2\lVert D^2\partial_t^2\phi_{rt}\rVert_{L^\infty}U(x)
\end{split}
\end{equation}
for any $x\in \mathbb{R}^n$, $t\in \mathbb{R}$. Here the map $U$ is defined as follows;
\begin{equation}
U: \mathbb{R}^n\to \mathbb{R}, \quad x\mapsto \int_{\mathbb{R}^n}\frac{\lvert u(y)-u(x)\rvert^2}{\lvert x-y\rvert^{n+\frac{1}{2}}}dy
\end{equation}
Since $u\in \dot{H}^\frac{1}{2}(\mathbb{R}^n)$, by Fubini's theorem $U(x)$ is well defined for a.e. $x\in\mathbb{R}^n$, and the map $U$ belongs to $L^1(\mathbb{R}^n)$.
Therefore, for almost every $x\in \mathbb{R}^n$, the map $t\mapsto (-\Delta)^\frac{1}{4}[\phi_t\circ u](x)$ is differentiable in $0$ with
\begin{equation}
\partial_t\vert_{t=0}(-\Delta)^\frac{1}{4}[\phi_t\circ u](x)=(-\Delta)^\frac{1}{4}[\partial_t\vert_{t=0}\phi_t\circ u](x).
\end{equation}

\end{proof}

\begin{rem}
We observe that in some cases of interest, the flow $\phi_t$ of the vector field $X$ is an affine function on $\mathbb{R}^n$ for $t$ in a neighbourhood of $0$, where the linear part is represented by a matrix $A_t$, of class $C^1$ as a function of $t$. This is for instance the case for the generators of  translations, rotations and dilations. In these cases it easier to show that the map $t\mapsto (-\Delta)^\frac{1}{4}[\phi_t\circ u](x)$ is differentiable for a.e. $x\in \mathbb{R}^n$, even if $X$ is only of class $C^1$, and that for any $i\in \mathbb{N}$, $(-\Delta)^\frac{1}{4}\left[\partial_t^{(i)}\phi_t\circ u\right]$ can be bounded by a square integrable function uniformly in $t$ (for $t$ in a neighbourhood of $0$). 
Indeed, assume that there exists $\varepsilon>0$ such that for $t\in (-\varepsilon, \varepsilon)$ the flow $\phi_t$ of $X$ is an affine function represented by a matrix $A_t$ and a vector $b_t$. Then for $i\in \mathbb{N}$, for $x\in \mathbb{R}^n$ and $t\in (-\varepsilon, \varepsilon)$
\begin{equation}
\begin{split}
&\partial_t^{(i)}(-\Delta)^\frac{1}{4}[\phi_t\circ u](x)=\partial_t^{(i)} PV\int_{\mathbb{R}^n}\frac{(A_t \cdot u(x)+b_t)-(A_t \cdot u(y)+b_t)}{\lvert x-y\rvert^{n+\frac{1}{2}}}dy\\
=&\partial_t^{(i)} A_t\cdot PV\int_{\mathbb{R}^n}\frac{u(x)-u(y)}{\lvert x-y\rvert^{n+\frac{1}{2}}}dy= \partial_t^{(i)}A_t(-\Delta)^\frac{1}{4}u(x).
\end{split}
\end{equation}
By the assumptions on $A_t$, the derivative is well defined. Moreover, if $t\in (-\varepsilon, \varepsilon)$,
\begin{equation}
\left\lvert \partial_t^{(i)}(-\Delta)^\frac{1}{4}[\phi_t\circ u]\right\rvert\leq \max_{\lvert t\rvert\leq \varepsilon}\left\lVert \partial_t^{(i)} A_t \right\rVert \left\lvert (-\Delta)^\frac{1}{4}u\right\rvert,
\end{equation}
and the function on the right hand side belongs to $L^2(\mathbb{R}^n)$.\\
In particular, for vector fields of this kind, Theorem \ref{prop_noether_thm_variaton_target_cor} holds even if we assume that $X$ is only of class $C^1$.
\end{rem}

\subsection{Half harmonic maps into spheres}\label{ssec: harmonic maps into spheres}
As observed above, an interesting example of Lagrangian of the form (\ref{eq_prop_definition_of_energy}) is given by the half Dirichlet energy
\begin{equation}\label{eq_definition_fractional_Dirichlet_energy_2}
E(u)=\int_{\mathbb{R}^n}\left\lvert (-\Delta)^\frac{1}{4}u(x)\right\rvert^2dx,
\end{equation}
defined for functions $u\in \dot{H}^\frac{1}{2}(\mathbb{R}^n, \mathbb{S}^{m})$, where $m\in \mathbb{N}$. Here we consider $\mathbb{S}^{m}$ as a submanifold of $\mathbb{R}^{m+1}$.
In \cite{divcurl}, K. Mazowiecka and A. Schikorra showed that critical points of (\ref{eq_definition_fractional_Dirichlet_energy_2}) are continuous, following an argument analogous to the one used in \cite{Heleinarticle} for the local case (for $n=2$). An important step in the proof of the continuity of minimizers consists in showing that for any $i, k\in \lbrace 1,...,m+1 \rbrace$ the quantity $\Omega_{ik}$, defined in (\ref{eq_lem_definition_Omega_i_k}) satisfies
\begin{equation}\label{eq_div_1_2_of_Omega_vanishes_intro_sec}
\text{div}_\frac{1}{2} \Omega_{ik}=0.
\end{equation}
From the analogy with the local case, one might expect (\ref{eq_div_1_2_of_Omega_vanishes_intro_sec}) to be a direct consequence of the fractional analogous to the Noether theorem (Theorem \ref{prop_noether_thm_variaton_target_cor}), but Theorem \ref{prop_noether_thm_variaton_target_cor} yields a different $\frac{1}{2}$-divergence free quantity (Lemma \ref{lem_1_2_unexpected_divergence_free_quantity}). Nevertheless we will show that we can recover (\ref{eq_div_1_2_of_Omega_vanishes_intro_sec}) by means of a further computation (Lemma \ref{lem_1_2_div_of_Omega_vanishes} or Lemma \ref{lem_alternative_argument_div12_Omega_ik=0}). Later we will show how (\ref{eq_div_1_2_of_Omega_vanishes_intro_sec}) implies the continuity of critical points of (\ref{eq_definition_fractional_Dirichlet_energy_2}) (Theorem \ref{prop_critical_points_are_continuous}).\\

\begin{lem}\label{lem_1_2_unexpected_divergence_free_quantity}
Let $n\in\mathbb{N}_{>0}$ and let $u\in \dot{H}^\frac{1}{2}(\mathbb{R}^n,\mathbb{S}^m)$ be a critical point of (\ref{eq_definition_fractional_Dirichlet_energy_2})
in $\dot{H}^\frac{1}{2}(\mathbb{R}^n,\mathbb{S}^m)$.
For any $i, k\in \lbrace 1,...,m+1 \rbrace$, let
\begin{equation}\label{eq_lem_definition_Lambda_ik}
\begin{split}
\Lambda_{ik}(x,y):=& \left((-\Delta)^\frac{1}{4}u^k(x)-(-\Delta)^\frac{1}{4}u^k(y)\right)u^i(y)\\
&-\left((-\Delta)^\frac{1}{4}u^i(x)-(-\Delta)^\frac{1}{4}u^i(y)\right)u^k(y)
\end{split}
\end{equation}
for any $x,y\in\mathbb{R}^n$.
Then
\begin{equation}\label{eq_div_1_2_of_Lambda_vanishes}
\text{div}_\frac{1}{2} \Lambda_{ik}=0.
\end{equation}
\end{lem}

\begin{proof}
First we observe that the integrand in (\ref{eq_definition_fractional_Dirichlet_energy_2}) is invariant under rotations in the target: for any $R\in SO(m+1)$ and any $u\in \dot{H}^\frac{1}{2}(\mathbb{R}^n, \mathbb{S}^m)$,
$(-\Delta)^\frac{1}{4}R u(x)=R(-\Delta)^\frac{1}{4}u(x)$ for any $x\in \mathbb{R}^n$ by the pointwise definition of fractional Laplacian. Therefore
\begin{equation}\label{Es_rot_inv_integrand} 
\left\lvert(-\Delta)^{\frac{1}{4}}(R u(x))\right\rvert^2=\left\lvert(-\Delta)^{\frac{1}{4}}u(x)\right\rvert^2\text{ for a.e. }x\in\mathbb{R}^n.
\end{equation}
We recall that the Lie algebra $\mathfrak{so}(m+1)$ of $SO(m+1)$ consists of the set of antisymmetric $(m+1)\times(m+1)$-matrices. Thus any $A\in\mathfrak{so}(n+1)$ induces a vector field $X_A$ on $\mathbb{S}^m$ given by
\begin{equation}
X_A:\mathbb{S}^m\to T\mathbb{S}^m, x\mapsto (x, A x),
\end{equation}
whose flow $\phi^A_t$ is a rotation at any $t\in\mathbb{R}$. Explicitly, for any $t\in\mathbb{R}$
\begin{equation}\label{eq_proof_lemma_definition_of_phiA_as_mult_by_exponential_fct}
\phi^A_t(x)=exp(tA) x,
\end{equation}
for $x\in \mathbb{S}^m$, where $exp$ denote the exponential map from $\mathfrak{so}(m+1)$ to $SO(m+1)$.
Equation (\ref{Es_rot_inv_integrand}) implies that for any $A\in\mathfrak{so}(n+1)$ and any $t\in \mathbb{R}$, $x\in\mathbb{R}^n$,
\begin{equation}\label{eq_prel_integ_inv}
\left\lvert(-\Delta)^{\frac{1}{4}}\left(\phi^A_t\circ u\right)(x)\right\rvert^2=\left\lvert(-\Delta)^{\frac{1}{4}} u(x)\right\rvert^2.
\end{equation}
This implies that the Lagrangian
\begin{equation}
L(x, X, Y)= \lvert Y\rvert^2\text{ for }(x, X, Y)\in \mathbb{R}^n\times\mathbb{R}^{m+1}\times\mathbb{R}^{m+1},
\end{equation}
corresponding to the energy (\ref{eq_definition_fractional_Dirichlet_energy_2}), satisfies condition (\ref{eq_prop_condition_infinitesimal_symmetry}) for $X_A$ for any $A\in\mathfrak{so}(m+1)$.
Moreover, by Remark \ref{rem_technical_assumtions_noether_thm}, assumption (\ref{eq_prop_assumption_integrability_of_p_derivative}) is satisfied. Therefore Theorem \ref{prop_noether_thm_variaton_target_cor} applies to (\ref{eq_definition_fractional_Dirichlet_energy_2}) for $X_A$ for any $A\in\mathfrak{so}(m+1)$.
Now for $i,k\in \lbrace 1, 2,..., m+1\rbrace$, $i\neq j$ we can choose $A=(a_{lm})_{lm}$ such that
\begin{equation}\label{eq_proof_prel_matrixdef}
a_{lm} = \begin{cases}
1 &\text{if $(l,m)=(k,i)$}\\
-1 &\text{if $(l,m)=(i,k)$}\\
0 &\text{otherwise.}
\end{cases}
\end{equation}
Then $A\in\mathfrak{so}(m+1)$, and for any $x\in \mathbb{R}^n$, 
\begin{equation}
X_A(u(x))=A u(x)=u^i(x)e^k-u^k(x)e^i,
\end{equation}
where $e^j$ denotes the $j^{th}$ unit vector in $\mathbb{R}^{m+1}$. Then Theorem \ref{prop_noether_thm_variaton_target_cor} for $X_A$ implies
\begin{equation}
\begin{split}
\text{div}_\frac{1}{2}&\left[ \left((-\Delta)^\frac{1}{4}u^k(x)-(-\Delta)^\frac{1}{4}u^k(y)\right)u^i(y)\right.\\
&\left.\phantom{[}-\left((-\Delta)^\frac{1}{4}u^i(x)-(-\Delta)^\frac{1}{4}u^i(y)\right)u^k(y)   \right]=0.
\end{split}
\end{equation}
\end{proof}

We will now see how Equation (\ref{eq_div_1_2_of_Omega_vanishes_intro_sec}) can be deduced from Lemma 
\ref{lem_1_2_unexpected_divergence_free_quantity}.
\begin{lem}\label{lem_alternative_argument_div12_Omega_ik=0}
For any $i, k\in \lbrace 1,...,m+1 \rbrace$, let
\begin{equation}\label{eq_lem_definition_Omega_i_k}
\begin{split}
\Omega_{ik}(x,y):=& u^k(x)\text{d}_\frac{1}{2}u^i(x,y)-u^i(x)\text{d}_\frac{1}{2}u^k(x,y).
\end{split}
\end{equation}
Then
\begin{equation}\label{eq_lem_vartar_Omega=0_ref}
\text{div}_\frac{1}{2}\left(\Omega_{ik}\right)=0.
\end{equation}
\end{lem}
\begin{proof}
Let's first consider a function $v\in\mathscr{S}(\mathbb{R}^n,\mathbb{R}^m)+\mathbb{R}^m$. We observe that for $i, k\in \lbrace 1,...,m+1 \rbrace$ if we denote by $\Omega_{ik}^v$ the expression defined in (\ref{eq_lem_definition_Omega_i_k}) for the function $v$ we obtain for any $\phi\in C^\infty_c(\mathbb{R}^n,\mathbb{R})$
\begin{equation}
\begin{split}
&\text{div}_\frac{1}{2}\left(\Omega_{ik}^v\right)[\phi]=\int_{\mathbb{R}^n}v^i(x)\int_{\mathbb{R}^n}\frac{(v^k(y)-v^k(x))(\phi(x)-\phi(y))}{\lvert x-y\rvert^{n+1}}\\
&\phantom{\text{div}_\frac{1}{2}\left(\Omega_{ik}^v\right)[\phi]=\int_{\mathbb{R}^n}}-v^k(x)\int_{\mathbb{R}^n}\frac{(v^i(y)-v^i(x)))(\phi(x)-\phi(y))}{\lvert x-y\rvert^{n+1}}dydx\\
=&\int_{\mathbb{R}^n}v^i(x)\left( -(-\Delta)^\frac{1}{2}(v^k\phi)+(-\Delta)^\frac{1}{2}v^k(x)\phi(x)+v^k(x)(-\Delta)^\frac{1}{2}\phi(x)\right)\\
&\phantom{\int_{\mathbb{R}^n}}-v^k(x)\left( -(-\Delta)^\frac{1}{2}(v^i\phi)+(-\Delta)^\frac{1}{2}v^i(x)\phi(x)+v^i(x)(-\Delta)^\frac{1}{2}\phi(x)\right)\\
=&2\int_{\mathbb{R}^n}\left(v^i(x)(-\Delta)^\frac{1}{2}v^k(x)-(-\Delta)^\frac{1}{2}v^i(x)v^k(x)\right)\phi(x).
\end{split}
\end{equation}
On the other hand, if denote by $\Lambda_{ik}^v$ the expression defined in (\ref{eq_lem_definition_Lambda_ik}) for the function $v$, we obtain for any $\phi\in C^\infty_c(\mathbb{R}^n,\mathbb{R})$
\begin{equation}
\begin{split}
&\text{div}_\frac{1}{2}\left(\Lambda_{ik}^v\right)\\
=&\int_{\mathbb{R}^n}\int_{\mathbb{R}^n}\left(v^i(x)\left((-\Delta)^\frac{1}{4}v^k(x)-(-\Delta)^\frac{1}{4}v^k(y)\right)-(-\Delta)^\frac{1}{4}v^k(x)(v^i(x)-v^i(y))\right.\\
&\left.+(\Delta)^\frac{1}{4}v^i(x)(v^k(x)-v^k(y))-v^k(x)\left((-\Delta)^\frac{1}{4}v^i(x)-(-\Delta)^\frac{1}{4}v^i(y)\right)\right)\phi(x)\frac{dydx}{\lvert x-y\rvert^{n+\frac{1}{2}}}\\
=&\int_{\mathbb{R}^n}\left(v^i(x)(-\Delta)^\frac{1}{2}v^k(x)-v^k(x)(-\Delta)^\frac{1}{2}v^i(x)\right)\phi(x)dx
\end{split}
\end{equation}
Therefore we conclude that
\begin{equation}\label{eq_proof_lemma_equivalence_Lambda_Omega_alternative_smooth}
\text{div}_\frac{1}{2}\left(\Lambda_{ik}^v\right)=2\text{div}_\frac{1}{2}\left(\Omega_{ik}^v\right)
\end{equation}
Now let $u\in \dot{H}^\frac{1}{2}(\mathbb{R}^n,\mathbb{S}^m)$.

By Lemma \ref{lem_approx_sobolev}, there exists a sequence $(u_n)_n$ in $\mathscr{S}(\mathbb{R}^n,\mathbb{R}^m)+\mathbb{R}^m$, bounded in $L^\infty(\mathbb{R}^n)$, such that $u_n\to u$ in $\dot{H}^\frac{1}{2}(\mathbb{R}^n)$ and locally in $L^2(\mathbb{R}^n)$. We claim that, by approximation through such a sequence, Equation (\ref{eq_proof_lemma_equivalence_Lambda_Omega_alternative_smooth}) remains true for $u$. For the right hand side, the convergence follows form Lebesgue's Dominated convergence Theorem along a subsequence of $(u_n)_n$, for which the convergence takes place a.e.. For the left hand side, it will be enough to show that for any $\phi\in C^\infty_c(\mathbb{R}^n,\mathbb{R})$, $\phi u_n\rightharpoondown\phi u$ weakly in $H^\frac{1}{2}(\mathbb{R}^n)$ along a subsequence. To proof the claim, we observe that by the argument in the proof of Lemma \ref{lem_Hs_cap_Linfty_is_an_algebra}, for any $n\in \mathbb{N}$
\begin{equation}
\lVert u_n \phi\rVert_{H^\frac{1}{2}}\leq 2\left(\lVert u_n\rVert [\phi]_{\dot{H}^\frac{1}{2}}+\lVert \phi\rVert [u_n]_{\dot{H}^\frac{1}{2}}\right)+\lvert \text{supp}(\phi)\rvert^\frac{1}{2}\lVert u_n\rVert_{L^\infty}.
\end{equation}
Thus since the sequence $(u_n)_n$ is bounded in $\dot{H}^\frac{1}{2}(\mathbb{R}^n)$ and $L^\infty(\mathbb{R}^n)$, the sequence $(u_n \phi)_n$ is bounded in $H^\frac{1}{2}(\mathbb{R}^n)$. Therefore, by Banach-Alaoglu Theorem, there exists $G\in H^\frac{1}{2}(\mathbb{R}^n,\mathbb{R}^m)$ and a subsequence of $(u_n \phi)$ converging weakly in $H^\frac{1}{2}(\mathbb{R}^n)$ to $G$. Since $u_n \phi\to u \phi $ in $L^2(\mathbb{R}^n)$, we conclude that $G=u\phi$. This concludes the proof of the claim.
Therefore, for any $H^\frac{1}{2}(\mathbb{R}^n,\mathbb{S}^1)$, Equation (\ref{eq_proof_lemma_equivalence_Lambda_Omega_alternative_smooth}) remains true. In particular, if $u$ is a critical point of (\ref{eq_definition_fractional_Dirichlet_energy_2}) in $H^\frac{1}{2}(\mathbb{R}^n,\mathbb{S}^1)$, by Lemma \ref{lem_1_2_unexpected_divergence_free_quantity} we have
\begin{equation}
\text{div}_\frac{1}{2}\left(\Omega_{ik}\right)=0
\end{equation}

\end{proof}

Next we present a different argument to deduce Equation (\ref{eq_div_1_2_of_Omega_vanishes_intro_sec}) from Lemma \ref{lem_1_2_unexpected_divergence_free_quantity}.
\begin{lem}\label{lem_1_2_div_of_Omega_vanishes}
Let $n\in\mathbb{N}_{>0}$ and let $u\in \dot{H}^\frac{1}{2}(\mathbb{R}^n,\mathbb{S}^m)$ be a critical point of (\ref{eq_definition_fractional_Dirichlet_energy_2})
in $\dot{H}^\frac{1}{2}(\mathbb{R}^n,\mathbb{S}^m)$.
For any $i, k\in \lbrace 1,...,m+1 \rbrace$, let $\Omega_{ik}$ be defined as in (\ref{eq_lem_definition_Omega_i_k}).
Then $(-\Delta)^\frac{1}{2}u$ belongs to $L^1(\mathbb{R}^n)$ and there holds
\begin{equation}\label{eq_lem_explcit_form_fraclap_on_sphere}
(-\Delta)^\frac{1}{2}u=\frac{1}{2}u\int_{\mathbb{R}^n}\lvert \text{d}_\frac{1}{2}u(x,y)\rvert^2\frac{dy}{\lvert x-y\rvert^n}
\end{equation}
almost everywhere in $\mathbb{R}^n$. In particular for a.e. $x\in \mathbb{R}^n$
\begin{equation}\label{eq_lem_Euler_Lagrange_vartar}
(-\Delta)^\frac{1}{2}u\wedge u=0.
\end{equation}
Moreover
\begin{equation}\label{eq_div_1_2_of_Omega_vanishes}
\text{div}_\frac{1}{2} \Omega_{ik}=0
\end{equation}
\end{lem}

\begin{rem}
In fact, for functions $u\in \dot{H}^\frac{1}{2}(\mathbb{R}^n,\mathbb{S}^m)$, it was shown in \cite{MillotSire} that Equation (\ref{eq_lem_explcit_form_fraclap_on_sphere}) is equivalent to Equation (\ref{eq_lem_Euler_Lagrange_vartar}), and by Lemma 3.1 in \cite{divcurl}, Equation (\ref{eq_lem_Euler_Lagrange_vartar}) is equivalent to Equation (\ref{eq_div_1_2_of_Omega_vanishes_intro_sec}).
\end{rem}

\begin{proof}
For $i\in\lbrace 1,..., m+1\rbrace$ let $v^i:=(-\Delta)^\frac{1}{4}u^i$.
First we claim that
\begin{equation}
\text{div}_\frac{1}{4}\left(\text{d}_\frac{1}{4}v^i(x,y)u^k(x)-\text{d}_\frac{1}{4}v^k(x,y)u^i(x)\right)=(-\Delta)^\frac{1}{2}u^iu^k-(-\Delta)^\frac{1}{2}u^ku^i
\end{equation}
as distributions, for any $i, k\in \lbrace 1,...,m+1\rbrace$.
To prove the claim we first observe that for any $x,y\in\mathbb{R}^n$
\begin{equation}
\begin{split}
\text{d}_\frac{1}{4}v^i(x,y)u^k(x)-\text{d}_\frac{1}{4}v^k(x,y)u^i(x)=&\text{d}_\frac{1}{4}\left(v^iu^k-v^ku^i\right)(x,y)\\
&-v^i(y)\text{d}_\frac{1}{4}u^k(x,y)+v^k(y)\text{d}_\frac{1}{4}u^i(x,y)
\end{split}
\end{equation}
Applying $\text{div}_\frac{1}{4}$ on both sides we obtain
\begin{equation}\label{eq_proof_lem_first_form_fraclap_as_div_of_d}
\begin{split}
&\text{div}_\frac{1}{4}(\text{d}_\frac{1}{4}v^i(x,y) u^k(x)-\text{d}_\frac{1}{4}v^k(x,y) u^i(x))\\
=&(-\Delta)^\frac{1}{4}(v^i u^k-v^k u^i)\\
&-\text{div}_\frac{1}{4}\left(v^i(y) \text{d}_\frac{1}{4}u^k(x,y)-v^k(y) \text{d}_\frac{1}{4} u^i(x,y)\right)
\end{split}
\end{equation}
as distributions.  Now we observe that for any $\phi\in C^\infty_c(\mathbb{R}^n)$, by the "Leibnitz rule" (\ref{eq_discussion_introduction_Leibnitz_rule_fraclaps}) there holds
\begin{equation}
\begin{split}
&\int_{\mathbb{R}^n}v^i(x)(-\Delta)^\frac{1}{4}(u^k\phi)(x)dx\\
=&\int_{\mathbb{R}^n}v^i(x)\left(v^k(x)\phi(x)+u^k(x)(-\Delta)^\frac{1}{4}\phi(x)\right)\\
&-\int_{\mathbb{R}^n}v^i(x)\left(\int_{\mathbb{R}^n}\text{d}_\frac{1}{4}u^k(x,y)\text{d}_\frac{1}{4}\phi(x,y)\frac{dy}{\lvert x-y\rvert^n}\right)dx
\end{split}
\end{equation}
Therefore
\begin{equation}
\begin{split}
&\int_{\mathbb{R}^n} v^i(x)(-\Delta)^\frac{1}{4}(u^k\phi)(x)-v^k(x)(-\Delta)^\frac{1}{4}(u^i\phi)(x)dx\\
=&\int_{\mathbb{R}^n}\left( v^i(x)u^k(x)-v^k(x)u^i(x)\right)(-\Delta)^\frac{1}{4}\phi(x)dx\\
&-\text{div}_\frac{1}{4}\left(v^i(x)\text{d}_\frac{1}{4}u^k(x,y)-v^k(x)\text{d}_\frac{1}{4}u^i(x,y)\right)[\phi].
\end{split}
\end{equation}
As this holds for any $\phi\in C^\infty_c(\mathbb{R}^n)$,
\begin{equation}
\begin{split}
(-\Delta)^\frac{1}{4}v^iu^k-(-\Delta)^\frac{1}{4}v^ku^i=&(-\Delta)^\frac{1}{4}(v^iu^k-v^ku^i)\\
&-\text{div}_\frac{1}{4}\left(v^i(x)\text{d}_\frac{1}{4}u^k(x,y)-v^k(x)\text{d}_\frac{1}{4}u^i(x,y)\right)
\end{split}
\end{equation}
as distributions.
Thus, by (\ref{eq_proof_lem_first_form_fraclap_as_div_of_d}),
\begin{equation}\label{eq_proof_lem_equality_Lambda_Euler_Lagrange}
\begin{split}
\text{div}_\frac{1}{4}(\text{d}_\frac{1}{4}v^i(x,y) u^k(x)-\text{d}_\frac{1}{4}v^k(x,y) u^i(x))=&(-\Delta)^\frac{1}{4}v^iu^k-(-\Delta)^\frac{1}{4}v^ku^i\\
=&(-\Delta)^\frac{1}{2}u^iu^k-(-\Delta)^\frac{1}{2}u^ku^i
\end{split}
\end{equation}
as distributions. This concludes the proof of the claim.
By Lemma \ref{lem_1_2_unexpected_divergence_free_quantity}, if
$u$ is a critical point of (\ref{eq_definition_fractional_Dirichlet_energy_2}) the left side of (\ref{eq_proof_lem_equality_Lambda_Euler_Lagrange}) is equal to zero. Therefore (\ref{eq_lem_Euler_Lagrange_vartar}) holds in the sense of distributions.\\
To show (\ref{eq_lem_explcit_form_fraclap_on_sphere}) let $(u_l)_l$ be a sequence in $\mathscr{S}(\mathbb{R}^n)+\mathbb{R}^m$ as in Lemma \ref{lem_approx_sobolev}, i.e. such that $u_l\to u$ in $L^2(K)$ for any compact $K\subset\mathbb{R}^n$ and $u_l\to u$ in $\dot{H}^\frac{1}{2}(\mathbb{R}^n)$ as $l\to \infty$\footnote{Notice that here, in contrast to the notation of Lemma \ref{lem_approx_sobolev}, we included the constant $c_l$ in the definition of $c_l$.}. Then for any $l\in\mathbb{N}$
\begin{equation}\label{eq_proof_lemma_computation_representation_fraclap_functions_into_spheres}
\begin{split}
u_l(x)\cdot(-\Delta)^\frac{1}{2}u_l(x)=& \int_{\mathbb{R}^n}\frac{u_l(x)-u_l(y)}{\lvert x-y\rvert}\cdot u_l(x)\frac{dy}{\lvert x-y\rvert^n}\\
=& \int_{\mathbb{R}^n}\frac{1-u_l(y)\cdot u_l(x)}{\lvert x-y\rvert}\frac{dy}{\lvert x-y\rvert^n}\\
=& \frac{1}{2}\int_{\mathbb{R}^n}\frac{\lvert u_l(x)-u_l(y)\rvert^2}{\lvert x-y\rvert}\frac{dy}{\lvert x-y\rvert^n}\\
=&\frac{1}{2}\int_{\mathbb{R}^n}\lvert \text{d}_\frac{1}{2}u_l(x,y)\rvert^2\frac{dy}{\lvert x-y\rvert^n}
\end{split}
\end{equation}
for $x\in \mathbb{R}^n$.
Moreover, by choice of $(u_l)_l$,
\begin{equation}
\begin{split}
&u_l\cdot(-\Delta)^\frac{1}{2}u_l\to u\cdot(-\Delta)^\frac{1}{2}u \text{ in }\mathscr{D}'(\mathbb{R}^n),\\
&\frac{1}{2}\int_{\mathbb{R}^n}\lvert \text{d}_\frac{1}{2}u_l(x,y)\rvert^2\frac{dy}{\lvert x-y\rvert^n}\to \frac{1}{2}\int_{\mathbb{R}^n}\lvert \text{d}_\frac{1}{2}u(x,y)\rvert^2\frac{dy}{\lvert x-y\rvert^n}\text{ in }L^1(\mathbb{R}^n)
\end{split}
\end{equation}
as $l\to \infty$. Therefore $u\cdot(-\Delta)^\frac{1}{2}u\in L^1(\mathbb{R}^n)$ and
\begin{equation}\label{eq_proof_lemma_explixit_formula_product_u_fraclap_12}
u\cdot(-\Delta)^\frac{1}{2}u(x)=\frac{1}{2}\int_{\mathbb{R}^n}\lvert \text{d}_\frac{1}{2}u(x,y)\rvert^2\frac{dy}{\lvert x-y\rvert^n}
\end{equation}
for a.e. $x\in \mathbb{R}^n$.
Since (\ref{eq_lem_Euler_Lagrange_vartar}) holds in the sense of distributions,
\begin{equation}
(-\Delta)^\frac{1}{2}u=\left( (-\Delta)^\frac{1}{2}u\cdot u \right)u
\end{equation}
as distributions. Thus, by (\ref{eq_proof_lemma_explixit_formula_product_u_fraclap_12}), $(-\Delta)^\frac{1}{2}u\in L^1(\mathbb{R}^n)$, and for a.e. $x\in \mathbb{R}^n$
\begin{equation}
(-\Delta)^\frac{1}{2}u(x)=\left(u(x)\cdot (-\Delta)^\frac{1}{2}u(x)\right)u(x)=\frac{1}{2}u(x)\int_{\mathbb{R}^n}\lvert \text{d}_\frac{1}{2}u(x,y)\rvert^2\frac{dy}{\lvert x-y\rvert^n}.
\end{equation}
This concludes the proof of (\ref{eq_lem_explcit_form_fraclap_on_sphere}), and thus in particular of (\ref{eq_lem_Euler_Lagrange_vartar}).
For the proof of (\ref{eq_div_1_2_of_Omega_vanishes}) we refer to \cite{divcurl} Lemma 3.1: there it was shown that for a function $u\in \dot{H}^\frac{1}{2}(\mathbb{R}^n, \mathbb{S}^{m})$, (\ref{eq_div_1_2_of_Omega_vanishes}) holds if and only if (\ref{eq_lem_explcit_form_fraclap_on_sphere}) holds.
\end{proof}
\begin{rem}
We observe that a key step in the previous argument consists in the representation
\begin{equation}\label{eq_rem_representation_formula_for_fraclap_into_spheres}
(-\Delta)^\frac{1}{2}u(x)=u(x)\int_{\mathbb{R}^n}\left\lvert \text{d}_\frac{1}{2}u(x,y)\right\rvert^2dy
\end{equation}
for a.e. $x\in \mathbb{R}^n$. This in fact allows to the deduce that $(-\Delta)^\frac{1}{2}u\in L^1(\mathbb{R}^n)$, providing significant informations about the regularity of $u$.
Representation (\ref{eq_rem_representation_formula_for_fraclap_into_spheres}) was obtained by means of computation (\ref{eq_proof_lemma_computation_representation_fraclap_functions_into_spheres}), where we used in an essential way the fact that $u$ takes value in a sphere. It would be interesting to find analogous representations for half harmonic maps taking values in more general manifolds.
\end{rem}

Next we show that considering s different form of the half Dirichlet energy and applying to it the argument of Noether Theorem for rotations in the target, one obtains directly Equation (\ref{eq_div_1_2_of_Omega_vanishes_intro_sec}) for any critical point in $\dot{H}^\frac{1}{2}(\mathbb{R}^n,\mathbb{S}^1)$ of the half Dirichlet energy.\\
For the following result, let
\begin{equation}\label{eq_discussion_other_representation_half_Dirichlet_energy:12}
E(u)=\left\langle (-\Delta)^\frac{1}{2}u, u \right\rangle
\end{equation}
for functions $u\in H^\frac{1}{2}(\mathbb{R}^n, \mathbb{S}^{m})$. Here and in the following, $\langle\cdot, \cdot\rangle$ denotes the action of an element of $\dot{H}^{-\frac{1}{2}}(\mathbb{R}^n)$ on an element of $\dot{H}^{\frac{1}{2}}(\mathbb{R}^n)$, defined as follows: let $a\in \dot{H}^{-\frac{1}{2}}(\mathbb{R}^n)$, $b\in \dot{H}^{\frac{1}{2}}(\mathbb{R}^n)$. Then
\begin{equation}\label{eq_discussion_definition_action_hom_-12_to_12}
\langle a, b\rangle:=\int_{\mathbb{R}^n} \left(\lvert \xi\rvert^{-\frac{1}{2}}\widehat{a}(\xi)\right)\cdot\left(\lvert \xi\rvert^\frac{1}{2}\overline{\widehat{b}(\xi)}\right)d\xi.
\end{equation}
Since $\lvert \xi\rvert^{-\frac{1}{2}}\widehat{a}(\xi)$, $\lvert \xi\rvert^\frac{1}{2}\overline{\widehat{b}(\xi)}\in L^2(\mathbb{R}^n)$, the integral in (\ref{eq_discussion_definition_action_hom_-12_to_12}) is well defined. Moreover, as fractional Laplacians are self-adjoint, the energy defined in (\ref{eq_discussion_other_representation_half_Dirichlet_energy:12}) corresponds to the half Dirichlet energy for any $u\in \dot{H}^\frac{1}{2}(\mathbb{R}^n,\mathbb{S}^m)$.

\begin{lem}\label{lem_direct_proof_divOmega=0}
Let $u\in \dot{H}^\frac{1}{2}(\mathbb{R}^n, \mathbb{S}^{m})$ be a critical point of the energy $E$ in $\dot{H}^\frac{1}{2}(\mathbb{R}^n, \mathbb{S}^{m})$. For any $i, k\in \lbrace 1,..., m+1\rbrace$, $i\neq k$ let $\Omega_{ik}$ be defined as in (\ref{eq_lem_definition_Omega_i_k}). Then
\begin{equation}\label{eq_div_1_2_of_Omega_vanishes2}
\text{div}_\frac{1}{2} \left(\Omega_{ik}\right)=0.
\end{equation}
\end{lem}

\begin{proof}
For any $A\in \mathfrak{so}(m+1)$, let $X_A$ and $\phi^A_t$ be defined as in the proof of Lemma \ref{lem_1_2_unexpected_divergence_free_quantity}.
We remark that for any function $v:\mathbb{R}^n\to\mathbb{R}^m$, for any $R\in SO(m+1)$,
\begin{equation}
R u(x)\cdot (-\Delta)^\frac{1}{2}R u(x)=u(x)\cdot (-\Delta)^\frac{1}{2} u(x)\text{ for any }x\in\mathbb{R}^n.
\end{equation}
In particular, if we extend the domain of $\phi^A_t$ to $\mathbb{R}^m$ identifying $\phi^A_t$ with the multiplication by the matrix $exp(tA)$ (see (\ref{eq_proof_lemma_definition_of_phiA_as_mult_by_exponential_fct})) we obtain that for any function $v:\mathbb{R}^n\to\mathbb{R}^m$
\begin{equation}
\phi^A_t\circ v(x)\cdot (-\Delta)^\frac{1}{2}\left( \phi^A_t\circ v \right)(x)=v(x)\cdot (-\Delta)^\frac{1}{2}u(x)
\end{equation}
for any $x\in\mathbb{R}^n$, $t\in \mathbb{R}$.
Now let $(u_n)_n$ be a sequence in $\mathscr{S}(\mathbb{R}^n)$ approximating $u$ in $\dot{H}^\frac{1}{2}(\mathbb{R}^n)$ as in Lemma \ref{lem_approx_sobolev}. 
Then by Taylor's Theorem, for any $A\in \mathfrak{so}(m+1)$ and for any $n\in\mathbb{N}$
\begin{equation}
\phi^A_t\circ u_n(x)=exp(tA) u_n(x)=u_n(x)+tAu_n(x)+t^2\partial_t\vert_{t=s_{x,t}}exp(tA)u_n(x)
\end{equation}
and
\begin{equation}
\begin{split}
&(-\Delta)^\frac{1}{2}(\phi^A_t\circ u_n)(x)\\
=&exp(tA)(-\Delta)^\frac{1}{2}u_n(x)\\=&(-\Delta)^\frac{1}{2}u_n(x)+tA (-\Delta)^\frac{1}{2}u_n(x)+t^2\partial_t\vert_{t=z_{x,t}}exp(tA)(-\Delta)^\frac{1}{2}u_n(x)
\end{split}
\end{equation}
for any $x\in \mathbb{R}^n$, $t\in \mathbb{R}$, for some $s_{x,t}, z_{x,y}$ between $0$ and $t$, depending on $x$ and $t$.
Since the exponential map is smooth, if we assume that $\lvert t\rvert\leq \delta$ for some $\delta>0$, then $\partial_t\vert_{t=s_{x,t}}exp(tA)$, $\partial_t\vert_{t=z_{x,t}}exp(tA)$ are bounded (in the space of $m$ times $m$ matrices) uniformly in $x$. Therefore
\begin{equation}\label{eq_proof_lem_rem_expansion_invariance_symmetry}
\begin{split}
&\phi^A_t\circ u_n\cdot (-\Delta)^\frac{1}{2}\left( \phi^A_t\circ u_n \right)\\
=&\left(u_n+tAu_n+o_{\dot{H}^\frac{1}{2}}(t)\right)\cdot\left( (-\Delta)^\frac{1}{2}u_n+tA(-\Delta)^\frac{1}{2}u_n +o_{\dot{H}^{-\frac{1}{2}}}(t) \right)
\end{split}
\end{equation}
as $t\to 0$, where the "error terms" $o_{\dot{H}^\frac{1}{2}}(t)$ $o_{\dot{H}^{-\frac{1}{2}}}(t)$ converges uniformly in $n$.
In particular , let $\psi\in C^\infty_c(\mathbb{R}^n)$. Then, substituting $t$ with $t\psi(x)$ in (\ref{eq_proof_lem_rem_expansion_invariance_symmetry}) for $x\in \mathbb{R}^n$ and integrating in $x$ we obtain
\begin{equation}
E(u)=\int_{\mathbb{R}^n}\left(u_n(x)+tAu_n(x)\right)\cdot\left( (-\Delta)^\frac{1}{2}u_n(x)+tA(-\Delta)^\frac{1}{2}u_n (x)\right)dx+o(t),
\end{equation}
as $t\to 0$, where again the "error term" $o(t)$ converges uniformly in $n$. Thus, as $u_n\to u$ in $\dot{H}^\frac{1}{2}(\mathbb{R}^n)$,
\begin{equation}\label{eq_proof_rem_lem_computation_invariance_rotation_expansion}
E(u)=  \left\langle (-\Delta)^\frac{1}{2}u+tA(-\Delta)^\frac{1}{2}u  ,u+ tAu \right\rangle+o(t).
\end{equation}
as $t\to 0$.\\
On the other hand, by the pointwise formula (\ref{eq_prel_pointwise_definition_fraclap}) and the fact that the fractional Laplacian is self-adjoint, we have, in the sense of temperated distributions,
\begin{equation}\label{eq_proof_rem_lem_equality_as_tempered_distr_also_for_H12}
\begin{split}
(-\Delta)^\frac{1}{2}(\psi (X^A\circ u))=&(-\Delta)^\frac{1}{2}\psi (X^A\circ u)+\psi(-\Delta)^\frac{1}{2}\left( X^A\circ u\right)\\&-\int_{\mathbb{R}^n}\text{d}_\frac{1}{2}\psi(x,y)\text{d}_\frac{1}{2}\left(X^A\circ u\right)(x,y)\frac{dy}{\lvert x-y\rvert^n}.
\end{split}
\end{equation}
As $\mathscr{S}(\mathbb{R}^n)$ is dense in $\dot{H}^\frac{1}{2}(\mathbb{R}^n)$ (thought of as quotient space) by Lemma \ref{lem_approx_sobolev}, equality (\ref{eq_proof_rem_lem_equality_as_tempered_distr_also_for_H12}) also holds as equality of linear continuous functionals acting on $\dot{H}^\frac{1}{2}(\mathbb{R}^n)$.
Thus, if we define $u_t$ as in (\ref{eq_proof_prop_Cauchy_prop_PicLind}), 
\begin{equation}\label{eq_proof_rem_lem_computation_expansion}
\begin{split}
E(u_t)=&E(u+t\psi (X^A\circ u)+o_{H^\frac{1}{2}}(t))=E(u+t\psi (X^A\circ u))+o(1)\\
=&\left\langle (-\Delta)^\frac{1}{2}u+t(-\Delta)^\frac{1}{2}(\psi (X^A\circ u)), u+t(\psi (X^A\circ u))\right\rangle+o(t)\\
=&\left\langle (-\Delta)^\frac{1}{2}u+tA(-\Delta)^\frac{1}{2}u  ,u+ tAu \right\rangle\\
&+t\int_{\mathbb{R}^n}(u(x)+t\psi(x)Au(x))\cdot\left(Au(x)(-\Delta)^\frac{1}{2}\psi(x)\right.\\
&\left.-\int_{\mathbb{R}^n}\text{d}_\frac{1}{2}\psi(x,y)\text{d}_\frac{1}{2}Au(x,y)\frac{dy}{\lvert x-y\rvert}\right)dx+o(t)
\end{split}
\end{equation}
where the first step follows from (\ref{eq_proof_prop_first_claim_asymptotic_ut}).\\
Finally, as $u$ is a critical point of $E$,
\begin{equation}\label{eq_proof_rem_lem_u_is_critical_point}
E(u_t)=E(u)+o(t)
\end{equation}
Thus, comparing (\ref{eq_proof_rem_lem_computation_invariance_rotation_expansion}), (\ref{eq_proof_rem_lem_computation_expansion}) and (\ref{eq_proof_rem_lem_u_is_critical_point}) we conclude that
\begin{equation}\label{eq_proof_rem_lem_conclusion}
\begin{split}
\int_{\mathbb{R}^n}u(x)\cdot\left(Au(x)(-\Delta)^\frac{1}{2}\psi(x)
-\int_{\mathbb{R}^n}\text{d}_\frac{1}{2}\psi(x,y)\text{d}_\frac{1}{2}Au(x,y)\frac{dy}{\lvert x-y\rvert}\right)dx=0
\end{split}
\end{equation}
Since $Au(x)\cdot u(x)=0$ for any $x\in \mathbb{R}^n$, the first summand in (\ref{eq_proof_rem_lem_conclusion}) vanishes. Since the equality holds for any $\psi\in C^\infty_0(\mathbb{R}^n)$, there holds
\begin{equation}
\text{div}_\frac{1}{2}\left( u\text{d}_\frac{1}{2}Au\right)=0.
\end{equation}
For $i,k\in \lbrace 1,2, ..., m+1\rbrace$, $i\neq k$, choosing the matrix $A$ as in (\ref{eq_proof_prel_matrixdef}) we obtain
\begin{equation}
\text{div}_\frac{1}{2}\left(\Omega_{ik}\right)=\text{div}_\frac{1}{2}\left( \frac{u^i(x)u^k(y)-u^k(x)u^i(y)}{\lvert x-y\rvert^\frac{1}{2}}\right)=0.
\end{equation}
\end{proof}

Finally we show how Equation (\ref{eq_lem_vartar_Omega=0_ref}) can be used to show that half-harmonic maps are continuous. We follow an approach slightly different from the one used by K. Mazowiecka and A. Schikorra in \cite{divcurl} and closer to the one developed by F. H\'{e}lein in \cite{Heleinarticle} for the local case.

\begin{thm}\label{prop_critical_points_are_continuous}(F. Da Lio, T. Rivi\`{e}re, \cite{3Termscomm})
Let $u\in \dot{H}^\frac{1}{2}(\mathbb{R},\mathbb{S}^m)$ be a critical point of the energy (\ref{eq_definition_fractional_Dirichlet_energy_2}), then $u$ is continuous.
\end{thm}

\begin{proof}
First we claim that for any $v\in \mathbb{R}^{m+1}$,
\begin{equation}\label{eq_proof_prop_representation_identity}
\sum_{i<k} v\cdot \left(A_{ik} u(x)\right) A_{ik} u(x)+(v\cdot u(x))u(x)=v
\end{equation}
for any $x\in \mathbb{R}^n$, where $A_{ik}$ is the matrix defined by (\ref{eq_proof_prel_matrixdef}). As both sides of (\ref{eq_proof_prop_representation_identity}) are linear in $v$, it is enough to show the identity for the canonical basis vectors $e^j$, $j\in \lbrace 1,...m+1\rbrace$: for $l\in \lbrace 1,...,m+1\rbrace$
\begin{equation}
\begin{split}
&\left[\sum_{i<k} e^j\cdot \left(A_{ik} u(x)\right) A_{ik} u(x)\right]_l\\
=&\sum_{i<k}(\delta_{jk}u^i(x)-\delta_{ji}u^k(x))\left[A_{ik} u(x)\right]_l+u^j(x)u^l(x)\\
=&\sum_{i<k}(\delta_{jk}u^i(x)-\delta_{ji}u^k(x))(\delta_{kl}u^i(x)-\delta_{li}u^k(x))+u^j(x)u^l(x)\\
=&\delta_{jl}\sum_{i<l}u^i(x)^2+\delta_{jl}\sum_{k>l}u^k(x)^2-(1+\delta_{jl})u^j(x)u^l(x)+u^j(x)u^l(x)\\
=&\delta_{jl}\sum_{i=1}^{m+1} u^i(x)^2=\delta_{jl},
\end{split}
\end{equation}
as $u(x)\in \mathbb{S}^m$ for any $x\in\mathbb{R}^n$.
Therefore for any $x,y\in\mathbb{R}^n$ we can write
\begin{equation}\label{eq_proof_prop_first_step_for_fraclap_12}
\text{d}_\frac{1}{2}u(x,y)=\sum_{i<k} \text{d}_\frac{1}{2}u(x,y)\cdot \left(A_{ik} u(x)\right) A_{ik} u(x)+\left(\text{d}_\frac{1}{2}u(x,y)\cdot u(x)\right)u(x).
\end{equation}
Now let $\phi\in C^\infty_c(\mathbb{R}^n, \mathbb{R}^{m+1})$, then
\begin{equation}\label{eq_proof_prop_projection_gradient}
\begin{split}
&\text{div}_\frac{1}{2}\left(\sum_{i<k} \text{d}_\frac{1}{2}u(x,y)\cdot \left(A_{ik} u(x)\right) A_{ik} u(x)+(\text{d}_\frac{1}{2}u(x,y)\cdot u(x))u(x)\right)[\phi]\\
=&\int_{\mathbb{R}^n}\int_{\mathbb{R}^n}\left(\sum_{i<k} \text{d}_\frac{1}{2}u(x,y)\cdot \left(A_{ik} u(x)\right) A_{ik} u(x)+(\text{d}_\frac{1}{2}(x,y)\cdot u(x))u(x)\right)\\
&\qquad\quad\cdot\frac{\phi(x)-\phi(y)}{\lvert x-y\rvert^\frac{1}{2}}\frac{dxdy}{\lvert x-y\rvert^n}\\
=&\int_{\mathbb{R}^n}\int_{\mathbb{R}^n}\left(\sum_{i<k} \text{d}_\frac{1}{2}u(x,y)\cdot \left(A_{ik} u(x)\right) \right)\frac{A_{ik} u(x)\cdot\phi(x)-A_{ik} u(y)\cdot\phi(y)}{\lvert x-y\rvert^\frac{1}{2}}\frac{dxdy}{\lvert x-y\rvert^n}\\
&+\int_{\mathbb{R}^n}\int_{\mathbb{R}^n}\sum_{i<k} \text{d}_\frac{1}{2}u(x,y)\cdot \left(A_{ik} u(x)\right)\frac{A_{ik}u(y)-A_{ik}u(x)}{\lvert x-y\rvert^\frac{1}{2}}\cdot\phi(y)\frac{dxdy}{\lvert x-y\rvert^n}\\
&+\int_{\mathbb{R}^n}\int_{\mathbb{R}^n}\left(\text{d}_\frac{1}{2}u(x,y)\cdot u(x)\text{d}_\frac{1}{2}u(x,y)\right)\cdot\phi(x)\frac{dxdy}{\lvert x-y\rvert^n}
\end{split}
\end{equation}
provided the integrals converge absolutely.
If we set
\begin{equation}
T(x):=\int_{\mathbb{R}^n}\text{d}_\frac{1}{2}u(x,y)\cdot u(x)\text{d}_\frac{1}{2}u(x,y)\frac{dy}{\lvert x-y\rvert}
\end{equation}
for $x\in \mathbb{R}$, we can rewrite the last term in (\ref{eq_proof_prop_projection_gradient}) as
\begin{equation}
\int_{\mathbb{R}}T(x)\phi(x)dx.
\end{equation}
By \cite{divcurl}, Section 3.2, $T\in L^1_{loc}(\mathbb{R})$, the last integral in (\ref{eq_proof_prop_projection_gradient}) converges absolutely, and there holds
\begin{equation}\label{eq_proof_prop_estimate_rest_term}
\left\lvert\int_{\mathbb{R}}T(x)\phi(x)dx\right\rvert\leq C \lVert (-\Delta)^\frac{1}{4}\phi\rVert_{L^{2,\infty}}
\end{equation}
for some constant $C$ independent from $\phi$.
Moreover we claim that the first term in the last expression of (\ref{eq_proof_prop_projection_gradient}) can be rewritten as
\begin{equation}
\text{div}_\frac{1}{2}\left(\sum_{i<k} \text{d}_\frac{1}{2}u(x,y)\cdot \left(A_{ik} u(x)\right)  \right) [A_{ik} u\cdot\phi]
\end{equation}
Actually here we used an improper notation, as we defined the operator $\text{div}_\frac{1}{2}$ of a function to be a distribution (if condition (\ref{eq_def_condition_def_div}) is satisfied).
What we mean here is that $\text{div}_\frac{1}{2}\left(\sum_{i<k} \text{d}_\frac{1}{2}u(x,y)\cdot \left(A_{ik} u(x)\right)  \right)$ is a well defined distribution that can be extended continuously to functions in $H^\frac{1}{2}(\mathbb{R})$. Once we checked this, it will follows that the first term in the last expression of (\ref{eq_proof_prop_projection_gradient}) vanishes, since, by (\ref{eq_div_1_2_of_Omega_vanishes}),
\begin{equation}
\text{div}_\frac{1}{2}\left(\sum_{i<k} \text{d}_\frac{1}{2}u(x,y)\cdot \left(A_{ik} u(x)\right)  \right)=0
\end{equation}
as distribution. To prove the claim, we compute for any $v\in H^\frac{1}{2}(\mathbb{R})$
\begin{equation}
\begin{split}
&\int_{\mathbb{R}}\int_{\mathbb{R}} \left\lvert\left(\sum_{i<k} \text{d}_\frac{1}{2}u(x,y)\cdot \left(A_{ik} u(x)\right)  \right) \frac{v(x)-v(y)}{\lvert x-y\rvert^\frac{1}{2}}\right\rvert\frac{dxdy}{\lvert x-y\rvert^n}\\
\leq & \sum_{i<k}\int_{\mathbb{R}}\int_{\mathbb{R}}\left\lvert\frac {u(x)-u(y)}{\lvert x-y\rvert^\frac{1}{2}}\right\rvert\left\lvert\frac{v(x)-v(y)}{\lvert x-y\rvert^\frac{1}{2}}\right\rvert \frac{dxdy}{\lvert x-y\rvert^n}\\
\leq& \binom{m+1}{2} [u]_{W^{\frac{1}{2}, 2}}[v]_{W^{\frac{1}{2}, 2}}.
\end{split}
\end{equation}
A similar computation shows that also the second integral in the last line of (\ref{eq_proof_prop_projection_gradient}) converges absolutely.
Therefore if we apply $\text{div}_\frac{1}{2}$ on both sides of (\ref{eq_proof_prop_first_step_for_fraclap_12}) we obtain the following identity of distributions:
\begin{equation}\label{eq_proof_prop_representation_fraclap_wih_error}
(-\Delta)^\frac{1}{2}u= \sum_{i<k}\left( \text{d}_\frac{1}{2}u(x,y)\cdot \left(A_{ik} u(x)\right)\right)\cdot_{\bigwedge_{od}^1\mathbb{R}} \text{d}_\frac{1}{2}\left(A_{ik}u\right)+T,
\end{equation}
where we used the notation introduced in (\ref{eq_definition_product_in_one_var}).
Now observe that since $u\in \dot{H}^\frac{1}{2}(\mathbb{R}, \mathbb{S}^m)$, $\text{d}_\frac{1}{2}u(x,y)\cdot \left(A_{ik} u(x)\right)\in L^2\left(\bigwedge^1_{od}\mathbb{R}\right)$, and by (\ref{eq_div_1_2_of_Omega_vanishes}), \newline$\text{div}_\frac{1}{2}\left(\text{d}_\frac{1}{2}u(x,y)\cdot \left(A_{ik} u(x)\right)\right)=0$
Therefore, since $T$ satisfy estimate (\ref{eq_proof_prop_estimate_rest_term}), Corollary \ref{cor_fractional_Wente_Lemma} implies that $u$ is continuous.
\end{proof}

\begin{rem}\noindent
\begin{enumerate}
\item The only step where the assumption $n=1$ is needed is estimate (\ref{eq_proof_prop_estimate_rest_term}).

\item The only difference between the proof of Theorem \ref{prop_critical_points_are_continuous}) presented above and the one in \cite{divcurl} is that here we obtained equation (\ref{eq_proof_prop_representation_fraclap_wih_error}) by projecting $\text{d}_\frac{1}{2}u(x,y)$ onto vector fields which are generators of infinitesimal symmetry for the energy (\ref{eq_definition_fractional_Dirichlet_energy_2}) (in order to apply the fractional analogous of Noether theorem) and onto $u(x)$. This allows to interpret the "error term" as the contribution from the normal component of $\text{d}_\frac{1}{2}u(x,y)$ (wrt $T_{u(x)}\mathbb{S}^m$). In the local cases, no such term appears; in fact, the role of $\text{d}_\frac{1}{2}u(x,y)$ is taken by $\nabla u(x)$ (for a.e. $x\in\mathbb{R}^2$), which lies in $T_{u(x)}\mathbb{S}^m$.
\end{enumerate}
\end{rem}

\section{Variations in the domain}

In this section we consider variational problems whose Lagrangians involving a $\frac{1}{4}$-fractional Laplacian and exhibiting symmetries under conformal variations in the domain. The first goal is to derive a Noether theorem for variations in the target for energies involving the fractional Laplacian.\\

\subsection{A Noether Theorem for Lagrangians involving a $\frac{1}{4}$-fractional Laplacian}

Let's first consider energies of the form
\begin{equation}\label{eq_local_energy_vardom}
E(u):=\int_{D^2} L\left(u(x), \nabla u(x)\right)dx,
\end{equation}
defined for functions $u\in W^{1,2}(D^2, \mathbb{R}^m)$. Here $D^2$ denotes the disc of radius $1$ centred at $0$ in $\mathbb{R}^2$, $L\in C^1(\mathbb{R}^m\times\mathbb{R}^{2m})$ for some $m\in\mathbb{N}_{>0}$ and $\lvert L(z,p)\rvert\leq C(1+\lvert z\rvert)$ for some constant $C$ for all $(z, p)\in \mathbb{R}^m\times\mathbb{R}^{2m}$.\\

\begin{defn}\label{def_statpoints_invariant}\noindent
\begin{enumerate}
\item $u\in W^{1,2}(D^2, \mathbb{R}^m)$ is said to be a \textbf{stationary point} of (\ref{eq_local_energy_vardom}) if for any vector field $X\in C^1(D^2, \mathbb{R}^2)$ there holds
\begin{equation}\label{eq_stationary_equation_local}
\frac{d}{dt}\bigg\vert_{t=0} E(u(x+tX(x)))=0.
\end{equation}
\item Let $X$ be a smooth vector field on $D^2$ and let's denote $\phi_t$ its flow.
The energy (\ref{eq_local_energy_vardom}) is said to be \textbf{invariant with respect to the variation generated by $X$} if for any $u\in u\in W^{1,2}(D^2, \mathbb{R}^m)$ there holds
\begin{equation}\label{eq_condition_invariance_vector_field}
L(u\circ\phi, \nabla (u\circ\phi_t))(x)=L(u, \nabla u)\circ \phi_t(x)\lvert\det D\phi_t(x)\rvert
\end{equation}
for a.e. $x\in D^2$, for $t\in\mathbb{R}$.
\end{enumerate}
\end{defn}

In this framework, the following Noether theorem holds:

\begin{thm}\label{thm_Noether_thm_vardom} (Theorem 2.1 in \cite{DaLio})
Let $E$ be defined as in (\ref{eq_local_energy_vardom}) and let $X$ be a vector field on $B_1(0)$. Let $u\in W^{1,2}(B^1(0), \mathbb{R}^m)$ be a stationary point of the energy (\ref{eq_local_energy_vardom}). If $E$ is invariant with respect to the variation generated by $X$. Let
\begin{equation}
J_X[u](x)=\left([\partial_{p^k_j}f\left(u,\nabla u\right)u^j(x)\partial_{x_l}u^j(x)X^l(x)]-f(u, \nabla u)(x)X^k(x)\right)_{k=1,2}
\end{equation}
for $x\in\mathbb{R}^n$ (we use implicit summation over repeated indices).
Then
\begin{equation}\label{eq_thm_div_of_Noether_current_equal_to_0}
\text{div}\left( J_X[u]\right)=0.
\end{equation}
\end{thm}
For the proof, one can plug $u\circ\phi_t$ into the energy (\ref{eq_local_energy_vardom}) and take the derivative in $t$ at time $t=0$. Comparing the resulting expression with the stationary equation (\ref{eq_stationary_equation_local}) satisfied by $u$, one obtains (\ref{eq_thm_div_of_Noether_current_equal_to_0}).\\
Theorem \ref{thm_Noether_thm_vardom} finds useful applications to energies of the kind (\ref{eq_local_energy_vardom}) exhibiting invariance under conformal transformations of the domain.
In this document we make use of the following definition for conformal maps:
\begin{defn}
Let $(\mathscr{M}, g)$, $(\mathscr{N}, h)$ be two Riemannian manifolds. A differentiable map $\phi$ from $\mathscr{M}$ to $\mathscr{N}$ is said to be \textbf{conformal} if for any point $p\in\mathscr{M}$, for any $X, Y\in T_p\mathscr{M}$
\begin{equation}\label{eq_def_definition_conformal_map}
h(d\phi_p X, d\phi_p Y)=e^{2\lambda(p)}g(X, Y)
\end{equation}
for some $\lambda(p)\in \mathbb{R}$. $\lambda(p)$ is called \textbf{conformal factor} of $\phi$ at $p$.
\end{defn}
In particular, for any $n\in\mathbb{N}_{>0}$, for any conformal map $\phi$ from $\mathbb{R}^n$ to $\mathbb{R}^n$,
\begin{equation}
D\phi(x)\cdot D\phi(x)^T=e^{2\lambda(x)}\text{Id}
\end{equation}
for any $x\in\mathbb{R}^n$, where $\text{Id}$ denotes the $n\times n$-identity matrix.\\
An important example of energies of the kind (\ref{eq_local_energy_vardom}) invariant under conformal transformations in the domain is given by the Dirichlet energy
\begin{equation}\label{eq_def_Dirichlet_energy_vartar}
E(u)=\int_{D^2}\lvert \nabla u(x)\rvert^2dx
\end{equation}
defined for functions $u\in H^1(D^2)$. To check that the energy (\ref{eq_def_Dirichlet_energy_vartar}) is indeed invariant under conformal variations in the domain, we compute

\begin{equation}
\begin{split}
\lvert \nabla (u\circ\phi)(x)\rvert^2=&\left(D\phi(x)^T\nabla u(\phi(x))\right)\cdot\left(D\phi(x)^T\nabla u(\phi(x))\right)\\
=&\left(D\phi(x)^TD\phi(x)\nabla u(\phi(x))\right)\cdot\left(\nabla u(\phi(x))\right)\\
=&e^{2\lambda(x)}\left\lvert \nabla u(\phi(x))\right\rvert^2
\end{split}
\end{equation}
for a.e. $x\in\mathbb{R}^n$, and since $n=2$ and $\phi$ is conformal
\begin{equation}
\lvert \det D\phi(x)\rvert=\left\lvert\partial_x\phi(x)\wedge\partial_y\phi(x)\right\rvert=e^{2\lambda(x)}.
\end{equation}
Therefore equation (\ref{eq_condition_invariance_vector_field}) is satisfied for every conformal diffeomorphism $\phi$.\\
We observe that the Dirichlet energy is conformal invariant only for $n=2$. Indeed if we take $X(x)=x$ for $x\in \mathbb{R}^n$, the generator of dilations, for its flow $\phi_t(x)=(1+t)x$ there holds
\begin{equation}
\lvert \nabla (u\circ\phi_t)(x)\rvert^2=(1+t)^2\lvert\nabla u((1+t)x)\rvert^2
\end{equation}
and
\begin{equation}
\lvert \det D\phi_t(x)\rvert=(1+t)^n.
\end{equation}
Thus equation (\ref{eq_condition_invariance_vector_field}) only holds for $n=2$.
If we look for analogous energies in dimension $n=1$, it is therefore natural to try to substitute $\nabla u$ with a $\frac{1}{2}$-homogeneous quantity. A natural candidate is given again by the half Dirichlet energy 
\begin{equation}\label{eq_energy_fraclap_1_4_square}
\int_{\mathbb{S}^1}\left\lvert (-\Delta)^\frac{1}{4}u(x)\right\rvert^2(x)dx.
\end{equation}  
Notice that energy (\ref{eq_energy_fraclap_1_4_square}) is defined for functions on $\mathbb{S}^1$. This is because, euristically, we can consider problem (\ref{eq_energy_fraclap_1_4_square}) as  a "trace" of problem (\ref{eq_local_energy_vardom}).
In order to make the minimization problem non-trivial, we will look for minimizers taking values in a given manifold.
In the following, we will try to follow the proof of theorem (\ref{thm_Noether_thm_vardom}) for a class of energies that includes (\ref{eq_energy_fraclap_1_4_square}). In section \ref{ssec: Stationary points of the energy} we will investigate in more details the properties of critical points of the energy (\ref{eq_energy_fraclap_1_4_square}).\\
Let $m\in\mathbb{N}_{>0}$. Let $L\in C^2(\mathbb{R}^m\times\mathbb{R}^m, \mathbb{R})$ so that for any $x,p\in\mathbb{R}^m$
\begin{equation}\label{eq_well_def_energy_vardom}
\lvert L(x,p)\rvert\leq C(1+\lvert p\rvert^2)
\end{equation}
for some $C>0$. For any $u\in H^1(\mathbb{S}^1,\mathbb{R}^m)$ we set
\begin{equation}\label{eq_def_1_4_Dirichlet_energy}
E(u):=\int_{\mathbb{S}^1}L(u(x),(-\Delta)^\frac{1}{4}u(x))dx.
\end{equation}
We observe that $E$ is well defined due to condition (\ref{eq_well_def_energy_vardom}).\\
In analogy to Definition \ref{def_statpoints_invariant}, we define the following concepts.

\begin{defn}\label{defn_stationary_point_invariant_with_respect_to_the_variation}\noindent
\begin{enumerate}
\item $u\in H^\frac{1}{2}(\mathbb{S}^1, \mathbb{R}^m)$ is said to be a \textbf{stationary point} of (\ref{eq_def_1_4_Dirichlet_energy}) if for any vector field $X\in C^1(\mathbb{S}^1, \mathbb{R})$\footnote{Recall that we consider $\mathbb{S}^1$ as a quotient space of $\mathbb{R}$.} there holds
\begin{equation}\label{eq_stationary_equation_nonlocal}
\frac{d}{dt}\bigg\vert_{t=0} E(u(x+tX(x)))=0.
\end{equation}
\item Let $X$ be a smooth vector field on $\mathbb{S}^1$ and let's denote $\phi_t$ its flow.
The energy (\ref{eq_def_1_4_Dirichlet_energy}) is said to be \textbf{invariant with respect to the variation generated by $X$} if for any $u\in H^\frac{1}{2}(\mathbb{S}^1, \mathbb{R}^m)$ there holds
\begin{equation}\label{eq_condition_invariance_vector_field_nonlocal}
f(u\circ\phi, (-\Delta)^\frac{1}{4} (u\circ\phi_t))(x)=f(u, (-\Delta)^\frac{1}{4} u)\circ \phi_t(x)\lvert D\phi_t(x)\rvert
\end{equation}
for a.e. $x\in \mathbb{S}^1$, for $t\in\mathbb{R}$.
\end{enumerate}
\end{defn}
We also introduce the concept of $s$-\textbf{fractional divergence} for functions defined on the circle, for $s\in (0,\frac{1}{2})$.
If $F:\mathbb{S}^1\times \mathbb{S}^1\to \mathbb{R}^m$ is a bounded function, let\footnote{Notice that this definition is not the direct analogous of Definition \ref{def_fractional_divergence_Rn}, as here fractional divergences are not defined as duals of "fractional gradients for functions on $\mathbb{S}^1$". However, this notation seems to be more convenient for the following computations.}
\begin{equation}\label{eq_definition_frac_div_14}
\text{div}_sF[\phi]:=\int_{\mathbb{S}^1}\int_{\mathbb{S}^1}F(x,y)\cdot(\phi(x)-\phi(y))K^s(x-y)dxdy
\end{equation}
for any $\phi\in C^\infty(\mathbb{S}^1,\mathbb{R}^m)$.
Later we will extend the concept of $\frac{1}{4}$-fractional divergence to other $F$s (see Equation (\ref{eq_discussion_extension_definition_fracdiv_14})).

In order to follow the proof of Theorem \ref{thm_Noether_thm_vardom}, we first assume that $u\in C^\infty(\mathbb{S}^1, \mathbb{R}^m)$. Let $X$ be a smooth vector field on $\mathbb{S}^1$ and let $\phi_t$ denote its flow.
First we claim that

\begin{equation}\label{eq_approx_formula_vardom}
u\circ\phi_t=u+tu'X+o_{C^1}(t)
\end{equation}
as $t\to 0$ (in the following, the small $o$ notation will always refer to $t\to 0$). Indeed, by Taylor's Theorem, for any $x\in \mathbb{S}^1$, $t\in \mathbb{R}$, 
\begin{equation}
\begin{split}
u(\phi_t(x))-u(x)-tu'(x)X(x)=t^2&\left[ u''(\phi_{s_x}(x))\partial_r\bigg\vert_{r={s_x}}\phi_r(x)\right.\\
&\left.\phantom{[}+u'(\phi_{s_x}(x))\partial_r^2\bigg\vert_{r={s_x}}\phi_r(x) \right]
\end{split}
\end{equation}
for some $s_x$ between $0$ and $t$, depending on $x$, and
\begin{equation}
\begin{split}
&\partial_x\left[u(\phi_t(x))-u(x)-tu'(x)X(x)\right]\\
=&u'(\phi_t(x))\partial_x\phi_t(x)-u'(x)-tu''(x)X(x)-tu'(x)X'(x)\\
=& \left(u'(\phi_t(x))-u'(x)\right)\partial_x\phi_t(x)-tu''(x)X(x)+u'(x)(\partial_x\phi_t(x)-1)-tu'(x)X'(x)\\=& t\left[ u''(\phi_r(x))\partial_r\phi_r(x)\partial_x\phi_r(x)+u'(x)\partial_x\partial_r\phi_r(x) \right]\bigg\vert^{r=s_x}_{r=0}+t(u'(\phi_t(x))\\
&-u'(x))\partial_r\bigg\vert_{r=s_x}\partial_x\phi_r(x)
\end{split}
\end{equation}
for some $s_x$ between $0$ and $t$, depending on $x$. Since all functions involved are smooth, (\ref{eq_approx_formula_vardom}) holds true.\\
From (\ref{eq_approx_formula_vardom}) we deduce that
\begin{equation}\label{eq_fraclap_comp_diffeo_step14_1}
(-\Delta)^\frac{1}{4} (u\circ\phi_t)=(-\Delta)^\frac{1}{4}u+t(-\Delta)^\frac{1}{4}(u'X)+o_{L^\infty}(t).
\end{equation}
Indeed, if a family $(v_t)_t$ in $C^1(\mathbb{S}^1)$ converges to $0$ in $C^1(\mathbb{S}^1)$ as $t\to 0$, then for any $x\in \mathbb{S}^1$
\begin{equation}
\begin{split}
\left\lvert (-\Delta)^\frac{1}{4}v_t\right\rvert\leq & B_\frac{1}{4}\int_{\mathbb{S}^1}\frac{\lvert v_t(x)-v_t(y)\rvert}{\left\lvert\sin\left(\frac{1}{2}(x-y)\right)\right\rvert^\frac{3}{2}}dy\\
\leq & B_\frac{1}{4}\lVert v_t\rVert_{C^1}\int_{\mathbb{S}^1}\frac{\lvert x-y\rvert}{\left\lvert\sin\left(\frac{1}{2}(x-y)\right)\right\rvert^\frac{3}{2}}dy\\
\leq & B_\frac{1}{4}\frac{\pi}{2} \lVert v_t\rVert_{C^1} \int_{\mathbb{S}^1}\frac{1}{\left\lvert\sin\left(\frac{1}{2}(x-y)\right)\right\rvert^\frac{1}{2}}dy
\end{split}
\end{equation}
where $B_\frac{1}{4}$ is the constant from (\ref{eq_estimates_kernel_fraclap_s}). In the last step we used the fact that if $\lvert x\rvert\leq \pi$, $\frac{2}{\pi}\left\lvert\sin\left(\frac{x}{2}\right)\right\rvert \leq \frac{1}{\pi}\lvert x\rvert \leq \left\lvert \sin\left(\frac{1}{2}x\right)\right\rvert$. Thus $(-\Delta)^\frac{1}{4}u_t$ converges to $0$ in $L^\infty(\mathbb{S}^1)$ as $t\to 0$. This implies that (\ref{eq_fraclap_comp_diffeo_step14_1}) holds true.\\
By the "Leibnitz rule" (\ref{eq_discussion_intro_Leibnitz_rule_fraclap_circle}) we have, for any $x\in \mathbb{S}^1$,
\begin{equation}\label{eq_fraclap_comp_diffeo_step14_2}
\begin{split}
(-\Delta)^\frac{1}{4}(u'X)(x)=&(-\Delta)^\frac{1}{4}u'(x)\ X(x)+(-\Delta)^\frac{1}{4}X(x)\ u'(x)\\&- \int_{\mathbb{S}^1}(u'(x)-u'(y))(X(x)-X(y))K^\frac{1}{4}(x-y)dy.
\end{split}
\end{equation}
Now we observe that since $u\in C^\infty(\mathbb{S}^1)$, $(-\Delta)^\frac{1}{4}u\in C^\infty(\mathbb{S}^1)$. Thus, by Taylor's Theorem,
\begin{equation}\label{eq_fraclap_comp_diffeo_step14_3}
(-\Delta)^\frac{1}{4}u(\phi_t(x))=(-\Delta)^\frac{1}{4}u(x)+t(-\Delta)^\frac{1}{4}u'(x)\ X(x)+o_{L^\infty}(t).
\end{equation}
Therefore, by (\ref{eq_fraclap_comp_diffeo_step14_1}),  (\ref{eq_fraclap_comp_diffeo_step14_2}) and  (\ref{eq_fraclap_comp_diffeo_step14_3}), for $t\in \mathbb{R}$, $x\in \mathbb{S}^1$
\begin{equation}
\begin{split}
(-\Delta)^\frac{1}{4}(u\circ\phi_t)(x)=&(-\Delta)^\frac{1}{4}u(x)+t\left( (-\Delta)^\frac{1}{4}u'(x)\ X(x)+(\Delta)^\frac{1}{4}X(x\ )u'(x)\right.\\&\left.-\int_{\mathbb{S}^1}(u'(x)-u'(y))(X(x)-X(y))K^\frac{1}{4}(x-y)dy \right)+o_{L^\infty}(t)\\
=&(-\Delta)^\frac{1}{4}u(\phi_t(x))+tu'(x)\ (-\Delta)^\frac{1}{4}X(x)\\&- t\int_{\mathbb{S}^1}(u'(x)-u'(y))(X(x)-X(y))K^\frac{1}{4}(x-y)dy+o_{L^\infty}(t).
\end{split}
\end{equation}
Then for any $x\in \mathbb{S}^1$, by the Mean value Theorem
\begin{equation}
\begin{split}
&L\left(u\circ\phi_t(x), (-\Delta)^\frac{1}{4}(u\circ\phi_t)(x)\right)=L\left(u(\phi_t(x)), (-\Delta)^\frac{1}{4}u(\phi_t(x))\right)\\
&\qquad +tD_pL\left(u(\phi_t(x)), s_{x,t}\right)
\cdot\left[u'(x)\ (-\Delta)^\frac{1}{4}X(x)\right.\\
&\qquad -\left. \int_{\mathbb{S}^1}(u'(x)-u'(y))(X(x)-X(y))K^\frac{1}{4}(x-y)dy\right]+o_{L^\infty}(t)
\end{split}
\end{equation}
for some $s_{x,t}$ between $(-\Delta)^\frac{1}{4}(u\circ\phi_t)(x)$ and $(-\Delta)^\frac{1}{4}u(\phi_t(x))$. As $L\in C^2(\mathbb{R}^m\times\mathbb{R}^{m})$ and $(-\Delta)^\frac{1}{4}(u\circ\phi_t)\to (-\Delta)^\frac{1}{4}u$, $(-\Delta)^\frac{1}{4}u(\phi_t(\cdot))\to (-\Delta)^\frac{1}{4}u$ in $L^\infty(\mathbb{S}^1)$ as $t\to 0$, we obtain
\begin{equation}\label{eq_explicit_form_precomposition_vardom14}
\begin{split}
&L\left(u\circ\phi_t(x), (-\Delta)^\frac{1}{4}(u\circ\phi_t)(x)\right)
=L\left(u(\phi_t(x)), (-\Delta)^\frac{1}{4}u(\phi_t(x))\right)\\
&\qquad+tD_pL\left(u(x), (-\Delta)^\frac{1}{4}u(x)\right)\cdot\left[u'(x)\ (-\Delta)^\frac{1}{4}X(x)\right.\\
&\qquad -\left. \int_{\mathbb{S}^1}(u'(x)-u'(y))(X(x)-X(y))K^\frac{1}{4}(x-y)dy\right]+o_{L^\infty}(t).
\end{split}
\end{equation}
Next we claim that
\begin{equation}\label{eq_proof_prop_before_inversion_integral_derivative_in_zero}
\begin{split}
&\frac{d}{dt}\bigg\vert_{t=0}\int_{\mathbb{S}^1}L\left(u\circ\phi_t(x), (-\Delta)^\frac{1}{4}(u\circ\phi_t)(x)\right)dx\\
=&\int_{\mathbb{S}^1}\frac{d}{dt}\bigg\vert_{t=0}L\left(u\circ\phi_t(x), (-\Delta)^\frac{1}{4}(u\circ\phi_t)(x)\right)dx.
\end{split}
\end{equation}
To prove the claim, it is enough to show that $(-\Delta)^\frac{1}{4}(u\circ\phi_t)$ is differentiable in $t$ with
\begin{equation}
\partial_t(-\Delta)^\frac{1}{4}(u\circ\phi_t)=(-\Delta)^\frac{1}{4}\partial_t(u\circ\phi_t)
\end{equation}
and that the function
\begin{equation}
\begin{split}
\partial_t L\left(u\circ\phi_t, (-\Delta)^\frac{1}{4}(u\circ \phi_t)\right)=&D_xL\left(u\circ\phi_t, (-\Delta)^\frac{1}{4}(u\circ \phi_t)\right)u'(X\circ\phi_t)\\
+&D_pL\left(u\circ\phi_t, (-\Delta)^\frac{1}{4}(u\circ \phi_t)\right)\partial_t (-\Delta)^\frac{1}{4}(u\circ\phi_t)
\end{split}
\end{equation}
has an integrable upper bound, uniformly in $t$ for $t$ in a neighbourhood of $0$. The second part of the claim follow directly from the regularity assumption on $u$ and $L$ once we prove the first part.
For that, let $\delta>0$; we define
\begin{equation}
[\phi_t]_{Lip}^\delta:=\sup_{\substack{\lvert s\rvert\leq \delta,\\ x,y\in \mathbb{S}^1}}\frac{\lvert \phi_s(x)-\phi_s(y)\rvert}{\lvert x-y\rvert}
\end{equation}
and
\begin{equation}
\lVert \phi_t\rVert_{L^\infty}^\delta:=\sup_{\lvert s\rvert\leq \delta}\lVert \phi_s\rVert_{L^\infty}.
\end{equation}
We define $[\phi'_t]_{Lip}^\delta$ and $\lVert \phi_t'\rVert_{L^\infty}^\delta$ similarly.
Now we compute, for $\lvert t\rvert\leq \delta$, $x,y\in \mathbb{S}^1$
\begin{equation}
\begin{split}
&\left\lvert \frac{u'(\phi_t(x))X(\phi_t(x))-u'(\phi_t(y))X(\phi_t(y))}{\left\lvert \sin\left(\frac{1}{2}(x-y)\right)\right\rvert^\frac{3}{2}} \right\rvert\\
\leq & \left\lvert \frac{[u'(\phi_t(x))-u'(\phi_t(y))]X(\phi_t(x))+u'(\phi_t(y))[X(\phi_t(x))-X(\phi_t(y))]}{\left\lvert \sin\left(\frac{1}{2}(x-y)\right)\right\rvert^\frac{3}{2}}\right\rvert\\
\leq & \frac{[u']_{Lip}[\phi_t]_{Lip}^\delta\lVert X\rVert_{L^\infty}+\lVert u'\rVert_{L^\infty}[X]_{Lip}[\phi_t]_{Lip}^\delta}{\left\lvert \sin\left(\frac{1}{2}(x-y)\right)\right\rvert^\frac{1}{2}}\frac{\lvert x-y\rvert}{\left\lvert \sin\left(\frac{1}{2}(x-y)\right)\right\rvert}.
\end{split}
\end{equation}
The expression in the last line is an integrable upper bound (up to multiplication by a constant), uniform in $t$, for the integrand in
\begin{equation}
(-\Delta)^\frac{1}{4}\partial_t(u\circ\phi_t)(x)=\int_{\mathbb{S}^1}\left(\partial_t(u\circ\phi_t)(x)-\partial_t(u\circ\phi_t)(y)\right)K^\frac{1}{4}(x-y)dy
\end{equation}
for any $x\in\mathbb{S}^1$ (see Equations (\ref{eq_pointwise_formula_fraclap_circle_s}) and (\ref{eq_estimates_kernel_fraclap_s}). Therefore $\partial_t (-\Delta)^\frac{1}{4}(u\circ\phi_t)$ is well defined and $\partial_t (-\Delta)^\frac{1}{4}(u\circ\phi_t)=(-\Delta)^\frac{1}{4}\partial_t(u\circ\phi_t)$.
This concludes the proof of the claim.\\
By (\ref{eq_explicit_form_precomposition_vardom14}) and (\ref{eq_proof_prop_before_inversion_integral_derivative_in_zero}) we have
\begin{equation}\label{eq_stationary_equation_vardom_testfcs14}
\begin{split}
\frac{d}{dt}\bigg\vert_{t=0}E(u\circ\phi_t)=&\int_{\mathbb{S}^1}\frac{d}{dx} L\left(u(x), (-\Delta)^\frac{1}{4}u (x)\right)X(x)dx\\
&+D_pL\left(u(x), (-\Delta)^\frac{1}{4}u(x)\right)\left[u'(x)(-\Delta)^\frac{1}{4}X(x)\right.\\
&\left.-\int_{\mathbb{S}^1}(u'(x)-u'(y))(X(x)-X(y))K^\frac{1}{4}(x-y)dy\right]dx.
\end{split}
\end{equation}
We compute
\begin{equation}
\begin{split}
&\int_{\mathbb{S}^1} D_pL\left(u(x), (-\Delta)^\frac{1}{4}u(x)\right)\left[u'(x)(-\Delta)^\frac{1}{4}X(x)\right.\\
&\phantom{\int_{\mathbb{S}^1}}\left.-\int_{\mathbb{S}^1}(u'(x)-u'(y))(X(x)-X(y))K^\frac{1}{4}(x-y)\right]\\
=&\int_{\mathbb{S}^1}\int_{\mathbb{S}^1} D_pL\left(u(x), (-\Delta)^\frac{1}{4}u(x)\right)\left[u'(x)(X(x)-X(y))\right.\\
&\phantom{\int_{\mathbb{S}^1}\int_{\mathbb{S}^1}}\left.-(u'(x)-u'(y))(X(x)-X(y))\right]K^\frac{1}{4}(x-y)dxdy\\
=&\int_{\mathbb{S}^1}\int_{\mathbb{S}^1}D_pL\left(u(x), (-\Delta)^\frac{1}{4}u(x)\right)u'(y)(X(x)-X(y))K^\frac{1}{4}(x-y)dxdy\\
=&\text{div}_\frac{1}{4}\left(D_pf\left(u(x), (-\Delta)^\frac{1}{4}u(x)\right)u'(y) \right)[X]
\end{split}
\end{equation}
Now if $u$ is a stationary point, by the previous computations
\begin{equation}\label{eq_stationary_equation_vardom_testfcs214}
\begin{split}
0=&\frac{d}{dt}\bigg\vert_{t=0} E(u\circ\phi_t)=\int_{\mathbb{S}^1}\frac{d}{dx}L\left(u(x),(-\Delta)^\frac{1}{4}u(x)\right)X(x)dx\\
&+\text{div}_\frac{1}{4}\left(D_pL\left(u(x),(-\Delta)^\frac{1}{4}u(x)u'(y)\right)\right)[X].
\end{split}
\end{equation}
Then, since (\ref{eq_stationary_equation_vardom_testfcs214}) holds for any smooth vector field $X$, $u$ is a stationary point of $E$ if and only if it obeys the stationary equation
\begin{equation}\label{eq_stationary_equation_vardom14}
\begin{split}
\frac{d}{dx}L\left(u,(-\Delta)^\frac{1}{4}u\right)+\text{div}_\frac{1}{4}\left(D_p L\left(u(x),(-\Delta)^\frac{1}{4}u(x)\right)u'(y)\right)=0
\end{split}
\end{equation}
in the sense of distributions. Let now assume that for any $x\in \mathbb{S}^1$
\begin{equation}\label{eq_diffeo_is_conformal14}
L\left(u\circ\phi_t(x),(-\Delta)^\frac{1}{4}(u\circ\phi_t)(x)\right)=L\left(u(\phi_t(x)),(-\Delta)^\frac{1}{4}u(\phi_t(x))\right)\phi_t'(x).
\end{equation}
If we derive (\ref{eq_diffeo_is_conformal14}) with respect to $t$ and evaluate in $t=0$ we obtain
\begin{equation}
\begin{split}\label{eq_derivative_diffeo_is_conformal14}
\frac{d}{dt}\bigg\vert_{t=0}L\left(u\circ\phi_t, (-\Delta)^\frac{1}{4}(u\circ\phi_t)\right)(x)=&\frac{d}{dt}\bigg\vert_{t=0}L\left(u(\phi_t(x)),(-\Delta)^\frac{1}{4}u(\phi_t(x))\right)\\
&+L\left(u(x),(-\Delta)^\frac{1}{4}u(x)\right)X'(x)\\
=&\frac{d}{dx}L\left(u, (-\Delta)^\frac{1}{4}u\right)(x)X(x)\\
&+L\left(u(x),(-\Delta)^\frac{1}{4}u(x)\right)X'(x)\\
=&\frac{d}{dx}\left[L\left(u,(-\Delta)^\frac{1}{4}u\right)X\right](x).
\end{split}
\end{equation}
On the other hand, it follows from (\ref{eq_explicit_form_precomposition_vardom14}) that for a.e. $x\in \mathbb{S}^1$
\begin{equation}\label{eq_derivative_of_f14}
\begin{split}
&\frac{d}{dt}\bigg\vert_{t=0}L\left(u\circ\phi_t(x), (-\Delta)^\frac{1}{4}u\circ\phi_t(x)\right)=\frac{d}{dx}L\left(u(x),(-\Delta)^\frac{1}{4}u(x)\right)X(x)\\
&\qquad+D_pL\left(u(x), (-\Delta)^\frac{1}{4}u(x)\right)\left[u'(x)(-\Delta)^\frac{1}{4}X(x)\right.\\
&\qquad\left.-\int_{\mathbb{S}^1}(u'(x)-u'(y))(X(x)-X(y))K^\frac{1}{4}(x-y)dy\right].
\end{split}
\end{equation}
Therefore, combining (\ref{eq_stationary_equation_vardom14}), (\ref{eq_derivative_diffeo_is_conformal14}) and (\ref{eq_derivative_of_f14}) we obtain for any critical point $u\in C^\infty(\mathbb{S}^1, \mathbb{R}^m)$ of $E$, for a.e. $x\in \mathbb{S}^1$,
\begin{equation}\label{eq_proof_prop_before_first_version_frac_Noether_current}
\begin{split}
0=&\frac{d}{dx}L\left(u,(-\Delta)^\frac{1}{4}u\right)X(x)+\text{div}_\frac{1}{4}\left[D_p L\left(u(x),(-\Delta)^\frac{1}{4}u(x)\right) u'(y)\right]X(x)\\
=&\frac{d}{dx}\left[L\left(u, (-\Delta)^\frac{1}{4}u\right)X\right](x)+\text{div}_\frac{1}{4}\left( D_pL\left(u(x),(-\Delta)^\frac{1}{4}u(x)\right)u'(y)\right) X(x)\\
&-D_pL\left(u(x), (-\Delta)^\frac{1}{4}u(x)\right)\left[u'(x)(-\Delta)^\frac{1}{4}X(x)\right.\\
&\left.-\int_{\mathbb{S}^1}(u'(x)-u'(y))(X(x)-X(y))K^\frac{1}{4}(x-y)dy\right].
\end{split}
\end{equation}
Finally we remark that for any $\phi\in C^\infty(\mathbb{S}^1)$
\begin{equation}\label{eq_proof_prop_before_from_EL_to Ncurrent}
\begin{split}
&\text{div}_\frac{1}{4}\left( D_pL\left(u(x),(-\Delta)^\frac{1}{4}u(x)\right)u'(y)\right) [X\phi]-\int_{\mathbb{S}^1}D_pL\left(u(x), (-\Delta)^\frac{1}{4}u(x)\right)\\
&\cdot\left[u'(x)(-\Delta)^\frac{1}{4}X(x)-\int_{\mathbb{S}^1}(u'(x)-u'(y))(X(x)-X(y))K^\frac{1}{4}(x-y)dy\right]\phi(x)dx\\
=&\int_{\mathbb{S}^1}\int_{\mathbb{S}^1}D_pL\left(u(x), (-\Delta)^\frac{1}{4}u(x)\right)u'(y)(X(x)\phi(x)-X(y)\phi(y))K^\frac{1}{4}(x-y)dydx\\
&-\int_{\mathbb{S}^1}\int_{\mathbb{S}^1}D_pL\left(u(x), (-\Delta)^\frac{1}{4}u(x)\right)(u'(x)(X(x)-X(y))-(u'(x)-u'(y))\\
&\cdot(X(x)-X(y)))K^\frac{1}{4}(x-y)\phi(x)dydx\\
=&\int_{\mathbb{S}^1}\int_{\mathbb{S}^1}D_pL\left(u(x), (-\Delta)^\frac{1}{4}u(x)\right)X(y)(\phi(x)-\phi(y))K^\frac{1}{4}(x-y)dxdy\\
=&\text{div}_\frac{1}{4}\left(D_pL\left(u(x), (-\Delta)^\frac{1}{4}u(x)\right)u'(y)u'(y)X(y)\right)[\phi].
\end{split}
\end{equation}
Therefore we can rewrite (\ref{eq_proof_prop_before_first_version_frac_Noether_current}) as follows: for any critical point $u\in C^\infty(\mathbb{S}^1, \mathbb{R}^m)$ of $E$, for a.e. $x\in \mathbb{S}^1$
\begin{equation}\label{eq_discussion_final_equation_argument_Noeththm_for_smooth_functions_14}
0=\frac{d}{dx}\left[L\left(u, (-\Delta)^\frac{1}{4}u\right)X\right](x)+\text{div}_\frac{1}{4}\left(D_pL\left(u(x), (-\Delta)^\frac{1}{4}u(x)\right)u'(y)X(y)\right).
\end{equation}
%
%

We summarize the previous computations in the following proposition.
\begin{prop}\label{prop_first_version_Noether_vardom_14}
Let $L$ and $E$ be defined as in (\ref{eq_well_def_energy_vardom}), (\ref{eq_def_1_4_Dirichlet_energy}). Let $X$ be a vector field on $\mathbb{R}^m$ such that its flow $\phi_t$ obeys condition (\ref{eq_diffeo_is_conformal14}) for $t\in\mathbb{R}$ in a neighbourhood of $0$.
Then, for any $u\in C^\infty(\mathbb{S}^1, \mathbb{R})$, there holds
\begin{equation}\label{eq_prop_fundamental_idea_of_Noether_vardom_for_tests}
\begin{split}
&\text{div}_\frac{1}{4}\left(D_pL\left(u(x),(-\Delta)^\frac{1}{4}u(x)\right)u'(y)X(y)\right)+\frac{d}{dx}\left[L\left(u,(-\Delta)^\frac{1}{4}u\right)X\right]\\
=&\text{div}_\frac{1}{4}\left(D_pL\left(u(x),(-\Delta)^\frac{1}{4}u(x)\right)u'(y)\right)X+\frac{d}{dx}\left[L\left(u,(-\Delta)^\frac{1}{4}u\right)\right]X
\end{split}
\end{equation}
in the sense of distributions.
In particular, if $u\in C^\infty(\mathbb{S}^1, \mathbb{R}^m)$ is a stationary point of $E$, there holds
\begin{equation}\label{eq_prop_conserved_quantity_14_C2}
\begin{split}
\text{div}_\frac{1}{4}\left(D_pL\left(u(x),(-\Delta)^\frac{1}{4}u(x)\right)u'(y)X(y)\right)+\frac{d}{dx}\left[L\left(u,(-\Delta)^\frac{1}{4}u\right)X\right](x)=0
\end{split}
\end{equation}
for any $x\in \mathbb{S}^1$.
\end{prop}
We now try to weaken the hypothesis that $u$ belongs to $C^\infty(\mathbb{S}^1)$. To this end we will first need to show that (\ref{eq_prop_conserved_quantity_14_C2}) is well defined, in the sense of distributions, even if $u\in H^\frac{1}{2}(\mathbb{S}^1)$. This will be done in Lemma \ref{lemma_approximation_identity_14}.
Let's first introduce the following space of functions.
\begin{defn}
Let $\mathbb{A}(\mathbb{S}^1)$ be the space of all the elements $u$ in $\mathscr{D}'(\mathbb{S}^1)$ such that
\begin{equation}
\lVert u\rVert_{\mathbb{A}}:=\sum_{k\in\mathbb{Z}}\lvert \widehat{u}(k)\rvert<\infty.
\end{equation}
$\mathbb{A}(\mathbb{S}^1)$ is called \textbf{Wiener algebra}.
\end{defn}
\begin{rem}\label{rem_Wiener_algebra}\noindent
\begin{enumerate}
\item $\mathbb{A}(\mathbb{S}^1)$ is a Banach algebra, where the multiplication is given by convolution
\item There is a continuous embedding $\mathbb{A}(\mathbb{S}^1)\hookrightarrow C^0(\mathbb{S}^1)$. Indeed, if $u\in \mathbb{A}(\mathbb{S}^1)$, $\lVert u\rVert_{L^\infty}\leq \lVert u\rVert_{\mathbb{A}}$ and the sequence of continuous functions given by $u_n(x)=\sum_{\lvert k\rvert\leq n}\widehat{u}(n)e^{inx}$ for $x\in \mathbb{S}^1$, $n\in\mathbb{N}$ converges uniformly to $f$.
\item For any $s>\frac{1}{2}$ there is a continuous embedding $H^s(\mathbb{S}^1)\hookrightarrow \mathbb{A}(\mathbb{S}^1)$. Indeed, for any $u\in H^s(\mathbb{S}^1)$,
\begin{equation}\label{eq_rem_embedding_Sobolev_in_Wiener}
\sum_{\substack{k\in\mathbb{Z}\\ k\neq 0}}\lvert \widehat{u}(k)\rvert\leq \left(\sum_{\substack{k\in\mathbb{Z}\\ k\neq 0}}\lvert \widehat{u}(k)\rvert^2\lvert k\rvert^{2s}\right)^\frac{1}{2}\left(\sum_{\substack{k\in\mathbb{Z}\\ k\neq 0}}\lvert k\rvert^{-2s}  \right)^\frac{1}{2}
\end{equation}
by Cauchy-Schwartz inequality. As $s>\frac{1}{2}$, the second factor is finite. Moreover $\lvert \widehat{u}(0)\rvert\lvert \lVert u\rVert_{H^s}$.
\item for any $\alpha\in(0,1]$, $\beta<\alpha$ there is a continuous embedding $C^\alpha(\mathbb{S}^1)\hookrightarrow H^\beta(\mathbb{S}^1)$ (see Theorem 1.13 in \cite{Schlag}). Thus, by the previous point, if $\alpha>\frac{1}{2}$, there is a continuous embedding $C^\alpha(\mathbb{S}^1)\hookrightarrow\mathbb{A}(\mathbb{S}^1)$.
\end{enumerate}
\end{rem}

We also introduce the normed space $\mathbb{A}^1(\mathbb{S}^1)$, consisting of all the elements $u\in \mathscr{D}'(\mathbb{S}^1)$ such that
\begin{equation}\label{eq_discussion_norm_A'}
\lVert u\rVert_{\mathbb{A}^1}:=\lvert \widehat{u}(0)\rvert+\sum_{\substack{k\in\mathbb{Z},\\ k\neq 0}}\lvert k\rvert\lvert \widehat{u}(k)\rvert<\infty.
\end{equation}
By the argument used in (\ref{eq_rem_embedding_Sobolev_in_Wiener}), we see that for any $s>\frac{1}{2}$ 
\begin{equation}\label{eq_embedding_Sobolev_in_Wiener_der}
H^{1+s}(\mathbb{S}^1)\hookrightarrow \mathbb{A}^1(\mathbb{S}^1).
\end{equation}
For any element $u$ of $\mathbb{A}^1(\mathbb{S}^1)$ we also define the seminorm
\begin{equation}
[u]_{\mathbb{A}^1}:=\sum_{\substack{k\in\mathbb{Z}\\ k\neq 0}}\lvert k\rvert\lvert \widehat{u}(k)\rvert.
\end{equation}

\begin{lem}\label{lemma_approximation_identity_14}
Let $w\in L^{2,1}(\mathbb{S}^1)$, $\phi\in \mathbb{A}^1(\mathbb{S}^1)$ and let $K^\frac{1}{4}$ be defined as in (\ref{def_Kernel_Ks_0_12}).
Let
\begin{equation}\label{eq_lem_approximation_identity_14}
G_{w,\phi}(y):=\int_{\mathbb{S}^1}w(x)(\phi(x)-\phi(y))K^\frac{1}{4}(x-y)dx
\end{equation}
for any $y\in \mathbb{S}^1$. The integral in (\ref{eq_lem_approximation_identity_14}) converges absolutely for a.e. $y\in \mathbb{S}^1$ and $G_{w,\phi}$ belongs to $L^2(\mathbb{S}^1)$. Moreover there is a constant $C$ depending only on $X$, such that
\begin{equation}\label{eq_lem_estimate_Gphi}
\lVert G_{w,\phi}\rVert_{L^2}\leq C[ \phi]_{\mathbb{A}^1}[w]_{\dot{H}^{-\frac{1}{2}}}.
\end{equation}
\end{lem}


\begin{proof}
First we claim that the integral in (\ref{eq_lem_approximation_identity_14}) converges absolutely for a.e. $y\in \mathbb{S}^1$ and $G_{w,\phi}$ defines a function in $L^2(\mathbb{S}^1)$ for any $w\in L^{2,1}(\mathbb{S}^1)$, $\phi\in \mathbb{A}^1(\mathbb{S}^1)$. In fact we estimate
\begin{equation}\label{eq_proof_lem_l2_norm_bounded_to_show_integrable_ae}
\begin{split}
&\int_{\mathbb{S}^1}\left(\int_{\mathbb{S}^1}\left\lvert w(x)(\phi(x)-\phi(y))K^\frac{1}{4}(x-y)\right\rvert dx\right)^2dy\\
\leq &\int_{\mathbb{S}^1}\left(\int_{\mathbb{S}^1}\left\lvert w(x)\frac{\pi[\phi]_{Lip}B_\frac{1}{4}}{\left\lvert \sin\left(\frac{1}{2}(x-y)\right)\right\rvert^\frac{1}{2}}\right\rvert dx\right)^2dy\\
\leq & 2\pi\left(\pi[\phi]_{Lip}B_\frac{1}{4} \right)^2\lVert w\rVert_{L^{2,1}}^2 \left\lVert\left\lvert \sin \left(\frac{1}{2}\cdot\right)\right\rvert^{-\frac{1}{2}}\right\rVert_{L^{2,\infty}}^2,
\end{split}
\end{equation}
where in the first step we used estimate (\ref{eq_estimates_kernel_fraclap_s}), while in the last step we used the fact that $L^{2,\infty}(\mathbb{S}^1)$ is the topological dual space of $L^{2,1}(\mathbb{S}^1)$, and therefore the $L^\infty$-norm of the internal integral of the second line can be bounded by the product of the two norms. This concludes the proof of the first claim.\newline
Next we observe that for any $n\in\mathbb{Z}$, $y\in \mathbb{S}^1$,
\begin{equation}
G_{w,{e^{in\cdot}}}(y)=\int_{\mathbb{S}^1}w(x)e^{iny}(e^{in(x-y)}-1)K^\frac{1}{4}(x-y)dx.
\end{equation}
For any $n\in\mathbb{Z}$, $x\in \mathbb{S}^1$ we define
\begin{equation}
H_n(x):=(e^{-inx}-1)K^\frac{1}{4}(x).
\end{equation}
Then, by (\ref{eq_estimates_kernel_fraclap_s}), $H_n\in L^{2,\infty}(\mathbb{S}^1)$ for all $n\in \mathbb{Z}$.
Now we compute for any $n,k\in \mathbb{Z}$
\begin{equation}
\begin{split}
\widehat{H_n}(k)=&\frac{1}{2\pi}\int_{\mathbb{S}^1}e^{-ikz}(e^{-inz}-1)K^\frac{1}{4}(z)dz=\frac{1}{2\pi}\int_{\mathbb{S}^1}(e^{-i(n+k)}z)-e^{-ikz})K^\frac{1}{4}(z)dz\\=&\int_{\mathbb{S}^1}(e^{-i(n+k)}z)-1))K^\frac{1}{4}(z)dz-\int_{\mathbb{S}^1}(e^{-ik}z)-1))K^\frac{1}{4}(z)dz.
\end{split}
\end{equation}
We observe that for any $m\in \mathbb{Z}$
\begin{equation}
\begin{split}
\frac{1}{2\pi}\int_{\mathbb{S}^1}(e^{-imz}-1)K^\frac{1}{4}(z)dz
=&\frac{1}{(2\pi)^2}\int_{\mathbb{S}^1}e^{-imy}\int_{\mathbb{S}^1}(e^{im(y-z)}-e^{imy})K^\frac{1}{4}(z)dzdy\\
=&\frac{1}{(2\pi)^2}\int_{\mathbb{S}^1}e^{-imy}(-\Delta)^\frac{1}{4}e^{im\cdot}(y)dy\\
=&\frac{1}{2\pi}\mathscr{F}\left((-\Delta)^\frac{1}{4}e^{im\cdot}\right)(m)\\
=&\frac{\lvert m\rvert^\frac{1}{2}}{2\pi},
\end{split}
\end{equation}
therefore for any $n,k\in \mathbb{Z}$
\begin{equation}
\widehat{H_n}(k)=\frac{1}{2\pi}\left(\lvert k+n\rvert^\frac{1}{2}-\lvert k\rvert^\frac{1}{2}\right).
\end{equation}
Now assume that $\phi$ is a trigonometric polynomial of degree $N$. Wlog we can assume that $\hat{\phi}(0)=0$. Then for a.e. $y\in \mathbb{S}^1$
\begin{equation}
\begin{split}
G_{w,\phi}(y)=&\sum_{0<\lvert n\rvert\leq N}\widehat{\phi}(n)\int_{\mathbb{S}^1}w(x)(e^{inx}-e^{iny})K^\frac{1}{4}(x-y)dx\\
=&\sum_{0<\lvert n\rvert\leq N}\widehat{\phi}(n)\int_{\mathbb{S}^1}w(x)e^{iny}H_n(y-x)dx=
\sum_{0<\lvert n\rvert\leq N}\hat{\phi}(n)e^{iny}w\ast H_n(y)
\end{split}
\end{equation}
For any $k\in\mathbb{Z}$ we have
\begin{equation}
\begin{split}
\widehat{G_{w,\phi}}(k)=&\sum_{0<\lvert n\rvert \leq N}\widehat{\phi}(n)\mathscr{F}(e^{in\cdot}(w)\ast H_n)(k)=
\sum_{0<\lvert n\rvert \leq N}\widehat{\phi}(n)\mathscr{F}(w\ast H_n)(k-n)\\
=&\sum_{0<\lvert n\rvert \leq N}\widehat{\phi}(n)\widehat{w}(k-n)\widehat{H_n}(k-n).
\end{split}
\end{equation}
Therefore, by Parseval's identity,
\begin{equation}\label{eq_proof_lemma_L2norm_Gphi}
\begin{split}
\lVert G_{u,\phi}\rVert_{L^2}^2=\sum_{k\in\mathbb{Z}}\left(\sum_{0<\lvert n\rvert \leq N}\widehat{\phi}(n)\widehat{w}(k-n)\widehat{H_n}(k-n)\right)^2
\end{split}
\end{equation}
%
Now we observe that if $\lvert k\rvert\geq 2\lvert n\rvert>0$, there holds $\lvert k-n\rvert\leq \frac{3}{2}\lvert k\rvert$ and $\frac{\lvert k\rvert}{2}\leq \lvert k+n\rvert$; therefore
\begin{equation}
\begin{split}
\left\lvert \lvert k+n\rvert^\frac{1}{2}-\lvert k\rvert^\frac{1}{2}\right\rvert\leq& \frac{1}{2}\int_{\lvert k\rvert\wedge\lvert n+k\rvert}^{\lvert k\rvert\vee\lvert n+k\rvert}s^{-\frac{1}{2}}ds\leq \frac{1}{2} \lvert n\rvert\left(\frac{\lvert k\rvert}{2}\right)^{-\frac{1}{2}}\\
\leq& \frac{1}{2}\left(\frac{3}{2}\right)^\frac{1}{2} \lvert n\rvert\lvert k-n\rvert^{-\frac{1}{2}}.
\end{split}
\end{equation}
Therefore Young's Inequality yields
\begin{equation}\label{eq_proof_lemma_L2norm_Gphi_firstsum}
\begin{split}
&\sum_{k\in\mathbb{Z}}\left(\sum_{0<\lvert n\rvert\leq\frac{\lvert k\rvert}{2}}\left\lvert \widehat{\phi}(n)\widehat{w}(k-n)\right\rvert\left\lvert\lvert k+n\rvert^\frac{1}{2}-\lvert k\rvert^\frac{1}{2}\right\rvert\right)^2\\
\leq&\frac{3}{4}\sum_{k\in\mathbb{Z}}\left(\sum_{0<\lvert n\rvert\leq\frac{\lvert k\rvert}{2}}\left\lvert \widehat{\phi}(n) n\right\rvert\left\lvert\widehat{w}(k-n)\right\rvert\lvert k-n\rvert^{-\frac{1}{2}}\right)^2\\
\leq& \frac{3}{4}\left(\sum_{0< \lvert n\rvert\leq N}\left\lvert \widehat{\phi}(n)n\right\rvert\right)^2\left(\sum_{k\in\mathbb{Z}}\left\lvert \widehat{w}(k)\right\rvert^2\lvert k\rvert^{-1}\right)=\frac{3}{4}[\phi]_{\mathbb{A}^1}^2[w]_{\dot{H}^{-\frac{1}{2}}}^2.
\end{split}
\end{equation}


On the other hand, we always have
\begin{equation}
\left\lvert\lvert n+k\rvert^\frac{1}{2}-\lvert k\rvert^\frac{1}{2}\right\rvert\leq \lvert n\rvert^\frac{1}{2},
\end{equation}
and if $\lvert k\rvert<2\lvert n\rvert$, $\lvert k-n\rvert\leq 3\lvert n\rvert$.
Therefore, for any $k\in\mathbb{Z}$,
\begin{equation}\label{eq_proof_lemma_L2norm_Gphi_secondsum}
\begin{split}
&\sum_{k\in\mathbb{Z}}\left(\sum_{\substack{0<\lvert n\rvert\leq N\\ \lvert k\rvert<2\lvert n\rvert}}\left\lvert \widehat{\phi}(n)\widehat{w}(k-n)\right\rvert\left\lvert\lvert k+n\rvert^\frac{1}{2}-\lvert k\rvert^\frac{1}{2}\right\rvert\right)^2\\
\leq&\sum_{k\in\mathbb{Z}}\left(\sum_{\substack{0<\lvert n\rvert\leq N\\ \lvert k\rvert<2\lvert n\rvert}}\left\lvert \widehat{\phi}(n)\right\rvert\lvert n\rvert^\frac{1}{2}\lvert k-n\rvert^{\frac{1}{2}}\left\lvert\widehat{w}(k-n)\right\rvert\lvert k-n\rvert^{-\frac{1}{2}}\right)^2\\
\leq & 9 \sum_{k\in\mathbb{Z}}\left(\sum_{0<\lvert n\rvert\leq N}\left\lvert \widehat{\phi}(n)\right\rvert\lvert n\rvert\left\lvert\widehat{w}(k-n)\right\rvert\lvert\lvert k-n\rvert^{-\frac{1}{2}}\right)^2\\
\leq&9\left(\sum_{0< \lvert n\rvert\leq N}\left\lvert \widehat{\phi}(n)n\right\rvert\right)^2\left(\sum_{k\in\mathbb{Z}}\left\lvert \widehat{w}(k)\right\rvert^2\lvert k\rvert^{-1}\right)=9[ \phi]_{\mathbb{A}^1}^2[ w]_{\dot{H}^{-\frac{1}{2}}}^2,
\end{split}
\end{equation}
where the second-last steps follow again from Young's inequality.
Therefore, combining (\ref{eq_proof_lemma_L2norm_Gphi_firstsum}) and (\ref{eq_proof_lemma_L2norm_Gphi_secondsum}) we obtain
\begin{equation}\label{eq_proof_lemma_intermediate_step_expression_with_seminorm}
\lVert G_{u, \phi}\rVert_{L^2}^2\leq\frac{1}{(2\pi)^2}\left(\frac{3}{4}+9\right)[\phi]_{\mathbb{A}^1}^2[w]_{\dot{H}^\frac{1}{2}}^2
\end{equation}
for any trigonometric polynomial $\phi$.
For the general case, let $\phi\in \mathbb{A}^1(\mathbb{S}^1)$ and for any $n\in \mathbb{N}$ let 
\begin{equation}
\phi_n:=D_n\ast \phi.
\end{equation}
Here $D_n$ denotes the $n^{th}$ Dirichlet kernel, which can be defined through its Fourier coefficients as follows:
\begin{equation}\label{eq_def_Dirichlet_kernel_Fourier}
\widehat{D_n}(k):=
\begin{cases}
1 \text{ if }\lvert k\rvert\leq n\\
0 \text{ if }\lvert k\rvert> n.
\end{cases}
\end{equation}
Then for any $n\in\mathbb{N}$ $\phi_n$ is a trigonometric polynomial, and
\begin{equation}
\lim_{n\to \infty} \phi_n =\phi\text{ in }\mathbb{A}^1(\mathbb{S}^1).
\end{equation}
As estimate (\ref{eq_lem_estimate_Gphi}) holds for any trigonometric polynomial, and since $\phi_n\to \phi$ in $\mathbb{A}^1(\mathbb{S}^1)$, the sequence $(G_{w,\phi_n})_n$ is a Cauchy sequence in $L^2(\mathbb{S}^1)$, and therefore converges towards an element $G_{w, \infty}\in L^2(\mathbb{S}^1)$. We would like to show that $G_{w,\infty}=G_{w,\phi}$. To this end let $\psi\in C^\infty(\mathbb{S}^1)$. Then for $n\in \mathbb{N}$ we consider
\begin{equation}\label{eq_proof_lemma_show_same_limit_Cauchy_limit_coeff}
\begin{split}
&\int_{\mathbb{S}^1}\psi(y)\left[G_{w,\phi}(y)-G_{w,\phi_n}(y)\right]dy\\
=&\int_{\mathbb{S}^1}\psi(y)\int_{\mathbb{S}^1}w(y)(\phi(x)-\phi_n(x)-(\phi(y)-\phi_n(y))K^\frac{1}{4}(x-y)dxdy.
\end{split}
\end{equation}
We remark that $\mathbb{A}^1(\mathbb{S}^1)\hookrightarrow C^1(\mathbb{S}^1)$ with
\begin{equation}\label{eq_proof_lemma_estimate_embedding_A'_C1}
\lVert u'\rVert_{L^\infty}\leq [u]_{\mathbb{A}^1}
\end{equation}
for any $u\in \mathbb{A}^1$.
Therefore the integral in (\ref{eq_proof_lemma_show_same_limit_Cauchy_limit_coeff}) converges absolutely in $\mathbb{S}^1\times\mathbb{S}^1$ by estimate (\ref{eq_estimates_kernel_fraclap_s}). Moreover we observe that, by (\ref{eq_proof_lemma_estimate_embedding_A'_C1}), for $x,y\in \mathbb{S}^1$, $n\in\mathbb{N}$
\begin{equation}
\lvert \phi(x)-\phi_n(x)-(\phi(y)-\phi_n(y))\rvert\leq [\phi-\phi_n]_{\mathbb{A}^1}\lvert x-y\rvert
\end{equation}
and therefore, by Lebesgue's Dominated convergence Theorem, the expression in (\ref{eq_proof_lemma_show_same_limit_Cauchy_limit_coeff}) converges to $0$ as $n\to \infty$.
By the uniqueness of the limit for convergence in $\mathscr{D}'(\mathbb{S}^1)$, we concludes that $G_{w,\infty}=G_{w,\phi}$.
\end{proof}
\begin{rem}\label{rem_to_lemma_approximation_identity_14}\noindent
\begin{enumerate}
\item by Remark \ref{rem_Wiener_algebra}, for $\phi$ to be in $\mathbb{A}^1(\mathbb{S}^1)$ it suffices that $\phi\in H^{1+\alpha}(\mathbb{S}^1)$ for some $\alpha>\frac{1}{2}$. In particular, this is true if $\phi\in C^{1,\alpha}(\mathbb{S}^1)$ for some $\alpha>\frac{1}{2}$.
\item Lemma \ref{lemma_approximation_identity_14} shows that for any $\phi\in C^\infty(\mathbb{S}^1)$, we have a continuous linear operator
\begin{equation}
F:L^{2,1}(\mathbb{S}^1)\to L^2(\mathbb{S}^1),\quad w\mapsto G_{w, \phi}
\end{equation}
By (\ref{eq_lem_estimate_Gphi}), $F$ can be extended to a continuous operator on $H^{-\frac{1}{2}}(\mathbb{S}^1)$ (still denoted by $F$), for which (\ref{eq_lem_estimate_Gphi}) remains valid.
\item We also observe that if $a\in L^{2,1}(\mathbb{S}^1)$, $b\in L^2(\mathbb{S}^1)$, for any $\phi\in C^\infty(\mathbb{S}^1)$ we have
\begin{equation}
\text{div}_\frac{1}{4}\left( a(x) \cdot b(y) \right)[\phi]=\int_{\mathbb{S}^1}G_{a, \phi}(x)\cdot b(x)dx.
\end{equation}
Therefore there holds
\begin{equation}\label{eq_rem_estimate_divergence_for_products}
\begin{split}
\left\lvert {div}_\frac{1}{4}\left( a(x)\cdot b(y)\right)[\phi]\right\rvert\leq &\lVert G_{a,\phi}\rVert_{L^2}\lVert b\rVert_{L^2}\\
\leq & C[\phi]_{\mathbb{A}^1}\lVert a\rVert_{H^{-\frac{1}{2}}}\lVert  b\rVert_{L^2},
\end{split}
\end{equation}
where $C$ is an independent constant.
\end{enumerate}
\end{rem}
By the previous Remark, we may extend the definition of
\begin{equation}
\text{div}_\frac{1}{4}\left[a(x)\cdot b(y)\right]
\end{equation}
to functions $a\in H^{-\frac{1}{2}}(\mathbb{S}^1)$, $b\in L^2(\mathbb{S}^1)$ as follows: for $\phi\in C^\infty(\mathbb{S}^1)$ let
\begin{equation}\label{eq_discussion_extension_definition_fracdiv_14}
\text{div}_\frac{1}{4}\left[a(x)\cdot b(y)\right][\phi]:=\int_{\mathbb{S}^1}G_{a,\phi}(x)\cdot b(x)dx,
\end{equation}
where $G_{a, \phi}:=F(a)$ was defined in the previous Remark.
Therefore the following estimate still holds:
\begin{equation}
\begin{split}
\left\lvert \text{div}_\frac{1}{4}\left( a(x)\cdot b(y)\right)[\phi]\right\rvert\leq & \lVert G_{a,\phi}\rVert_{L^2}\lVert b\rVert_{L^2}\\
\leq & C[\phi]_{\mathbb{A}^1}\lVert a\rVert_{H^{-\frac{1}{2}}} \lVert b\rVert_{L^2}
\end{split}
\end{equation}
where $C$ is an independent constant.\\
In particular, we will be interested in the case where $a=u'X$ and \newline$b=D_pL\left(u, (-\Delta)^\frac{1}{4}u\right)$, for $X$ and $L$ as in Proposition \ref{prop_first_version_Noether_vardom_14} and $u\in H^\frac{1}{2}(\mathbb{S}^1, \mathbb{R}^m)$. To this end, we first check that we can apply the previous Lemma for this choice of $a$ and $b$ (assuming that $D_pf\left(u, (-\Delta)^\frac{1}{4}u\right)\in L^2(\mathbb{S}^1)$).

\begin{lem}\label{lem_verify_condition_approxlem}
Let $u\in H^\frac{1}{2}(\mathbb{S}^1, \mathbb{R}^m)$, let $X$ a smooth vector field on $\mathbb{S}^1$. Then $u'X\in H^{-\frac{1}{2}}(\mathbb{S}^1)$ and
\begin{equation}
\lVert u'X \rVert_{H^{-\frac{1}{2}}}\leq C_X \lVert u\rVert_{H^\frac{1}{2}}
\end{equation}
for some constant $C_X$ depending only on $X$.
\end{lem}

\begin{proof}
We compute $\lVert u'X \rVert_{H^{-\frac{1}{2}}}$ exploiting the fact that $H^{-\frac{1}{2}}(\mathbb{S}^1)^\ast\approx H^\frac{1}{2}(\mathbb{S}^1)$:
\begin{equation}\label{eq_proof_lemma_estimate_negative_sobolev_norm_using_duality}
\begin{split}
\lVert u'X\rVert_{H^{-\frac{1}{2}}}\leq &\sup_{\substack{\phi\in H^\frac{1}{2}(\mathbb{S}^1)\\ \lVert \phi \rVert_{H^\frac{1}{2}}\leq 1}} \langle u' X, \phi \rangle
=\sup_{\substack{\phi\in H^\frac{1}{2}(\mathbb{S}^1)\\ \lVert \phi \rVert_{H^\frac{1}{2}}\leq 1}} \langle u' , X\phi \rangle\\
\leq & \sup_{\substack{\phi\in H^\frac{1}{2}(\mathbb{S}^1)\\ \lVert \phi \rVert_{H^\frac{1}{2}}\leq 1}}\lVert u'\rVert_{H^{-\frac{1}{2}}}\lVert X \phi\rVert_{H^\frac{1}{2}}
\leq  C\sup_{\substack{\phi\in H^\frac{1}{2}(\mathbb{S}^1)\\ \lVert \phi \rVert_{H^\frac{1}{2}}\leq 1}} \lVert u\rVert_{H^\frac{1}{2}}\lVert X\rVert_{C^3}\lVert \phi\rVert_{H^\frac{1}{2}}\\ 
\leq & C \lVert u\rVert_{H^\frac{1}{2}}\lVert X\rVert_{C^3}
\end{split}
\end{equation}
for some constant $C$ independent from $u$ and $X$. In the third step we made use of Lemma \ref{lem_multiplication_with_test_function_continuous}.
\end{proof}
Thus, by the previous Lemmas, if $u\in H^\frac{1}{2}(\mathbb{S}^1,\mathbb{R}^m)$ and $L$ and $X$ are like in Proposition \ref{prop_first_version_Noether_vardom_14}, if $D_pL\left(u, (-\Delta)^\frac{1}{4}u\right)\in L^2(\mathbb{S}^1)$,
\begin{equation}\label{eq_discussion_div14_is_well_defined_by_previous_results}
\text{div}_\frac{1}{4}\left( D_pL\left(u(x), (-\Delta)^\frac{1}{4}u(x)\right)X(y)u'(y) \right)
\end{equation}
is a well defined distribution, and there holds the following estimate: for any $\phi\in C^\infty(\mathbb{S}^1)$,
\begin{equation}\label{eq_discussion_estimate_for_extended_fracdiv_14}
\begin{split}
&\left\lvert \text{div}_\frac{1}{4}\left( D_pL\left(u(x),(-\Delta)^\frac{1}{4}u(x)\right)u'(y)X(y)\right)[\phi]\right\rvert\\
\leq &\lVert G_{u'X,\phi}\rVert_{L^2}\left\lVert D_pL\left(u, (-\Delta)^\frac{1}{4}u\right)\right\rVert_{L^2}\\
\leq & C[\phi]_{\mathbb{A}^1}\lVert u\rVert_{H^{\frac{1}{2}}}\lVert X\rVert_{C^3}\left\lVert D_pL\left(u, (-\Delta)^\frac{1}{4}u\right) \right\rVert_{L^2}.
\end{split}
\end{equation}
for some independent constant $C$.\\

We are now able to formulate the analogous of Noether's Theorem (Theorem \ref{thm_Noether_thm_vardom}) for variations in the domain for stationary points of the energy (\ref{eq_def_1_4_Dirichlet_energy}).

\begin{thm}\label{prop_final_version_Noether_thm_vardom}
Let $L\in C^2(\mathbb{R}^m\times\mathbb{R}^m, \mathbb{R})$ so that for any $x,p\in\mathbb{R}^m$
\begin{equation}\label{eq_well_def_energy_vardom_2}
\lvert L(x,p)\rvert\leq C(1+\lvert p\rvert^2)
\end{equation}
for some constant $C>0$. For any $u\in H^\frac{1}{2}(\mathbb{S}^1,\mathbb{R}^m)$ we set
\begin{equation}\label{eq_def_1_4_Dirichlet_energy_2}
E(u):=\int_{\mathbb{S}^1}L\left(u(x),(-\Delta)^\frac{1}{4}u(x)\right)dx.
\end{equation}
Let $X$ be a smooth vector field on $\mathbb{S}^1$ and assume that its flow $\phi_t$ satisfies
\begin{equation}\label{eq_diffeo_is_conformal14_2}
L\left(u\circ\phi_t(x),(-\Delta)^\frac{1}{4}(u\circ\phi_t)(x)\right)=L\left(u(\phi_t(x)),(-\Delta)^\frac{1}{4}u(\phi_t(x))\right)\phi_t'(x)
\end{equation}
for a.e. $x\in \mathbb{S}^1$, for $t$ in a neighbourhood of the origin.
Assume that $u\in H^\frac{1}{2}(\mathbb{S}^1, \mathbb{R}^m)$ satisfy the following stationarity equation:
\begin{equation}\label{eq_prop_stationary_equation_14}
\frac{d}{dx}L\left(u, (-\Delta)^\frac{1}{4}u\right)X+\text{div}_\frac{1}{4}\left[ D_pL\left(u(x), (-\Delta)^\frac{1}{4}u(x) \right)u'(y)\right]X=0
\end{equation}
as distributions,
and assume that there exists a sequence $(u_n)_n$ in $C^\infty(\mathbb{S}^1)$ such that $u_n\to u$ in $H^\frac{1}{2}(\mathbb{S}^1)$ and
\begin{equation}\label{eq_prop_assumption_convergence_weak}
D_pL\left(u_n, (-\Delta)^\frac{1}{4}u_n\right)\rightharpoondown D_pL\left(u, (-\Delta)^\frac{1}{4}u\right)\text{ weakly in }L^2(\mathbb{S}^1).
\end{equation}
Then, in the sense of distributions
\begin{equation}\label{eq_prop_result_Noether_thm_vardom_14}
\frac{d}{dx}\left[L\left(u, (-\Delta)^\frac{1}{4}u\right)X\right](x)+\text{div}_\frac{1}{4}\left[D_pL\left(u(x), (-\Delta)^\frac{1}{4}u(x)\right)u'(y)X(y)\right]=0.
\end{equation}
\end{thm}
\begin{proof}
Let $(u_n)_n$ be a sequence in $C^\infty_c(\mathbb{S}^1, \mathbb{R}^m)$ such that $u_n\to u$ in $H^\frac{1}{2}(\mathbb{S}^1)$.
By Lemma \ref{lemma_approximation_identity_14} and Remark \ref{rem_to_lemma_approximation_identity_14}, for any $\phi\in C^\infty(\mathbb{S}^1)$
\begin{equation}
G_{u_n'X,\phi}\to G_{u'X,\phi} \text{ in }L^2(\mathbb{S}^1).
\end{equation}
Thus, by the assumptions on $(u_n)_n$,
\begin{equation}
\begin{split}
&\text{div}_\frac{1}{4}\left[D_pL\left(u_n(x), (-\Delta)^\frac{1}{4}u_n(x)\right)u_n'(y)X(y)\right][\phi]\\=&\int_{\mathbb{S}^1}G_{u_n'X,\phi}(x)D_pL\left(u_n(x), (-\Delta)^\frac{1}{4}u_n(x)\right)dx\\
\to&\int_{\mathbb{S}^1}G_{u'X,\phi}(x)D_pL\left(u(x), (-\Delta)^\frac{1}{4}u(x)\right)dx\\
=&\text{div}_\frac{1}{4}\left[D_pL\left(u(x), (-\Delta)^\frac{1}{4}u(x)\right)u'(y)X(y)\right][\phi].
\end{split}
\end{equation}
as $n\to \infty$.
By the same argument with $Y\equiv 1$  we obtain
\begin{equation}
\begin{split}
&\text{div}_\frac{1}{4}\left[D_pL\left(u_n(x), (-\Delta)^\frac{1}{4}u_n(x)\right)u_n'(y)\right][\phi]\\
\to &\text{div}_\frac{1}{4}\left[D_pL\left(u(x), (-\Delta)^\frac{1}{4}u(x)\right)u'(y)\right][\phi].
\end{split}
\end{equation}
Now we observe that for any $n\in \mathbb{N}$, computation (\ref{eq_proof_prop_before_from_EL_to Ncurrent}) implies
\begin{equation}
\begin{split}
&\frac{d}{dx}L\left(u_n,(-\Delta)^\frac{1}{4}u_n\right)X+\text{div}_\frac{1}{4}\left[D_p L\left(u_n(x),(-\Delta)^\frac{1}{4}u_n(x)\right)u_n'(y)\right]X\\
=&\text{div}_\frac{1}{4}\left[D_pL\left(u_n(x),(-\Delta)^\frac{1}{4}u_n(x)\right)u_n'(y)X(y)\right]+\frac{d}{dx}\left[L\left(u_n,(-\Delta)^\frac{1}{4}u_n\right)X\right]
\end{split}
\end{equation}
as distributions.
Therefore
\begin{equation}\label{eq_proof_prop_computation_to_pass_to_the_limit_Noether_vardom}
\begin{split}
0=&-\int_{\mathbb{S}^1}L\left(u(x), (-\Delta)^\frac{1}{4}u(x)\right)\frac{d}{dx}\left(X\phi\right)(x)dx\\
&+\text{div}_\frac{1}{4}\left(D_pL\left(u(x), (-\Delta)^\frac{1}{4}u(x)\right)u'(y)\right)[X\phi]\\
=&\lim_{n\to \infty}\int_{\mathbb{S}^1}L\left(u_n(x), (-\Delta)^\frac{1}{4}u_n(x)\right)\frac{d}{dx}\left(X\phi\right)(x)dx\\
&+\lim_{n\to \infty}\text{div}_\frac{1}{4}\left(D_pL\left(u_n(x), (-\Delta)^\frac{1}{4}u_n(x)\right)u_n'(y)\right)[X\phi]\\
=&\lim_{n\to\infty}\text{div}_\frac{1}{4}\left(D_pL\left(u_n(x), (-\Delta)^\frac{1}{4}u_n(x)\right)u_n'(y)X(y)\right]\\
&-\lim_{n\to \infty}\int_{\mathbb{S}^1}f\left(u_n(x), (-\Delta)^\frac{1}{4}u_n(x)\right)X(x)\phi'(x)dx\\
=& \text{div}_\frac{1}{4}\left(D_pL\left(u(x), (-\Delta)^\frac{1}{4}u(x)\right)u'(y)X(y)\right].
\end{split}
\end{equation}
As (\ref{eq_proof_prop_computation_to_pass_to_the_limit_Noether_vardom}) holds for any $\phi\in C^\infty(\mathbb{S}^1)$ we conclude that
\begin{equation}
\text{div}_\frac{1}{4}\left[D_pL\left(u(x), (-\Delta)^\frac{1}{4}u(x)\right)u'(y)X(y)\right]+\frac{d}{dx}\left[L\left(u, (-\Delta)^\frac{1}{4}u\right)X\right](x)=0.
\end{equation}
in the sense of distributions.
\end{proof}
\begin{rem}
It is possible to rewrite Equation (\ref{eq_prop_result_Noether_thm_vardom_14}) as a unique $\frac{1}{4}$-divergence-free quantity. Indeed we observe that for any integrable function $g$ defined on $\mathbb{S}^1$ there holds, in the sense of distributions,
\begin{equation}
\frac{d}{dx}g=(-\Delta)^\frac{1}{2} Hg=\text{div}_\frac{1}{4}\left[Hg(x)\frac{K^\frac{1}{2}(x-y)}{K^\frac{1}{4}(x-y)}\right],
\end{equation}
where $\text{div}_\frac{1}{4}$ has to be understood as a principal value.\newline
Therefore Equation (\ref{eq_prop_result_Noether_thm_vardom_14}) can be rewritten as
\begin{equation}
\begin{split}
\text{div}_\frac{1}{4}&\left[H\left(L\left(u, (-\Delta)^\frac{1}{4}u\right)X\right)(x)\frac{K^\frac{1}{2}(x-y)}{K^\frac{1}{4}(x-y)}\right.\\
&\phantom{[}\left.+D_pL\left(u(x), (-\Delta)^\frac{1}{4}u(x)\right)u'(y)X(y) \right]=0,
\end{split}
\end{equation}
where $\text{div}_\frac{1}{4}$ acts on the first summand as a principal value.
\end{rem}
By the following result, the technical assumption (\ref{eq_prop_assumption_convergence_weak}) is satisfied in most cases of interest.

\begin{lem}\label{lem_condition_for_prop_def_Noeth_vardom}
Let $L$, $u$ be as in Theorem \ref{prop_final_version_Noether_thm_vardom}.
Assume that there exists a constant $C$ such that for any $x, p\in\mathbb{R}^m$
\begin{equation}\label{eq_lemma_bound_for_f_most_cases_condition}
\lvert D_pL(x,p)\rvert\leq C(1+\lvert p\rvert).
\end{equation}
Then there exists a sequence $(u_n)_n$ in $C^\infty(\mathbb{S}^1)$ such that $u_n\to u$ in $H^\frac{1}{2}(\mathbb{S}^1)$ and such that
\begin{equation}
D_pL\left(u_n, (-\Delta)^\frac{1}{4}u_n\right)\rightharpoondown D_pL\left(u, (-\Delta)^\frac{1}{4}u\right)\text{ weakly in }L^2(\mathbb{S}^1).
\end{equation}
\end{lem}

\begin{proof}
For any $n\in \mathbb{N}$ let 
\begin{equation}
u_n:=D_n\ast u,
\end{equation}
where $D_n$ denotes the Dirichlet kernel introduced in (\ref{eq_def_Dirichlet_kernel_Fourier}). Then $u_n\in C^\infty(\mathbb{S}^1)$ for any $n\in\mathbb{N}$, $u_n\to u$ in $H^\frac{1}{2}(\mathbb{S}^1)$ and in particular $(-\Delta)^\frac{1}{4}u_n\to (-\Delta)^\frac{1}{4}u$ in $L^2(\mathbb{S}^1)$. Thus, by condition (\ref{eq_lemma_bound_for_f_most_cases_condition}), the sequence $\left( D_pL\left(u, (-\Delta)^\frac{1}{4}u\right) \right)_n$ is bounded in $L^2(\mathbb{S}^1)$. Then, by Banach-Alaoglu Theorem, there exists $G\in L^2(\mathbb{S}^1)$ such that
\begin{equation}\label{eq_proof_lemma_weak_L2_convergence_subsequence}
D_pL\left(u_n, (-\Delta)^\frac{1}{4}u_n\right)\rightharpoondown G\text{ weakly in $L^2(\mathbb{S}^1)$ along a subsequence}.
\end{equation}
To complete the proof, it is enough to show that $G=D_pL\left(u, (-\Delta)^\frac{1}{4}u\right)$ a.e..
To this end, we observe that since $u_n\to u$ in $H^\frac{1}{2}(\mathbb{S}^1)$
\begin{equation}
D_pL\left(u_n, (-\Delta)^\frac{1}{4}u_n\right)\rightharpoondown G\text{ a.e. along a subsequence},
\end{equation}
where the subsequence can be chosen to be a subsequence of the subsequence in (\ref{eq_proof_lemma_weak_L2_convergence_subsequence}).
By contradiction, let's assume that $D_pL\left(u, (-\Delta)^\frac{1}{4}u\right)$ differs from $G$ on a set $B\subset \mathbb{S}^1$ of positive Lebesgue measure $b:=\mathscr{L}(B)>0$. Then, by Egorov's Theorem, there exists a set $A\subset \mathbb{S}^1$ such that $\mathscr{L}(\mathbb{S}^1\smallsetminus A)\leq\frac{b}{2}$ such that $D_pL\left(u_n, (-\Delta)^\frac{1}{4}u_n\right)\to D_pL\left(u, (-\Delta)^\frac{1}{4}u\right)$ uniformly in $A$, and therefore also weakly in $L^2(A)$.
By the uniqueness of the limit for weak convergence in $L^2(A)$, $D_pL\left(u, (-\Delta)^\frac{1}{4}u\right)$ and $G$ coincide a.e. in $A$. This contradicts the assumption.
\end{proof}
We observe that Theorem \ref{prop_final_version_Noether_thm_vardom} holds in particular for the half Dirichlet energy (\ref{eq_energy_fraclap_1_4_square}).
Finally, we show that equation (\ref{eq_prop_result_Noether_thm_vardom_14}) can be represented in a different way. To this end we will use Lemma \ref{lem_explicit_form_fracdiv}, which we will proof for general $s\in (0,\frac{1}{2}]$.\\
For this, we first observe that for any $s\in (0,1)$ we the definition of $\text{div}_s(a(x)b(y))$ given in (\ref{eq_definition_frac_div_14}) can be extended to any $a,b\in H^{s}(\mathbb{S}^1,\mathbb{R}^m)$, analogously to (\ref{eq_discussion_extension_definition_fracdiv_14}). For the definition of $s$-fractional divergence when $s=\frac{1}{2}$ we refer to Equations (\ref{eq_def_fractional_divergence_12_circle}) and (\ref{eq_discussion_extension_definition_fracdiv_12}).
\begin{lem}\label{lem_explicit_form_fracdiv}
Let $s\in\left(0,\frac{1}{2}\right]$, $a, b\in H^{s}(\mathbb{S}^1, \mathbb{R}^m)$. Then
\begin{equation}\label{eq_lem_explicit_form_fracdiv}
\text{div}_s(a(x)\cdot b(y))=b\cdot(-\Delta)^s a-a\cdot (-\Delta)^sb
\end{equation}
in the sense of distributions.
\end{lem}
\begin{proof}
Assume first that $a, b\in C^\infty(\mathbb{S}^1, \mathbb{R}^m)$.
Then for any $\phi\in C^\infty(\mathbb{S}^1,\mathbb{R})$
\begin{equation}\label{eq_proof_lemma_explicit_form_divergence_first_step_def}
\text{div}_s(a(x)b(y))[\phi]=\int_{\mathbb{S}^1}\int_{\mathbb{S}^1}a(x)\cdot b(y)(\phi(x)-\phi(y))K^s(x-y)dxdy.
\end{equation}
We remark that since $a,b$ are smooth and $K^s$ obeys estimate (\ref{eq_estimates_kernel_fraclap_s}), the integral (\ref{eq_proof_lemma_explicit_form_divergence_first_step_def}) converges absolutely. By interchanging the variables $x$ and $y$ we obtain
\begin{equation}\label{eq_proof_lemma_explicit_form_divergence_second_step_changevar}
\begin{split}
&\text{div}_s(a(x)b(y))[\phi]=\int_{\mathbb{S}^1}\int_{\mathbb{S}^1}(a(x)b(y)-a(y)b(x))\phi(x)K^s(x-y)dxdy
\\=&\int_{\mathbb{S}^1}\int_{\mathbb{S}^1}(a(x)\cdot(b(y)-b(x))+(a(x)-a(y))\cdot b(x))\phi(x)K^s(x-y)dxdy.
\end{split}
\end{equation}
From the last expression we deduce that the integrals in (\ref{eq_proof_lemma_explicit_form_divergence_second_step_changevar}) converge absolutely. Therefore we compute
\begin{equation}
\begin{split}
\text{div}_s(a(x)b(y))[\phi]=&\int_{\mathbb{S}^1}\phi(x)\left( -a(x) \cdot PV\int_{\mathbb{S}^1}(b(x)-b(y))K^s(x-y)dy\right.\\&\left.\phantom{\int_{\mathbb{S}^1}\phi(x)\bigg(}+b(x)\cdot PV\int_{\mathbb{S}^1}(a(x)-a(y))K^s(x-y)dy\right)dx\\
&\int_{\mathbb{S}^1}\phi(x)\left(-a(x)\cdot (-\Delta)^s b(x)+b(x)\cdot (-\Delta)^sa(x) \right)dx.
\end{split}
\end{equation}
Since this is true for any $\phi\in C^\infty(\mathbb{S}^1,\mathbb{R})$, we conclude that whenever $a,b\in C^\infty(\mathbb{S}^1,\mathbb{R}^m)$
\begin{equation}
\text{div}_s(a(x)b(y))=-a\cdot (-\Delta)^s b+b\cdot (-\Delta)^sa
\end{equation}
in the sense of distributions.\\
We observe that identity (\ref{eq_lem_explicit_form_fracdiv}) can be proved also with the help of Fourier series, if we still assume that $a,b\in C^\infty(\mathbb{S}^1,\mathbb{R}^m)$. In order to simplify the notation, let's denote
\begin{equation}
v(x):=\begin{pmatrix}
    a(x)\\
    b(x)
    \end{pmatrix} 
\end{equation}
for any $x\in \mathbb{S}^1$ and
\begin{equation}
R:=\begin{bmatrix}
    0     & I_m  \\
    -I_m      & 0 
\end{bmatrix}.
\end{equation}
If we rewrite the vector valued function $V:=v(\cdot-z)+v(\cdot+z)-2v(\cdot)$ as a Fourier series, we obtain
\begin{equation}
\begin{split}
&\int_{\mathbb{S}^1} (a(x)b(y)-a(y)b(x))K^s(x-y)dy\\
=&\frac{1}{2}\int_{\mathbb{S}^1}v(x)\cdot R\sum_{n\in\mathbb{Z}}(e^{izn}+e^{-izn}-2)\widehat{v}e^{inx}K^s(z)dz\\
=&v(x)\cdot R\int_{\mathbb{S}^1}\sum_{n\in\mathbb{Z}}(\cos(zn)-1)K^s(z)\widehat{v}e^{inx}dz
\end{split}
\end{equation}
Since $V$ is smooth, the decay of its Fourier coefficients allows to invert summation and integration:
\begin{equation}
\begin{split}
&\int_{\mathbb{S}^1} (a(x)b(y)-a(y)b(x))K^s(x-y)dy\\
=&v(x)\cdot R \sum_{n\in\mathbb{Z}}\int_{\mathbb{S}^1}\left(\cos(2\pi zn)-1\right)K^s(z)dz \widehat{v}(n)e^{inx}\\
=&v(x)\cdot R\sum_{n\in\mathbb{Z}}-(\Delta)^s(\cos(zn)-1)(0)\widehat{v}(n)e^{inx}\\
=&-v(x)\cdot R \sum_{n\in\mathbb{Z}}\lvert n\rvert^{2s} \widehat{v}(n)e^{inx}\\
=&-v(x)\cdot R (-\Delta)^sv(x).
\end{split}
\end{equation}
By the first equality in (\ref{eq_proof_lemma_explicit_form_divergence_second_step_changevar}) we deduce Equation (\ref{eq_lem_explicit_form_fracdiv}) whenever $a,b\in C^\infty(\mathbb{S}^1,\mathbb{R}^m)$.\\
Now if $a,b\in H^{s}(\mathbb{S}^1)$ let 
$(a_n)_n$, $(b_n)_n$ be sequences in $C^\infty(\mathbb{S}^1, \mathbb{R}^m)$ converging respectively to $a$ and $b$ in $H^s(\mathbb{S}^1)$. By the previous computation, for any $n\in \mathbb{N}$
\begin{equation}\label{eq_proof_lemma_expform_fracdiv_for_approx}
\text{div}_s(a_n(z)b_n(y))(x)=b_n(x)\cdot(-\Delta)^s a_n(x)-a_n(x)\cdot (-\Delta)^sb_n(x)
\end{equation}
Now for any $\phi\in C^\infty(\mathbb{S}^1,\mathbb{R}^m)$
\begin{equation}
\begin{split}
&\int_{\mathbb{S}^1} \left[b_n(x)\cdot(-\Delta)^s a_n(x)-a_n(x)\cdot (-\Delta)^sb_n(x)\right]\phi(x)dx\\
=&\int_{\mathbb{S}^1} (-\Delta)^\frac{s}{2}(b_n\phi)(x)\cdot(-\Delta)^\frac{s}{2}a_n(x)-(-\Delta)^\frac{s}{2}(a_n\phi)(x)\cdot (-\Delta)^\frac{s}{2}b_n(x)dx.
\end{split}
\end{equation}
By Lemma \ref{lem_multiplication_with_test_function_continuous}, $b_n\phi\to b\phi$ and $a_n\phi\to a\phi$ in $H^s(\mathbb{S}^1)$. Thus
\begin{equation}\label{eq_proof_lemma_expform_fracdiv_for_approx2}
\begin{split}
&\int_{\mathbb{S}^1} [b_n(x)\cdot(-\Delta)^s a_n(x)-a_n(x)\cdot (-\Delta)^sb_n(x)]\phi(x)dx\\
\to &\left\langle b\cdot(-\Delta)^s a-a\cdot (-\Delta)^sb,\phi\right\rangle
\end{split}
\end{equation}
as $n\to\infty$.
On the other hand, by estimate (\ref{eq_rem_estimate_divergence_for_products})\footnote{Estimate (\ref{eq_rem_estimate_divergence_for_products}) was proved for $s=\frac{1}{4}$, but can be proved with similar arguments when $s\in \left(0, \frac{1}{2}\right)$.} (or, if $s=\frac{1}{2}$, by estimate (\ref{eq_lem_estimate_Gphi_12})),
\begin{equation}\label{eq_proof_lemma_expform_fracdiv_for_approx3}
\text{div}_s(a_n(z)b_n(y))[\phi]\to \text{div}_s(a(z)b(y))[\phi]
\end{equation}
as $n\to\infty$.
Combining (\ref{eq_proof_lemma_expform_fracdiv_for_approx}), (\ref{eq_proof_lemma_expform_fracdiv_for_approx2}) and (\ref{eq_proof_lemma_expform_fracdiv_for_approx3}) we obtain the desired result.
\end{proof}
As an application of the previous results, we consider again the energy
\begin{equation}\label{eq_fractional_Dirichlet_energy_for_example_14}
E(u)=\int_{\mathbb{S}^1}\lvert(-\Delta)^\frac{1}{4}u(x)\rvert^2dx
\end{equation}
for $u\in H^\frac{1}{2}(\mathbb{S}^1)$, corresponding to the Lagrangian $f(p)=\lvert p\rvert^2$ for any $p\in\mathbb{R}^m$.
We claim that in this framework, the conditions of Theorem \ref{prop_final_version_Noether_thm_vardom} are satisfied if we choose $X\equiv 1$.
Geometrically, if we identify $\mathbb{S}^1$ with $\partial D^2$ through the map $i$ introduced in (\ref{eq_definition_identification_map_i}), we observe that $X$ corresponds to the generator of rotations.
Its flow (on $\mathbb{S}^1$, here again thought of as a quotient space of $\mathbb{R}$) is given by $\phi_t(x)=x+t$ for any $x\in \mathbb{S}^1$, $t\in\mathbb{R}$. Then for any $u\in H^\frac{1}{2}(\mathbb{S}^1)$, for a.e. $x\in \mathbb{S}^1$,
\begin{equation}
\begin{split}
f(u\circ\phi_t(x),(-\Delta)^\frac{1}{4}(u\circ\phi_t)(x))=&\lvert(-\Delta)^\frac{1}{4}(u+t)(x)\rvert^2=\lvert(-\Delta)^\frac{1}{4}u(x)\rvert^2\\
=&f(u(\phi_t(x)),(-\Delta)^\frac{1}{4}u(\phi_t(x)))\phi_t'(x),
\end{split}
\end{equation}
therefore (\ref{eq_diffeo_is_conformal14_2}) is satisfied. Moreover, $f$ satisfies condition (\ref{eq_lemma_bound_for_f_most_cases_condition}) in Lemma \ref{lem_condition_for_prop_def_Noeth_vardom}, therefore condition (\ref{eq_prop_assumption_convergence_weak}) holds.
Thus Theorem (\ref{prop_final_version_Noether_thm_vardom}) yields that any stationary point $u\in H^\frac{1}{2}(\mathbb{S}^1, \mathbb{R}^m)$ of $E$ satisfies
\begin{equation}\label{eq_prop_result_Noether_thm_vardom_14_3}
\text{div}_\frac{1}{4}\left[2(-\Delta)^\frac{1}{4}u(x)u'(y)\right]+\frac{d}{dx}\left[f\left(u, (-\Delta)^\frac{1}{4}u\right)\right](x)=0
\end{equation}
as distribution. We observe that (\ref{eq_prop_result_Noether_thm_vardom_14_3}) corresponds to the stationary equation of $E$, so that in this case, Proposition \ref{prop_final_version_Noether_thm_vardom} yields no new information. By Lemma \ref{lem_explicit_form_fracdiv}, if $u$ is sufficiently regular, Equation (\ref{eq_prop_result_Noether_thm_vardom_14_3}) can be rewritten as
\begin{equation}\label{eq_discussion_Noether_formula_for_rotations_from_14}
2u'(x)\cdot (-\Delta)^\frac{1}{2}u(x)=0
\end{equation}
for a.e. $x\in\mathbb{S}^1$.\\

In order to find more application for Noether's theorem, we observe that the energy can be rewritten in different ways, in particular, if $u$ is regular enough,
\begin{equation}\label{eq_discussion_half_Dirichlet_energy_integrand_conformally_invariant}
E(u)=\int_{\mathbb{S}^1}(-\Delta)^\frac{1}{2}u(x)u(x)dx.
\end{equation}
In section \ref{ssec: Lagrangians involving 12Laplacians} we will derive an analogous result to the Noether's Theorem derived in this section for some Lagrangian involving a $\frac{1}{2}$-fractional Laplacian, in section \ref{ssec: Traces of conformal maps and half Dirichlet energy} we will show that the integrand of the energy (\ref{eq_discussion_half_Dirichlet_energy_integrand_conformally_invariant}) is conformally invariant and give some applications of the analogous result.

\subsection{Lagrangians involving $\frac{1}{2}$-Laplacians}\label{ssec: Lagrangians involving 12Laplacians}
Let $F$, $B$ be two Lipschitz function from $\mathbb{R}^m$ to $\mathbb{R}^m$ and $\mathbb{R}$ respectively. Let
\begin{equation}\label{eq_discussion_definition_fct_f_in_terms_of_F_G}
L: \mathbb{R}^m\times\mathbb{R}^m\to \mathbb{R}, \quad (x,p)\mapsto p\cdot F(x)+B(x).
\end{equation}
For any $u\in C^\infty(\mathbb{S}^1,\mathbb{R}^m)$, let
\begin{equation}\label{eq_def_12_Dirichlet_energy_conformal}
E(u):=\int_{\mathbb{S}^1}L\left(u(x), (-\Delta)^\frac{1}{2}u(x)\right)dx.
\end{equation}
If $u\in C^\infty(\mathbb{S}^1,\mathbb{R}^m)$, the integrand in (\ref{eq_def_12_Dirichlet_energy_conformal}) is well defined for any $x\in \mathbb{S}^1$ and the integral converges absolutely. For a general $u\in H^\frac{1}{2}(\mathbb{S}^1, \mathbb{R}^m)$, one can generalize the definition of $E$ interpreting the product $(-\Delta)^\frac{1}{2}u\cdot ( F\circ u)$ as the action of a distribution in $H^{-\frac{1}{2}}(\mathbb{S}^1)$ on a function in $H^\frac{1}{2}(\mathbb{S}^1)$, i.e.
\begin{equation}\label{eq_discussion_definition_E(u)_in_distributional_sense}
E(u)=\int_{\mathbb{S}^1} (-\Delta)^\frac{1}{4}(F\circ u)(x)\cdot(-\Delta)^\frac{1}{4}u(x)+B(x)dx.
\end{equation}
For a smooth function $u$, one can shows, similarly as in the previous case, that $u$ is a stationary point of the energy (\ref{eq_def_12_Dirichlet_energy_conformal})\footnote{Stationary points of (\ref{eq_def_12_Dirichlet_energy_conformal}) are defined analogously to the ones of (\ref{eq_def_1_4_Dirichlet_energy}).} if and only if $u$ satisfies the stationarity equation
\begin{equation}\label{eq_stationary_equation_vardom}
\begin{split}
&\frac{d}{dx} L\left(u, (-\Delta)^\frac{1}{2}u\right)+\text{div}_\frac{1}{2}\left[u'(y)F(u(x))\right]=0.
\end{split}
\end{equation}
Here we made use of the $\frac{1}{2}$-fractional divergence, defined as follows: for any functions $a$, $b\in C^1(\mathbb{R}^m, \mathbb{R})$, the \textbf{$\frac{1}{2}$-fractional divergence} of $F$ is the distribution given by
\begin{equation}\label{eq_def_fractional_divergence_12_circle}
\text{div}_\frac{1}{2}F[\phi]:=PV\int_{\mathbb{S}^1}\int_{\mathbb{S}^1}a(x)b(y)\cdot(\phi(x)-\phi(y))K^\frac{1}{2}(x-y)dxdy
\end{equation}
for any $\phi\in C^\infty(\mathbb{S}^1)$.
Now let $X$ be a smooth vector field on $\mathbb{S}^1$ and let $\phi_t$ denote its flow. Assume that $u\in C^\infty(\mathbb{S}^1, \mathbb{R}^m)$ is a critical point of $E$, and that $u$ satisfies
\begin{equation}\label{eq_diffeo_is_conformal14_2_discussion}
L\left(u(\phi_t(x)),(-\Delta)^\frac{1}{2}(u\circ\phi_t)(x)\right)=L\left(u(\phi_t(x)),(-\Delta)^\frac{1}{2}u(\phi_t(x))\right)\phi_t'(x)
\end{equation}
for $x\in \mathbb{S}^1$, for $t$ in a neighbourhood of the origin.
Then, deriving (\ref{eq_diffeo_is_conformal14_2_discussion}) in $t$ and comparing the resulting expression with (\ref{eq_stationary_equation_vardom}) in a similar way as in the previous section, one obtains
\begin{equation}\label{eq_discussion_result_Noether_vardom_smooth_12}
\frac{d}{dx}\left[L\left(u,(-\Delta)^\frac{1}{2}u\right)X\right](x)+\text{div}_\frac{1}{2}\left[F(u(x))u'(y)X(y)\right]=0
\end{equation}
for any $x\in\mathbb{S}^1$.\\
Next we try to show that (\ref{eq_discussion_result_Noether_vardom_smooth_12}) remains true (in the sense of distributions) for any critical points of $E$ satisfying (\ref{eq_diffeo_is_conformal14_2_discussion}) a.e. in $\mathbb{S}^1$. To this end, we first need a result analogous to Lemma \ref{lemma_approximation_identity_14}.\newline
First, in analogy to the definition given in (\ref{eq_discussion_norm_A'}), we introduce the normed space $\mathbb{A}^{\frac{3}{2}}(\mathbb{S}^1)$, consisting of all elements $u\in \mathscr{D}'(\mathbb{S}^1)$ such that
\begin{equation}
\lVert u\rVert_{\mathbb{A}^{\frac{3}{2}}}:=\lvert \widehat{u}(0)\rvert+\sum_{\substack{k\in \mathbb{Z},\\ k\neq 0}}\lvert k\rvert^\frac{3}{2}\lvert \widehat{u}(k)\rvert<\infty.
\end{equation}
By the argument used in (\ref{eq_rem_embedding_Sobolev_in_Wiener}), we see that for any $s>\frac{1}{2}$ 
\begin{equation}\label{eq_embedding_Sobolev_in_Wiener_32}
H^{\frac{3}{2}+s}(\mathbb{S}^1)\hookrightarrow \mathbb{A}^\frac{3}{2}(\mathbb{S}^1).
\end{equation}
For any element $u$ of $\mathbb{A}^\frac{3}{2}(\mathbb{S}^1)$ we define the seminorm
\begin{equation}
[u]_{\mathbb{A}^\frac{3}{2}}:=\sum_{k\in\mathbb{Z},\\ k\neq 0}\lvert k\rvert^\frac{3}{2}\lvert \widehat{u}(k)\rvert.
\end{equation}
We can now proceed to state the result analogous to Lemma \ref{lemma_approximation_identity_14}.
\begin{lem}\label{lemma_approximation_identity_12}
Let $w\in H^\frac{1}{2}(\mathbb{S}^1)$, $\phi\in \mathbb{A}^\frac{3}{2}(\mathbb{S}^1)$ and let $K^\frac{1}{2}$ be defined as in (\ref{eq_explit_formula_Kernel_12}).
Let
\begin{equation}\label{eq_definition_Gwphi_12}
G_{w,\phi}(x):=PV\int_{\mathbb{S}^1}w(y)(\phi(y)-\phi(x))K^\frac{1}{2}(y-x)dy
\end{equation}
for $x\in \mathbb{S}^1$. Then $G_{w,\phi}(x)$ is well defined for a.e. $x\in \mathbb{S}^1$, $G_{w,\phi}$ belongs to $H^{-\frac{1}{2}}(\mathbb{S}^1)$, and there is a constant $C$ depending only on $X$, such that
\begin{equation}\label{eq_lem_estimate_Gphi_12}
\lVert G_{w,\phi}\rVert_{H^{-\frac{1}{2}}}\leq C[ \phi]_{\mathbb{A}^\frac{3}{2}}\lVert w\rVert_{H^{-\frac{1}{2}}}.
\end{equation}
\end{lem}
\begin{proof}
In order to show that (\ref{eq_definition_Gwphi_12}) is well defined, we observe that for a.e. $x\in \mathbb{S}^1$
\begin{equation}\label{eq_proof_lemma_Gwphi_well_def_12}
\begin{split}
G_{w,\phi}(x)=&\int_{\mathbb{S}^1}(w(y)-w(x))(\phi(y)-\phi(x))K^\frac{1}{2}(x-y)dy\\
&+w(x)PV\int_{\mathbb{S}^1}(\phi(x)-\phi(y))K^\frac{1}{2}(x-y)dy\\
=&(-\Delta)^\frac{1}{2}(w\phi)(x)-w(x)(-\Delta)^\frac{1}{2}\phi(x)-(\Delta)^\frac{1}{2}w(x)\phi(x)+w(x)(-\Delta)^\frac{1}{2}\phi(x) \\
=&(-\Delta)^\frac{1}{2}(w\phi)(x) -(\Delta)^\frac{1}{2}w(x)\phi(x).
\end{split}
\end{equation}
From expression (\ref{eq_proof_lemma_Gwphi_well_def_12}) we see that $G_{w,\phi}$ is a well defined distribution, and that for any $n\in \mathbb{Z}$
\begin{equation}
\begin{split}
\mathscr{F}[G_{w,\phi}](n)=&\lvert n\rvert \widehat{w\phi}(n)-\sum_{k\in\mathbb{Z}}\lvert n-k\rvert \widehat{w}(n-k)\widehat{\phi}(k)\\
=&\lvert n \rvert\sum_{k\in\mathbb{Z}}\widehat{w}(n-k)\widehat{\phi}(k)-\sum_{k\in\mathbb{Z}}\lvert n-k\rvert \widehat{w}(n-k)\widehat{\phi}(k)\\
=&\sum_{k\in\mathbb{Z}}\left[\lvert n\rvert-\lvert n-k\rvert\right]\widehat{\phi}(k)\widehat{w}(n-k)
\end{split}
\end{equation}
Now we observe that if $\lvert n\rvert\geq 2\lvert k\rvert$, $\lvert n-k\rvert\geq \frac{1}{2}\lvert n\rvert$. Therefore
\begin{equation}
\frac{\lvert n\rvert-\lvert n-k\rvert}{(1+\lvert n\rvert^2)^\frac{1}{4}}\leq \frac{2^\frac{1}{2}\lvert k\rvert}{(1+\lvert n-k\rvert^2)^\frac{1}{4}}.
\end{equation}
Thus, by Young's inequality
\begin{equation}\label{eq_proof_lemma_approximation_12_step_1}
\begin{split}
&\sum_{n\in\mathbb{Z}}\left(\sum_{\lvert k\rvert\geq 2\lvert n\rvert}\widehat{\phi}(n)\widehat{w}(k-n)[\lvert k\rvert-\lvert n-k\rvert]\right)^2(1+\lvert k\rvert^2)^{-\frac{1}{2}}\\
\leq& 2\sum_{n\in \mathbb{Z}}\left(\sum_{k\in\mathbb{Z}}\lvert\widehat{\phi}(k)\rvert\lvert k\rvert\lvert\widehat{w}(n-k)\rvert\lvert(1+\lvert n-k\rvert^2)^{-\frac{1}{4}}\rvert\right)^2\leq 2[\phi]_{\mathbb{A}^1}^2\lVert w\rVert_{H^{-\frac{1}{2}}}^2.
\end{split}
\end{equation}
On the other hand, if $\lvert n\rvert<2 \lvert k\rvert$, $\lvert n-k\rvert\leq 3\lvert k\rvert$. Therefore
\begin{equation}
\frac{\lvert n\rvert -\lvert n-k\rvert}{\left(1+\lvert n\rvert^2\right)^\frac{1}{4}}\leq \frac{\lvert k\rvert}{(1+\lvert n-k\rvert^2)^\frac{1}{4}}\left(1+\lvert n-k\rvert^2\right)^\frac{1}{4}\leq \frac{2^\frac{1}{4}3^\frac{1}{2}\lvert k\rvert^\frac{3}{2}}{\left(1+\lvert n-k\rvert^2\right)^\frac{1}{4}}.
\end{equation}
Thus, by Young's inequality
\begin{equation}\label{eq_proof_lemma_approximation_12_step_2}
\begin{split}
&\sum_{n\in\mathbb{Z}}\left(\sum_{\lvert n\rvert<2\lvert k\rvert}\widehat{\phi}(k)\widehat{w}(n-k)[\lvert n\rvert-\lvert n-k\rvert]\right)^2(1+\lvert n\rvert^2)^{-\frac{1}{4}}\\
\leq & 2^\frac{1}{2}3 \sum_{n\in\mathbb{Z}}\left(\sum_{\lvert n\rvert<2\lvert k\rvert}\lvert\widehat{\phi}(k)\rvert\lvert k\rvert^\frac{3}{2}\lvert\widehat{w}(n-k)\rvert\lvert (1+\lvert n-k\rvert^2)^\frac{1}{4}\right)^2\leq 2^\frac{1}{2}3 [\phi]_{\mathbb{A}^\frac{3}{2}}^2\lVert w\rVert_{H^{-\frac{1}{2}}}.
\end{split}
\end{equation}
Then, combining (\ref{eq_proof_lemma_approximation_12_step_1}) and (\ref{eq_proof_lemma_approximation_12_step_2}) we obtain
\begin{equation}
\lVert G_{w,\phi}\rVert_{H^{-\frac{1}{2}}}\leq C[ \phi]_{\mathbb{A}^\frac{3}{2}}\lVert w\rVert_{H^{-\frac{1}{2}}}
\end{equation}
for some independent constant $C$.
\end{proof}

\begin{rem}\label{rem_to_lemma_approximation_identity_12}\noindent
\begin{enumerate}
\item Lemma \ref{lemma_approximation_identity_14} shows that for any $\phi\in C^\infty(\mathbb{S}^1)$, we have a continuous linear operator
\begin{equation}
F:H^\frac{1}{2}(\mathbb{S}^1)\to H^{-\frac{1}{2}}(\mathbb{S}^1),\quad w\mapsto G_{w, \phi}.
\end{equation}
for any $p\in (2, \infty]$. By (\ref{eq_lem_estimate_Gphi}), $F$ can be extended to a continuous operator on $H^{-\frac{1}{2}}(\mathbb{S}^1)$ (still denoted by $F$), for which (\ref{eq_lem_estimate_Gphi_12}) remains valid.
\item We also observe that if $a$, $b\in C^1(\mathbb{S}^1)$, for any $\phi\in C^\infty(\mathbb{S}^1)$ we have
\begin{equation}
\text{div}_\frac{1}{2}\left( a(x)b(y) \right)[\phi]=\langle G_{a, \phi},b\rangle,
\end{equation}
therefore there holds
\begin{equation}\label{eq_rem_estimate_divergence_for_products_12}
\begin{split}
\left\lvert {div}_\frac{1}{2}\left( a(x)b(y)\right)[\phi]\right\rvert\leq &\lVert G_{a,\phi}\rVert_{H^{-\frac{1}{2}}}\lVert b\rVert_{H^\frac{1}{2}}\\
\leq & C[\phi]_{\mathbb{A}^\frac{3}{2}}\lVert a\rVert_{H^{-\frac{1}{2}}}\lVert  b\rVert_{H^\frac{1}{2}},
\end{split}
\end{equation}
where $C$ is an independent constant.
\end{enumerate}
\end{rem}
By the previous Remark, we may extend the definition of
\begin{equation}
\text{div}_\frac{1}{2}\left(a(x)b(y)\right)
\end{equation}
to functions $a\in H^{-\frac{1}{2}}(\mathbb{S}^1)$, $b\in H^\frac{1}{2}(\mathbb{S}^1)$ as follows: for $\phi\in C^\infty(\mathbb{S}^1)$ let
\begin{equation}\label{eq_discussion_extension_definition_fracdiv_12}
\text{div}_\frac{1}{2}\left[a(x)b(y)\right][\phi]:=\langle G_{a,\phi}, b\rangle,
\end{equation}
where $G_{a, \phi}:=F(a)$ was defined in the previous Remark.
Therefore the following estimate still holds:
\begin{equation}
\begin{split}
\left\lvert \text{div}_\frac{1}{2}\left( a(x)b(y)\right)[\phi]\right\rvert\leq & \lVert G_{a,\phi}\rVert_{H^{-\frac{1}{2}}}\lVert b\rVert_{H^\frac{1}{2}}\\
\leq & C[\phi]_{\mathbb{A}^\frac{3}{2}}\lVert a\rVert_{H^{-\frac{1}{2}}} \lVert b\rVert_{H^\frac{1}{2}}
\end{split}
\end{equation}
where $C$ is an independent constant.\\
By Lemma \ref{lem_verify_condition_approxlem}, we may apply Lemma \ref{lemma_approximation_identity_12} to $a=u'X$ and $b=F\circ u$, where $F$ is a Lipschitz function from $\mathbb{R}^m$ to $\mathbb{R}^m$ as in (\ref{eq_discussion_definition_fct_f_in_terms_of_F_G}), $X$ is any smooth vector field on $\mathbb{S}^1$ and $u\in H^\frac{1}{2}(\mathbb{S}^1, \mathbb{R}^m)$.\\
We are now able to formulate the analogous of Proposition \ref{prop_final_version_Noether_thm_vardom} for Lagrangians involving $\frac{1}{2}$-fractional Laplacians.

\begin{thm}\label{prop_final_version_Noether_thm_vardom_12}
Let $L$ and $E$ be defined as in (\ref{eq_discussion_definition_fct_f_in_terms_of_F_G}) and (\ref{eq_discussion_definition_E(u)_in_distributional_sense}).
Let $X$ be a smooth vector field on $\mathbb{S}^1$ and assume that for any $v\in C^\infty(\mathbb{S}^1,\mathbb{R}^m)$ its flow $\phi_t$ satisfies
\begin{equation}\label{eq_diffeo_is_conformal12_2}
L\left(v(\phi_t(x)),(-\Delta)^\frac{1}{2}(v\circ\phi_t)(x)\right)=L\left(v(\phi_t(x)),(-\Delta)^\frac{1}{2}v(\phi_t(x))\right)\phi_t'(x)
\end{equation}
for a.e. $x\in \mathbb{S}^1$, for $t$ in a neighbourhood of the origin.
Assume that $u\in H^\frac{1}{2}(\mathbb{S}^1, \mathbb{R}^m)$ satisfy the following stationarity condition:
\begin{equation}
\begin{split}
&\text{div}_\frac{1}{2}\left[F(u(x))u'(y)\right]+\frac{d}{dx} L\left(u, (-\Delta)^\frac{1}{2}u\right)=0
\end{split}
\end{equation}
as distributions.
Then, in the sense of distributions
\begin{equation}\label{eq_prop_result_Noether_thm_vardom_12}
\text{div}_\frac{1}{2}\left[F(u(x))u'(y)X(y)\right]+\frac{d}{dx}\left[L\left(u, (-\Delta)^\frac{1}{2}u\right)X\right](x)=0.
\end{equation}
\end{thm}

\begin{rem}
If $u$ is regular enough, we can rewrite the quantity on the left hand side of  (\ref{eq_prop_result_Noether_thm_vardom_12}) as follows:

\begin{equation}\label{eq_prop_result_Noether_thm_vardom_12_rewritten}
\begin{split}
&\text{div}_\frac{1}{2}\left[F(u(x))u'(y)X(y)\right]+\frac{d}{dx}\left[L\left(u, (-\Delta)^\frac{1}{2}u\right)X\right](x)\\
=&\text{div}_\frac{1}{2}\left[F(u(x))u'(y)X(y)+H\left[L\left(u, (-\Delta)^\frac{1}{2}u\right)X\right]\bigg\vert^{x}_{y}\right].
\end{split}
\end{equation}
in the sense of distributions.
\end{rem}

\begin{proof}
We follow the idea of the proof of Proposition \ref{prop_final_version_Noether_thm_vardom}.
Let $(u_n)_n$ be a sequence in $C^\infty(\mathbb{S}^1,\mathbb{R}^m)$ such that $u_n\to u$ in $H^\frac{1}{2}(\mathbb{S}^1)$ as $n\to\infty$. Then by Lemma \ref{lem_condition_for_prop_def_Noeth_vardom_2} $F(u_n)\rightharpoondown F(u)$ weakly in $H^\frac{1}{2}(\mathbb{S}^1)$ as $n\to \infty$. By Lemma \ref{lemma_approximation_identity_12} and Remark \ref{rem_to_lemma_approximation_identity_12}, for any $\phi\in C^\infty(\mathbb{S}^1)$
\begin{equation}
G_{u_n'X,\phi}\to G_{u'X,\phi} \text{ in }H^{-\frac{1}{2}}(\mathbb{S}^1)
\end{equation}
as $n\to \infty$. Therefore
\begin{equation}
\begin{split}
&\text{div}_\frac{1}{2}\left[F(u_n(x))u_n'(y)X(y)\right][\phi]=\left\langle G_{u_n'X,\phi},F\circ u_n\right\rangle\\
\to &\left\langle G_{u'X,\phi}, F\circ u \right\rangle
=\text{div}_\frac{1}{2}\left[F(u(x))u'(y)X(y)\right][\phi]
\end{split}
\end{equation}
as $n\to \infty$.
By the same argument with $Y\equiv 1$  we obtain
\begin{equation}
\begin{split}
&\text{div}_\frac{1}{2}\left[F(u_n(x))u_n'(y)\right][\phi]\to\text{div}_\frac{1}{2}\left[F(u(x))u'(y)\right][\phi]
\end{split}
\end{equation}
as $n\to \infty$. Now we observe that since assumption (\ref{eq_diffeo_is_conformal12_2}) is satisfied, for any $n\in \mathbb{N}$ the computation preceding Proposition \ref{prop_first_version_Noether_vardom_14} implies
\begin{equation}
\begin{split}
&\frac{d}{dx}L\left(u_n,(-\Delta)^\frac{1}{2}u_n\right)X+\text{div}_\frac{1}{2}\left[F(u_n(x))u_n'(y)\right]X\\
=&\text{div}_\frac{1}{2}\left[F(u_n(x))u_n'(y)X(y)\right]+\frac{d}{dx}\left[L\left(u_n,(-\Delta)^\frac{1}{2}u_n\right)X\right]
\end{split}
\end{equation}
as distributions.
Therefore
\begin{equation}\label{eq_proof_prop_computation_to_pass_to_the_limit_Noether_vardom_12}
\begin{split}
0=&-\int_{\mathbb{S}^1}L\left(u(x), (-\Delta)^\frac{1}{2}u(x)\right)\frac{d}{dx}\left(X\phi\right)(x)dx+\text{div}_\frac{1}{2}\left[F(u(x))u'(y)\right][X\phi]\\
=&\lim_{n\to \infty}\int_{\mathbb{S}^1}L\left(u_n(x), (-\Delta)^\frac{1}{2}u_n(x)\right)\frac{d}{dx}\left(X\phi\right)(x)dx\\&+\text{div}_\frac{1}{2}\left[F(u_n(x))u_n'(y)\right][X\phi]\\
=&\lim_{n\to\infty}\text{div}_\frac{1}{2}\left[D_pL\left(u_n(x), (-\Delta)^\frac{1}{2}u_n(x)\right)u_n'(y)X(y)\right]\\&-\int_{\mathbb{S}^1}L\left(u_n(x), (-\Delta)^\frac{1}{2}u_n(x)\right)X(x)\phi'(x)dx\\
=& \text{div}_\frac{1}{2}\left[F(u(x))u'(y)X(y)\right].
\end{split}
\end{equation}
As (\ref{eq_proof_prop_computation_to_pass_to_the_limit_Noether_vardom_12}) holds for any $\phi\in C^\infty(\mathbb{S}^1)$ we conclude that
\begin{equation}
\text{div}_\frac{1}{2}\left[F(u(x))u'(y)X(y)\right]+\frac{d}{dx}\left[L\left(u, (-\Delta)^\frac{1}{4}u\right)X\right](x)=0.
\end{equation}
in the sense of distributions.
\end{proof}

\begin{lem}\label{lem_condition_for_prop_def_Noeth_vardom_2}
Let $F:\mathbb{R}^m\to\mathbb{R}^m$ be a Lipschitz function and let $u\in H^\frac{1}{2}(\mathbb{S}^1, \mathbb{R}^m)$.
For any sequence $(u_n)_n$ in $C^\infty(\mathbb{S}^1,\mathbb{R}^m)$ such that $u_n\to u$ in $H^\frac{1}{2}(\mathbb{S}^1)$ as $n\to \infty$, there holds
\begin{equation}
F\circ u_n \rightharpoondown F\circ u\text{ weakly in }H^\frac{1}{2}(\mathbb{S}^1)\text{ along a subsequence}
\end{equation}
as $n\to \infty$.
\end{lem}

\begin{proof}
Let $(u_n)_n$ be a sequence in $C^\infty(\mathbb{S}^1,\mathbb{R}^m)$ such that $u_n\to u$ in $H^\frac{1}{2}(\mathbb{S}^1)$ as $n\to \infty$. Then, along a subsequence, $u_n\to u$ a.e. as $n\to \infty$. In particular, there exists $x_0\in \mathbb{S}^1$ and $C>0$ such that, for any element of the subsequence above, $\lvert u_n(x_0)\rvert\leq C$. Since $F$ is Lipschitz, for any $u_n$ in the subsequence, for any $x\in \mathbb{S}^1$ there holds
\begin{equation}
\lvert F\circ u_n(x)\rvert\leq [F]_{Lip}\left(\lvert u_n(x)\rvert+C \right)+\lvert F(u_n(x_0))\rvert,
\end{equation}
and therefore
\begin{equation}
\lVert F\circ u_n\lVert_{H^\frac{1}{2}}\leq [F]_{Lip}\left([u_n]_{\dot{H}^\frac{1}{2}}+\lVert u_n\rVert_{L^2} +(2\pi)^\frac{1}{2}C\right)+(2\pi)^\frac{1}{2}\max_{\lvert y\rvert\leq C}\lvert F(y)\rvert.
\end{equation}
In particular, a subsequence of $(F\circ u_n)_n$ is bounded in $H^\frac{1}{2}(\mathbb{S}^1)$. By Banach-Alaoglu Theorem, there exists $G\in H^\frac{1}{2}(\mathbb{S}^1)$ such that $F\circ u_n\rightharpoondown G\text{ weakly in }H^\frac{1}{2}(\mathbb{S}^1)$ as $n\to \infty$ along a further subsequence. Arguing as in the proof of Lemma \ref{lem_condition_for_prop_def_Noeth_vardom}, one can prove that $G=F\circ u$. Therefore $F\circ u_n\rightharpoondown F\circ u$ weakly in $H^\frac{1}{2}(\mathbb{S}^1)$ along a subsequence, as $n\to\infty$.
\end{proof}
In particular, Theorem \ref{prop_final_version_Noether_thm_vardom_12} holds when $E$ is the half Dirichlet energy. In this case, if $u$ is regular enough, Lemma \ref{lem_explicit_form_fracdiv} implies that Equation (\ref{eq_prop_result_Noether_thm_vardom_12}) can be rewritten as follows:
\begin{equation}\label{eq_discussion_Noether_formula_for_rotations_from_12}
\begin{split}
&(-\Delta)^\frac{1}{2}u(x)\cdot u'(x)X(x)-(-\Delta)^\frac{1}{2}(u'X)(x)u(x)+(-\Delta)^\frac{1}{2}u(x)\cdot u'(x)X(x)\\+&\frac{d}{dx}\left((-\Delta)^\frac{1}{2}uX\right)(x)u(x)
\end{split}
\end{equation}
for a.e. $x\in \mathbb{S}^1$.
When $X$ is the generator of rotations (i.e. $X\equiv 1$ in $\mathbb{S}^1$) we recover Equation (\ref{eq_discussion_Noether_formula_for_rotations_from_14}).

\subsection{Half Dirichlet energy and traces of automorphisms of $D^2$}\label{ssec: Traces of conformal maps and half Dirichlet energy}
In this section we show that the half Dirichlet energy satisfies condition (\ref{eq_diffeo_is_conformal14}) whenever the flow $\phi_t$ consists of traces of automorphisms of $D^2$, i.e. conformal, orientation preserving diffeomorphisms of $D^2$. To this end we first give a geometrical characterization of the $\frac{1}{2}$-fractional Laplacian for functions on $\mathbb{S}^1$. For the following we identify $\mathbb{S}^1$ with $\partial D^2$.\\
Let $u\in H^\frac{1}{2}(\mathbb{S}^1,\mathbb{R}^m)$. We recall that the harmonic extension $\tilde{u}$ is given by
\begin{equation}\label{eq_harmonic_extension_Poisson_kernel}
\tilde{u}(r, \theta)=P_r\ast u (\theta),
\end{equation}
for any $r\in [0,1)$, $\theta\in [0, 2\pi)$, where $P_r$ is the Poisson kernel defined in (\ref{eq_intro_circle_definition_Poisson_kernel}).
Then $\tilde{u}$ is smooth in $D^2$, and
\begin{equation}
\partial_r\tilde{u}(r,\theta)=\sum_{n\in\mathbb{Z}}\lvert n\rvert r^{\lvert n\rvert-1}\widehat{u}(n)e^{in\theta}.
\end{equation}
Therefore, by Lebsgue's Dominated convergence Theorem,
\begin{equation}
\lim_{r\to 1^-}\left\lVert \partial_r \tilde{u}(r,\theta)-\sum_{n\in\mathbb{Z}}\lvert n\rvert \widehat{u}(n)e^{in\theta}\right\rVert_{H^{-\frac{1}{2}}(\mathbb{S}^1)}=0,
\end{equation}
i.e.
\begin{equation}\label{eq_conv_to_frac_lap}
\lim_{r\to 1^-}\left\lVert \partial_r \tilde{u}(r,\theta)-(-\Delta)^\frac{1}{2}u(\theta)\right\rVert_{H^{\frac{1}{2}}(\mathbb{S}^1)}=0.
\end{equation}
Thus we can think of the fractional Laplacian $(-\Delta)^\frac{1}{2}$ as the operator mapping a function $u\in H^\frac{1}{2}(\mathbb{S}^1,\mathbb{R}^m)$ to the radial derivative of its harmonic extension $\tilde{u}$.\\
Let's now fix $u\in C^\infty(\mathbb{S}^1,\mathbb{R}^m)$, let $\tilde{u}$ be its harmonic extension in $D^2$ and let $\phi$ be a biholomorphic map from $D^2$ to $D^2$, i.e. $\phi\in \mathscr{M}(D^2)$, the M\"{o}bius group of the disc.
We recall that any element of $\mathscr{M}(D^2)$ takes the form
\begin{equation}\label{eq_general_form_Mobius2}
D^2\to D^2,\quad z\mapsto\mu\frac{z-a}{\bar{a}z-1}
\end{equation}
for $\mu, a\in\mathbb{C}$ with $\lvert \mu\rvert=1$, $\lvert a\rvert<1$. Thus $\phi$ extend to a diffeomorphism of $\overline{D^2}$, which we still denote $\phi$.
We observe that $\tilde{u}\circ \phi$ is harmonic as composition of an harmonic map with an holomorphic function.
Now, for any $r\in [0,1)$,
\begin{equation}
\partial_r (\tilde{u}\circ\phi)=D\tilde{u}\circ \phi\cdot\partial_r\phi.
\end{equation}
We also observe that since $\phi$ restricts to a diffeomorphism of $\mathbb{S}^1$ and it is conformal, for any $\theta\in \mathbb{S}^1$
\begin{equation}
\partial_r \phi(1,\theta)=\frac{\lvert \partial_r \phi(1,\theta)\rvert}{\lvert \phi(1,\theta)\rvert}\phi(1,\theta)=\lvert \partial_r \phi (1, \theta)\rvert\phi(1, \theta),
\end{equation}
therefore for any $\theta\in \mathbb{S}^1$
\begin{equation}\label{eq_lim_boundary_smooth_case}
(-\Delta)^\frac{1}{2}(u\circ \phi)(\theta)= D\tilde{u}(\phi(1,\theta))\cdot\partial_r\phi(1,\theta)=(-\Delta)^\frac{1}{2}u(\phi(1,\theta))\lvert\partial_r\phi(1,\theta)\rvert,
\end{equation}
Finally, since $\phi$ is orientation preserving and conformal, for any $\theta\in \mathbb{S}^1$
\begin{equation}
\lvert\partial_r\phi(1,\theta)\rvert=e^{\lambda_\phi(1,\theta)},
\end{equation}
where $\lambda_\phi$ is the conformal factor of $\phi$ defined in (\ref{eq_def_definition_conformal_map}).
We conclude that
\begin{equation}\label{eq_fraclap_of_precomposition}
(-\Delta)^\frac{1}{2}(u\circ\phi)=(-\Delta)^\frac{1}{2}u\circ\phi e^{\lambda_\phi}.
\end{equation}
In particular, if we multiply both sides of (\ref{eq_fraclap_of_precomposition}) by $u\circ\phi$ we can conclude that the half Dirichlet energy satisfies condition (\ref{eq_diffeo_is_conformal12_2}) for any flow consisting of traces of biholomorphic maps from $D^2$ to $D^2$.\\

Now we compute the vector field induced by some elements of $\mathscr{M}(D^2)$, and we apply Theorem \ref{prop_final_version_Noether_thm_vardom_12}.\\
First, let's consider rotations of the disc: let $a=1$ and $\mu=e^{it}$ for $t\in \mathbb{R}$. On $\mathbb{S}^1$, thought of as a quotient space of $\mathbb{R}$, this corresponds to the flow $\phi_t(x)=x+t$ for $x\in \mathbb{S}^1$, $t\in\mathbb{R}$. This is the flow generated by $X\equiv1$. Therefore, in this case, equation (\ref{eq_prop_result_Noether_thm_vardom_14}) is just equation (\ref{eq_prop_stationary_equation_14}). Recall that also when we considered the Lagrangian (\ref{eq_fractional_Dirichlet_energy_for_example_14}), applying Noether to the flow corresponding to rotations we obtained the stationary equation (see (\ref{eq_prop_result_Noether_thm_vardom_14_3})).\\
Now let $\mu=1$, $a=e^{i\delta}t$ for some $\delta\in [0,1)$ and $t\in \mathbb{R}$ in a neighbourhood of $0$.
The corresponding flow in $\mathbb{C}$ is given by $\tilde{\phi}_t(e^{i\theta})=\frac{e^{i\theta}-te^{i\delta}}{1-te^{i(\theta-\delta)}}$ for $\theta\in [\theta, 2\pi)$. This is the flow of the vector field 
\begin{equation}
Y(e^{i\theta})=e^{i(2\theta-\delta)}-e^{i\delta}=2ie^{i\theta}\sin(\theta-\delta)
\end{equation}
for any $\theta\in [0,2\pi)$. Thus, on $\mathbb{S}^1$, $Y$ takes the form $Y(x)=2\sin( x-\delta)$ for any $\theta\in \mathbb{S}^1$. Geometrically, $Y$ induces a "dialation on $\mathbb{S}^1$" around the point $\delta$. Then equation (\ref{eq_prop_result_Noether_thm_vardom_12}) becomes
\begin{equation}\label{eq_discussion_traces_of_conformal_maps_conserved_quantity_moving_a}
\begin{split}
&\text{div}_\frac{1}{2}\left(F(u(x))u'(y)\sin( y-\delta)\right)
+\frac{d}{dx}\left[L\left(u,(-\Delta)^\frac{1}{2}u\right)\sin( x-\delta)\right](x)=0
\end{split}
\end{equation}
in the sense of distributions, where $F$ is defined as in (\ref{eq_discussion_definition_fct_f_in_terms_of_F_G}).
In particular, in the case of the half Dirichlet energy, Equation (\ref{eq_discussion_traces_of_conformal_maps_conserved_quantity_moving_a}) becomes
\begin{equation}\label{eq_discussion_traces_of_conformal_maps_conserved_quantity_moving_a_2}
\text{div}_\frac{1}{2}\left(u(x)\cdot u'(y)\sin( y-\delta)\right)
+\frac{d}{dx}\left[u\cdot(-\Delta)^\frac{1}{2}u\sin( x-\delta)\right](x)=0.
\end{equation}

Next we observe that the argument of Noether Theorem can be applied to the half Dirichlet energy to obtain Pohozaev identities. Let $u\in H^1(\mathbb{S}^1,\mathbb{R}^m)$. Following the argument that precedes Proposition \ref{prop_first_version_Noether_vardom_14} for the half Dirichlet energy in the form
\begin{equation}\label{eq_half_Dirichlet_energy_invariant_form_for_Pohozaev}
E(v)=\int_{\mathbb{S}^1}(-\Delta)^\frac{1}{2}v(x)\cdot v(x)dx,
\end{equation}
well defined for any $v\in H^1(\mathbb{S}^1,\mathbb{R}^m)$, and making use of the fact that the energy (\ref{eq_half_Dirichlet_energy_invariant_form_for_Pohozaev}) satisfies condition (\ref{eq_diffeo_is_conformal12_2}) for any flow consisting of traces of biholomorphic maps from $D^2$ to $D^2$, we obtain for any $v\in C^\infty(\mathbb{S}^1,\mathbb{R}^m)$
\begin{equation}\label{eq_discussion_modified_result_proof_Noeththm_for_smooth_14}
\begin{split}
&\frac{d}{dx}\left[\frac{1}{2}v(-\Delta)^\frac{1}{2}v\right]X+\text{div}_\frac{1}{2}\left(\frac{1}{2}v(x)v'(y)\right)X\\
=&\frac{d}{dx}\left[\frac{1}{2}v(-\Delta)^\frac{1}{2}vX\right]+\text{div}_\frac{1}{2}\left(\frac{1}{2}v(x)v'(y)X(y)\right)
\end{split}
\end{equation}
for any smooth vector field $X$ on $\mathbb{S}^1$ whose flow $\phi_t$ consists of traces of automorphisms of $D^2$.
Thus applying Lemma \ref{lem_explicit_form_fracdiv} to (\ref{eq_discussion_modified_result_proof_Noeththm_for_smooth_14}) we obtain
\begin{equation}\label{eq_discussion_Pohozaev_pointwise_circle_simpler}
v\cdot\frac{d}{dx}\left( (-\Delta)^\frac{1}{2}v X\right)=v\cdot (-\Delta)^\frac{1}{2}(v'X)
\end{equation}
for any smooth vector field $X$ on $\mathbb{S}^1$ whose flow $\phi_t$ consists of traces of automorphisms of $D^2$.
Now let $(u_n)_n$ in $C^\infty(\mathbb{S}^1, \mathbb{R})$ such that $u_n\to u $ in $H^\frac{3}{2}(\mathbb{S}^1)$, then Equation (\ref{eq_discussion_Pohozaev_pointwise_circle_simpler}) holds for $u_n$ for any $n\in \mathbb{N}$. By considering the limit $n\to \infty$, we observe that Equation (\ref{eq_discussion_Pohozaev_pointwise_circle_simpler}) remains true in the sense of distributions for $u$ in place of $v$.\\
Applying the distributions on both sides to the smooth function $\phi\equiv 1$ we obtain
\begin{equation}
\int_{\mathbb{S}^1}u'(x)(-\Delta)^\frac{1}{2}u(x)X(x)dx=0.
\end{equation}
In particular, if $X\equiv1$ (the vector field corresponding to rotations), we obtain the Pohozaev identity
\begin{equation}
\int_{\mathbb{S}^1}u'(x)(-\Delta)^\frac{1}{2}u(x)dx=0,
\end{equation}
while if we choose, for some $\delta\in \mathbb{S}^1$, $X(x)=2\sin(x-\delta)$ for every $x\in \mathbb{S}^1$ (the vector field corresponding to "dilations around $\delta$"), we obtain the Pohozaev identity
\begin{equation}
\int_{\mathbb{S}^1}u'(x)(-\Delta)^\frac{1}{2}u(x)\sin(x-\delta)dx=0.
\end{equation}
We summarize the previous computations in the following Proposition:
\begin{prop}\label{prop_Pohozaev_identities_on_the_circle}
Let $u\in H^1(\mathbb{S}^1, \mathbb{R})$. Then, for any vector field $X$ on $\mathbb{S}^1$ whose flow is a family of traces of automorphisms of $D^2$, there holds
\begin{equation}\label{eq_prop_Pohozaev_pointwise_circle}
u\frac{d}{dx}\left[(-\Delta)^\frac{1}{2}uX\right]=u(-\Delta)^\frac{1}{2}\left(u'X\right)
\end{equation}
in the sense of distributions. In particular
\begin{equation}
\int_{\mathbb{S}^1}u'(x)(-\Delta)^\frac{1}{2}u(x)dx=0,
\end{equation}
and for any $\delta\in \mathbb{S}^1$
\begin{equation}\label{eq_prop_Pohozaev_circle_dilation_integral_form}
\int_{\mathbb{S}^1}u'(x)(-\Delta)^\frac{1}{2}u(x)\sin(x-\delta)dx=0.
\end{equation}
\end{prop}
\begin{rem}\noindent
\begin{enumerate}
\item Proposition \ref{prop_Pohozaev_identities_on_the_circle} can be applied, for instance, whenever a function $u\in H^\frac{1}{2}(\mathbb{S}^1,\mathbb{R}^m)$ satisfies the equation
\begin{equation}\label{eq_eq_discussion_typical_application_Pohozaev_circle}
(-\Delta)^\frac{1}{2}u=f\circ u\quad\text{weakly in }\mathbb{S}^1
\end{equation}
for some Lipschitz function $f:\mathbb{R}^m\to\mathbb{R}^m$. In fact, since $u\in H^\frac{1}{2}(\mathbb{S}^1)$ and $f$ is Lipschitz, $f\circ u\in H^\frac{1}{2}(\mathbb{S}^1)$, thus if $u$ satisfies Equation (\ref{eq_eq_discussion_typical_application_Pohozaev_circle}), $u\in H^1(\mathbb{S}^1)$ and thus the assumption of Proposition \ref{prop_Pohozaev_identities_on_the_circle} is satisfied.
\item We remark that if $X(x)=2\sin(x-\delta)$ for every $x\in \mathbb{S}^1$, identity (\ref{eq_prop_Pohozaev_pointwise_circle}) also follows from the identity of distributions
\begin{equation}\label{eq_discussion_identity_Pohozaev_valid_for_any _distr_circle}
\frac{d}{dx}\left[(-\Delta)^\frac{1}{2}uX\right]=(-\Delta)^\frac{1}{2}\left(u'X\right),
\end{equation}
valid for any $u\in \mathscr{D}'(\mathbb{S}^1)$. Identity (\ref{eq_discussion_identity_Pohozaev_valid_for_any _distr_circle}) can be proved computing the Fourier coefficients of both sides: in fact, for any $n\in \mathbb{Z}$
\begin{equation}
\begin{split}
&\mathscr{F}\left(\frac{d}{dx}\left(\sin(x-\delta))(-\Delta)^\frac{1}{2}u\right)\right)(n)\\=&-in\left(\frac{e^{-i\delta}}{2i}\widehat{(-\Delta)^\frac{1}{2}u}(n-1)-\frac{e^{i\delta}}{2i}\widehat{(-\Delta)^\frac{1}{2}u}(n+1)\right)\\
=&-n\lvert n-1\rvert \frac{e^{-i\delta}}{2}\widehat{u}(n-1)+n\lvert n+1\rvert\frac{e^{i\delta}}{2}\widehat{u}(n+1),
\end{split}
\end{equation}
on the other hand
\begin{equation}
\begin{split}
&\mathscr{F}\left((-\Delta)^\frac{1}{2}(u'\sin(x-\delta)) \right)(n)\\=&\lvert n\rvert\left(\frac{e^{-i\delta}}{2i}\widehat{u'}(n-1)-\frac{e^{i\delta}}{2i}\widehat{u'}(n+1)\right)\\
=&-\lvert n\rvert (n-1)\frac{e^{-i\delta}}{2}\widehat{u}(n-1)+\lvert n\rvert(n+1)\frac{e^{i\delta}}{2}\widehat{u}(n+1).
\end{split}
\end{equation}
\end{enumerate}
\end{rem}

\subsection{Functions on the real line}\label{ssec:Functions on the real line}
In this section we briefly consider variations in the domain for functions defined in $\mathbb{R}$.\newline
First let's consider energies of the form
\begin{equation}\label{eq_discussion_first_energy_vardom_functions_Euclidean}
E(u)=\int_{\mathbb{R}}L\left(u(x), (-\Delta)^\frac{1}{4}u(x)\right)dx,
\end{equation}
where $L:\mathbb{R}^m\times\mathbb{R}^m\to \mathbb{R}$ is a function of class $C^2$ such that for some constants $C$ and $r$ with $r\geq 2$
\begin{equation}
\left\lvert L(x,p)\right\rvert\leq C\left(\lvert x\rvert^2+\lvert x\rvert^r+\lvert p\rvert^2\right).
\end{equation}
Under these assumptions, the integral in (\ref{eq_discussion_first_energy_vardom_functions_Euclidean}) converges absolutely for any $u\in H^\frac{1}{2}(\mathbb{R}, \mathbb{R}^m)$.\newline
Now let $u\in C^\infty_c(\mathbb{R}^n, \mathbb{R}^m)$ and assume that $u$ is a stationary point of $E$.
Also, let $X$ be a smooth vector field on $\mathbb{R}^m$. Assume that for some $\delta>0$, the flow $\phi_t$ of $X$ exists for any $t\in (-\delta, \delta)$.
Then one can repeat the argument of the proof of Proposition \ref{prop_first_version_Noether_vardom_14} (with the obvious adaptations\footnote{In particular the kernel $K^\frac{1}{4}$ has to be substituted by the function $z\mapsto C_{1,\frac{1}{4}}\frac{1}{\lvert z\rvert^\frac{3}{2}}$, where $C_{1,\frac{1}{4}}$ was defined in (\ref{eq_introduction_definition_constant_C1s}).}) to obtain the following result:

\begin{prop}\label{prop_first_noeth_vardom_14_whole_R}
Let $L$, $X$ and $\phi_t$ be defined as above. Let $u\in C^\infty_c(\mathbb{R}, \mathbb{R}^m)$. Assume that
for $t$ in a neighbourhood of $0$, for a.e. $x\in \mathbb{R}$
\begin{equation}\label{eq_prop_condition_symmetry_14_real_line}
L\left(u\circ \phi_t(x), (-\Delta)^\frac{1}{4}(u\circ\phi_t)\right)(x)=L\left(u(\phi_t(x)), (-\Delta)^\frac{1}{4}u(\phi_t(x))\right)\phi'_t(x).
\end{equation}
Then for any $x\in \mathbb{R}$ there holds
\begin{equation}\label{eq_lemma_result_equivalence_Noether_14_whole_R}
\begin{split}
&\text{div}_\frac{1}{4}\left[\frac{D_pL\left(u(x), (-\Delta)^\frac{1}{4}u(x)\right)u'(y)}{\lvert x-y\rvert^\frac{1}{4}}\right]X+\frac{d}{dx}\left[L\left(u, (-\Delta)^\frac{1}{4}u\right)\right]X\\=&\text{div}_\frac{1}{4}\left[\frac{D_pL\left(u(x), (-\Delta)^\frac{1}{4}u(x)\right)u'(y)X(y)}{\lvert x-y\rvert^\frac{1}{4}}\right]+\frac{d}{dx}\left[L\left(u, (-\Delta)^\frac{1}{4}u\right)X\right]
\end{split}
\end{equation}
in the sense of distributions. In particular, if $u\in C^\infty_c(\mathbb{R}, \mathbb{R}^m)$ satisfies the stationarity equation
\begin{equation}
\begin{split}
\text{div}_\frac{1}{4}\left[\frac{D_pL\left(u(x), (-\Delta)^\frac{1}{4}u(x)\right)u'(y)}{\lvert x-y\rvert^\frac{1}{4}}\right]X+\frac{d}{dx}\left[L\left(u, (-\Delta)^\frac{1}{4}u\right)\right]X=0
\end{split}
\end{equation}
in the sense of distributions,
there holds
\begin{equation}
\text{div}_\frac{1}{4}\left[\frac{D_pL\left(u(x), (-\Delta)^\frac{1}{4}u(x)\right)u'(y)X(y)}{\lvert x-y\rvert^\frac{1}{4}}\right]+\frac{d}{dx}\left[L\left(u, (-\Delta)^\frac{1}{4}u\right)X\right]=0
\end{equation}
in the sense of distributions.
\end{prop}
As in the case of $\mathbb{S}^1$, this result can be extended by approximation to a broader class of stationary points.
To this end, we observe that the argument of the proof of Lemma \ref{lem_explicit_form_fracdiv} yields
\begin{lem}\label{lem_explicit_form_fracdiv_realline}
Let $s\in (0,\frac{1}{2})$, $a,b\in C^\infty_c(\mathbb{R}, \mathbb{R}^m)$. Then
\begin{equation}
\text{div}_s\left(\frac{a(x)\cdot b(y)}{\lvert x-y\rvert^s}\right)=b\cdot(-\Delta)^s a-a\cdot (-\Delta)^sb.
\end{equation}
\end{lem}
The previous Lemma allows to extend the definition of $\text{div}_s$ to arguments of the form $\frac{a(x)\cdot b(y)}{\lvert x-y\rvert^s}$, where $a,b\in H^s(\mathbb{R}, \mathbb{R}^m)$, for any $s\in (0,1)$.
A result similar to Proposition \ref{prop_first_noeth_vardom_14_whole_R} holds for stationary points of energies of the form
\begin{equation}
E(u)=\int_{\mathbb{R}}L\left(u(x), (-\Delta)^\frac{1}{2}u(x)\right)dx,
\end{equation}
in analogy with Theorem \ref{prop_final_version_Noether_thm_vardom_12}.
In this context we observe that a family of flows on $\mathbb{R}$ satisfying condition (\ref{eq_prop_condition_symmetry_14_real_line}) a.e. in $\mathbb{R}$ and for any $t\in \mathbb{R}$ is given by the traces of families of holomorphic diffeomorphisms of the upper half plane $\mathbb{H}$ that can be extended to $\overline{\mathbb{H}}$, i.e. by affine functions from the real line to itself; for instance one can consider the flow of the vector field $X\equiv1$, corresponding to translations, or the vector field $X(x)=x$ for any $x\in \mathbb{R}$, corresponding to dilations.\\
Next we observe that also in this framework, the argument of Noether Theorem can be applied to the half Dirichlet energy to obtain Pohozaev identities. Let $u\in H^1(\mathbb{R},\mathbb{R}^m)$. Following the argument that precedes Proposition \ref{prop_first_version_Noether_vardom_14} (with the obvious modifications) for the half Dirichlet energy in the form
\begin{equation}\label{eq_half_Dirichlet_energy_invariant_form_for_Pohozaev_line}
E(v)=\int_{\mathbb{R}}(-\Delta)^\frac{1}{2}v(x)\cdot v(x)dx,
\end{equation}
well defined for any $v\in H^1(\mathbb{R},\mathbb{R}^m)$, and making use of the fact that the energy (\ref{eq_half_Dirichlet_energy_invariant_form_for_Pohozaev_line}) satisfies condition (\ref{eq_prop_condition_symmetry_14_real_line}) for any flow consisting of affine functions from the real line to itself, we obtain for any $v\in C^\infty_c(\mathbb{R},\mathbb{R}^m)$
\begin{equation}\label{eq_discussion_fundamental_Noththm_vardom_Pohozaev_12_whole_R}
\begin{split}
&\frac{d}{dx}\left[\frac{1}{2}v(-\Delta)^\frac{1}{2}v\right]X+\text{div}_\frac{1}{2}\left(\frac{1}{2}v(x)v'(y)\right)X\\
=&\frac{d}{dx}\left[\frac{1}{2}v(-\Delta)^\frac{1}{2}vX\right]+\text{div}_\frac{1}{2}\left(\frac{1}{2}v(x)v'(y)X(y)\right)
\end{split}
\end{equation}
for any vector field $X$ on $\mathbb{R}$ for any flow consisting of affine functions from the real line to itself. By Lemma \ref{lem_explicit_form_fracdiv_realline}, Equation (\ref{eq_discussion_fundamental_Noththm_vardom_Pohozaev_12_whole_R}) implies that
\begin{equation}\label{eq_discussion_Pohozaev_local_real_line_whole_simple}
v\frac{d}{dx}\left[(-\Delta)^\frac{1}{2}vX \right]=v(-\Delta)^\frac{1}{2}(v'X)
\end{equation}
for any $v\in C^\infty_c(\mathbb{R},\mathbb{R}^m)$, for any vector field $X$ on $\mathbb{R}$ for any flow consisting of affine functions from the real line to itself. Arguing by approximation, we conclude that
\begin{equation}\label{eq_discussion_final_Pohozaev_for whole_R_pointwise}
u\frac{d}{dx}\left[(-\Delta)^\frac{1}{2}uX\right]=u(-\Delta)^\frac{1}{2}(u'X)
\end{equation}
in the sense of distributions for any vector field $X$ on $\mathbb{R}$ whose flow consists of affine functions from the real line to itself. We can regard Equation (\ref{eq_discussion_final_Pohozaev_for whole_R_pointwise}) as a Pohozaev identity. In fact, heuristically, if the expressions in (\ref{eq_discussion_final_Pohozaev_for whole_R_pointwise}) were integrable, we would obtain, after integrating by parts,
\begin{equation}\label{eq_discussion_heuristic_argument_Pohozaev_whole_line}
\int_{\mathbb{R}}(-\Delta)^\frac{1}{2}u(x)u'(x)X(x)dx=0
\end{equation}
for any vector field $X$ on $\mathbb{R}$ whose flow consists of affine functions from the real line to itself.

We summarize the previous computations in the following Proposition:
\begin{prop}\label{prop_Pohozaev_identity_for_the_whole_R}
Let $u\in H^1(\mathbb{R}, \mathbb{R}^m)$.
Then for any vector field $X$ on $\mathbb{R}$ whose flow consists of affine functions from the real line to itself, there holds
\begin{equation}\label{eq_prop_conclusion_Pohozaev_identity_for_the_whole_R}
u\frac{d}{dx}\left[(-\Delta)^\frac{1}{2}uX\right]=u(-\Delta)^\frac{1}{2}(u'X)
\end{equation}
in the sense of distributions.
\end{prop}

\begin{rem}\noindent
\begin{enumerate}
\item Proposition \ref{prop_Pohozaev_identity_for_the_whole_R} can be applied, for instance, whenever a function $u\in H^\frac{1}{2}(\mathbb{R},\mathbb{R}^m)$ satisfies the equation
\begin{equation}\label{eq_eq_discussion_typical_application_Pohozaev_realline}
(-\Delta)^\frac{1}{2}u=f\circ u\quad\text{weakly in }\mathbb{R}
\end{equation}
for some Lipschitz function $f:\mathbb{R}^m\to\mathbb{R}^m$ with $f(0)=0$. In fact, since $u\in H^\frac{1}{2}(\mathbb{R})$ and $f$ is Lipschitz with $f(0)=0$, $f\circ u\in H^\frac{1}{2}(\mathbb{R})$, thus if $u$ satisfies Equation (\ref{eq_eq_discussion_typical_application_Pohozaev_realline}), $u\in H^1(\mathbb{S}^1)$ and thus the assumption of Proposition \ref{prop_Pohozaev_identity_for_the_whole_R} is satisfied.

\item Let $u\in H^\frac{1}{2}(\mathbb{R},\mathbb{R})$ and assume that $u$ satisfies (\ref{eq_eq_discussion_typical_application_Pohozaev_realline}) for some function $f\in C^1(\mathbb{R},\mathbb{R})$. Then, under some additional condition on $f$, we can deduce from Proposition \ref{prop_Pohozaev_identity_for_the_whole_R} that
\begin{equation}\label{eq_equivalent_Pohozaev_identity_for_the_whole_line}
\frac{d}{dx}\left[(-\Delta)^\frac{1}{2}ux\right]=(-\Delta)^\frac{1}{2}(u'x)
\end{equation}
in the sense of distributions.
In fact, assume that there exists a constant $C>0$ such that
\begin{equation}
\lvert f(x)x\rvert+\lvert f'(x)x^2\rvert\leq C\lvert x\rvert^p
\end{equation}
for some $p>2$, for any $x\in \mathbb{R}$. Then the previous point and Proposition \ref{prop_Pohozaev_identity_for_the_whole_R}, $u$ satisfies
\begin{equation}
u\frac{d}{dx}\left[(-\Delta)^\frac{1}{2}uX\right]=u(-\Delta)^\frac{1}{2}(u'X)
\end{equation}
in the sense of distributions.
Moreover, by the Unique Continuation Principle (Theorem 1.4 in \cite{UCP}), since $u$ satisfies Equation (\ref{eq_eq_discussion_typical_application_Pohozaev_realline}), if $u$ is not the trivial solution $u\equiv0$, then $u(x)$ is different from $0$ for a.e. $x\in \mathbb{R}$, and therefore Equation (\ref{eq_prop_conclusion_Pohozaev_identity_for_the_whole_R}) implies Equation (\ref{eq_equivalent_Pohozaev_identity_for_the_whole_line}). On the other hand, if $u\equiv0$ then Equation (\ref{eq_equivalent_Pohozaev_identity_for_the_whole_line}) is trivially satisfied.

\item We also remark that Equation (\ref{eq_equivalent_Pohozaev_identity_for_the_whole_line}) actually holds in the sense of distribution for any temperate distribution $u\in \mathscr{S}'(\mathbb{R})$. To see this we first claim that it is enough to show the identity for Schwartz functions; in fact, for any $\phi\in \mathscr{S}(\mathbb{R})$, 
\begin{equation}
\left\langle(-\Delta)^\frac{1}{2}u'x,\phi\right\rangle=-\left\langle u, \frac{d}{dx}\left(x(-\Delta)^\frac{1}{2}\phi\right)\right\rangle
\end{equation}
and
\begin{equation}
\left\langle\frac{d}{dx}\left((-\Delta)^\frac{1}{2}ux\right),\phi\right\rangle=-\left\langle u, (-\Delta)^\frac{1}{2}\left(x\frac{d}{dx}\phi\right)\right\rangle.
\end{equation}
In order to show the identity for Schwartz functions, we compute, for any $\phi\in \mathscr{S}(\mathbb{R})$, for a.e. $\xi\in \mathbb{R}$
\begin{equation}
\mathscr{F}\left((-\Delta)^\frac{1}{2}(x\phi')\right)(\xi)=\lvert \xi\rvert\frac{d}{d\xi}(\xi\widehat{\phi})
\end{equation}
and
\begin{equation}
\mathscr{F}\left(\frac{d}{dx}\left(x(-\Delta)^\frac{1}{2}\phi\right)\right)(\xi)=\xi\frac{d}{d\xi}\left(\lvert \xi\rvert\widehat{\phi}\right).
\end{equation}
As the previous expression coincide a.e. in $\mathbb{R}$, we conclude that identity (\ref{eq_equivalent_Pohozaev_identity_for_the_whole_line}) holds for any $u\in \mathscr{S}'(\mathbb{R})$.
\end{enumerate}
\end{rem}
In section \ref{sec:Pohozaev identities} we will exploit the conformal invariance of the half Dirichlet energy to study the following related problem: we will try to derive Pohozaev identities for functions $u\in H^\frac{1}{2}(\mathbb{R},\mathbb{R})$ solving 
\begin{equation}
\begin{cases}
(-\Delta)^\frac{1}{2}u=f(u) &\text{weakly in }\Omega\\
u=0&\text{in}\Omega^c,
\end{cases}
\end{equation}
where $\Omega$ is a bounded domain in $\mathbb{R}$.

\subsection{The fractional Hopf differential}\label{ssec: Stationary points of the energy}
For this section, let $m\in\mathbb{N}$.
We consider the functional
\begin{equation}\label{eq_intro_ssec_half_Dirichlet_energy}
E(v):=\int_{\mathbb{S}^1}\lvert(-\Delta)^\frac{1}{4}v(x)\rvert^2dx
\end{equation}
defined for any $v\in H^\frac{1}{2}(\mathbb{S}^1, \mathbb{R}^m)$.\\
We will use the formalism developed for the fractional Noether theorems discussed above to obtain a characterization of half harmonic maps in terms of their Fourier coefficients (Lemma \ref{lem_indentity_Fourier_coefficients}).
We will then use this characterization to show that a function $u\in H^\frac{1}{2}(\mathbb{S}^1,\mathbb{R}^m)$ is a stationary point of $E$ if and only if its harmonic extension in $D^2$ is conformal (Proposition \ref{prop_harmonic_extension_of_critical_pts_are_conformal_Fourier}), a result first obtained by F. Da Lio in \cite{CompBubble} and V. Millot and Y. Sire in \cite{MillotSire}.
Finally we will introduce a fractional analogous of the Hopf differential (Definition \ref{defn_fractional_Hopf}), and we will show how it can be used to charactezize stationary point of the energy $E$ in a simple way and thus to reformulate in a simple way the proof that a function $u\in H^\frac{1}{2}(\mathbb{S}^1,\mathbb{R}^m)$ is a stationary point of $E$ if and only if its harmonic extension in $D^2$ is conformal.
Throughout this section, we will also try to show the similarities between the problem at hand and its local analogous: the study of stationary point in $H^1(D^2)$ of the energy
\begin{equation}\label{eq_intro_local_energy}
E_{D^2}(v):=\int_{D^2}\lvert \nabla v(z)\rvert^2dz.
\end{equation}
For the next Lemma, recall that, by discussion before equation (\ref{eq_discussion_div14_is_well_defined_by_previous_results}), if $u\in H^\frac{1}{2}(\mathbb{S}^1, \mathbb{R}^m)$,
\begin{equation}
\text{div}_\frac{1}{4}\left[ 2(-\Delta)^\frac{1}{4}u(x)X(y)u'(y) \right]
\end{equation}
is a well defined distribution. By estimate (\ref{eq_discussion_estimate_for_extended_fracdiv_14}), it lies in $\mathbb{A}^1(\mathbb{S}^1)^\ast$, the topological dual of $\mathbb{A}^1(\mathbb{S}^1)$.

\begin{lem}\label{lem_distribution_description_for_derivative}
For any $u\in H^\frac{1}{2}(\mathbb{S}^1, \mathbb{R}^m)$, for any $\phi\in C^\infty(\mathbb{S}^1)$ let
\begin{equation}\label{eq_lem_definition_A[u]}
A_{u}[\phi]:=\text{div}_\frac{1}{4}\left[2(-\Delta)^\frac{1}{4}u(x)\cdot u'(y)\right] [\phi]-\int_{\mathbb{S}^1}\left\lvert (-\Delta)^\frac{1}{4}u(x)\right\rvert^2\phi'(x)dx.
\end{equation}
Then $A_u$ is a well defined distribution, belonging to $\mathbb{A}^1(\mathbb{S}^1)^\ast$ and to $H^{-1-s}(\mathbb{S}^1)$ for any $s>\frac{1}{2}$.
Moreover if $X$ if a smooth vector field on $\mathbb{S}^1$ with flow denoted by $\phi_t$, there holds
\begin{equation}\label{eq_lem_Au_vanishes_as_distr}
A_u[X]=\frac{d}{dt}\bigg\vert_{t=0}E(u\circ \phi_t).
\end{equation}
In particular, $u\in H^\frac{1}{2}(\mathbb{S}^1,\mathbb{R}^m)$ is a stationary point of the energy (\ref{eq_intro_ssec_half_Dirichlet_energy}) if and only if $A_u=0$ as distribution.
\end{lem}
\begin{proof}
First we observe that by the remark before the Lemma, $A_u$ is a well defined distribution belonging to $\mathbb{A}^1(\mathbb{S}^1)^\ast$. Thus, by embedding (\ref{eq_embedding_Sobolev_in_Wiener_der}), $A_u$ belongs to $H^{-1-s}(\mathbb{S}^1)$ for any $s>\frac{1}{2}$.
Let $(u_n)_n$ be a sequence in $C^\infty(\mathbb{S}^1)$ approximating $u$ in $H^\frac{1}{2}(\mathbb{S}^1)$ and almost everywhere (such a sequence exists by Lemma \ref{lem_approximation_in_H_12}).
By (\ref{eq_stationary_equation_vardom_testfcs214}),
\begin{equation}
A_{u_n}[X]=\frac{d}{dt}\bigg\vert_{t=0} E(u_n\circ \phi_t).
\end{equation}
By Lemma \ref{lem_continuity_derivative_energy}, $\frac{d}{dt}\vert_{t=0}E(u_n\circ\phi_t)\to\frac{d}{dt}\vert_{t=0}E(u_n\circ\phi_t)$ as $n\to \infty$.
Thus, to complete the proof, it is sufficient to show that $A_{u_n}\to A_{u}$ as distributions, as $n\to\infty$.\\
Now we observe that since $u_n\to u$ in $H^\frac{1}{2}(\mathbb{S}^1)$, $(-\Delta)^\frac{1}{4}u_n\to (-\Delta)^\frac{1}{4}u$ in $L^2(\mathbb{S}^1)$ as $n\to\infty$. Therefore
\begin{equation}
\int_{\mathbb{S}^1} \lvert(-\Delta)^\frac{1}{4}u_n\rvert^2\frac{d}{dx}Xdx\to\int_{\mathbb{S}^1} \lvert(-\Delta)^\frac{1}{4}u\rvert^2\frac{d}{dx}Xdx
\end{equation}
as $n\to\infty$.
Moreover, by estimate (\ref{eq_discussion_estimate_for_extended_fracdiv_14}),
\begin{equation}
\text{div}_\frac{1}{4}[2(-\Delta)^\frac{1}{4}u_n(x)\cdot u_n'(y)][X]\to \text{div}_\frac{1}{4}[2(-\Delta)^\frac{1}{4}u(x)\cdot u'(y)][X].
\end{equation}
as $n\to\infty$.
Therefore
\begin{equation}\label{eq_proof_lemma_convergence_An}
A_{u_n}\to A_u
\end{equation}
as $n\to\infty$ in the sense of distributions. This concludes the proof of the Lemma.
\end{proof}

\begin{lem}\label{lem_continuity_derivative_energy}
Let $u\in H^\frac{1}{2}(\mathbb{S}^1,\mathbb{R}^m)$ and let $(u_n)_n$ be a sequence in $C^\infty(\mathbb{S}^1,\mathbb{R}^m)$ approximating $u$ in $H^\frac{1}{2}(\mathbb{S}^1, \mathbb{R}^m)$ and a.e. (such a sequence exists by Lemma \ref{lem_approximation_in_H_12}). Let $X$ be a smooth vector field on $\mathbb{S}^1$ and let $\phi_t$ denote its flow.
Then, $\frac{d}{dt}\big\vert_{t=0} E(u\circ \phi_t)$ and $\frac{d}{dt}\big\vert_{t=0} E(u_n\circ \phi_t)$ are well defined for all $n\in\mathbb{N}$, and
\begin{equation}\label{eq_lem_continuity_derivative_energy}
\frac{d}{dt}\bigg\vert_{t=0} E(u_n\circ \phi_t)\to \frac{d}{dt}\bigg\vert_{t=0} E(u\circ \phi_t)
\end{equation}
as $n\to \infty$.
\end{lem}
\begin{proof}
First we observe that for any $v\in H^\frac{1}{2}(\mathbb{S}^1, \mathbb{R}^m)$, by Plancherel's identity and Lemma \ref{lem_appendix_circle_eqivalent_Sobolev_seminorm},
\begin{equation}
\begin{split}
E(v)=&\int_{\mathbb{S}^1}\lvert (-\Delta)^\frac{1}{4}v(x)\rvert^2dx=\sum_{n\in\mathbb{Z}} \mathscr{F}\left((-\Delta)^\frac{1}{4}v\right)[n]^2=\sum_{n\in\mathbb{Z}}\lvert n\rvert\hat{v}[n]^2=[v]_{\dot{H}^\frac{1}{4(2\pi)^2}}^2\\=&\frac{1}{2}\int_{\mathbb{S}^1}\int_{\mathbb{S}^1}\frac{\lvert v(x)-v(y)\rvert^2}{\sin^2\left(\frac{1}{2}(x-y)\right)}.
\end{split}
\end{equation}
Therefore, for any $t\in\mathbb{R}$,
\begin{equation}
\begin{split}
&\left\lvert E(u_n\circ \phi_t)-E(u_n)-(E(u\circ \phi_t)-E(u))\right\rvert\\
=&\frac{1}{4(2\pi)^2}\left\lvert\int_{\mathbb{S}^1}\int_{\mathbb{S}^1}\frac{\lvert u_n\circ\phi_t(x)-u_n\circ\phi_t(y)\rvert^2-\lvert u_n(x)-u_n(y)\rvert^2}{\sin^2\left(\frac{1}{2}(x-y)\right)}\right.\\
 &\left.-\frac{\lvert u\circ\phi_t(x)-u\circ\phi_t(y)\rvert^2-\lvert u(x)-u(y)\rvert^2}{\sin^2\left(\frac{1}{2}(x-y)\right)}dxdy\right\rvert \\
=&\frac{1}{4(2\pi)^2}\left\lvert\int_{\mathbb{S}^1}\int_{\mathbb{S}^1}\frac{\lvert u_n(x')-u_n(y')\rvert^2-\lvert u(x')-u(y')\rvert^2}{\sin^2\left(\frac{1}{2}(\phi_{-t}(x')-\phi_{-t}(y'))\right)}\lvert D_{x'}\phi_{-t}(x')\rvert\lvert D_{y'}\phi_{-t}(y')\rvert dx'dy'\right.\\
&\left.-\int_{\mathbb{S}^1}\int_{\mathbb{S}^1}\frac{\lvert u_n(x)-u_n(y)\rvert^2-\lvert u(x)-u(y)\rvert^2}{\sin^2\left(\frac{1}{2}(x-y)\right)}dxdy\right\rvert.
\end{split}
\end{equation}
Now we observe that since $\phi_0=Id_{\mathbb{S}^1}$, for any $\delta>0$, $x,y\in \mathbb{S}^1$,
\begin{equation}
\lvert D_{x}\phi_{-t}(x)\rvert\lvert D_{y}\phi_{-t}(y)-1\rvert\leq
\lvert t\rvert\sup_{\substack{x,y\in \mathbb{S}^1,\\ \lvert s\rvert\leq \delta}}\left\lvert \partial_s\lvert D_{x'}\phi_s(x')\rvert\lvert D_{y'}\phi_s(y')\rvert\right\rvert
\end{equation}
for any $t\in \mathbb{R}$ such that $\lvert t\rvert\leq \delta$.\\
Moreover, for any $t\in \mathbb{R}$, for any $x,y\in \mathbb{S}^1$
\begin{equation}\label{eq_proof_lemma_continuity_ derivative_energy}
\begin{split}
&\left\lvert \sin^{-2}(\frac{1}{2}(\phi_{-t}(x)-\phi_{-t}(y)))-\sin^{-2}\left(\frac{1}{2}(x-y)\right)\right\rvert\\
=&\lvert t\rvert\left\lvert \frac{2\sin\left(\frac{1}{2}(\phi_{-s}(x)-\phi_{-s}(y))\right)\cos\left(\frac{1}{2}(\phi_{-s}(x)-\phi_{-s}(y))\right)\frac{1}{2}(\partial_r\vert_{r=s}\phi_{-r}(x)-\partial_r\vert_{r=s}\phi_{-r}(y))}{\sin^4\left(\frac{1}{2}(\phi_{-s}(x)-\phi_{-s}(y))\right)} \right\rvert
\end{split}
\end{equation}
for some $s$ between $0$ and $t$, depending on $x$, $y$ and $t$.
We notice that we can choose $\delta>0$ such that there exists $\lambda>0$ with
\begin{equation}
\inf_{x\in \mathbb{S}^1, \lvert t\rvert\leq \delta}\lvert D_x\phi_t(x)\rvert>\lambda.
\end{equation}
Then for any $x,y\in \mathbb{S}^1$, if $\lvert s\rvert\leq \delta$
\begin{equation}
\left\lvert \sin\left(\frac{1}{2}(\phi_{-s}(x)-\phi_{-s}(y))\right)\right\rvert\geq \frac{\lvert \phi_{-s}(x)-\phi_{-s}(y)\rvert}{\pi}\geq \frac{\lambda}{\pi}\lvert x-y\rvert\geq \frac{2\lambda}{\pi}\left\lvert \sin\left(\frac{1}{2}(x-y)\right)\right\rvert.
\end{equation}
Therefore, for some constant $C_X$ depending on $X$, if $\lvert t\rvert\leq \delta$ the expression in (\ref{eq_proof_lemma_continuity_ derivative_energy}) can be bounded from above by
\begin{equation}
\frac{C_X\lvert t\rvert}{\sin^2\left(\frac{1}{2} (x-y)\right)}.
\end{equation}
Thus, if $\lvert t\rvert\leq \delta$,
\begin{equation}\label{eq_proof_lemma_estimate_continuity of derivative}
\begin{split}
&\lvert E(u_n\circ  \phi_t)-E(u_n)-(E(u\circ \phi_t)-E(u))\rvert\\
\leq & \frac{1}{4(2\pi)^2}\int_{\mathbb{S}^1}\int_{\mathbb{S}^1}\left\lvert \lvert u_n(x)-u_n(y)\rvert^2-\lvert u(x)-u(y)\rvert^2\right\rvert\\
&\cdot\left\lvert \sin^{-2}\left(\frac{1}{2}(\phi_{-t}(x)-\phi_{-t}(y))\right)-\sin^{-2}\left(\frac{1}{2}(x-y)\right)\right\rvert\lvert D_x\phi_{-t}(x)\rvert\lvert D_y\phi_{-t}(y)\rvert dxdy\\
&+\frac{1}{4(2\pi)^2}\int_{\mathbb{S}^1}\int_{\mathbb{S}^1}\left\lvert\frac{ \lvert u_n(x)-u_n(y)\rvert^2-\lvert u(x)-u(y)\rvert^2}{\sin^2\left(\frac{1}{2}(x-y)\right)}\right\rvert\lvert 1-\lvert D_x\phi_{-t}(x)\rvert\lvert D_{y}\phi_{-t}(y)\rvert\rvert dxdy\\
\leq & C_X \lvert t\rvert \int_{\mathbb{S}^1}\int_{\mathbb{S}^1}\left\lvert\frac{\lvert u_n(x)-u_n(y)\rvert^2-\lvert u(x)-u(y)\rvert^2}{\sin^2\left(\frac{1}{2}(x-y)\right)}\right\rvert dxdy,
\end{split}
\end{equation}
for a possibly different constant $C_X$ depending only on $X$.
Now we claim that the integral in the last line of (\ref{eq_proof_lemma_estimate_continuity of derivative}) converges to $0$ as $n\to \infty$.
To this end we define
\begin{equation}
f: \mathbb{S}^1\times \mathbb{S}^1\to \mathbb{R},\quad(x,y)\mapsto \frac{\lvert u(x)-u(y)\rvert^2}{\left\lvert\sin\left(\frac{1}{2}(x-y)\right)\right\rvert^2}
\end{equation}
and for any $n\in\mathbb{N}$
\begin{equation}
f_n: \mathbb{S}^1\times \mathbb{S}^1\to \mathbb{R},\quad(x,y)\mapsto \frac{\lvert u_n(x)-u_n(y)\rvert^2}{\left\lvert\sin\left(\frac{1}{2}(x-y)\right)\right\rvert^2}.
\end{equation}
In order to prove the claim, it is enough to show that $f_n\to f$ in $L^1(\mathbb{S}^1\times \mathbb{S}^1)$ as $n\to \infty$.
To show this, we observe that for any $n\in\mathbb{N}$
\begin{equation}
\lvert f\rvert+\lvert f_n\rvert-\lvert f_n-f\rvert\geq 0.
\end{equation}
Since $f_n\to f$ a.e., by Fatou's Lemma we obtain
\begin{equation}
\int_{\mathbb{S}^1}\int_{\mathbb{S}^1} 2\lvert f\rvert\leq\liminf_{n\to\infty}\int_{\mathbb{S}^1}\int_{\mathbb{S}^1} \lvert f\rvert+\lvert f_n\rvert-\lvert f_n-f\rvert.
\end{equation}
Since $u_n\to u$ in $H^\frac{1}{2}(\mathbb{S}^1)$ , $\lVert f_n\rVert_{L^1}\to\lVert f\rVert_{L^1}$ as $n\to \infty$, and thus we conclude that
\begin{equation}
\limsup_{n\to\infty}\int_{\mathbb{S}^1}\int_{\mathbb{S}^1}\lvert f_n-f\rvert=0.
\end{equation}
This concludes the proof of the claim.
Now we observe that dividing the first and the last line of (\ref{eq_proof_lemma_estimate_continuity of derivative}) by $t\neq 0$ and letting $t$ tend to $0$, we obtain
\begin{equation}
\begin{split}
&\lim_{n\to \infty} \limsup_{t\to 0} \left( \frac{E(u_n\circ  \phi_t)-E(u_n)}{t}-\frac{E(u\circ \phi_t)-E(u)}{t}\right)=0,\\
&\lim_{n\to \infty} \liminf_{t\to 0} \left( \frac{E(u_n\circ  \phi_t)-E(u_n)}{t}-\frac{E(u\circ \phi_t)-E(u)}{t}\right)=0.
\end{split}
\end{equation}
For any $n\in \mathbb{N}$, since $u_n$ is smooth, $E(u\circ\phi_t)$ is differentiable, therefore
\begin{equation}
\begin{split}
&\lim_{n\to \infty} \frac{d}{dt}\bigg\vert_{t=0}E(u_n\circ\phi_t)+\limsup_{t\to 0}-\frac{E(u\circ \phi_t)-E(u)}{t}=0,\\
&\lim_{n\to \infty} \frac{d}{dt}\bigg\vert_{t=0}E(u_n\circ\phi_t)+\liminf_{t\to 0}-\frac{E(u\circ \phi_t)-E(u)}{t}=0.
\end{split}
\end{equation}

Therefore $E(u\circ\phi_t)$ is differentiable in $t$ in $0$, and
\begin{equation}\label{eq_lem_continuity_derivative_energy_2}
\frac{d}{dt}\bigg\vert_{t=0} E(u_n\circ \phi_t)\to \frac{d}{dt}\bigg\vert_{t=0} E(u\circ \phi_t)
\end{equation}
as $n\to \infty$.
\end{proof}

In the next Lemma, we will use the previous results to obtain a characterization of stationary points of the energy (\ref{eq_intro_ssec_half_Dirichlet_energy}) in terms of their Fourier coefficients.

\begin{lem}\label{lem_indentity_Fourier_coefficients}
Let $u\in H^\frac{1}{2}(\mathbb{S}^1,\mathbb{R}^m)$. Then $u$ is a stationary point of $E$ if and only if for any $k\in\mathbb{N}_{>0}$
\begin{equation}\label{eq_lem_indentity_Fourier_coefficients_identity}
\sum_{\substack{a,b \in{\mathbb{N}}_{>0},\\ a+b=k}}a\widehat{u}(a)\cdot_{\mathbb{C}^m} b\widehat{u}(b)=0.
\end{equation}
where $\cdot_{\mathbb{C}^m}$ denotes the $\mathbb{C}$-scalar product in $\mathbb{C}^m$.
\end{lem}

\begin{proof}
First assume that $u\in H^\frac{1}{2}(\mathbb{S}^1)$ is a stationary point of $E$. Let $(u_n)_n$ be a sequence in $C^\infty(\mathbb{S}^1,\mathbb{R}^m)$ approximating $u$ in $H^\frac{1}{2}(\mathbb{S}^1)$ and almost everywhere (such a sequence exists by Lemma \ref{lem_approximation_in_H_12}).
For any $n\in\mathbb{N}$ we have, by Lemma \ref{lem_explicit_form_fracdiv},
\begin{equation}\label{eq_proof_lem_dentity_distribution_for_stationary_equation}
\begin{split}
A_{u_n}=&\text{div}_\frac{1}{4}\left[ 2(-\Delta)^\frac{1}{4}u_n(x)u'(y) \right]+\frac{d}{dx}\left[\left\lvert(-\Delta)^\frac{1}{4}u\right\rvert^2\right]\\
=&2(-\Delta)^\frac{1}{2}u_n\cdot u_n'-2(-\Delta)^\frac{1}{4}u_n\cdot (-\Delta)^\frac{1}{4}u_n+2(-\Delta)^\frac{1}{4}u_n\cdot(-\Delta)^\frac{1}{4}u_n'\\
=&2(-\Delta)^\frac{1}{2}u_n\cdot u_n'
\end{split}
\end{equation}
as distributions.
Since $u$ is a stationary point of $E$, Lemma \ref{lem_distribution_description_for_derivative} implies that $A_u[X]=0$ for any smooth vector field $X$ on $\mathbb{S}^1$, and therefore $A_u=0$ as a distribution. By Lemmas \ref{lem_continuity_derivative_energy} and \ref{lem_distribution_description_for_derivative}, $A_{u_n}\to A_u$ in the sense of distributions, therefore
\begin{equation}\label{eq_proof_lemma_product_at_the_boundary}
2(-\Delta)^\frac{1}{2}u_n\cdot u_n'\to 0
\end{equation}
as distributions, as $n\to\infty$.
For any $n\in\mathbb{Z}$, let $v_n:=(-\Delta)^\frac{1}{2}u_n$ and let
\begin{equation}\label{eq_proof_lemma_definition_positive_negative_parts}
{v_n}_+:=\sum_{k>0}\widehat{v_n}(k)e^{ik\cdot}\quad {v_n}_-:=\sum_{k<0}\widehat{v_n}(k)e^{ik\cdot}
\end{equation}
(where the convergence takes place in $L^2(\mathbb{S}^1)$; observe that $\widehat{v_n}(0)=0$).
Since $v_n$ takes values in $\mathbb{R}^m$, for any $k\in\mathbb{Z}$ $\widehat{v_n}(k)=\overline{\widehat{v_n}(-k)}$. Therefore, by (\ref{eq_disc_intro_comparison_H_and_fraclap}),
\begin{equation}\label{eq_proof_lem_identity_for_approx_for_Fourier}
\begin{split}
(-\Delta)^\frac{1}{2}u_n\cdot u_n'=&v_n\cdot H v_n=-i({v_n}_++\overline{{v_n}_+})\cdot({v_n}_+-\overline{{v_n}_+})\\
=&-i\left({v_n}_+\cdot_{\mathbb{C}^m} {v_n}_+-\overline{{v_n}_+}\cdot_{\mathbb{C}^m}\overline{{v_n}_+}\right),
\end{split}
\end{equation}
where $H$ denotes the Hilbert transform. Thus (\ref{eq_proof_lemma_product_at_the_boundary}) implies
\begin{equation}
{v_n}_+\cdot_{\mathbb{C}^m} {v_n}_+-\overline{{v_n}_+}\cdot_{\mathbb{C}^m}\overline{{v_n}_+}\to 0\text{ as distributions}
\end{equation}
as distributions, as $n\to \infty$.
For any $k\in\mathbb{N}_{>0}$, we can take $e^{-ik\cdot}$ as test function; we obtain
\begin{equation}\label{eq_proof_lemma_approximation_sum_of_frequencies}
\int_{\mathbb{S}^1} {v_n}_+\cdot{v_n}_+(x)e^{-ikx}dx-\int_{\mathbb{S}^1} \overline{{v_n}_+\cdot{v_n}_+}(x)e^{-ikx}dx\to 0
\end{equation}
as $n\to \infty$. The second integral in (\ref{eq_proof_lemma_approximation_sum_of_frequencies}) vanishes, as $\overline{{v_n}_+\cdot_{\mathbb{C}^m}{v_n}_+}$ consists only of negative frequencies. The first integral is equal to
\begin{equation}
\begin{split}
\sqrt{2\pi}\mathscr{F}({v_n}_+\cdot_{\mathbb{C}^m} {v_n}_+)(k)=&\sqrt{2\pi}\sum_{\substack{a,b\in\mathbb{N},\\ a+b=k}}\widehat{v_n}(a)\cdot_{\mathbb{C}^m}\widehat{v_n}(b)\\
=&\sqrt{2\pi}\sum_{\substack{a,b\in\mathbb{N},\\ a+b=k}} a\widehat{u_n}(a)\cdot_{\mathbb{C}^m} b\widehat{u_n}(b).
\end{split}
\end{equation}
Since $u_n\to u$ in $H^\frac{1}{2}(\mathbb{S}^1)$, $\widehat{u_n}(k)\to \widehat{u}(k)$ as $n\to \infty$ for all $k\in\mathbb{Z}$.
Therefore, (\ref{eq_proof_lemma_approximation_sum_of_frequencies}) implies that
\begin{equation}
\sum_{\substack{a,b\in{\mathbb{N}}_{>0},\\ a+b=k}}a\widehat{u}(a)\cdot_{\mathbb{C}^m} b\widehat{u}(b)=0.
\end{equation}
This shows the first implication. For the converse, let $u\in H^\frac{1}{2}(\mathbb{S}^1)$ so that (\ref{eq_lem_indentity_Fourier_coefficients_identity}) holds. Let $(u_n)_n$ be a sequence in $C^\infty(\mathbb{S}^1)$ converging to $u$ in $H^\frac{1}{2}(\mathbb{S}^1)$. For any $n\in\mathbb{N}$ let $v_n, {v_n}_+, {v_n}_-$ be defined as before. Then, for any $k\in\mathbb{N}_{>0}$
\begin{equation}\label{eq_proof_lemma_Fouriercoeff_convergence_coeff_product}
\mathscr{F}({v_n}_+\cdot_{\mathbb{C}^m}{v_n}_+)(k)=\sum_{\substack{a,b\in\mathbb{N},\\ a+b=k}}a\widehat{u_n}(a)\cdot_{\mathbb{C}^m}b\widehat{u_n}(b)\to 0
\end{equation}
as $n\to\infty$.
Since for any $n\in\mathbb{N}$ ${v_n}_+\cdot_{\mathbb{C}^m}{v_n}_+$ consists only of positive frequencies, and since for any $k\in\mathbb{N}_{>0}$ $\mathscr{F}(\overline{{v_n}_+}\cdot_{\mathbb{C}^m}\overline{{v_n}_+})(k)=  \mathscr{F}({v_n}_+\cdot_{\mathbb{C}^m}{v_n}_+)(-k)$, it follows from (\ref{eq_proof_lemma_Fouriercoeff_convergence_coeff_product}) that for any $k\in\mathbb{Z}$,
\begin{equation}
\begin{split}
&\mathscr{F}({v_n}_+\cdot_{\mathbb{C}^m} {v_n}_+)(k)-\mathscr{F}(\overline{{v_n}_+}\cdot_{\mathbb{C}^m}\overline{{v_n}_+})(k)\\
=&\mathscr{F}({v_n}_+\cdot_{\mathbb{C}^m} {v_n}_+)(k)-\mathscr{F}({v_n}_+\cdot_{\mathbb{C}^m} {v_n}_+)(-k)\to 0
\end{split}
\end{equation}
as $n\to \infty$ (here we exploit the fact that if $k>0$, the second term vanishes, while if $k<0$, the first one does).
By (\ref{eq_proof_lem_identity_for_approx_for_Fourier}), this implies that for any $k\in\mathbb{Z}$
\begin{equation}
A_{u_n}[e^{-ik\cdot}]=2\mathscr{F}\left((-\Delta)^\frac{1}{2}u_n\cdot u_n'\right)(k)\to 0
\end{equation}
as $n\to \infty$ (here we used the notation introduced in (\ref{eq_lem_definition_A[u]}) and computation  (\ref{eq_proof_lem_dentity_distribution_for_stationary_equation})).
Since, by Lemmas \ref{lem_continuity_derivative_energy} and \ref{lem_distribution_description_for_derivative}, $A_{u_n}\to A_u$ in the sense of distributions, we obtain that
\begin{equation}\label{eq_proof_lem_Au_vanishes_on_ek}
A_u[e^{-ik\cdot}]=0
\end{equation}
for any $k\in \mathbb{Z}$. We claim that this is equivalent to $A_u=0$ in the sense of distributions, which by (\ref{eq_lem_Au_vanishes_as_distr}) implies that $u$ is a stationary point of $E$.
To prove the claim, let $\phi\in C^\infty(\mathbb{S}^1,\mathbb{R}^m)$ and for any $n\in\mathbb{N}$ let $\phi_n:=D_n\ast\phi$, where $D_n$ is the $n^{th}$-Dirichlet kernel. Then, as $A_u$ is linear, by (\ref{eq_proof_lem_Au_vanishes_on_ek}) $A_u[\phi_n]=0$ for any $n\in\mathbb{N}$.
Since $A_u\in \mathbb{A}^1(\mathbb{S}^1)^\ast$ by Lemma \ref{lem_distribution_description_for_derivative}, to prove the claim it is enough to show that $\phi_n\to \phi$ in $\mathbb{A}^1(\mathbb{S}^1)$.
In fact, for any $n\in\mathbb{N}$
\begin{equation}\label{eq_proof_lem_convergence_testfct_in_A'}
\begin{split}
\lVert \phi_n-\phi\rVert_{\mathbb{A}^1}=&\sum_{\lvert k\rvert\geq n+1}\left\lvert \widehat{\phi}(k)\right\rvert\lvert k\rvert=\sum_{\lvert k\rvert\geq n+1}\left\lvert \frac{\widehat{\phi^{(3)}}(k)}{k^3}\right\rvert\lvert k\rvert\\ \leq & \lVert \phi^{(3)}\rVert_{L^1}\sum_{\lvert k\rvert\geq n+1}\frac{1}{\lvert k\rvert^2},
\end{split}
\end{equation}
and the expression on the right of (\ref{eq_proof_lem_convergence_testfct_in_A'}) tends to $0$ as $n\to 0$. This concludes the proof of the claim. 
\end{proof}
\begin{rem}
Since $u$ takes values in $\mathbb{R}^m$, for any $k\in\mathbb{N}$ $\widehat{v_n}(k)=\overline{\widehat{v_n}(-k)}$. Therefore, for any $k\in\mathbb{Z}_{<0}$, (\ref{eq_lem_indentity_Fourier_coefficients_identity}) is equivalent to
\begin{equation}
\sum_{\substack{a,b\in{\mathbb{Z}}_{<0},\\ a+b=k}}a\widehat{u}(a)\cdot_{\mathbb{C}^m} b\widehat{u}(b)=0.
\end{equation}
\end{rem}
\begin{rem}\label{rem_DaLio_Cabre}
Let $u\in H^\frac{1}{2}(\mathbb{S}^1,\mathbb{R}^m$); for any $a\in \mathbb{Z}$ let
\begin{equation}
\alpha_a=Re(\widehat{u}(a))\quad \beta_a:=Im(\widehat{u}(a)).
\end{equation}
Then for any $k\in \mathbb{N}_{>0}$ the left hand side of Equation (\ref{eq_lem_indentity_Fourier_coefficients_identity}) can be rewritten as
\begin{equation}\label{eq_real_vanishing_quantity_DaLio_Cabre}
\sum_{\substack{a,b\in\mathbb{N}_{>0}\\ a+b=k}}ab(\alpha_a\alpha_b-\beta_a\beta_b),
\end{equation}
while its imaginary part can be rewritten as
\begin{equation}\label{eq_imaginary_vanishing_quantity_DaLio_Cabre}
\sum_{\substack{a,b\in\mathbb{N}_{>0}\\ a+b=k}}ab(\alpha_a\beta_b-\alpha_a\beta_b).
\end{equation}
The fact that the quantities in (\ref{eq_real_vanishing_quantity_DaLio_Cabre}) and (\ref{eq_imaginary_vanishing_quantity_DaLio_Cabre}) vanish for any $k\in \mathbb{N}_{>0}$ if $u$ is a stationary point of $E$ was first published by F. Da Lio in \cite{DaLio} (Proposition 1.2), and was also obtained in the work of preparation of \cite{Cabre}.
\end{rem}
In the following, we will show that the harmonic extension in $D^2$ of stationary point of $E$ in $H^\frac{1}{2}(\mathbb{S}^1)$ is conformal in $D^2$. This result was first obtained in \cite{CompBubble} and \cite{MillotSire}. First we will show how the characterization of stationary points o the energy (\ref{eq_intro_ssec_half_Dirichlet_energy}) in terms of their Fourier coefficients (Lemma \ref{lem_indentity_Fourier_coefficients}), combined with an important property of the Hopf differential (Lemma \ref{lem_Hopf_diff_holomorphic_Schoen}), can be used to obtain the desired result (Proposition \ref{prop_harmonic_extension_of_critical_pts_are_conformal_Fourier}). We will then give another proof of the result (Proposition \ref{lem_stationary_points_ have_conformal_harmonic_extension}), based on a variational characterization of conformal maps  (Proposition \ref{prop_variational_criterion_conformality}) and an approximation argument obtained from Lemmas \ref{lem_distribution_description_for_derivative} and \ref{lem_continuity_derivative_energy}.\\
In order to study the properties of harmonic maps in $D^2$, we fix the following conventions:
for any $u\in H^1(D^2)$ let
\begin{equation}
[ u]_{\dot{H}^1}:=\lVert \nabla u\rVert_{L^2(D^2)}
\end{equation}
and
\begin{equation}
\lVert u\rVert_{H^1}:=\lVert u\rVert_{L^2(D^2)}+\lVert \nabla u\rVert_{L^2(D^2)}.
\end{equation}
where $\nabla u$ denotes the distributional gradient of $u$.

\begin{lem}\label{lem_harmonic_extension_conformal_first_condition}
Let $u\in H^\frac{1}{2}(\mathbb{S}^1, \mathbb{R}^m)$. Then $u$ is a stationary point of $E$ in $H^\frac{1}{2}(\mathbb{S}^1)$ if and only if its harmonic extension $\tilde{u}$ obeys
\begin{equation}\label{eq_lemma_orthogonality_of_the_derivatives}
\partial_\theta \tilde{u}\cdot\partial_r\tilde{u}=0\text{ in $D^2$}.
\end{equation}
\end{lem}
\begin{proof}
Assume first that $u$ is a stationary point of $E$. We recall that the harmonic extension of $u$ is given by $\tilde{u}(r,\theta)=P_r\ast u(\theta)$, where $P_r $ is the Poisson kernel introduced in (\ref{eq_intro_circle_definition_Poisson_kernel}).
Let $s\in (0,1)$ and let $u_s:=P_s\ast u$.
As in (\ref{eq_proof_lemma_definition_positive_negative_parts}), let $v_s:=(-\Delta)^\frac{1}{2}u_s$ and
\begin{equation}
{v_s}_+:=\sum_{k>0}\widehat{v_s}(k)e^{ik\cdot}\quad {v_s}_-:=\sum_{k<0}\widehat{v_s}(k)e^{ik\cdot}
\end{equation}
(where the convergence takes place in $L^2(\mathbb{S}^1)$).
Then, for any $\theta\in \mathbb{S}^1$,
\begin{equation}\label{eq_proof_lemma_harmonic_extension_conformal_first_condition_computation}
\begin{split}
&\partial_r \tilde{u}(s, \theta)\cdot \partial_\theta \tilde{u}(s, \theta)=\frac{1}{s}(-\Delta)^\frac{1}{2}u_s(\theta)\cdot \partial_\theta u_s(\theta)=\frac{1}{s}v_s(\theta)\cdot Hv_s(\theta)\\
=&\frac{-1}{s}({v_s}_++\overline{{v_s}_+})(\theta)\cdot  i({v_s}_+-\overline{{v_s}_+})(\theta)=\frac{-i}{s}({v_s}_+\cdot_{\mathbb{C}^m} {v_s}_+-\overline{{v_s}_-\cdot_{\mathbb{C}^m} {v_s}_+})(\theta),
\end{split}
\end{equation} 
where we used the fact that $\widehat{v_s}(0)=0$, and that $u$, and thus $\tilde{u}$, take values in $\mathbb{R}^m$.
Now we observe that for any $k\in \mathbb{Z}_{\leq 0}$, $\mathscr{F}\left(v_+\cdot_{\mathbb{C}^m} v_+\right)(k)=0$, as $v_+$ consists only of positive frequencies. Moreover, for any $k\in\mathbb{Z}_{>0}$,
\begin{equation}\label{eq_proof_lemma_first_condition_conformality_Fcoeff_product_vs}
\begin{split}
\mathscr{F}\left({v_s}_+\cdot_{\mathbb{C}^m} {v_s}_+\right)(k)=&\sum_{\substack{a, b\in \mathbb{N}_{>0},\\ a+b=k}}\widehat{{v_s}}(a)\cdot_{\mathbb{C}^m}  \widehat{{v_s}}(b)\\
=&\sum_{\substack{a, b\in \mathbb{N}_{>0},\\ a+b=k}}a\widehat{P_s}(a)\widehat{u}(a)\cdot_{\mathbb{C}^m}  b\widehat{P_s}(b)\widehat{u}(b)\\
=&\sum_{\substack{a, b\in \mathbb{N}_{>0},\\ a+b=k}}as^a\widehat{u}(a)\cdot_{\mathbb{C}^m}  bs^b\widehat{u}(b)\\
=& s^k\sum_{\substack{a, b\in \mathbb{N}_{>0},\\ a+b=k}}a\widehat{u}(a)\cdot_{\mathbb{C}^m} b\widehat{u}(b)=0,
\end{split}
\end{equation}
where the last equality follows from Lemma \ref{lem_indentity_Fourier_coefficients}. Therefore, for any $s\in(0,1)$,
\begin{equation}
{v_s}_+\cdot_{\mathbb{C}^m} {v_s}_+=0.
\end{equation}
Thus, by (\ref{eq_proof_lemma_harmonic_extension_conformal_first_condition_computation}),
\begin{equation}
\partial_r \tilde{u}\cdot \partial_\theta \tilde{u}=0\text{ in $D^2$}.
\end{equation}
This shows the first implication.
To show the converse, let $u\in H^\frac{1}{2}(\mathbb{S}^1, \mathbb{R}^m)$ so that (\ref{eq_lemma_orthogonality_of_the_derivatives}) holds. We will show that (\ref{eq_lem_indentity_Fourier_coefficients_identity}) holds as well. By Lemma \ref{lem_indentity_Fourier_coefficients}, this will conclude the proof.\\
For any $s\in (0,1)$, $k\in\mathbb{Z}_{>0}$, by (\ref{eq_lemma_orthogonality_of_the_derivatives}), (\ref{eq_proof_lemma_harmonic_extension_conformal_first_condition_computation}) and (\ref{eq_proof_lemma_first_condition_conformality_Fcoeff_product_vs})
\begin{equation}
\begin{split}
0=&\mathscr{F}\left(\partial_\theta \tilde{u}\cdot\partial_r\tilde{u}(s,\cdot)\right)(k)
=\frac{- i}{s}\mathscr{F}\left({v_s}_+\cdot_{\mathbb{C}^m} {v_s}_+\right)(k)\\
=&\frac{- i}{s}\sum_{\substack{a,b\in\mathbb{N}_{>0}\\ a+b=k}}s^a a\widehat{u}(a)\cdot_{\mathbb{C}^m} s^b b\widehat{u}(b)
=- i s^{k-1}\sum_{\substack{a,b\in\mathbb{N}_{>0}\\ a+b=k}}a\widehat{u}(a)\cdot_{\mathbb{C}^m} b\widehat{u}(b),
\end{split}
\end{equation}
therefore (\ref{eq_lem_indentity_Fourier_coefficients_identity}) holds.
\end{proof}

We recall that a weakly harmonic function $u$ in $H^1(D^2)$ satisfies (\ref{eq_lemma_orthogonality_of_the_derivatives}) if and only if 
\begin{equation}\label{eq_discussion_equivalent_condition_conformality}
\lvert\partial_\theta \tilde{u}\rvert=\frac{1}{r^2}\lvert\partial_r \tilde{u}\rvert\text{ in }D^2.
\end{equation}
This can be shown by means of the Hopf differential.

\begin{defn}
Let $u\in H^1(D^2, \mathbb{R}^m)$. The \textbf{Hopf differential} of $u$ is the form
\begin{equation}\label{eq_definition_Hopf_differential_D2}
\mathscr{H}(u)=\frac{\partial u}{\partial z}\cdot_\mathbb{C}\frac{\partial u}{\partial z}dz\otimes dz.
\end{equation}
\end{defn}
We remark that in polar coordinates, the Hopf differential of a function $u\in H^1(D^2)$ is given by
\begin{equation}\label{eq_comm_Hopf_diff_in_polar_coordinates}
\mathscr{H}(u)(z)=\frac{\overline{z}}{4 r^2}\left[\lvert \partial_ru(z)\rvert^2-\frac{1}{r^2}\lvert\partial_\theta u(z)\rvert^2-2i\partial_r u(z)\cdot\frac{1}{r}\partial_\theta u(z)\right]dz\otimes dz
\end{equation}
for any $z\in D^2$, where $z=re^{i\theta}$.
One of the most useful properties of the Hopf differential is given by the following Lemma.

\begin{lem}\label{lem_Hopf_diff_holomorphic_Schoen}(Lemma 1.1 in \cite{SchoenHarmonic})
Assume that $u\in H^1(D^2,\mathbb{R}^m)$ is a stationary point of the energy $E_{D^2}$ introduced in (\ref{eq_intro_local_energy}), in the sense that if $X$ is a compactly supported smooth vector field on $D^2$, and $\phi_t$ is its flow, then
\begin{equation}
\frac{d}{dt}\bigg\vert_{t=0} E_{D^2}(u\circ \phi_t)=0.
\end{equation}
Then $\mathscr{H}(u)$ is holomorphic in $D^2$, in the sense that $h:z\mapsto \frac{\partial u}{\partial z}\cdot_{\mathbb{C}^m}\frac{\partial u}{\partial z}$ is an holomorphic function in $D^2$.
\end{lem}
We will use this result to show that (\ref{eq_lemma_orthogonality_of_the_derivatives}) is equivalent to (\ref{eq_discussion_equivalent_condition_conformality}) if $u$ is weakly harmonic. For this, we first recall that if a function $u\in H^1(\mathbb{S}^1)$ is a critical point of $E_{D^2}$ in $H^1(\mathbb{S}^1)$ (or equivalently $u$ solves $\Delta u=0$ in the weak sense), then $u$ is smooth in $D^2$ and therefore is also a stationary point of $E_{D^2}$. Indeed, if $u$ is a smooth function on $D^2$ and $\phi_t$ is the flow of a compactly supported smooth vector field $X$ on $D^2$, by Taylor's Theorem we have for any $t\in \mathbb{R}$, $x\in D^2$

\begin{equation}
\begin{split}
d (u\circ\phi_t)(x)=& d u\circ \phi_t(x)d\phi_t(x)\\
=&d u(x)+td(d u)(x)X(x)+td u(x) dX(x)+\mathscr{O}_{L^\infty}(t^2)\\
=&d (u+tduX)(x)+\mathscr{O}_{L^\infty}(t^2).
\end{split}
\end{equation}
Thus, if $u$ is a critical point of $E_{D^2}$,
\begin{equation}
\frac{d}{dt}\bigg\vert_{t=0}E_{D^2}(u\circ\phi_t)=\frac{d}{dt}\bigg\vert_{t=0}E_{D^2}(u+tduX)=0.
\end{equation}

\begin{lem}\label{lem_equivalent_conditions_conformality_Hopf}
Let $u\in H^1(D^2)$ be weakly harmonic, i.e. assume that $u$ solves $\Delta u=0$ in the weak sense. Then the following conditions are equivalent:
\begin{enumerate}
\item  $h_1:=\partial_\theta {u}\cdot\partial_r{u}\equiv0\text{ in $D^2$}$,
\item $h_2:=\lvert\partial_r {u}\rvert^2-\frac{1}{r^2}\lvert\partial_\theta {u}\rvert^2\equiv0\text{ in }D^2$,
\item $\mathscr{H}(u)\equiv0 \text{ in }D^2.$
\end{enumerate}
\end{lem}
\begin{proof}
By the discussion before the Lemma, if $u\in H^1(\mathbb{S}^1)$ is weakly harmonic in $D^2$, then $u$ is smooth in $D^2$ and it is a stationary point of $E_{D^2}$. Therefore by Lemma \ref{lem_Hopf_diff_holomorphic_Schoen} its Hopf differential is holomorphic in $D^2$. By (\ref{eq_comm_Hopf_diff_in_polar_coordinates}), this means that the function $h$ defined by
\begin{equation}
h(z)=\frac{\overline{z}}{4 r^2}\left[\lvert \partial_ru(z)\rvert^2-\frac{1}{r^2}\lvert\partial_\theta u(z)\rvert^2-2i\partial_r u(z)\cdot\frac{1}{r}\partial_\theta u(z)\right]
\end{equation}
for any $z\in D^2$ is an holomorphic function on $D^2$. In particular, $h$ is bounded in $\frac{1}{2}D^2$.
Therefore the function $f(z):=zh(z)$ is still holomorphic in $D^2$ and $f(0)=0$. Now we observe that
\begin{equation}
f(z)=\frac{1}{4}(h_2(z)-\frac{2i}{r} h_1(z)),
\end{equation}
and therefore
\begin{equation}
Re(f(z))=\frac{1}{4}h_2(z), \quad Im(f(z))=-\frac{1}{2r} h_1(z)
\end{equation}
for any $z\in D^2$.
Thus if $h_2\equiv0$, then $Re(f)\equiv0$, and thus, by the Cauchy-Riemann equations, $f$ is constant, but since $f(0)=0$, $f$ must be equal to $0$ in $D^2$, so that $h\equiv0$ in $D^2$ and therefore $\mathscr{H}(u)\equiv0$ in $D^2$.
Similarly, if $h_1\equiv0$, then $Im(f)\equiv0$ so that $f$ must be constant. Thus one can conclude as before that $\mathscr{H}(u)\equiv0$ in $D^2$. Conversely, if $\mathscr{H}(u)\equiv0$ in $D^2$, then $f\equiv0$ on $D^2$. Thus both its real and imaginary part vanish on $D^2$, therefore $h_1\equiv0$ and $h_2\equiv0$ on $D^2$.
\end{proof}
Combining Lemma \ref{lem_harmonic_extension_conformal_first_condition} and \ref{lem_equivalent_conditions_conformality_Hopf} we obtain the desired result:
\begin{prop}\label{prop_harmonic_extension_of_critical_pts_are_conformal_Fourier}
Let $u\in H^\frac{1}{2}(\mathbb{S}^1,\mathbb{R}^m)$. Then $u$ is a stationary point of the energy (\ref{eq_intro_ssec_half_Dirichlet_energy}) if and only if its harmonic extension in $D^2$ is conformal. 
\end{prop}

We now gives another proof of Proposition \ref{prop_harmonic_extension_of_critical_pts_are_conformal_Fourier}, based on a variational characterisation of conformal maps (Proposition \ref{prop_variational_criterion_conformality}). We observe that in the proof of Proposition \ref{prop_variational_criterion_conformality} it is used again the fact that the Hopf differential of a stationary point of $E_{D^2}$ is holomorphic.

\begin{prop}\label{lem_stationary_points_ have_conformal_harmonic_extension}
Let $u\in H^\frac{1}{2}(\mathbb{S}^1,\mathbb{R}^m)$. Then $u$ is a stationary point of the energy (\ref{eq_intro_ssec_half_Dirichlet_energy}) if and only if its harmonic extension in $D^2$ is conformal. 
\end{prop}
\begin{proof}
We follow the proof of Theorem 2.2 in \cite{BlowupAnalysis}: let $(u_n)_n$ be a sequence in $C^\infty(\mathbb{S}^1, \mathbb{R}^m)$ approximating $u$ in $H^\frac{1}{2}(\mathbb{S}^1)$. We denote $(\tilde{u}_n)_n$ the sequence of the corresponding harmonic extensions. Then, by Lemma \ref{lem_continuity_harmonic_extension_operator}, $\tilde{u}_n$ converges to $\tilde{u}$ in $H^1(D^2)$ and, up to a subsequence, a.e. as $n\to \infty$.
Now we claim that for any smooth vector field $\tilde{X}$ on $\overline{D^2}$ such that $\tilde{X}(z)\cdot z=0$ for $z\in\partial D^2$, there holds
\begin{equation}
\frac{d}{dt}\bigg\vert_{t=0}\int_{D^2}\left\lvert \nabla\tilde{u}(z+t\tilde{X}(z))\right\rvert^2 dz=0.
\end{equation}
Let $n\in\mathbb{N}$ and let $\tilde{X}$ be a smooth vector field on $\overline{D^2}$ such that $\tilde{X}(z)\cdot z=0$ for $z\in\partial D^2$. For all $z\in\partial D^2$, let $Y_n(z):=d\tilde{u}_n(z)\cdot \tilde{X}(z)=du_n(z)\cdot\tilde{X}(z)$, and since $\tilde{u}_n$ is harmonic in $D^2$,
\begin{equation}
\int_{D^2} \nabla \tilde{u}_n(z)\cdot\nabla Y_n(z)-\int_{D^2}\text{div}(\nabla \tilde{u}_n\cdot Y_n)(z)dz=-\int_{D^2}\Delta \tilde{u}(z)Y(z)dz=0
\end{equation}
by Gauss' Theorem. Therefore
\begin{equation}
\begin{split}
\frac{d}{dt}\bigg\vert_{t=0}\int_{D^2}\lvert \nabla\tilde{u}_n(z+t\tilde{X}(z))\rvert^2 dz=&\int_{D^2}\nabla\tilde{u}_n(z)\cdot \nabla Y_n(z)dz\\
=&\int_{D^2}\text{div}(\nabla \tilde{u}_n\cdot Y_n)(z)dz\\
=&\int_{\partial D^2}\partial_r \tilde{u}_n(x)\cdot Y_n(x)dx\\
=&\int_{\mathbb{S}^1}(-\Delta)^\frac{1}{2}u_n (x) u_n'(x) \tilde{X}(x)dx.
\end{split}
\end{equation}
By (\ref{eq_proof_lemma_product_at_the_boundary}), the last expression converges to $0$ as $n\to \infty$. By Lemma \ref{lem_continuity_derivative_of_energy_on_D2},
\begin{equation}
\frac{d}{dt}\bigg\vert_{t=0}\int_{D^2}\lvert \nabla\tilde{u}(z+t\tilde{X}(z))\rvert^2 dz=\lim_{n\to \infty}\frac{d}{dt}\bigg\vert_{t=0}\int_{D^2}\lvert \nabla\tilde{u}_n(z+t\tilde{X}(z))\rvert^2 dz=0.
\end{equation}
This completes the proof of the claim. The Proposition now follows from Proposition \ref{prop_variational_criterion_conformality}.
\end{proof}

\begin{lem}\label{lem_continuity_derivative_of_energy_on_D2}
Let $u\in H^1(D^2,\mathbb{R}^m)$ and let $(u_n)_n$ be a sequence in $C^\infty(D^2,\mathbb{R}^m)$ approximating $u$ in $H^1(D^2)$ and a.e..
Let $X$ be a smooth vector field on $\overline{D^2}$ such that $X(x)\cdot x=0$ for $x\in\partial D^2$ and let $\phi_t$ denote its flow. Then
\begin{equation}
\lim_{n\to\infty}\frac{d}{dt}\bigg\vert_{t=0}\int_{D^2} \lvert \nabla u_n\circ\phi_t\rvert^2=\frac{d}{dt}\bigg\vert_{t=0}\int_{D^2} \lvert \nabla u\circ\phi_t\rvert^2.
\end{equation}
\end{lem}
\begin{proof}
Let $n\in \mathbb{N}$ and let $X$ and $\phi_t$ be like in the assumptions of the Lemma.
We compute, for any $t\in\mathbb{R}$,
\begin{equation}\label{eq_proof_lem_continuity_derivative_of_energy_on_D2_first_estimate}
\begin{split}
&\lvert E(u_n\circ\phi_t)-E(u_n)-(E(u\circ \phi_t)-E(u))\rvert\\
=&\left\lvert\int_{D^2}\lvert d u_n(\phi_t(x)) d\phi_t(x)\rvert^2-\lvert d u_n(x)\rvert^2-\left(\lvert d u(\phi_t(x))d\phi_t(x)\rvert^2-\lvert du(x)\rvert^2\right)dx\right\rvert\\
=&\left\lvert\int_{D^2}\left(\lvert d u_n(x')d\phi_t(\phi_{-t}(x'))\rvert^2-\lvert d u_n(x')d\phi_t(\phi_{-t}(x'))\rvert^2\right)\lvert \det ( d_{x'}\phi_{-t}(x'))\rvert dx'\right.\\
&\left.-\int_{D^2}\lvert d u_n(x)\rvert^2-\lvert d u(x)\rvert^2dx\right\rvert\\
\leq& \int_{D^2} \left\lvert\lvert d u_n(x)d\phi_t(\phi_{-t}(x))\rvert^2-\lvert d u(x)d\phi_t(\phi_{-t}(x))\rvert^2\right\rvert\left\lvert\lvert \det(d\phi_{-t}(x))\rvert-1\right\rvert dx\\
&+\int_{D^2}\left\lvert\lvert d u_n(x)d\phi_t(\phi_{-t}(x))\rvert^2-\lvert d u(x)d\phi_t(\phi_{-t}(x))\rvert^2\right.\\
&\phantom{+\int_{D^2}}\left.-\left(\lvert d u_n(x)\rvert^2-\lvert d u(x)\rvert^2\right)\right\rvert dx.
\end{split}
\end{equation}
Now we observe that for any fixed $0<\delta\leq 1$ for $\lvert t\rvert\leq \delta$, by the mean value theorem, for any $x\in D^2$
\begin{equation}
\lvert\lvert \det(d\phi_{-t}(x))\rvert-1\rvert\leq \lvert t\rvert\sup_{\lvert t\rvert\leq \delta, x'\in D^2}\left\lvert\partial_t \det(\phi_{-t}(x'))\right\rvert,
\end{equation}
and
\begin{equation}
\lvert d\phi_t(\phi_{-t}(x))-Id(x)\rvert\leq \lvert t\rvert\sup_{\lvert t\rvert\leq \delta, x'\in \mathbb{S}^1} \left\lvert\partial_t d\phi_t(\phi_{-t}(x'))\right\rvert,
\end{equation}
where $Id(x)$ denotes the identity map from $T_x\mathbb{S}^1$ to itself. We set
\begin{equation}
C_1:=\sup_{\lvert t\rvert\leq \delta, x\in D^2} \left\lvert\partial_t \det(\phi_{-t}(x))\right\rvert,\quad C_2:=\sup_{\lvert t\rvert\leq \delta, x\in D^2}\left\lvert\partial_t d\phi_t(\phi_{-t}(x))\right\rvert.
\end{equation}
Moreover, for any $x\in D^2$
\begin{equation}
\begin{split}
&\lvert d u_n(x)d\phi_t(\phi_{-t}(x))\rvert^2-\lvert d u(x)d\phi_t(\phi_{-t}(x))\rvert^2-(\lvert d u_n(x)\rvert^2-\lvert d u(x)\rvert^2)\\
=&du_n(x)(d\phi_t(\phi_{-t}(x))+Id(x))\cdot du_n(x)(d\phi_t(\phi_{-t}(x))-Id(x))\\
&-du(x)(d\phi_t(\phi_{-t}(x))+Id(x))\cdot du(x)(d\phi_t(\phi_{-t}(x))-Id(x))\\
=&(du_n(x)-du(x))(d\phi_t(\phi_{-t}(x))+Id(x))\cdot du_n(x)(d\phi_t(\phi_{-t}(x))-Id(x))\\
&+du(x)(d\phi_t(\phi_{-t}(x))+Id(x))\cdot(du_n(x)-du(x))(d\phi_t(\phi_{-t}(x))-Id(x)),
\end{split}
\end{equation}
Now let $\varepsilon>0$. If $n$ is so large that $\lVert u-u_n\rVert_{H^1}\leq \varepsilon$, if $\lvert t\rvert\leq \delta$, we obtain
\begin{equation}
\begin{split}
&\int_{D^2}\left\lvert\lvert d u_n (x)d\phi_t(\phi_{-t}(x))\rvert^2-\lvert d u(x)d\phi_t(\phi_{-t}(x))\rvert^2-(\lvert d u_n(x)\rvert^2-\lvert d u(x)\rvert^2)\right\rvert dx\\
\leq &\int_{D^2}\lvert d u_n(x) -d u(x)\rvert(\lvert t\rvert C_2+2)\rvert \lvert d u_n(x)\rvert C_2\\
&+\lvert du(x)\rvert(\lvert t\rvert C_2+2)\lvert d u_n(x) -d u(x)\rvert\lvert t\rvert C_2dx\\
\leq & 2\lVert du_n-du\rVert_{L^2}(\lvert t\rvert C_2+2)(\lVert du\rVert_{L^2}+\varepsilon)\lvert t\rvert C_2\leq \lvert t\rvert C_3\lVert u-u_n\rVert_{H^1},
\end{split}
\end{equation}
where $C_3:=2(C_2+2)C_2(\lVert u\rVert_{H^1}^2+\varepsilon)$.\\
Similarly, if $n$ is large enough, if $\lvert t\rvert\leq \delta$ we obtain
\begin{equation}
\begin{split}
&\int_{D^2}\left\lvert \lvert d u_n (x)d\phi_t(\phi_{-t}(x))\rvert^2-\lvert du(x)d\phi_t(\phi_t(x))\rvert^2\right\rvert dx\\
=&\int_{D^2}\lvert(d u_n(x)+d u(x))d\phi_t(\phi_{-t}(x))\cdot(d u_n(x)-du(x))d\phi_t(\phi_{-t}(x))\rvert dx\\
\leq&\lVert d u_n+du\rVert_{L^2}(\lvert t\rvert C_\phi^2+2)\lVert du_n-du\rVert_{L^2}(\lvert t\rvert C_2+2)\leq C_4\lVert u-u_n\rVert_{H^1},
\end{split}
\end{equation}
where $C_4:=(2\lVert u\rVert_{H^1}+\varepsilon)(C_2+2)$, and therefore
\begin{equation}
\begin{split}
&\int_{D^2} \left\lvert d u_n(x)d\phi_t(\phi_{-t}(x))\rvert^2-\lvert d u(x)d\phi_t(\phi_{-t}(x))\rvert^2\right\rvert\lvert\lvert \det(d\phi_{-t}(x))\rvert-1\rvert dx\\
\leq &C_\phi^1C_\phi^4\lvert t \rvert\lVert u-u_n\rVert_{H^1}.
\end{split}
\end{equation}
Thus if $n$ is large enough and $\lvert t\rvert\leq \delta$ there holds
\begin{equation}\label{eq_proof_lemma_estimate_continuity of derivative_D2}
\lvert E_{D^2}(u_n\circ\phi_t)-E_{D^2}(u_n)-(E_{D^2}(u\circ \phi_t)-E_{D^2}(u))\rvert\leq C\lvert t\rvert \lVert u-u_n\rVert_{H^1}.
\end{equation}
for some constant $C$ depending on $X$ and $u$.
Starting from equation (\ref{eq_proof_lemma_estimate_continuity of derivative_D2}) and arguing as in the last part of the proof of Lemma \ref{lem_continuity_derivative_energy}, we conclude that
the map $t\mapsto E_{D^2}(u\circ\phi_t)$ is differentiable in in $0$, and
\begin{equation}
\frac{d}{dt}\bigg\vert_{t=0}E(u_n\circ\phi_t)\to \frac{d}{dt}\bigg\vert_{t=0}E(u\circ\phi_t)
\end{equation}
as $t\to 0$. This concludes the proof of the Lemma.
\end{proof}

\begin{prop}\label{prop_variational_criterion_conformality}(Proposition V.3 in \cite{ConfInv})
Let $\tilde{u}$ in $H^1(D^2,\mathbb{R}^m)$. Then $\tilde{u}$ is conformal in $D^2$ if and only if $\tilde{u}$ satisfies
\begin{equation}
\frac{d}{dt}\bigg\vert_{t=0}\int_{D^2} \lvert \nabla \tilde{u}_t(x)\rvert^2dx=0,
\end{equation}
where $\tilde{u}_t(x):=\tilde{u}(x+tX(x))$ for any $x\in D^2$, for every $X\in C^\infty(\overline{D}^2,\mathbb{R}^2)$ such that $X(z)\cdot z=0$ for $z\in \partial D^2$.
\end{prop}

\begin{rem}
We recall that if $u\in H^\frac{1}{2}(\mathbb{S}^1, \mathscr{M})$, for some smooth closed manifold $\mathscr{M}$, is a critical point of $E$ in $H^\frac{1}{2}(\mathbb{S}^1, \mathscr{M})$, i.e. if $u$ is an $\frac{1}{2}$-harmonic map in $H^\frac{1}{2}(\mathbb{S}^1, \mathscr{M})$, then $u\in C^\infty(\mathbb{S}^1)$ (see \cite{FreeBoundary}, Appendix E). Therefore any such $u$ is a fortiori also a stationary point of $E$, and all results of this section apply to it.

\end{rem}

In order to simplify the characterization of stationary points $u$ of the energy (\ref{eq_intro_ssec_half_Dirichlet_energy}) in terms of the distribution $A_u$ introduced in Lemma \ref{lem_distribution_description_for_derivative}, we now introduce a fractional analogous of the Hopf differential.
%
%
%
%

\begin{defn}\label{defn_fractional_Hopf}
For $u\in H^\frac{1}{2}(\mathbb{S}^1)$ let $v:=(-\Delta)^\frac{1}{2}u\in H^{-\frac{1}{2}}(\mathbb{S}^1)$ and let $v_+$ the distribution whose Fourier coefficients are given by
\begin{equation}\label{eq_def_fractional_Hopf_Fourier_coeffs_of_v_+}
\widehat{v_+}(k)=\begin{cases}
\widehat{v}(k)&\text{ if }k\geq 1\\
0&\text{ if }k\leq 0
\end{cases}
\end{equation}
We define the \textbf{$\frac{1}{2}$-fractional differential} of $u$ to be
\begin{equation}
\mathscr{H}_\frac{1}{2}(u):=v_+\cdot_{\mathbb{C}^m} v_+,
\end{equation}
where $\cdot_{\mathbb{C}^m}$ denotes the $\mathbb{C}$-scalar product in $\mathbb{C}^m$, corresponding to the convolution of the Fourier coefficients.
\end{defn}
In the following Lemma we prove that $\mathscr{H}_\frac{1}{2}(u)$ is a well defined distribution.
\begin{lem}\label{lem_fractional_Hopf_differential_well_def}
For any $u\in H^\frac{1}{2}(\mathbb{S}^1)$, $\mathscr{H}_\frac{1}{2}(u)$ is a well defined element of $\dot{H}^{-3}(\mathbb{S}^1)$, and the map
\begin{equation}
\mathscr{H}_\frac{1}{2}: H^\frac{1}{2}(\mathbb{S}^1)\to H^{-3}(\mathbb{S}^1),\quad u\mapsto \mathscr{H}_\frac{1}{2}(u)
\end{equation}
is continuous.
\end{lem}
\begin{proof}
For any $u\in H^\frac{1}{2}(\mathbb{S}^1)$, let $v:=(-\Delta)^\frac{1}{2}u$ and let $v_+$ be the distribution defined in (\ref{eq_def_fractional_Hopf_Fourier_coeffs_of_v_+}).
First we observe that for any $k\in \mathbb{Z}$
\begin{equation}
\lvert \widehat{v}(k)\rvert^2=\lvert k\rvert^2\lvert \widehat{u}(k)\rvert^2\leq \lvert k\rvert^2 [u]_{\dot{H}^\frac{1}{2}}^2\lvert k\rvert^{-1}=[u]_{\dot{H}^\frac{1}{2}}^2\lvert k\rvert.
\end{equation}
Therefore for any $k\in \mathbb{Z}_{>0}$
\begin{equation}
\begin{split}
\left\lvert \mathscr{F}(v_+\cdot_{\mathbb{C}^m} v_+)(k)\right\rvert=&\left\lvert\sum_{\substack{ a,b\in\mathbb{N}_{>0}\\ a+b=k}}\widehat{v}(a)\cdot_{\mathbb{C}^m} \widehat{v}(b)\right\rvert\leq
\left\lvert\sum_{\substack{ a,b\in\mathbb{N}_{>0}\\ a+b=k}}[u]_{\dot{H}^\frac{1}{2}}^2\lvert m\rvert^\frac{1}{2}\lvert n\rvert^\frac{1}{2}\right\rvert\\
\leq & k[u]_{\dot{H}^\frac{1}{2}}^2 k^\frac{1}{2}k^\frac{1}{2}\leq [u]_{\dot{H}^\frac{1}{2}}^2 k^2
\end{split}
\end{equation}
We conclude that
\begin{equation}
\begin{split}
\left\lVert \mathscr{H}_\frac{1}{2}(u)\right\rVert_{H^{-3}}^2=&\sum_{k\in\mathbb{Z}_{>0}}\left\lvert \mathscr{F}\left(v_+\cdot v_+\right)(k)\right\rvert^2 (1+\lvert k\rvert^2)^{-3}\leq\sum_{k\in\mathbb{Z}_{>0}}[u]_{\dot{H}^\frac{1}{2}}^4 k^4 (1+\lvert k\rvert^2)^{-3}\\
\leq &[u]_{\dot{H}^\frac{1}{2}}^4\sum_{k\in\mathbb{Z}_{>0}}k^{-2}=\frac{\pi^2}{3}[u]_{\dot{H}^\frac{1}{2}}^4.
\end{split}
\end{equation}
\end{proof}

To see the analogy between the Hopf differential $\mathscr{H}$ and the $\frac{1}{2}$-fractional Hopf differential $\mathscr{H}_\frac{1}{2}$ consider $u\in C^\infty(\mathbb{S}^1)$. In this case, $\mathscr{H}(u)$ (as a function) extends continuously to $\overline{D^2}$ and $\mathscr{H}_\frac{1}{2}(u)$ is a smooth complex-valued function on $\mathbb{S}^1$. We compute
\begin{equation}
\begin{split}
Re\left(\mathscr{H}_\frac{1}{2}(u)\right)=&\frac{1}{4}\left[\lvert v_++\overline{v_+}\rvert^2-\lvert-i(v_+-\overline{v_+})\rvert^2\right]=\frac{1}{4}\left[\lvert v\rvert^2-\lvert H v\rvert^2\right]\\
=&\frac{1}{4}\left[\lvert (-\Delta)^\frac{1}{2}u\rvert^2-\lvert u'\rvert^2\right]=\frac{1}{4}\left[ \lvert\partial_r \tilde{u}\rvert^2-\lvert\partial_\theta \tilde{u}\rvert^2 \right]\bigg\vert_{\partial D^2}.
\end{split}
\end{equation}
and
\begin{equation}
\begin{split}
Im(\mathscr{H}_\frac{1}{2}(u))=&\frac{1}{2i}(v_+\cdot{\mathbb{C}^m} v_+-\overline{v_+}\cdot_{\mathbb{C}^m} \overline{v_+})=\frac{1}{2i}(v_++\overline{v_+})\cdot_{\mathbb{C}^m}(v_+-\overline{v_+})\\
=&\frac{1}{2}v\cdot H v
=\frac{1}{2}(-\Delta)^\frac{1}{2}u\cdot u'=-\frac{1}{2}\partial_r\tilde{u}\cdot \partial_\theta\tilde{u}\bigg\vert_{\partial D^2},
\end{split}
\end{equation}
where "$\cdot$" denotes the $\mathbb{R}$-scalar product in $\mathbb{R}^{m}$.
Therefore, in this case, $\mathscr{H}_\frac{1}{2}(u)$ coincides -up to multiplication by $-\overline{z}$- with the restriction of the Hopf differential $\mathscr{H}(\tilde{u})$ to $\partial D^2$, where $\tilde{u}$ is the harmonic extension of $u$.\\
Next we show how the $\frac{1}{2}$-fractional Hopf differential can be used to characterize stationary points of the energy (\ref{eq_intro_ssec_half_Dirichlet_energy}) in a simple way, and thus to reformulate in a simpler way the proof of the fact that for any $u\in H^\frac{1}{2}(\mathbb{S}^1,\mathbb{R}^m)$, $u$ is a stationary point of the energy (\ref{eq_intro_ssec_half_Dirichlet_energy}) if and only if its harmonic extension in $D^2$ is conformal.
A key observation for the proof of the following result is that $\mathscr{H}_\frac{1}{2}(u)$ consists only of positive frequencies;
therefore $Re[{\mathscr{H}_\frac{1}{2}(u)}]$ vanishes if and only if $Im[{\mathscr{H}_\frac{1}{2}(u)}]$ vanishes.

\begin{thm}\label{lem_equivalent_conditions_vanishing_H(u)}
Let $u\in H^\frac{1}{2}(\mathbb{S}^1)$ and let $\tilde{u}$ be its harmonic extension.
Then the following statements are equivalent:
\begin{enumerate}
\item  $\frac{1}{r}\lvert\partial_\theta \tilde{u}\rvert=\lvert \partial_r\tilde{u}\rvert$ in $D^2$,
\item $\partial_\theta \tilde{u}\cdot \partial_r \tilde{u}=0$ in $D^2$,
\item $\mathscr{H}_\frac{1}{2}(u)=0$,
\item $u$ is a stationary point of $E$.
\end{enumerate}
In particular, $u$ is a stationary point of the energy (\ref{eq_intro_ssec_half_Dirichlet_energy}) if ad only if $\tilde{u}$ is conformal in $D^2$.
\end{thm}
\begin{proof}
First we show that $1.$ is equivalent to $3.$, let's denote $\tilde{u}$ the harmonic extension of $u$, and for any $r\in (0,1)$ let $u_r(\theta):=\tilde{u}(r,\theta)$ for any $\theta\in \mathbb{S}^1$. Also, let $v_r:=(-\Delta)^\frac{1}{2}u_r$ and ${v_r}_+:=\sum_{k\in\mathbb{Z}_{>0}}\widehat{v_r}(k)e^{ik\cdot}$. Then for any $r\in (0,1)$
\begin{equation}
\begin{split}
&\frac{1}{r^2}\lvert\partial_\theta \tilde{u}(\cdot, r)\rvert^2-\lvert \partial_r\tilde{u}(\cdot,r)\rvert^2=\frac{1}{r^2}\left(\lvert u_r'\rvert^2-\lvert (-\Delta)^\frac{1}{2}u_r\rvert^2\right)\\
=&\frac{1}{r^2}\left(\lvert{v_r}_++\overline{{v_r}_+}\rvert^2-\lvert {v_r}_+-\overline{{v_r}_+}\rvert^2\right)=\frac{4}{r^2}Re({v_r}_+\cdot_{\mathbb{C}^m}{v_r}_+).
\end{split}
\end{equation}
Now we observe that for any $k\in \mathbb{Z}$ $\widehat{v_r}(k)=\lvert k\rvert\widehat{u_r}(k)=\lvert k\rvert r^k\widehat{u}(k)$, therefore for any $k\in\mathbb{Z}$
\begin{equation}\label{eq_proof_lemma_Fouriercoeff_product_positive_parts}
\mathscr{F}\left({v_r}_+\cdot_{\mathbb{C}^m}{v_r}_+\right)(k)=\sum_{\substack{a,b\in\mathbb{Z}_{>0}\\ a+b=k}}r^kab\widehat{u}(a)\cdot_{\mathbb{C}^m}\widehat{u}(b)=r^k\mathscr{F}\left(\mathscr{H}_\frac{1}{2}(u)\right)(k).
\end{equation}
As ${v_r}_+\cdot_{\mathbb{C}^m}{v_r}_+$ consists only of positive coefficients, its real part vanishes if and only if all its Fourier coefficients vanish. Therefore $\frac{1}{r}\lvert\partial_\theta \tilde{u}\rvert=\lvert \partial_r\tilde{u}\rvert$ in $D^2$ if and only if $\mathscr{H}_\frac{1}{2}(u)=0$.
Now we show that $2.$ is equivalent to $3.$.
For any $r\in (0,1)$,
\begin{equation}
\begin{split}
\partial_\theta \tilde{u}(\cdot, r)\cdot \partial_r \tilde{u}(\cdot,r)=&\frac{1}{r}u_r'\cdot (-\Delta)^\frac{1}{2}u_r\\
=&\frac{1}{r}Hv_r\cdot v_r=- \frac{i}{r}({v_r}_+-\overline{v_r}_+)\cdot ({v_r}_++\overline{{v_r}_+})\\=&-\frac{i}{r} ({v_r}_+\cdot_{\mathbb{C}^m} {v_r}_+-\overline{{v_r}_+}\cdot_{\mathbb{C}^m}\overline{{v_r}_+})=2 Im({v_r}_+\cdot_{\mathbb{C}^m} {v_r}_+).
\end{split}
\end{equation}
Again, since ${v_r}_+\cdot_{\mathbb{C}^m} {v_r}_+$ consists only of positive frequencies, its imaginary part vanishes if and only if all its Fourier coefficients vanish. Thus, by (\ref{eq_proof_lemma_Fouriercoeff_product_positive_parts}), $\partial_\theta \tilde{u}\cdot \partial_r \tilde{u}=0$ in $D^2$ if and only if $\mathscr{H}_\frac{1}{2}(u)=0$, therefore $2.$ is equivalent to $4.$. This concludes the proof of the Lemma.

Finally, by Lemma \ref{lem_harmonic_extension_conformal_first_condition}, $u$ is a stationary point of $E$ if and only if (\ref{eq_lemma_orthogonality_of_the_derivatives}) holds.
\end{proof}

\begin{cor}\label{cor_characterization_conformal_functions_D2}
Let $u\in H^1(D^2)$ so that
\begin{equation}
\Delta u=0\text{ in the weak sense in }D^2.
\end{equation}
Let $\overline{u}:=Tr (u)$.
Then the following conditions are equivalent:
\begin{enumerate}
\item  $\frac{1}{r}\lvert\partial_\theta u\rvert=\lvert \partial_ru\rvert$ in $D^2$,
\item $\partial_\theta u\cdot \partial_r u=0$ in $D^2$,
\item $\mathscr{H}_\frac{1}{2}(\overline{u})=0$,
\item $\overline{u}$ is a stationary point of $E$.
\end{enumerate}
In particular, if any of these conditions hold, then $u$ is conformal in $D^2$.
\end{cor}
\begin{proof}
By Lemmas \ref{lem_trace_operator} and \ref{lem_continuity_harmonic_extension_operator}, $u$ is the harmonic extension of $\overline{u}$. Therefore the result follows directly from Theorem \ref{lem_equivalent_conditions_vanishing_H(u)}.
\end{proof}

\section{Pohozaev identities}\label{sec:Pohozaev identities}
In this section, we aim to give an analogous formula to the Pohozaev identity for functions $u\in H^\frac{1}{2}(\mathbb{R})$ satisfying
\begin{equation}\label{eq_intro_Poh_nonlocal_Cauchy}
\begin{cases}
(-\Delta)^\frac{1}{2} u=f(u)&\text{ weakly in }\Omega\\
u=0&\text{ in }\Omega^c.
\end{cases}
\end{equation}
Here, $f:\mathbb{R}\to\mathbb{R}$ is a function in $C^{1,\alpha}(\mathbb{R})$ for some $\alpha>0$, and $\Omega\subset \mathbb{R}$ is an open bounded interval.\\
First we recall the analogous local Cauchy problem in two dimension:
\begin{equation}\label{eq_intro_Poh_local_Cauchy}
\begin{cases}
\Delta u=f(u)&\text{ weakly in }\Omega\\
u=0&\text{ on }\partial \Omega.
\end{cases},
\end{equation}
where $f\in C^1(\mathbb{R},\mathbb{R})$ and $\Omega$ is a bounded domain in $\mathbb{R}^2$ with $C^1$ boundary.
If $u\in C^2(\Omega)\cap C^1(\overline{\Omega})$ satisfies (\ref{eq_intro_Poh_local_Cauchy}), then the following Pohozaev identity holds:
\begin{equation}\label{eq_Pohozaev_local}
\int_{\Omega}x\cdot\nabla F(u)(x)dx=\int_{\partial\Omega}x\cdot\nabla u(x)\partial_{\nu} u(x)dx-\int_{\partial\Omega} x\cdot\nu(x)\frac{\lvert \nabla u(x)\rvert^2}{2}dx,
\end{equation}
where $F$ is the primitive function of $f$ such that $F(0)=0$, $\nu(x)$ is the the unit outward normal to $\partial\Omega$ at $x$ for any $x\in\partial\Omega$, and $\partial_\nu$ denotes the derivative in direction $\nu(x)$.
For convex domains $\Omega$ containing the origin, the Pohozaev identity (\ref{eq_Pohozaev_local}) can be deduce from the conformal invariance of the Dirichlet energy as follows. First we observe that by the divergence theorem
\begin{equation}
\begin{split}
\int_\Omega x\cdot \nabla F(u)(x)dx=&\int_\Omega x\cdot\nabla u(x)\Delta u(x)dx\\=&\frac{d}{d\lambda}\bigg\vert_{\lambda=1^-}\int_{\Omega} u_\lambda(x)\Delta u(x)dx\\
=&\frac{d}{d\lambda}\bigg\vert_{\lambda=1^-}\left[\int_{\partial\Omega}u_\lambda(x)\nabla u(x)\cdot\nu(x)dx-\int_\Omega \nabla u_\lambda(x)\cdot\nabla u(x)dx\right],
\end{split}
\end{equation}
where the derivatives has to be interpreted as a left derivatives.
Next we observe that by Leibnitz rule
\begin{equation}
\frac{d}{d\lambda}\bigg\vert_{\lambda=1^-}\int_\Omega \nabla u_\lambda(x)\cdot\nabla u(x)dx=\frac{1}{2}\frac{d}{d\lambda}\bigg\vert_{\lambda=1^-}\int_{\Omega}\lvert \nabla u_\lambda(x)\rvert^2dx,
\end{equation}
and since the Dirichlet energy is conformally invariant, for any $\lambda>0$
\begin{equation}
\int_{\frac{1}{\lambda}\Omega}\lvert \nabla u_\lambda(x)\rvert^2dx=\int_{\Omega}\lvert \nabla u(x)\rvert^2dx,
\end{equation}
and thus
\begin{equation}
\begin{split}
\frac{d}{d\lambda}\bigg\vert_{\lambda=1^-}\int_\Omega \lvert \nabla u_\lambda(x)\rvert^2dx=&-\lim_{\lambda\to 1^-}\int_{\frac{1}{\lambda}\Omega\smallsetminus \Omega}\frac{\lvert \nabla u_\lambda(x)\rvert^2}{\lambda-1}dx\\
=&\int_{\partial \Omega}x\cdot\nu(x)\lvert \nabla u(x)\rvert^2dx.
\end{split}
\end{equation}
Therefore
\begin{equation}
\int_\Omega x\cdot \nabla F(u)(x)dx=\int_{\partial \Omega} x\cdot \nabla u(x) \partial_\nu u(x)dx-\frac{1}{2}\int_{\partial \Omega}x\cdot\nu(x)\lvert \nabla u(x)\rvert^2dx.
\end{equation}
This concludes the proof of identity (\ref{eq_Pohozaev_local}).
In analogy to the local case and to the Pohozaev identities (\ref{eq_prop_Pohozaev_circle_dilation_integral_form}) and (\ref{eq_prop_conclusion_Pohozaev_identity_for_the_whole_R}), we would like to find an explicit expression for
\begin{equation}\label{eq_discussion_intro_Pohozaev_integral_to_be_evaluated_Pohozaev_interval}
\int_{\Omega} x(-\Delta)^\frac{1}{2}u(x)u'(x)dx,
\end{equation}
i.e. the integral of the product of the "generator of dilations in the domain" $\frac{d}{d\lambda}\big\vert_{\lambda=1}u(\lambda x)=u'(x)x$ with the left hand side of Equation (\ref{eq_intro_Poh_nonlocal_Cauchy}), for functions $u$ solving (\ref{eq_intro_Poh_nonlocal_Cauchy}). To this end we will make use of the invariance of the half Dirichlet energy under conformal variations of the domain, in particular under dilations. In fact, since the half Dirichlet energy satisfies (\ref{eq_prop_condition_symmetry_14_real_line}) with $X(x)=x$ for any $x\in \mathbb{R}$, we have
\begin{equation}
\begin{split}
0=&\frac{d}{d\lambda}\bigg\vert_{\lambda=1}\int_\mathbb{R}\left\lvert (-\Delta)^\frac{1}{4}u_\lambda(x)\right\rvert^2dx=2\frac{d}{d\lambda}\int_{\mathbb{R}}(-\Delta)^\frac{1}{4}u_\lambda(x)(-\Delta)^\frac{1}{4}u(x)dx\\
=&2\frac{d}{d\lambda}\int_\mathbb{R}u_\lambda(x)(-\Delta)^\frac{1}{2}u(x)dx
\end{split}
\end{equation}
if we assume that we can invert the order of differentiation and integration, and therefore
\begin{equation}\label{eq_discussion_heuristic_proof_Pohozaev_2}
\begin{split}
\int_\Omega xu'(x)(-\Delta)^\frac{1}{2}u(x)dx=&-\frac{d}{d\lambda}\bigg\vert_{\lambda=1}\int_{\Omega^c} u_\lambda (x)(-\Delta)^\frac{1}{2}u(x)dx\\=&-\frac{d}{d\lambda}\bigg\vert_{\lambda=1}\int_{( \frac{a}{\lambda}, \frac{b}{\lambda})\smallsetminus (a,b)} u_\lambda(x)(-\Delta)^\frac{1}{2}u(x)dx,
\end{split}
\end{equation}
where we denoted $a$ and $b$ the boundary points of $\Omega$ and where we assumed again that is possible to invert the order of differentiation and integration\footnote{We will see in the proof of Theorem \ref{prop_Pohozaev_identity_for_open_bounded_intervals_final} that it is possible to follow the idea outlined above for the left derivative in $\lambda$. In this case the first Equality in (\ref{eq_discussion_heuristic_proof_Pohozaev_2}) holds up to a multiplicative constant.}.
Thus we see that the quantity in (\ref{eq_discussion_intro_Pohozaev_integral_to_be_evaluated_Pohozaev_interval}) only depends on the values of $u$ and $(-\Delta)^\frac{1}{2}u$ in an arbitrary small neighbourhood of the boundary points $a$ and $b$. For this reason, in the following we will try to compute the asymptotics of $u(x)$ and $(-\Delta)^\frac{1}{2}u(x)$ for $x$ approaching the boundary points $a$ and $b$.\\
A similar argument can be used to show that the quantity
\begin{equation}
\int_\Omega (-\Delta)^\frac{1}{2}u(x)u'(x)dx,
\end{equation}
which we can regard as the integral of the product of the "generator of translation" $\frac{d}{d\lambda}\big\vert_{\lambda=1}u(x+\lambda)=u'(x)$ with the left hand side of Equation (\ref{eq_intro_Poh_nonlocal_Cauchy}),
only depends on the values of $u$ and $(-\Delta)^\frac{1}{2}u$ in an arbitrary small neighbourhood of the boundary points $a$ and $b$. Both results discussed here will be proved in Theorem \ref{prop_Pohozaev_identity_for_open_bounded_intervals_final}.\\
We remark that in the previous arguments we didn't use directly the fact that $u$ satisfies Equation (\ref{eq_intro_Poh_nonlocal_Cauchy}). However, we will see that this assumption ensures that $u$ is regular enough to carry out in a rigorous way the arguments sketched above (see also Remark \ref{rem_regularity_Pohozaev}). 

Before entering into the details of the argument, we make some observations about solutions $u$ of the Cauchy problem (\ref{eq_thm_Poh_nonlocal_Cauchy_general}). For the following, let $n\in\mathbb{N}_{>0}$ and let $\Omega\subset \mathbb{R}^n$ be a bounded open subset. 
\begin{lem}\label{lemma_equivalence minimization}
Let $f:\mathbb{R}\to \mathbb{R}$ be a $C^1$ function. Let $p\geq 1$ such that
\begin{equation}
\frac{1}{2}-\frac{1}{p}\leq \frac{s}{n},
\end{equation}
and assume that for any $x\in \mathbb{R}$
\begin{equation}\label{eq_lemma_bound_on_f}
\lvert f'(x)\rvert\leq C(1+\lvert x\rvert^p)
\end{equation}
for some constant $C$.
Let $s\in (0,1)$. Then $u\in H^s(\mathbb{R}^n)$ solves 
\begin{equation}\label{eq_lemma_fractional_Dirichlet_problem}
\begin{cases}
(-\Delta)^s u&=f(u)\text{ weakly in }\Omega\\
u&=0\text{ in }\Omega^c
\end{cases}
\end{equation}
if and only if $u$ is a critical point of the functional
\begin{equation}
E(u)=\int_\Omega \frac{1}{2}\left\lvert(-\Delta)^\frac{s}{2}u(x)\right\rvert^2-F(u(x))dx,
\end{equation}
where $F$ is the primitive function of $f$ such that $F(0)=0$; i.e. for any $\phi\in C^\infty_0(\Omega)$,
\begin{equation}
\frac{d}{dt}\bigg\vert_{t=0}E(u+t\phi)=0.
\end{equation}
\end{lem}
\begin{proof}
Let $\phi\in C^\infty_0(\Omega)$. Then for any $t\in\mathbb{R}$, $x\in \Omega$,
\begin{equation}
F(u(x)+t\phi(x))=F(u)(x)+tf(u)(x)\phi(x)+t^2f'(u(x)+ts_{x, t}\phi(x))\phi^2(x),
\end{equation}
for some $s_{x, t}\in (0,1)$ depending on $x$ and $t$.
By Sobolev's Embedding, if $p$ is chosen as in the assumptions, $H^s(\mathbb{R}^n)\hookrightarrow L^p_{loc}(\mathbb{R}^n)$. Therefore, since $f$ satisfies (\ref{eq_lemma_bound_on_f}), there holds
\begin{equation}
\int_\Omega F(u(x)+t\phi(x))dx=\int_\Omega [F(u(x))+tf(u(x))\phi(x) ]dx+o(t).
\end{equation}
Thus
\begin{equation}
\begin{split}
E(u+t\phi)=\int_\Omega &\frac{1}{2}\left\lvert (-\Delta)^\frac{s}{2} u(x)\right\rvert^2+t(-\Delta)^\frac{s}{2}u(x)(-\Delta)^\frac{s}{2}\phi(x)+t^2\frac{1}{2}\left\lvert (-\Delta)^\frac{s}{2}\phi(x)\right\rvert^2\\
&-F(u(x))-tf(u(x))\phi(x) dx+o(t)
\end{split}
\end{equation}
This implies that
\begin{equation}\label{eq_proof_lemma_equivalence_critical_and_equation}
\frac{d}{dt}\bigg\vert_{t=0}E(u+t\phi)=\int_\Omega (-\Delta)^\frac{s}{2}u(x)(-\Delta)^\frac{s}{2}\phi(x)-f(u(x))\phi(x)dx.
\end{equation}
Now $u$ solves (\ref{eq_lemma_fractional_Dirichlet_problem}) in the sense of distributions if and only if the right hand side of (\ref{eq_proof_lemma_equivalence_critical_and_equation}) vanishes for any $\phi\in C^\infty_0(\Omega)$ and $u$ is a critical point of $E$ if and only if the left hand side of (\ref{eq_proof_lemma_equivalence_critical_and_equation}) vanishes for any $\phi\in C^\infty_0(\Omega)$. This concludes the proof of the Lemma.
\end{proof}

%
We also recall some regularity results for solutions of the Cauchy problem (\ref{eq_intro_Poh_nonlocal_Cauchy}). These will be useful in the proof of equation (\ref{eq_Poh_simplified}).

\begin{prop}\label{prop_interior_regularity}
(see the proof of Proposition 2.8 in \cite{Silvestre}) Let $s\in (0,1)$, let $u\in H^s\cap L^\infty(\mathbb{R}^n)$ be a weak solution of $(-\Delta)^s u=f(u)$ in $B_1(0)$, where $f\in C^{0,\alpha}(\mathbb{R}^n)$ for some $\alpha\in (0,1]$. Assume that  $\alpha+2s$ is not an integer. Then
\begin{equation}
\lVert u\rVert_{C^{\alpha+2s}(B_\frac{1}{2}(0))}\leq C(\lVert f\rVert_{C^\alpha(B_1(0))}+\lVert u\rVert_{L^\infty(\mathbb{R}^n)}),
\end{equation}
with a constant $C$ depending only on $\alpha$, $s$ and $n$.
\end{prop}

\begin{prop}\label{prop_bdry_regularity}
(Proposition 1.1 in \cite{RosOtonBdryreg}) Let $s\in (0,1)$, let $\Omega\subset \mathbb{R}^n$ be any bounded $C^2$ domain and let $u\in H^s(\mathbb{R}^n)$ be a weak solution to
\begin{equation}
\begin{cases}
(-\Delta)^s u=f &\text{ in the weak sense in }\Omega\\
u=0 &\text{ in }\Omega^c
\end{cases}
\end{equation}
with $f\in L^\infty(\Omega)$. Then $u\in C^s(\overline{\Omega})$ and
\begin{equation}
\lVert u\rVert_{C^s(\overline{\Omega})}\leq C\lVert f\rVert_{L^\infty(\Omega)},
\end{equation}
with a constant $C$ depending only on $s$ and $n$.
\end{prop}
If we assume that $f(u)\in L^\infty(\Omega)$ in (\ref{eq_intro_Poh_nonlocal_Cauchy}),  Proposition \ref{prop_bdry_regularity} implies that $u\in C^\frac{1}{2}(\overline{\Omega})$. Then we can apply Proposition \ref{prop_interior_regularity}:

\begin{lem}\label{lem_interior_regularity_sulution_Poh}
Let $u\in H^\frac{1}{2}(\mathbb{R}^n)$ be a solution of
\begin{equation}
\begin{cases}
(-\Delta)^s u=f(u)&\text{ weakly in }\Omega\\
u=0&\text{ in }\Omega^c,
\end{cases}
\end{equation}
where $\Omega\subset \mathbb{R}^n$ is an open bounded set. Assume that $f\in C^{1,\alpha}(\mathbb{R})$ for some $\alpha>0$ and that $f(u)\in L^\infty(\mathbb{R})$. Then for any open subset $\tilde{\Omega}$ compactly contained in $\Omega$, $u\in C^{2,\varepsilon}(\tilde{\Omega})$ for some $\varepsilon>0$.
\end{lem} 
\begin{proof}
As observed above, since $f(u)\in L^\infty$, $u\in C^{\frac{1}{2}}(\overline{\Omega})$ by Proposition \ref{prop_bdry_regularity}. Thus, since $f$ is Lipschitz, $f(u)\in C^{\frac{1}{2}}$. Then, by Proposition \ref{prop_interior_regularity}, $u\in C^{\frac{3}{2}}(\Omega_1)$ and therefore $f(u)\in C^{1+\beta}(\Omega)$, where $\beta=\min\lbrace\alpha,\frac{1}{2}\rbrace$ and $\Omega_1$ is an open subset compactly contained in $\Omega$. In particular, $u'\in C^\frac{1}{2}(\Omega_1)$ and satisfies $(-\Delta)^\frac{1}{2}u'=(f(u))'$ weakly, where $(f(u))'\in C^{0,\beta}(\Omega_1)$. Thus, by Proposition \ref{prop_interior_regularity}, $u'\in C^{1,\beta}(\Omega_2)$ for any open subset $\Omega_2$ compactly contained in $\Omega_1$. Now one observes that it is possible to choose $\Omega_1$, $\Omega_2$ such that $\tilde{\Omega}\subset\Omega_2$. In this way one obtain that $u\in C^1(\tilde{\Omega})$ and $u'\in C^{1,\beta}(\tilde{\Omega})$. This concludes the proof of the Lemma.
\end{proof}

\subsection{Asymptotics of $u$ and $(-\Delta)^\frac{1}{2}u$ around the boundary points}\label{ssec:Asymptotics}

In this section, we try to study the asymptotic of $u$ and $(-\Delta)^\frac{1}{2}u$ around the boundary points $a$ and $b$ for solutions $u\in H^\frac{1}{2}(\mathbb{R},\mathbb{R})$ of (\ref{eq_intro_Poh_nonlocal_Cauchy}). We will then use them to give a proof of the Pohozaev identity (\ref{eq_Poh_simplified}).
First we make some remarks about the regularity of $u$:

\begin{rem}
Let $u\in H^\frac{1}{2}(\mathbb{R},\mathbb{R})$ be a solution of (\ref{eq_intro_Poh_nonlocal_Cauchy}). By Proposition \ref{prop_bdry_regularity}, $u\in C^\frac{1}{2}(\overline{\Omega})$ and by Lemma \ref{lem_interior_regularity_sulution_Poh}, $u\in C^2(\tilde{\Omega})$ for any domain $\tilde{\Omega}$ compactly contained in $\Omega$.
Also, we claim that $(-\Delta)^\frac{1}{2}u\in L^1\cap L^{2,\infty}(\mathbb{R})$. Indeed $(-\Delta)^\frac{1}{2}u=f(u)$ in $\Omega$ and thus in particular $(-\Delta)^\frac{1}{2}u\in L^2(\Omega)$. Moreover, if $x\in \Omega^c$,
\begin{equation}\label{eq_proof_prop_estimate_derivative}
\begin{split}
\left\lvert(-\Delta)^\frac{1}{2}u(x)\right\rvert&=\left\lvert\frac{1}{\pi} \int_a^b\frac{u(y)}{\lvert x-y\rvert^2}dy\right\rvert\leq\frac{1}{\pi}\int_a^b \frac{[u]_{C^\frac{1}{2}}\lvert y-x\rvert^\frac{1}{2}}{\lvert x-y\rvert^2}dy\\
&\leq \frac{1}{\pi}[u]_{C^\frac{1}{2}}\int_a^b\frac{1}{\lvert y-x\rvert^\frac{3}{2}}dy=\frac{2}{\pi}[u]_{C^\frac{1}{2}}\left\lvert \frac{1}{\lvert b-x\rvert^\frac{1}{2}}-\frac{1}{\lvert a-x\rvert^\frac{1}{2}}\right\rvert,
\end{split}
\end{equation}
where we used the fact that $u\equiv0$ in $\Omega^c$. Now the expression on the right of (\ref{eq_proof_prop_estimate_derivative}) belongs to $L^1\cap L^{2,\infty}(\Omega^c)$. This   completes the proof of the claim.\\
Then $(-\Delta)^\frac{1}{2}u\in L^1\cap L^{2-\varepsilon}(\mathbb{R})$ for any $\varepsilon>0$, and therefore $u'=H(-\Delta)^\frac{1}{2}u\in L^s(\mathbb{R})$ for any $s\in (1,2)$, since the Hilbert transform $H$ is strong $(p,p)$ for $p\in (1,\infty)$. Since $u$ is supported in $\overline{\Omega}$, this implies in particular that $u'\in L^1(\mathbb{R})$.
\end{rem}
In the following we will say that a function $u$  belongs to $C^2([a,b])$
if $(u^2)''$ is well defined on $(a,b)$ and extends to a bounded continuous function on $[a,b]$. Then also $u^2$ and $(u^2)'$ extends to bounded continuous functions on $[a,b]$, and we also denote the extensions by $u^2$, $(u^2)'$ and $(u^2)''$. The previous results show that if $f\in C^{1,\alpha}(\mathbb{R})$ for some $\alpha>0$, any solution $u$ of (\ref{eq_intro_Poh_nonlocal_Cauchy}) is in fact of class $C^2$ inside $(a,b)$. Therefore, if $u\in H^\frac{1}{2}(\mathbb{R},\mathbb{R})$ satisfies Equation (\ref{eq_intro_Poh_nonlocal_Cauchy}), the assumption $u\in C^2([a,b])$ is only an assumption on the regularity of $u$ around the boundary points $a$ and $b$.

\begin{lem}\label{lem_asymptotic_u}
Assume that $u\in H^\frac{1}{2}(\mathbb{R},\mathbb{R})$ is a solution of (\ref{eq_intro_Poh_nonlocal_Cauchy}). Moreover assume that $u^2\in C^2([a,b])$.
Let
\begin{equation}
L_a:=\lim_{x\to a^+}\frac{u^2(x)}{x-a},\quad L_b:=\lim_{x\to b^-}\frac{u^2(x)}{x-b}
\end{equation}
and set
\begin{equation}\label{eq_lem_def_alpha_ab}
\alpha_a:=\sqrt{ L_a},\quad \alpha_b:=\sqrt{- L_b}.
\end{equation}
Then there exists $\varepsilon>0$ and bounded functions $\beta_a$, $\zeta_a$, $\beta_b$, $\zeta_b$ defined respectively on $(a, a+\varepsilon)$ and $(b-\varepsilon,b)$ such that if $(u^2)'(a)\neq 0$ and $x\in (a, a+\varepsilon)$,
\begin{equation}\label{eq_lem_asymptotic_for_u_a}
u(x)=\text{sgn}[u(x)]\alpha_a(x-a)^\frac{1}{2}+\beta_a(x)(x-a)^\frac{3}{2},
\end{equation}
and if $(u^2)'(b)\neq 0$ $x\in (b-\varepsilon, b)$,
\begin{equation}\label{eq_lem_asymptotic_for_u_b}
u(x)=\text{sgn}[u(x)]\alpha_b(b-x)^\frac{1}{2}+\beta_b(x)(b-x)^\frac{3}{2}.
\end{equation}
Moreover, if $(u^2)'(a)= 0$ and $x\in (a, a+\varepsilon)$
\begin{equation}\label{eq_lem_asymptotic_for_u_a_derivative0}
u(x)=\zeta_a(x)(x-a),
\end{equation}
and if $(u^2)'(b)= 0$ and $x\in (b-\varepsilon, b)$
\begin{equation}\label{eq_lem_asymptotic_for_u_b_derivative0}
u(x)=\zeta_b(x)(b-x).
\end{equation}
\end{lem}

\begin{proof}
First we remark that since $u^2\in C^2([a,b])$, the limits $L_a$ and $L_b$ are well defined.
We will show (\ref{eq_lem_asymptotic_for_u_a}) and (\ref{eq_lem_asymptotic_for_u_a_derivative0}); the proofs of (\ref{eq_lem_asymptotic_for_u_b}) and (\ref{eq_lem_asymptotic_for_u_b_derivative0}) are analogous.
First assume that $(u^2)'(a)\neq 0$. Then, since $(u^2)'(x)=2u(x)u'(x)$ for $x\in(a,b)$ and $u(a)=0$, $\lim_{x\to a^+} u'(x)$ exists and is equal either to $\infty$ or to $-\infty$.
Then we can choose $\varepsilon>0$ such that $u(x)$ has the same sign on $(a,a+\varepsilon)$. By Taylor's Theorem, since $u^2\in C^2([a,b])$ and $u(a)=0$,
\begin{equation}\label{eq_proof_lemma_Taylor_for_square}
u^2(x)=(u^2)'(a)(x-a)+(u^2)''(\xi_x)(x-a)^2
\end{equation}
for $x\in(a,a+\varepsilon)$, for some $\xi_x\in (a,a+\varepsilon)$ depending on $x$. Therefore, if $(u^2)'(a)\neq 0$
\begin{equation}
u(x)=\text{sgn}[u(x)]\sqrt{(u^2)'(a)(x-a)+(u^2)''(\xi_x)(x-a)^2}
\end{equation}
for $x\in (a, a+\varepsilon)$. Then, by Taylor's Theorem
\begin{equation}\label{eq_proof_lem_first_form_asymp}
u(x)=\text{sgn}[u(x)]\left(\sqrt{(u^2)'(a)}(x-a)^\frac{1}{2}+\frac{1}{2}\frac{(u^2)''(\xi_x)(x-a)^2}{\sqrt{(u^2)'(a)(x-a)+z_x}}\right)
\end{equation}
for $x\in(a,a+\varepsilon)$ for some $z_x\in \mathbb{R}$ depending on $x$, such that $\lvert z_x\rvert\leq \lvert (u^2)''(\xi_x)\rvert(x-a)^2$. By taking a smaller $\varepsilon$ if necessary, we can ensure that the argument of the square root remains positive and that $(u^2)'(a)+\frac{z_x}{(x-a)}$ is uniformly bounded below by a positive constant for any $x\in (a, a+\varepsilon)$.
Therefore, if we set
\begin{equation}
\beta_a(x):=\text{sgn}(u(x))\frac{(u^2)''(\xi_x)}{2\sqrt{(u^2)'(a)+\frac{z_x}{(x-a)}}}
\end{equation}
for any $x\in (a, a+\varepsilon)$ , $\beta_a$ is a bounded function and
\begin{equation}
u(x)=\text{sgn}[u(x)]\alpha_a(x-a)^\frac{1}{2}+\beta_a(x)(x-a)^\frac{3}{2},
\end{equation}
for any $x\in (a, a+\varepsilon)$. This shows (\ref{eq_lem_asymptotic_for_u_a}).\\
In order to show (\ref{eq_lem_asymptotic_for_u_a_derivative0}), assume that $(u^2)'(a)=0$. Then (\ref{eq_proof_lemma_Taylor_for_square}) remains true, but the first term is equal to zero. Therefore, if we set
\begin{equation}
\zeta_a(x)=\text{sgn}[u(x)]\sqrt{\lvert(u^2)''(\xi_x)\rvert}
\end{equation}
for any $x\in (a,a+\varepsilon)$, $\zeta_a$ is bounded and we have
\begin{equation}
u(x)=\zeta_a(x)(x-a)
\end{equation}
for any $x\in (a,a+\varepsilon)$, as desired.
\end{proof}

\begin{lem}
Assume that $u\in H^\frac{1}{2}(\mathbb{R},\mathbb{R})$ is a solution of (\ref{eq_intro_Poh_nonlocal_Cauchy}). Moreover assume that $u^2\in C^2([a,b])$.
Then there exists $\varepsilon>0$ and bounded functions $\delta_a$, $\delta_b$ defined respectively on $(a, a+\varepsilon)$ and $(b-\varepsilon,b)$ such that if $(u^2)'(a)\neq 0$, if $x\in (a-\varepsilon, a)$,
\begin{equation}\label{eq_lem_asymptotic_for_fraclap_u_a}
(-\Delta)^\frac{1}{2}u(x)=-\frac{\alpha_a}{2}(a-x)^{-\frac{1}{2}}+\delta_a(x),
\end{equation}
and if $(u^2)'(b)\neq 0$, if $x\in (b, b+\varepsilon)$,
\begin{equation}\label{eq_lem_asymptotic_for_fraclap_u_b}
(-\Delta)^\frac{1}{2}u(x)=-\frac{\alpha_b}{2}(x-b)^{-\frac{1}{2}}+\delta_b(x),
\end{equation}
where $\alpha_a$, $\alpha_b$ were defined in (\ref{eq_lem_def_alpha_ab}).
Moreover, there exist constants $C_a$, $C_b$ such that if $(u^2)'(a)=0$, if $x\in(a-\varepsilon, a)$
\begin{equation}
\left\lvert(-\Delta)^\frac{1}{2}u(x)\right\rvert\leq C_a\left\lvert\log(a-x)\right\rvert+\delta_a(x),
\end{equation}
and if $(u^2)'(b)=0$, if $x\in(b, b+\varepsilon)$
\begin{equation}\label{eq_lem_asymptotic_for_fraclap_u_b_when_derivative_vanishes_at_b}
\left\vert(-\Delta)^\frac{1}{2}u(x)\right\rvert=C_b\left\lvert\log(x-b)\right\rvert+\delta_b(x).
\end{equation}
\end{lem}

\begin{proof}
We show (\ref{eq_lem_asymptotic_for_fraclap_u_b}). The proof of (\ref{eq_lem_asymptotic_for_fraclap_u_a}) is analogous. First we assume that $(u^2)'\neq 0$.
Let $\varepsilon>0$. For any $x\in (b, b+\varepsilon)$ we set
\begin{equation}
g(x):=\frac{(-\Delta)^\frac{1}{2}u(x)}{(x-b)^{-\frac{1}{2}}}.
\end{equation}
We claim that $\lim_{x\to b^+}g(x)=\frac{1}{2}\text{sgn}(u(b_-))$, where $\text{sgn}(u(b_-))$ denotes the sign of $u(x)$ as $x$ approaches $b$ from the right. Indeed, let $\tilde{\varepsilon}$, $\alpha_b$, $\beta_b$ as in Lemma \ref{lem_asymptotic_u}, and such that for any $x\in (b-\tilde{\varepsilon},b)$, $u(x)$ has the same sign. Then
\begin{equation}\label{eq_proof_lemma_asymptotics_three_limits}
\begin{split}
\lim_{x\to b^+} g(x)=&-\lim_{x\to b^+}\frac{\frac{1}{\pi}\int_a^{b-\tilde{\varepsilon}}\frac{u(y)}{(x-y)^2}dy}{(x-b)^{-\frac{1}{2}}}-\alpha_b\lim_{x\to b^+}\frac{\frac{1}{\pi}\int_{b-\tilde{\varepsilon}}^b\text{sgn}(u(y))\frac{(b-y)^\frac{1}{2}}{(x-y)^2}dy}{(x-b)^{-\frac{1}{2}}}\\
&-\lim_{x\to b^+}\frac{\frac{1}{\pi}\int_{b-\tilde{\varepsilon}}^b\frac{\beta_b(y)(b-y)^\frac{3}{2}}{(x-y)^2}dy}{(x-b)^{-\frac{1}{2}}}
\end{split}
\end{equation}
provided the limits on the right hand side exist.
In fact, the first term converges to zero (as the numerator is uniformly bounded); the third term converges to zero as well. Indeed, for any $x\in (b, b+\varepsilon)$
\begin{equation}
\left\lvert\frac{\int_{b-\tilde{\varepsilon}}^b\frac{\beta_b(y)(b-y)^\frac{3}{2}}{(x-y)^2}dy}{(x-b)^{-\frac{1}{2}}}\right\rvert\leq\frac{\lVert \beta_b\rVert_{L^\infty}\int_{b-\tilde{\varepsilon}}^b(b-y)^\frac{1}{2}dy}{(x-b)^{-\frac{1}{2}}},
\end{equation}
thus the numerator is uniformly bounded, and the third term converges to $0$.\\
Finally we claim that
\begin{equation}\label{eq_proof_lemma_limit_to_prove_existence_of}
\lim_{x\to b^+}\frac{\frac{1}{\pi}\int_{b-\tilde{\varepsilon}}^b\frac{(b-y)^\frac{1}{2}}{(x-y)^2}dy}{(x-b)^{-\frac{1}{2}}}=\frac{1}{2}.
\end{equation}
Indeed, substituting $x'=x-b$, $y'=b-y$,
\begin{equation}
\lim_{x\to b^+}\frac{\frac{1}{\pi}\int_{b-\tilde{\varepsilon}}^b\frac{(b-y)^\frac{1}{2}}{(x-y)^2}dy}{(x-b)^{-\frac{1}{2}}}
=\lim_{x'\to 0^+} \frac{1}{\pi}{x'}^\frac{1}{2}\int_0^{\tilde{\varepsilon}}\frac{{y'}^\frac{1}{2}}{(x'+y')^2}dy'
\end{equation}
Now we substitute $z=\frac{y'}{x'}$ and we obtain
\begin{equation}
\lim_{x'\to 0^+}\frac{1}{\pi}'{x'}^\frac{1}{2}\int_0^\frac{\tilde{\varepsilon}}{x'}\frac{(x'z)^\frac{1}{2}}{(x'+x'z)^2}x'dz=\frac{1}{\pi}\int_0^\infty\frac{z^\frac{1}{2}}{(1+z)^2}dz.
\end{equation}
To evaluate this integral, we think of it as a complex integral: let $R>0$ and let $\gamma_R$ be the path in $\mathbb{C}$ consisting of the straight line from $0$ to $R$, the anticlockwise circle centered in zero and starting at $R$ and the straight line from $R$ to $0$. Then, if we denote as $I$ the integral we would like to evaluate, we obtain
\begin{equation}\label{eq_proof_lem_integral_through_complex_integration}
2I=\lim_{R\to \infty}\frac{1}{\pi}\int_{\gamma_R}\frac{z^\frac{1}{2}}{(1+z)^2}dz
\end{equation}
(on the segment from $0$ to $R$, $z^\frac{1}{2}$ is evaluated on its principal branch, while on the segment from $R$ to $0$ is evaluated on its other branch).
To compute the integral on the right hand side of (\ref{eq_proof_lem_integral_through_complex_integration}) we use the Residue Theorem:
\begin{equation}
\lim_{R\to \infty}\frac{1}{\pi}\int_{\gamma_R}\frac{z^\frac{1}{2}}{(1+z)^2}dz=2\pi i Res(-1),
\end{equation}
where $Res(-1)$ stays for the residue at $-1$ of $\frac{1}{\pi}\frac{z^\frac{1}{2}}{(1+z)^2}$.
Using the fact that the first terms of the Laurent expansion of $z^\frac{1}{2}$ around $-1$ are $i-\frac{i}{2}(z+1)$, we obtain $Res(-1)=\frac{-i}{2\pi}$. Therefore
\begin{equation}
I=\frac{1}{2}2\pi i\frac{-i}{2\pi}=\frac{1}{2}.
\end{equation}
Thus the limits in (\ref{eq_proof_lemma_asymptotics_three_limits}) are well defined, and $\lim_{x\to b^+}g(x)=\frac{1}{2}\text{sgn}(u(b_-))$.
This concludes the proof of the first claim. Now we claim that $k(x):=g'(x)(x-b)^\frac{1}{2}$ is uniformly bounded for  $x$ in $(b,b+\varepsilon)$. Indeed we can compute
\begin{equation}
\begin{split}
g'(x)(x-b)^\frac{1}{2}=&\int_a^bu(y)\left(\frac{2(x-b)}{(x-y)^3}-\frac{1}{2(x-y)^2}\right)dy\\
=&\int_a^{b-\tilde{\varepsilon}}u(y)\left(\frac{2(x-b)}{(x-y)^3}-\frac{1}{2(x-y)^2}\right)dy\\
&+\int_{b-\tilde{\varepsilon}}^b \text{sgn}(u(y))\alpha_b(b-y)^\frac{1}{2}\left(\frac{2(x-b)}{(x-y)^3}-\frac{1}{2(x-y)^2}\right)dy\\
&+\int_{b-\tilde{\varepsilon}}^b \beta_b(y)(b-y)^\frac{3}{2}\left(\frac{2(x-b)}{(x-y)^3}-\frac{1}{2(x-y)^2}\right)dy.
\end{split}
\end{equation}
Now the first term is clearly bounded. To see that that the third one is bounded as well we compute
\begin{equation}
\left\lvert\frac{(b-y)^\frac{3}{2}(x-b)}{(x-y)^3}\right\rvert\leq \left\lvert\frac{(b-y)^\frac{3}{2}}{(x-y)^2}\right\rvert \leq \frac{1}{\lvert b-y\rvert^\frac{1}{2}};\quad \left\lvert\frac{(b-y)^\frac{3}{2}}{(x-y)^2}\right\rvert\leq \frac{1}{\lvert b-y\rvert^\frac{1}{2}}
\end{equation}
for any $x\in (b, b-\varepsilon)$, $y\in (b-\tilde{\varepsilon}, b)$. As $\frac{1}{\lvert b-y\rvert^\frac{1}{2}}$ is integrable in $(b-\tilde{\varepsilon})$, the third term is bounded.\\
For the second term, we compute
\begin{equation}\label{eq_proof_lemma_evaluation_second_integral}
\begin{split}
&\int_{b-\tilde{\varepsilon}}^b(b-y)^\frac{1}{2}\frac{2(x-b)-\frac{1}{2}(x-y)}{(x-y)^3}dy\\
=&(b-y)^\frac{1}{2}\frac{(x-b)}{(x-y)^2}\bigg\rvert_{y=b-\tilde{\varepsilon}}^{y=b}-\frac{1}{2}\int_{b-\tilde{\varepsilon}}^b\frac{-(b-y)^{-\frac{1}{2}}(x-b)+(b-y)^\frac{1}{2}}{(x-y)^2}dy\\
=&-\tilde{\varepsilon}^\frac{1}{2}\frac{(x-b)}{(x-b+\tilde{\varepsilon})^2}-\frac{1}{2}\int_{b-\tilde{\varepsilon}}^b\frac{(b-y)^{-\frac{1}{2}}[-(x-b)+(b-y)]}{(x-y)^2}dy.
\end{split}
\end{equation}
Now the first term in the last line of (\ref{eq_proof_lemma_evaluation_second_integral}) is uniformly bounded for $x\in (b, b-\varepsilon)$. For the second, we first make the change of variables $y'=b-y$, $x'=x-b$ and then $y'=z^2$ to obtain
\begin{equation}\label{eq_proof_lemma_integral_wolfram}
\begin{split}
-\frac{1}{2}\int_0^{\tilde{\varepsilon}}\frac{{y'}^{-\frac{1}{2}}(y'-x')}{(x'+y')^2}dy'=&-\int_0^{\sqrt{\tilde{\varepsilon}}}\frac{z^2-x'}{(x'+z^2)^2}dz.
\end{split}
\end{equation}
Since $x'$ is positive, $\lvert z^2-x'\rvert\leq z^2+x'$. Therefore the absolute value of the integral is bounded by
\begin{equation}
\int_0^{\sqrt{\tilde{\varepsilon}}}\frac{1}{z^2+x'}dz=\frac{1}{\sqrt{x'}}\arctan\left(\sqrt{\frac{\tilde{\varepsilon}}{x'}}\right),
\end{equation}
and this is uniformly bounded for $x\in (b, b-\varepsilon)$.
Therefore $k$ is uniformly bounded. Thus $g'$ is integrable in $(b, b+\varepsilon)$.
This implies that for $x\in (b, b+\varepsilon)$,
\begin{equation}
g(x)=g(b)+\int_b^xk(y)(y-b)^{-\frac{1}{2}}dy=-\frac{1}{2}\text{sgn}(u(b_-))\alpha_b+\mathscr{O}(\lvert x-b\rvert^\frac{1}{2})
\end{equation}
and thus
\begin{equation}
(-\Delta)^\frac{1}{2}u(x)=-\frac{1}{2}\text{sgn}(u(b_-))\alpha_b(x-b)^{-\frac{1}{2}}+\mathscr{O}(\lvert x-b\rvert^\frac{1}{2})(x-b)^{-\frac{1}{2}}.
\end{equation}
This concludes the proof of (\ref{eq_lem_asymptotic_for_fraclap_u_b}) whenever $(u^2)'(b)\neq 0$ (after taking $\varepsilon$ smaller if necessary).\newline
Now assume that $(u^2)'(b)= 0$. By means of (\ref{eq_lem_asymptotic_for_u_b_derivative0}) we compute for any $x\in (b, b+\varepsilon)$
\begin{equation}
(-\Delta)^\frac{1}{2}u(x)=-\frac{1}{\pi}\int_a^{b-\tilde{\tilde{\varepsilon}}}\frac{u(y)}{(x-y)^2}dy-\frac{1}{\pi}\int_{b-\tilde{\varepsilon}}^b\frac{\zeta_b(y)(b-y)}{(x-y)^2}dy.
\end{equation}
for some bounded function $\zeta_b$.
The first term is clearly bounded. For the second we compute
\begin{equation}
\begin{split}
\int_{b-\tilde{\varepsilon}}^b\frac{b-y}{(x-y)^2}dy=&\int_{b-\tilde{\varepsilon}}^b\frac{b-x}{(x-y)^2}+\frac{1}{x-y} dy\\
=&1-\frac{b-x}{x-b+\tilde{\varepsilon}}+\log(x-b)+\log(x-b+\tilde{\varepsilon}).
\end{split}
\end{equation}
Setting $C_b:= \frac{1}{\pi}\lVert \zeta_b\rVert_{L^\infty}$ and
\begin{equation}
\delta_b(x):=\frac{1}{\pi}\int_a^{b-\tilde{\tilde{\varepsilon}}}\frac{\lvert u(y)\rvert}{(x-y)^2}dy+\frac{1}{\pi}\lVert \zeta_b\rVert_{L^\infty}\left(1-\frac{b-x}{x-b+\tilde{\varepsilon}}+\left\lvert\log(x-b+\tilde{\varepsilon})\right\rvert\right)
\end{equation}
for $x\in (b, b+\varepsilon)$, we obtain the desired equation.
\end{proof}
\subsection{The Pohozaev identity for bounded intervals}\label{ssec:The Pohozaev identity for open, bounded intervals}

Combining the previous estimates we obtain the following result:
\begin{lem}\label{lem_derivative_on_the_boundary}
Assume that $a<0<b$. Assume that $u\in H^\frac{1}{2}(\mathbb{R},\mathbb{R})$ is a solution of (\ref{eq_intro_Poh_nonlocal_Cauchy}). Moreover assume that $u^2\in C^2([a,b])$.
Then the limits
\begin{equation}\label{eq_lem_derivative_on_the_boundary_existence_limit}
\lim_{x\to a^+}\frac{u^2(x)}{a-x},\quad \lim_{x\to b^-}\frac{u^2(x)}{x-b}
\end{equation}
exist and are finite. Moreover
\begin{equation}\label{eq_lem_derivative_boundary_term_for_asymptotic_a}
\frac{d}{d\lambda}\bigg\rvert_{\lambda=1^-} \int_\frac{a}{\lambda}^{a} u(\lambda x)(-\Delta)^\frac{1}{2}u(x)dx=-\frac{\pi}{4}\lim_{x\to a^+}\frac{u^2(x)}{a-x} a
\end{equation}
and 
\begin{equation}\label{eq_lem_derivative_boundary_term_for_asymptotic}
\frac{d}{d\lambda}\bigg\rvert_{\lambda=1^-} \int_{b}^{\frac{b}{\lambda}} u(\lambda x)(-\Delta)^\frac{1}{2}u(x)dx=-\frac{\pi}{4}\lim_{x\to b^-}\frac{u^2(x)}{x-b} b.
\end{equation}
The two derivatives in $\lambda$ have to be understood as left derivatives.
\end{lem}
\begin{proof}
First we observe that since $u^2\in C^2([a,b])$, $u^2$ can be extended to a function of class $C^2$ defined in an open interval containing $[a,b]$. Let's call $\overline{u^2}$ such an extension. Then, in particular, $\overline{u^2}$ is differentiable in $a$ and $b$, and since $u(a)=0$, $u(b)=0$, its derivative in $a$ and $b$ is given by
\begin{equation}
\partial\overline{u^2}(a)=-\lim_{x\to a^+}\frac{u^2(x)}{a-x},\quad \partial\overline{u^2}(b)=\lim_{x\to b^-}\frac{u^2(x)}{x-b}.
\end{equation}
Therefore the limits in (\ref{eq_lem_derivative_on_the_boundary_existence_limit}) exist and are finite.\\
Next we prove (\ref{eq_lem_derivative_boundary_term_for_asymptotic}); the proof of (\ref{eq_lem_derivative_boundary_term_for_asymptotic_a}) is analogous.
First we assume that $(u^2)'(b)\neq 0$. Let $\lambda$ be so that $\lvert \frac{b}{\lambda} -b\rvert$ are smaller than $\varepsilon$ for both previous Lemmas. Then we can rewrite the integral on the left hand side of (\ref{eq_lem_derivative_boundary_term_for_asymptotic}) as
\begin{equation}
\begin{split}
\int_{b}^\frac{b}{\lambda}&\left[\text{sgn}(u(b_-)) \alpha_b(b-\lambda x)^\frac{1}{2}+\beta_b(\lambda x)(b-\lambda x)^\frac{3}{2}\right]\\
&\cdot\left[-\text{sgn}(u(b_-))\frac{\alpha_b}{2}\left(x-b\right)^{-\frac{1}{2}}+\delta_b\left(x\right)\right]dx.
\end{split}
\end{equation}
Now let
\begin{equation}
\phi(\lambda)=\int_b^{\frac{b}{\lambda }} \delta_b(x)(b-\lambda x)^\frac{1}{2}dx.
\end{equation}
for $\lambda\in (1-\tilde{\varepsilon}, 1]$ for some small $\tilde{\varepsilon}>0$. Then $\phi(1)=0$. Moreover $\phi$ is bounded from above by
\begin{equation}
\begin{split}
\phi_+(\lambda)=&\int_b^{\frac{b}{\lambda }}\lVert \delta_b\rVert_{L^\infty}(b-\lambda x)^\frac{1}{2}dx=\frac{2}{3}\lVert \delta_b\rVert_{L^\infty}b^\frac{3}{2}\frac{(1-\lambda)^\frac{3}{2}}{\lambda},
\end{split}
\end{equation}
and from below by $\phi_-:=-\phi_+$.
Now $\phi_+(1)=0$ and $\phi_-(1)=0$. Moreover $\phi_+$ and $\phi_-$ are both left differentiable in $1$ with left derivative equal to $0$. Therefore $\phi$ is also left differentiable in $1$ with left derivative equal to $0$.\\
One can apply a similar argument to see that the function
\begin{equation}
\psi(\lambda)=\int_b^{\frac{b}{\lambda }} \beta_b(\lambda x)(b-\lambda x)^\frac{3}{2}\delta_b\left(x\right)dx
\end{equation}
defined for $\lambda\in (1-\tilde{\varepsilon},1]$ is left differentiable in $1$ with left derivative equal to $0$.
Moreover, let's set
\begin{equation}
\chi(\lambda)=\int_b^{\frac{b}{\lambda }} \beta_b(\lambda x)\frac{(b-\lambda x)^\frac{3}{2}}{(x-b)^\frac{1}{2}}dx
\end{equation}
for $\lambda\in (1-\tilde{\varepsilon}, 1]$. We observe that $\chi(1)=0$.
Then $\chi$ is bounded from above by
\begin{equation}
\chi_+(\lambda)=(b-\lambda b)^\frac{3}{2}\lVert \beta\rVert_{L^\infty}\int_{b}^{\frac{b}{\lambda}}\frac{1}{(x-b)^\frac{1}{2}}dx
\end{equation}
and from below by $\chi_-:=-\chi_+$.
We observe that $\chi_+(1)=\chi_-(1)=0$ and both $\chi_+$ and $\chi_-$ are left differentiable with left derivative equal to $0$. Therefore also $\chi$ is left differentiable with left derivative equal to $0$.
We conclude that
\begin{equation}
\frac{d}{d\lambda}\bigg\rvert_{\lambda=1^-} \int_{ b}^\frac{b}{\lambda} u_{\lambda }(x)(-\Delta)^\frac{1}{2}u(x)dx=-\frac{\alpha_b^2}{2} \frac{d}{d\lambda}\bigg\rvert_{\lambda=1^-}\int_{b}^\frac{b}{\lambda}\frac{(b-\lambda x)^\frac{1}{2}}{(x-b)^\frac{1}{2}}dx,
\end{equation}
provided the left derivative on the right hand side exists. Now we observe that, substituting $y=\frac{x}{b}$,
\begin{equation}
\int_{b}^\frac{b}{\lambda}\frac{(b-\lambda x)^\frac{1}{2}}{(x-b)^\frac{1}{2}}dx=b\int_1^\frac{1}{\lambda}\frac{(b-\lambda y b)^\frac{1}{2}}{(yb-b)^\frac{1}{2}}dy=b\int_1^\frac{1}{\lambda}\frac{(1-\lambda y)^\frac{1}{2}}{(y-1)^\frac{1}{2}}dy.
\end{equation}
One can compute
\begin{equation}
\frac{d}{d\lambda}\bigg\rvert_{\lambda=1^-}\int_1^\frac{1}{\lambda}\frac{(1-\lambda y)^\frac{1}{2}}{(y-1)^\frac{1}{2}}dy=\frac{\pi}{2}.
\end{equation}
Indeed, observing that the integral tends to $0$ as $\lambda\to 1^-$, we can rewrite the derivative as
\begin{equation}
\lim_{\lambda\to 1^-}\frac{1}{1-\lambda}\int_1^\frac{1}{\lambda}\frac{(1-\lambda y)^\frac{1}{2}}{(y-1)^\frac{1}{2}}dy.
\end{equation}
Substituting $x=\frac{(y-1)\lambda}{1-\lambda}$ we obtain that for any $\lambda\in (1-\tilde{\varepsilon}, 1)$ the integral is equal to
\begin{equation}
\begin{split}
\int_0^1\frac{(1-x)^\frac{1}{2}}{x^\frac{1}{2}}dx=&2\int_0^\frac{\pi}{2}\frac{\left(1-\cos^2(\theta)\right)^\frac{1}{2}}{\cos(\theta)}\cos(\theta)\sin(\theta)d\theta\\
=&2\int_0^\frac{\pi}{2}\sin(\theta)^2d\theta
=2\int_0^\frac{\pi}{2}\frac{1-\cos(2\theta)}{2}d\theta=\frac{\pi}{2},
\end{split}
\end{equation}
where we used the substitution $x=\cos^2(\theta)$.
Thus we obtain
\begin{equation}
\frac{d}{d\lambda}\bigg\rvert_{\lambda=1^-} \int_{ b}^\frac{b}{\lambda} u_{\lambda }(x)(-\Delta)^\frac{1}{2}u(x)dx=-\frac{\pi}{4}\lim_{x\to b^-}\frac{u^2(x)}{x-b}.
\end{equation}
This concludes the proof of (\ref{eq_lem_derivative_boundary_term_for_asymptotic}) whenever $(u^2)'\neq 0$.\newline
Finally, if $(u^2)'(b)=0$, we let
\begin{equation}
\sigma(\lambda):=\int_b^\frac{b}{\lambda}u(\lambda x)(-\Delta)^\frac{1}{2}u(x)dx
\end{equation}
and
\begin{equation}
\tau(\lambda):=\int_b^{\frac{b}{\lambda}}\left\lvert\zeta_b(x\lambda)\right\rvert(b-\lambda x)(C_b\left\lvert\log(x-b)\right\rvert+\delta_b(x))dx
\end{equation}
for $\lambda\in (1-\varepsilon, 1]$ for some $\varepsilon>0$.
Then $\sigma(1)=0$ and $\tau(1)=0$, and by estimate (\ref{eq_lem_asymptotic_for_fraclap_u_b_when_derivative_vanishes_at_b}), there holds
\begin{equation}\label{eq_proof_lemma_derivative_of_sigma_estimated_by_derivative_of_tau}
\tau(\lambda)\geq \lvert \sigma(\lambda)\rvert
\end{equation}
for any $\lambda\in (1-\varepsilon, 1]$. Therefore, if $\tau$ is left differentiable in $1$, then also $\sigma$ is left differentiable in $1$, and there holds
\begin{equation}
\left\lvert\frac{d}{d\lambda}\sigma\bigg\rvert_{\lambda=1^-}\right\rvert\leq \frac{d}{d\lambda}\tau\bigg\rvert_{\lambda=1^-}.
\end{equation}
In order to show that $\tau$ is left differentiable in $1$, we observe that since $\tau(1)=0$, if the limits exist
\begin{equation}\label{eq_proof_lemma_computation_limit_if_left_derivative_of_tau_exists}
\begin{split}
\frac{d}{d\lambda}\tau\bigg\rvert_{\lambda=1^-}=&\lim_{\lambda\to 1^-}\frac{1}{1-\lambda}\int_b^{\frac{b}{\lambda}}\left\lvert\zeta_b(x\lambda)\right\rvert(b-\lambda x)(C_b\left\lvert\log(x-b)\right\rvert+\delta_b(x))dx
\end{split}
\end{equation}
Now we observe that for any $\lambda\in (1-\varepsilon, 1]$
\begin{equation}
\begin{split}
&\left\lvert \frac{1}{1-\lambda}\int_b^{\frac{b}{\lambda}}\left\lvert\zeta_b(x\lambda)\right\rvert(b-\lambda x)(C_b\left\lvert\log(x-b)\right\rvert+\delta_b(x))dx\right\rvert\\
\leq &\lVert \zeta_b\rVert_{L^\infty}\lVert \delta_b\rVert_{L^\infty}\frac{b}{\lambda}\frac{1}{\frac{b}{\lambda}-b}\int_b^\frac{b}{\lambda}(b-\lambda x)dx\\
&+\lVert \zeta_b\rVert_{L^\infty}b\int_b^\frac{b}{\lambda}\frac{(b-\lambda x)}{b-\lambda b}\left\lvert\log(x-b)\right\rvert dx\\
\leq &\lVert \zeta_b\rVert_{L^\infty}\left(\lVert \delta_b\rVert_{L^\infty}\frac{b}{\lambda}\frac{b}{2}(1-\lambda)+\int_b^\frac{b}{\lambda}\left\lvert\log(x-b)\right\rvert dx\right)
\end{split}
\end{equation}
%
%
and the expression on the last line tends to $0$ as $\lambda\to 1^-$. Therefore the limit in (\ref{eq_proof_lemma_computation_limit_if_left_derivative_of_tau_exists}) exists and $\frac{d}{d\lambda}\tau\rvert_{\lambda=1^-}=0$.
By (\ref{eq_proof_lemma_derivative_of_sigma_estimated_by_derivative_of_tau}), this implies that $\frac{d}{d\lambda}\sigma\rvert_{\lambda=1^-}=0$.
This concludes the proof.
\end{proof}

We are now able to compute an explicit expression for the integral (\ref{eq_discussion_intro_Pohozaev_integral_to_be_evaluated_Pohozaev_interval}), i.e. a Pohozaev identity for the Cauchy problem (\ref{eq_intro_Poh_nonlocal_Cauchy}).
We will follow the idea outlined at the beginning of this section (before Lemma \ref{lemma_equivalence minimization}).

\begin{thm}\label{prop_Pohozaev_identity_for_open_bounded_intervals_final}
Assume that $u\in H^\frac{1}{2}(\mathbb{R}, \mathbb{R})$ is a solution of (\ref{eq_intro_Poh_nonlocal_Cauchy}). Moreover assume that $u^2\in C^2([a,b])$.
Then
\begin{equation}\label{eq_prop_statement_Pohozaev_proof_asymptotic_translations}
\int_\Omega u'(x)(-\Delta)^\frac{1}{2}u(x)dx=\frac{\pi}{8}\left[ \lim_{x\to a^+}\frac{u^2(x)}{x-a}-\lim_{x\to b^-}\frac{u^2(x)}{b-x}\right]
\end{equation}
and
\begin{equation}\label{eq_prop_statement_Pohozaev_proof_asymptotic}
\int_\Omega u'(x)x(-\Delta)^\frac{1}{2}u(x)dx=\frac{\pi}{8}\left[\lim_{x\to a^+}\frac{u^2(x)}{x-a}a-\lim_{x\to b^-}\frac{u^2(x)}{b-x}b\right].
\end{equation}
\end{thm}
\begin{rem}\noindent
\begin{enumerate}
\item Since $u^2$ is assumed to be in $C^2([a,b])$, we can rewrite the expression on the right hand sides of (\ref{eq_prop_statement_Pohozaev_proof_asymptotic_translations}) and (\ref{eq_prop_statement_Pohozaev_proof_asymptotic}) as
\begin{equation}
\frac{\pi}{8}\left(\partial_+ u^2(a)+\partial_-u^2(b)\right)
\end{equation}
and
\begin{equation}
\frac{\pi}{8}\left(\partial_+ u^2(a)a+\partial_-u^2(b)b \right)
\end{equation} 
respectively, where $\partial_+$ and $\partial_-$ denote right and left derivatives.
\item For functions $u\in H^\frac{1}{2}(\mathbb{R},\mathbb{R})$ satisfying Equation (\ref{eq_intro_Poh_nonlocal_Cauchy}), there holds
\begin{equation}
\int_a^b x u'(x)(-\Delta)^\frac{1}{2}u(x)dx=\int_{\mathbb{R}}(F(u))'(x)dx=0
\end{equation}
by the Fundamental Theorem of Calculus, where $F$ is the primitive function of $f$ such that $F(0)=0$. Thus, if we also assume that $u^2\in C^2([a,b])$, Theorem \ref{prop_Pohozaev_identity_for_open_bounded_intervals_final} implies that also the right hand side in Equation (\ref{eq_prop_statement_Pohozaev_proof_asymptotic_translations}) is equal to $0$.
\end{enumerate}
\end{rem}

\begin{proof}
Let's first assume that $a<0<b$. First we claim that we can rewrite (\ref{eq_prop_statement_Pohozaev_proof_asymptotic}) as
\begin{equation}
\int_\mathbb{R}x u'(x)(-\Delta)^\frac{1}{2}u(x)dx=\frac{d}{d\lambda}\bigg\rvert_{\lambda=1^-} \int_a^b
u(\lambda x)(-\Delta)^\frac{1}{2}u(x)dx.
\end{equation}
The derivative in $\lambda$ is a left derivative. To prove the claim, we compute
\begin{equation}
\frac{d}{d\lambda}\bigg\rvert_{\lambda=1^-} \int_a^b
u(\lambda x)(-\Delta)^\frac{1}{2}u(x)dx=\lim_{\lambda\to 1^-}\frac{1}{\lambda-1}\int_a^b(u(\lambda x)-u(x))(-\Delta)^\frac{1}{2}u(x)dx.
\end{equation}
We observe that for, if $x\neq 0$, $\lambda\neq 1$, by the Mean value Theorem we have
\begin{equation}
\frac{u(\lambda x)-u(x)}{\lambda-1}=u'(\xi_{x, \lambda})
\end{equation}
for some $\xi_{x,\lambda}$ between $x\lambda$ and $x$ depending on $\lambda$ and $x$. 
In order to estimate $u'(\xi_{x,\lambda})$ we remark that if $(u^2)'(a)\neq 0$, $u'(x)=\frac{(u^2)'(x)}{u(x)}$ for $x\in (a, a+\varepsilon)$ for some $\varepsilon>0$. We observe that since $u^2\in C^2([a,b])$, $(u^2)'=2u'u$ can be extended to a $C^1$ function on $[a,b]$. On the other hand, if $(u^2)'(a)=0$, then by (\ref{eq_lem_asymptotic_for_u_a_derivative0}) $u$ is Lipschitz in $x\in (a, a+\varepsilon)$ and therefore $u'$ is bounded on $(a,a+\varepsilon)$. Thus in both cases 
\begin{equation}
\lvert u'(x)\rvert=\mathscr{O}(( x-a)^{-\frac{1}{2}})\text{ as }x\to a^+.
\end{equation}
Similarly, one can show that
\begin{equation}
\lvert u'(x)\rvert=\mathscr{O}((b-x)^{-\frac{1}{2}})\text{ as }x\to b^-.
\end{equation}
In particular, if we set $d(x)=\min\lbrace \lvert x-a\rvert,\vert x-b\rvert\rbrace$ for any $x\in (a,b)$ and $d(x)=0$ for any $x\in (a,b)^c$,
there exists a constant $C$ with
\begin{equation}
\lvert u'(x)\rvert\leq C d(x)^{-\frac{1}{2}}
\end{equation}
in $(a,b)$. Then, since $d^{-\frac{1}{2}}$ is convex
\begin{equation}\label{eq_proof_prop_integrable_uniform_bound_difference}
\left\lvert \frac{u(\lambda x)-u(x)}{\lambda-1}\right\rvert\leq C d(\xi_{x,\lambda})^{-\frac{1}{2}}\leq C\min\lbrace d(x),d(\lambda x)\rbrace^{-\frac{1}{2}}\leq C \lambda^{-\frac{1}{2}}d^{-\frac{1}{2}}(x)
\end{equation}
for $x\in (a,b)$. Since $d^{-\frac{1}{2}}$ is integrable in $(a,b)$, by Lebesgue's Dominated convergence Theorem we obtain
\begin{equation}\label{eq_proof_prop_Pohozaev_first_claim_derivative}
\lim_{\lambda\to 1^-}\int_a^b\frac{u(\lambda x)-u(x)}{\lambda-1}(-\Delta)^\frac{1}{2}u(x)dx=\int_a^b x u'(x)(\Delta)^\frac{1}{2}u(x)dx.
\end{equation}
This concludes the proof of the first claim. Now we observe that for any $\lambda>0$, for a.e. $x\in \mathbb{R}$
\begin{equation}
\begin{split}
\left\lvert (-\Delta)^\frac{1}{4}u_\lambda(x)\right\rvert^2-\left\lvert (-\Delta)^\frac{1}{4}u(x)\right\rvert^2=&
2(-\Delta)^\frac{1}{4}u(x)\left((-\Delta)^\frac{1}{4}u_\lambda(x)-(-\Delta)^\frac{1}{4}u(x)\right)\\&+\left((-\Delta)^\frac{1}{4}u_\lambda(x)-(-\Delta)^\frac{1}{4}u(x)\right)^2,
\end{split}
\end{equation}
and therefore
\begin{equation}\label{eq_proof_Pohozaev_rewriting_derivative_square}
\begin{split}
\int_{\mathbb{R}}(u_\lambda(x)-u(x))(-\Delta)^\frac{1}{2}u(x)dx=&\frac{1}{2}\int_\mathbb{R}\left\lvert (-\Delta)^\frac{1}{4}u_\lambda(x)\right\rvert^2-\left\lvert (-\Delta)^\frac{1}{4}u(x)\right\rvert^2dx\\
&-\frac{1}{2}\int_{\mathbb{R}}\left((-\Delta)^\frac{1}{4}u_\lambda(x)-(-\Delta)^\frac{1}{4}u(x)\right)^2dx.
\end{split}
\end{equation}
Since the half Dirichlet energy is invariant under dilations, the first integral on the right hand side of  (\ref{eq_proof_Pohozaev_rewriting_derivative_square}) vanishes. Thus, by Equation (\ref{eq_proof_prop_Pohozaev_first_claim_derivative}), there holds
\begin{equation}\label{eq_proof_prop_Pohozaev_dividing_limit_in_regions}
\begin{split}
\int_a^b x u'(x)(\Delta)^\frac{1}{2}u(x)dx=&-\frac{1}{2}\lim_{\lambda\to 1^-}\int_{\mathbb{R}}\frac{u_\lambda(x)-u(x)}{\lambda-1}(-\Delta)^\frac{1}{2}(u_\lambda-u)(x)dx\\
&-\lim_{\lambda\to 1^-}\int_{(a,b)^c}\frac{u(\lambda x)-u(x)}{\lambda-1}(-\Delta)^\frac{1}{2}u(x)dx.
\end{split}
\end{equation}
Now we observe that the expression on the right hand side of Equation (\ref{eq_proof_prop_Pohozaev_dividing_limit_in_regions}) can be rewritten as
\begin{equation}\label{eq_proof_prop_Pohozaev_limits_rewritten}
\begin{split}
&-\frac{1}{2}\lim_{\lambda\to 1^-}\int_{(a,b)^c}\frac{u_\lambda(x)-u(x)}{\lambda-1}(-\Delta)^\frac{1}{2}(u+u_\lambda)(x)dx\\
&+\frac{1}{2}\lim_{\lambda\to 1^-}\int_a^b\frac{u_\lambda(x)-u(x)}{\lambda-1}(-\Delta)^\frac{1}{2}(u_\lambda-u)(x)dx.
\end{split}
\end{equation}
Since $u$ satisfies Equation (\ref{eq_intro_Poh_nonlocal_Cauchy}), for any $\lambda>0$, any $x\in \left(\frac{a}{\lambda},\frac{b}{\lambda}\right)$ there holds
\begin{equation}\label{eq_proof_prop_Pohozaev_invariance_and_equation}
(-\Delta)^\frac{1}{2}u_\lambda(x)=\lambda(-\Delta)^\frac{1}{2}u(\lambda x)=\lambda f(u(\lambda x)).
\end{equation}
Since, by Proposition \ref{prop_bdry_regularity}, $u\in  C^\frac{1}{2}(\mathbb{R})$, $(-\Delta)^\frac{1}{2}(u_\lambda-u)(x)$ is bounded uniformly for $\lambda\in (0,1)$, $x\in (a,b)$, and $(-\Delta)^\frac{1}{2}(u_\lambda-u)\to 0$ a.e. in $(a,b)$ as $\lambda\to 1$. Thus, by estimate (\ref{eq_proof_prop_integrable_uniform_bound_difference}) and Lebesgue's Dominated convergence Theorem,
the second integral in (\ref{eq_proof_prop_Pohozaev_limits_rewritten}) tends to $0$ as $\lambda\to 1^-$.
Next we observe that for any $\lambda>0$, $\frac{u_\lambda-u}{\lambda-1}$ is supported in $\left[\frac{a}{\lambda},\frac{b}{\lambda}\right]$. By equation (\ref{eq_proof_prop_Pohozaev_invariance_and_equation}), $(-\Delta)^\frac{1}{2}u_\lambda(x)$ is bounded by a constant uniformly for $\lambda\in \left(\frac{1}{2},1\right)$ and $x\in \left[\frac{a}{\lambda},\frac{b}{\lambda}\right]$. Thus, since $\left\lvert (a,b)^c\smallsetminus \left[\frac{a}{\lambda},\frac{b}{\lambda}\right]\right\rvert\to 0$ as $\lambda\to 1$, we deduce from estimate (\ref{eq_proof_prop_integrable_uniform_bound_difference}) that
\begin{equation}
\lim_{\lambda\to 1^-}\int_{(a,b)^c}\frac{u_\lambda(x)-u(x)}{\lambda-1}(-\Delta)^\frac{1}{2}u_\lambda(x)dx=0.
\end{equation}
Therefore we conclude that
\begin{equation}
\begin{split}
\int_a^b xu'(x)(-\Delta)^\frac{1}{2}u(x)dx=&\frac{1}{2}\lim_{\lambda\to 1^-}\int_{(a,b)^c}\frac{u_\lambda(x)-u(x)}{\lambda-1}(-\Delta)^\frac{1}{2}u(x)dx\\
&\frac{1}{2}\lim_{\lambda\to 1^-}\int_{\left(\frac{a}{\lambda},a\right)\cup\left(b,\frac{b}{\lambda}\right)}\frac{u_\lambda(x)-u(x)}{\lambda-1}(-\Delta)^\frac{1}{2}u(x)dx
\end{split}
\end{equation}

We can thus apply Lemma \ref{lem_derivative_on_the_boundary} to obtain
\begin{equation}\label{eq_proof_prop_final_equation_Pohozaev}
\int_\mathbb{R}x u'(x)(-\Delta)^\frac{1}{2}u(x)dx=-\frac{\pi}{8}\left[\lim_{x\to b^-}\frac{u^2(x)}{b-x}b-\lim_{x\to a^+}\frac{u^2(x)}{x-a}a\right].
\end{equation}
This concludes the proof whenever $a<0<b$.
Now for the general case let $d\in (a,b)$ and let $\tau_d(x):=x+d$ for any $x\in \mathbb{R}$.
We observe that if $u$ solves (\ref{eq_intro_Poh_nonlocal_Cauchy}) weakly for $\Omega=(a,b)$, then $u_d:=u\circ\tau_d$ is a weak solution of
\begin{equation}
\begin{cases}
(-\Delta)^\frac{1}{2}u_d=f(u_d)\text{ in }\Omega_d\\
u_d=0\text{ in }\Omega_d^c,
\end{cases}
\end{equation}
where $\Omega_d:=[a', b']:=[a-d, b-d]$. As $a'<0<b'$, we can apply Equation (\ref{eq_proof_prop_final_equation_Pohozaev}) to $u_d$, $a'$ and $b'$.
Therefore, by the change of variable $y=d+x$
\begin{equation}\label{eq_proof_prop_computation_translation_c}
\begin{split}
\int_{\mathbb{R}}(y-d)u'(y)(-\Delta)^\frac{1}{2}u(y)dy=&\int_{\mathbb{R}}x u_d'(x)(-\Delta)^\frac{1}{2}u_d(x)dx\\
=&-\frac{\pi}{8}\left[\lim_{x\to b'^-}\frac{u_d^2(x)}{b'-x}b'-\lim_{x\to a'^+}\frac{u_d^2(x)}{x-a'}a'\right]\\
=&-\frac{\pi}{8}\left[\lim_{x\to b^-}\frac{u^2(x)}{b-x}(b-d)-\lim_{x\to a^+}\frac{u^2(x)}{x-a}(a-d)\right].
\end{split}
\end{equation}
Since Equation (\ref{eq_proof_prop_computation_translation_c}) holds for any $d\in (a,b)$, we can differentiate the first and the last expressions in (\ref{eq_proof_prop_computation_translation_c}) in $d$. We obtain
\begin{equation}\label{eq_proof_prop_Poozaev_Pohozaev_translation}
-\int_\mathbb{R}u'(y)(-\Delta)^\frac{1}{2}u(y)dy=-\frac{\pi}{8}\left[ \lim_{x\to a^+}\frac{u^2(x)}{x-a}-\lim_{x\to b^-}\frac{u^2(x)}{b-x}\right].
\end{equation}
On the one hand this proves Equation (\ref{eq_prop_statement_Pohozaev_proof_asymptotic_translations}), on the other hand, if we plug Equation (\ref{eq_proof_prop_Poozaev_Pohozaev_translation}) in (\ref{eq_proof_prop_computation_translation_c}) we obtain
\begin{equation}
\int_\mathbb{R}x u'(x)(-\Delta)^\frac{1}{2}u(x)dx=-\frac{\pi}{8}\left[\lim_{x\to b^-}\frac{u^2(x)}{b-x}b-\lim_{x\to a^+}\frac{u^2(x)}{x-a}a\right].
\end{equation}
This concludes the proof of the Proposition.
\end{proof}
\begin{rem}\label{rem_regularity_Pohozaev}
We remark that in the previous arguments, the assumption that $u$ satisfies Equation (\ref{eq_intro_Poh_nonlocal_Cauchy}) was never used directly, but only to deduce regularity properties for $u$ and $(-\Delta)^\frac{1}{2}u$. In fact, the arguments in the proofs in Sections \ref{ssec:Asymptotics} and \ref{ssec:The Pohozaev identity for open, bounded intervals} remain true for any function $u\in C^\frac{1}{2}(\mathbb{R},\mathbb{R})$ such that
\begin{enumerate}
\item $\text{supp}(u)\subset[a,b]$
\item $u^2\in C^2([a,b])$
\item $(-\Delta)^\frac{1}{2}u\in C^0([a,b])$.
\end{enumerate}
In particular, for any function $u\in C^\frac{1}{2}(\mathbb{R},\mathbb{R})$ satisfying the conditions above, Equations (\ref{eq_prop_statement_Pohozaev_proof_asymptotic_translations}) and (\ref{eq_prop_statement_Pohozaev_proof_asymptotic}) hold true.
\end{rem}

The second part of the previous result can be regarded as a special case of the following Pohozaev identity, proved by X. Ros-Oton and J. Serra in \cite{RosOtonSerra}.
\begin{thm}\label{thm_Pohozaev_identity_RosOton_Serra}
(Theorem 1.1 in \cite{RosOtonSerra}) Let $n\in\mathbb{N}_{>0}$. Let $s\in (0,1)$, let $\Omega\subset \mathbb{R}^n$ be a bounded and $C^{1,1}$ domain, $f$ be a locally Lipschitz function, $u$ be a bounded solution of 
\begin{equation}\label{eq_thm_Poh_nonlocal_Cauchy_general}
\begin{cases}
(-\Delta)^su=f(u) &\in \Omega\\
u=0 &\in \Omega^c,
\end{cases}
\end{equation}
and let 
\begin{equation}
\delta(x)=\text{dist}(x,\partial \Omega).
\end{equation}
Then
\begin{equation}
\frac{u}{\delta^s}\bigg\rvert_{\Omega}\in C^\alpha(\overline{\Omega}) \text{ for some }\alpha\in (0,1),
\end{equation}
meaning that $\frac{u}{\delta^s}\bigg\rvert_{\Omega}$ has a continuous extension to $\overline{\Omega}$ which is $C^\alpha(\overline{\Omega})$, and the following identity holds
\begin{equation}\label{eq_thm_Poh_RosOton}
(2s-n)\int_{\Omega}uf(u)dx+2n\int_{\Omega}F(u)dx=\Gamma (1+s)^2\int_{\partial\Omega}\left(\frac{u}{\delta^s}\right)^2(x\cdot\nu)d\sigma,
\end{equation}
where $F(t):=\int_0^t f(x)dx$ and $\nu$ is the unit outward normal to $\partial\Omega$ at $x$.
\end{thm}
In fact, in the case $n=1$, $s=\frac{1}{2}$, with $\Omega=(a,b)$, (\ref{eq_thm_Poh_RosOton}) reduces to
\begin{equation}\label{eq_Poh_simplified}
2\int_a^bF(u)dx=\frac{\pi}{4}\left[\lim_{x\to a^+}\frac{u^2(x)}{x-a}a-\lim_{x\to b^-}\frac{u^2(x)}{b-x}b\right];
\end{equation}
integrating by part the term on the left hand side we obtain Equation (\ref{eq_prop_statement_Pohozaev_proof_asymptotic}).

\appendix

\section{Results about Sobolev spaces and harmonic extensions}

\subsection{Results for function on Euclidean spaces}
In the following, let $n\in \mathbb{N}_{>0}$.

\begin{lem}\label{lem_equivalent_definitions_homogeneous_sobolev}
Let $s\in (0,1)$ and let $\dot{W}^{s,2}(\mathbb{R}^n)$, $\dot{H}^s(\mathbb{R}^n)$ be defined as in (\ref{def_hom_sobolev_w}) and (\ref{def_hom_sobolev_h}), then
\begin{equation}
\dot{W}^{s,2}(\mathbb{R}^n)=\dot{H}^s(\mathbb{R}^n),
\end{equation}
as sets, and the respective seminorms are equivalent, so that the induced normed spaces are isomorphic.
\end{lem}
\begin{proof}
(Adapted from the proof of Proposition 3.4 in \cite{Hitchhiker}) Let $u\in\dot{W}^{s,2}(\mathbb{R}^n)$, then
\begin{equation}\label{eq_proof_lem_app_explicit_form_norm}
[u]_{\dot{W}^{s,2}}^2=\int_{\mathbb{R}^n}\int_{\mathbb{R}^n}\frac{\lvert u(x)-u(y)\rvert}{\lvert x-y\rvert^{n+2s}}dxdy=\int_{\mathbb{R}^n}\left(\int_{\mathbb{R}^n}\left\lvert\frac{u(y+z)-u(y)}{\lvert z\rvert^{\frac{n}{2}+s}}\right\rvert^2dy\right)dz.
\end{equation}
Thus the internal integral on the right side of (\ref{eq_proof_lem_app_explicit_form_norm}) is finite for a.e. $z\in \mathbb{R}^n$. For any such $z$, by Plancherel's identity
\begin{equation}\label{eq_proof_lem_app_explicit_form_norm2}
\begin{split}
\int_{\mathbb{R}^n}\left\lvert\frac{u(y+z)-u(y)}{\lvert z\rvert^{\frac{n}{2}+s}}\right\rvert^2dy&=\left\lVert \mathscr{F}\left(\frac{u(\cdot+z)-u(\cdot)}{\lvert z\rvert^{\frac{n}{2}+s}}\right)\right\rVert_{L^2}^2\\
&=\left\lVert \frac{e^{iz\cdot\xi}-1}{\lvert z\rvert^{\frac{n}{2}+s}}\mathscr{F}u(\xi)\right\rVert_{L^2}^2.
\end{split}
\end{equation}
Since this quantity is bounded for a.e. $z\in\mathbb{R}^n$, $\mathscr{F}u$ is locally square integrable in $\mathbb{R}^n\smallsetminus\lbrace 0\rbrace$. Therefore $\mathscr{F}u$ can be rewritten as the sum of a  locally square integrable function $f$ with the property that
\begin{equation}
\frac{e^{iz\cdot\xi}-1}{\lvert z\rvert^{\frac{n}{2}+s}}f(\xi)\in L^2(\mathbb{R}^n)
\end{equation}
for a.e. $z\in\mathbb{R}^n$, and a distribution $g$ supported in $\lbrace 0\rbrace$. By Proposition 2.4.1 in \cite{Grafakos1}, there exist $k\in \mathbb{N}$ and complex numbers $c_{\alpha}$ such that
\begin{equation}
g=\sum_{\lvert \alpha\rvert\leq k } c_\alpha\partial^{\alpha}\delta_{0},
\end{equation}
where $\alpha$ runs over all $n$-multiindices of degree smaller or equal to $k$.
Now let $\phi\in\mathscr(\mathbb{R}^n)$. Then for a.e. $z\in \mathbb{R}^n$
\begin{equation}\label{eq_proof_lem_app_contradiction_higher_order}
\begin{split}
\int_{\mathbb{R}^n}\frac{e^{iz\cdot\xi}-1}{\lvert z\rvert^{\frac{n}{2}+s}}\mathscr{F}u(\xi)\phi(\xi)d\xi=&\int_{\mathbb{R}^n}f(\xi)\frac{e^{iz\cdot\xi}-1}{\lvert z\rvert^{\frac{n}{2}+s}}\phi(\xi)d\xi\\
&+(-1)^{\lvert \alpha\rvert}\sum_{\lvert \alpha\rvert\leq k } c_\alpha\partial^\alpha\bigg\vert_{\xi=0}\left(\frac{e^{iz\cdot\xi}-1}{\lvert z\rvert^{\frac{n}{2}+s}}\phi(\xi) \right)
\end{split}
\end{equation}
For a.e. $z\in \mathbb{R}^n$, the left hand side can be bounded by $\lVert\phi\rVert_{L^2}$, where $C$ is a constant independent from $\phi$. The same holds for the first term on the right side, and thus also for the second. This implies that $c_\alpha=0$ for any multiindex $\alpha$ such that $\lvert \alpha\rvert>0$.
Therefore $\mathscr{F}u$ is a tempered distribution of order $0$, and combining (\ref{eq_proof_lem_app_explicit_form_norm}) and (\ref{eq_proof_lem_app_explicit_form_norm2}) we obtain
\begin{equation}
[u]_{\dot{W}^{s,2}}^2=\int_{\mathbb{R}^n}\int_{\mathbb{R}^n}\left\lvert\frac{e^{iz\cdot \xi}-1}{\lvert z\rvert^{\frac{n}{2}-s}}\right\rvert^2\lvert\mathscr{F}u(\xi)\rvert^2d\xi dz
\end{equation}
By (3.12) in \cite{Hitchhiker},
\begin{equation}
\int_{\mathbb{R}^n}\left\lvert\frac{e^{iz\cdot \xi}-1}{\lvert z\rvert^{\frac{n}{2}-s}}\right\rvert^2dz=C_{n,s}\lvert \xi\rvert^{2s}
\end{equation}
for a positive constant $C_{n,s}$ depending on $n$ and $s$. Therefore
\begin{equation}
[u]_{\dot{W}^{s,2}}^2=C_{n,s}\int_{\mathbb{R}^n}\lvert \xi\rvert^{2s}\lvert\mathscr{F}u(\xi)\rvert^2 d\xi 
\end{equation}
Thus $u\in \dot{H}^\frac{1}{2}(\mathbb{R}^n)$ and $[u]_{\dot{H}^s}=C_{n,s}^{\frac{1}{2}}[u]_{W^{s,2}}$.\newline
Conversely, assume that $u\in \dot{H}^s(\mathbb{R}^n)$, then $\lvert\xi\rvert^s\mathscr{F}u(\xi)\in L^2(\mathbb{R}^n)$. First we claim that $u\in L^2_{loc}(\mathbb{R}^n)$. In fact, let $\eta\in C^\infty_c(B_1(0), [0,1])$ and such that $\eta\equiv 1$ in $B_\frac{1}{2}(0)$. Then we can rewrite the distribution $\mathscr{F}u$ as follows
\begin{equation}\label{eq_proof_lemma_appendix_decomposition_Fu_for_local_integrability}
\mathscr{F}u= \eta \mathscr{F}u+(1-\eta)\mathscr{F}u.
\end{equation}
Now the first summand in (\ref{eq_proof_lemma_appendix_decomposition_Fu_for_local_integrability}) is compactly supported, therefore its inverse Fourier transform is smooth (by Theorem 2.3.21 in \cite{Grafakos1}), while the second summand lies in $L^2(\mathbb{R}^n)$, since $\lvert\xi\rvert^s\mathscr{F}u(\xi)\in L^2(\mathbb{R}^n)$ and $(1-\eta)\mathscr{F}u$ is equal to $0$ in $B_\frac{1}{2}(0)$. Therefore the inverse Fourier transform of the second summand lies in $L^2(\mathbb{R}^n)$. We conclude that $u\in L^2_{loc}(\mathbb{R}^n)$.
Moreover, since $\lvert\xi\rvert^s\mathscr{F}u(\xi)\in L^2(\mathbb{R}^n)$, one can use again equation (3.12) in \cite{Hitchhiker} to obtain
\begin{equation}
[u]_{\dot{W}^{s,2}}^2=\int_{\mathbb{R}^n}\left(\lvert \xi\rvert\mathscr{F}u(\xi)\right)^2d\xi=C_{n,s}^{-1}\int_{\mathbb{R}^n}\int_{\mathbb{R}^n}\left(\left\lvert\frac{e^{iz\cdot \xi-1}}{\lvert z\rvert^{\frac{n}{2}+s}}\right\rvert\mathscr{F}u(\xi)\right)^2d\xi dz.
\end{equation}
Thus, by (\ref{eq_proof_lem_app_explicit_form_norm}) and (\ref{eq_proof_lem_app_explicit_form_norm2}), $u\in \dot{W}^{s,2}$ and $[u]_{W^{s,2}}=C_{n,s}^{-\frac{1}{2}}[u]_{\dot{H}^s}$.
\end{proof}

\begin{lem}\label{lem_DaLio_Pigati_locally_L2}
(Adapted from Lemma B.1 in \cite{FreeBoundary}) Given $u\in \dot{H}^\frac{1}{2}(\mathbb{R}^n)$, for any $j\in\mathbb{N}$ there holds
\begin{equation}
\lVert u\rVert_{L^2\left(B_{2^j}(0)\right)}\leq C_{n,s}2^{j(s+\frac{n}{2})}[u]_{\dot{H}^\frac{1}{2}} +C_n2^\frac{jn}{2}\left\lvert (u)_{B_1(0)}\right\rvert
\end{equation}
for some independent constants $C_{n,s}$ and $C_n$ depending respectively on $n,s$ and $n$.
\end{lem}

\begin{proof}
Let $u\in \dot{H}^\frac{1}{2}(\mathbb{R}^n)$. For the following, denote $\alpha(n)$ the volume of the ball of radius $1$ in $\mathbb{R}^n$.
First we observe that for any $j\in \mathbb{N}$
\begin{equation}
\begin{split}
\left\lVert u-(u)_{B_{2^j}(0)}\right\rVert_{L^2\left(B_{2^j}(0)\right)}^2=&\int_{B_{2^j}(0)}\left( \frac{1}{\left\lvert B_{2^j}(0)\right\rvert}\int_{B_{2^j}(0)}u(x)-u(y)dy\right)^2dx\\
\leq&\int_{B_{2^j}(0)}\frac{1}{\left\lvert B_{2^j}(0)\right\rvert}\int_{B_{2^j}(0)}\lvert u(x)-u(y)\rvert^2 dydx\\
=&\frac{1}{\alpha(n)2^{jn}}\int_{B_{2^j}(0)}\int_{B_{2^j}(0)}\lvert u(x)-u(y)\rvert^2dxdy\\
\leq&\frac{2^{n+2s+2js}}{\alpha(n)}\int_{B_{2^j}(0)}\int_{B_{2^j}(0)}\frac{\lvert u(x)-u(y)\rvert}{\lvert x-y\rvert^{n+2s}}dxdy.
\end{split}
\end{equation}
Therefore
\begin{equation}
\left\lVert u-(u)_{B_{2^j}(0)}\right\rVert_{L^2\left(B_{2^j}(0)\right)}\leq \frac{2^{\frac{n}{2}+s+js}}{\alpha(n)^\frac{1}{2}}[u]_{\dot{H}^\frac{1}{2}(\mathbb{R}^n)}
\end{equation}
Thus for any $j\in \mathbb{N}$
\begin{equation}
\begin{split}
\lvert (u)_{B_{2^{j-1}}(0)}-(u)_{B_{2^{j}}(0)}=& \left\lvert \frac{1}{\left\lvert B_{2^{j-1}}(0)\right\rvert}\int_{B_{2^{j-1}}(0)} u(x)-(u)_{B_{2^{j}}(0)} dx \right\rvert\\
\leq& \frac{1}{\alpha(n)2^{(j-1)n}}\int_{B_{2^{j}}(0)}\lvert u(x)-(u)_{B_{2^{j}}(0)}\rvert dx\\
\leq & \frac{1}{\alpha(n)2^{(j-1)n}} \left\lvert B_{2^j}(0)\right\rvert^\frac{1}{2}\lVert u-(u)_{B_{2^j}(0)}\rVert_{L^2\left( B_{2^j}(0)\right)}\\
\leq & \frac{1}{\alpha(n)2^{(j-1)n}} (\alpha(n)2^{jn})^\frac{1}{2}\frac{2^{\frac{n}{2}+s+js}}{\alpha(n)^\frac{1}{2}}[u]_{\dot{H}^\frac{1}{2}(\mathbb{R}^n)}\\
=&\frac{2^{-\frac{jn}{2}+\frac{3}{2}n+(1+j)s}}{\alpha(n)}[u]_{\dot{H}^\frac{1}{2}(\mathbb{R}^n)}.
\end{split}
\end{equation}
We conclude that
\begin{equation}
\begin{split}
\lVert u\rVert_{L^2\left(B_{2^j}(0)\right)}\leq &\lVert u-(u)_{B_{2^j}(0)}\rVert_{L^2\left(B_{2^j}(0)\right)}+\lVert (u)_{B_1}\rVert_{L^2\left(B_{2^j}(0)\right)}\\
&+\sum_{l=1}^j\lVert (u)_{B_{2^{l-1}}(0)}-(u)_{B_{2^{l}}(0)}\rVert_{L^2\left(B_{2^j}(0)\right)}\\
\leq & \frac{2^{\frac{n}{2}+(1+j)s}}{\alpha(n)^{\frac{1}{2}}}[u]_{\dot{H}^\frac{1}{2}(\mathbb{R}^n)}+\alpha(n)^\frac{1}{2}2^\frac{jn}{2}\left\lvert(u)_{B_1(0)}\right\rvert\\
&+2^\frac{jn}{2}\alpha(n)^\frac{1}{2}\sum_{l=1}^j\left\lvert (u)_{B_{2^{l-1}}(0)}-(u)_{B_{2^l}(0)}\right\rvert\\
\leq&\left( \frac{2^{\frac{n}{2}+(1+j)s}}{\alpha(n)^{\frac{1}{2}}}+2^{\frac{n}{2}+s}\left(\frac{1-2^{(\frac{n}{2}+s)j}}{1-2^{\frac{n}{2}+s}}\right)\frac{2^{\frac{3}{2}n+s}}{\alpha(n)^\frac{1}{2}} \right)[u]_{\dot{H}^\frac{1}{2}(\mathbb{R}^n)}\\
&+\alpha(n)^\frac{1}{2}2^\frac{jn}{2}\left\lvert(u)_{B_1(0)}\right\rvert.
\end{split}
\end{equation}
This concludes the proof of the Lemma.
\end{proof}

\begin{lem}\label{lem_approx_sobolev}
(Adapted from Lemma B.2 in \cite{FreeBoundary}) Let $s\in (0,2)$ and let $u\in \dot{H}^s(\mathbb{R}^n)$. Then there exists a sequence $(u_k)_k$ in $\mathscr{S}(\mathbb{R}^n)$ and a sequence $(c_k)_k$ in $\mathbb{R}^m$ such that
\begin{equation}
\lim_{k\to \infty} [u_k-u]_{\dot{H}^s}=\lim_{k\to \infty} [u_k+c_k-u]_{\dot{H}^s}=0,
\end{equation}
for any compact $K\in \mathbb{R}^n$
\begin{equation}
\lim_{k\to \infty}\lVert u_k+c_k-u\rVert_{L^2(K)}=0
\end{equation}
and for all $k\in \mathbb{N}$, $\widehat{u_k}\equiv 0$ in a neighbourhood of the origin.
\end{lem}
\begin{proof}
For $k\in\mathbb{N}_{\geq 2}$, let $\phi_k\in C_c^\infty(\mathbb{R}^n, [0,1])$ such that
\begin{equation}
\widehat{\phi_k}(\xi)=\begin{cases}
1 &\text{if $\frac{1}{k}\leq \lvert \xi\rvert\leq k$}\\
0 &\text{if $\lvert \xi \rvert\leq \frac{1}{2k}$ or $2k\leq \lvert \xi\rvert$}.
\end{cases}
\end{equation}
Let $\psi\in C^\infty_0(B_1(0))$ such that $\int_{\mathbb{R}^n}\psi(x)dx=1$.
For any $t\in (0,1)$ let $\psi_t(x):=\frac{1}{t}\psi\left(\frac{x}{t}\right)$ for all $x\in \mathbb{R}^n$.
For any $k\in\mathbb{N}_{\geq 2}$ let $t_k\in (0,1)$ such that
\begin{equation}
\left\lVert\psi_{t_k}\ast(\widehat{u_k}\mathds{1}_{B_{2k+1}(0)})-\widehat{u_k}\mathds{1}_{B_{2k+1}(0)}\right\rVert_{L^2}^2\leq \frac{1}{k(2k)^{2s}}
\end{equation}
For any $k\in \mathbb{N}_{\geq 2}$, let $u_k$ be defined by
\begin{equation}
\widehat{u_k}:=\left( \psi_{t_k}\ast\widehat{u} \right)\phi_k.
\end{equation}
Then, by construction, for any $k\in\mathbb{N}_{\geq 2}$ $\widehat{u_k}\in \mathscr{S}(\mathbb{R}^n)$ and thus $u_k\in \mathscr{S}(\mathbb{R}^n)$. Moreover $\widehat{u_k}\equiv 0$ in $B_{\frac{1}{2k}}(0)$.
Now we compute for any $k\in\mathbb{N}_{\geq 2}$
\begin{equation}
[ u_k-u]_{\dot{H}^s}^2\leq 2\int_{\mathbb{R}^n}\lvert\xi\rvert^{2s}(\psi_{t_k}\ast \widehat{u}(\xi)-\widehat{u}(\xi))^2\phi_k(\xi)^2d\xi+2\int_{\mathbb{R}^n}\lvert \xi\rvert^{2s}\widehat{u}(\xi)^2(\phi_k(\xi)-1)^2d\xi.
\end{equation}
We observe that
\begin{equation}\label{eq_proof_lemma_appendix_first_pert_triangular}
\begin{split}
&\left\lvert \int_{\mathbb{R}^n}\lvert\xi\rvert^{2s}(\psi_{t_k}\ast \widehat{u}(\xi)-\widehat{u}(\xi))^2\phi_k(\xi)^2d\xi \right\rvert\\
\leq &\left\lVert\left(\psi_{t_k}\ast(\widehat{u_k}\mathds{1}_{B_{2k+1}(0)})-\widehat{u_k}\mathds{1}_{B_{2k+1}(0)}\right)^2\right\rVert_{L^1} \lVert \lvert \xi\rvert^{2s}\phi_k\rVert_{L^\infty}\\
\leq& C_n \frac{1}{k( 2k)^{2s}}( 2k)^{2s}\leq C_n\frac{1}{k}
\end{split}
\end{equation}
for some constant $C_n$ depending on $n$. On the other hand, since $u\in \dot{H}^\frac{1}{2}$, $\lvert \xi\rvert^{2s}\widehat{u}(\xi)\in L^2(\mathbb{R}^2)$. Therefore, by Lebesgue's Dominated convergence Theorem, 
\begin{equation}\label{eq_proof_lemma_appendix_second_pert_triangular}
\lim_{k\to \infty}\int_{\mathbb{R}^n}\lvert \xi\rvert^{2s}\widehat{u}(\xi)^2(\phi_n(\xi)-1)^2d\xi=0.
\end{equation}
Combining (\ref{eq_proof_lemma_appendix_first_pert_triangular}) and (\ref{eq_proof_lemma_appendix_second_pert_triangular}) we obtain
\begin{equation}\label{eq_proof_lemma_appendix_convergence_Sobolev_seminorm}
\lim_{k\to \infty} [ u_n-u]_{\dot{H}^\frac{1}{2}}=0.
\end{equation}
For any $k\in\mathbb{N}_{\geq 2}$ choose $c_k$ such that $(u_k+c_k)_{B_1(0)}=(u)_{B_1(0)}$. Then, for any fixed $j\in\mathbb{N}$, Lemma \ref{lem_DaLio_Pigati_locally_L2} implies that
\begin{equation}
\lVert u_k+c_k-u\rVert_{L^2(B_{2^j}(0))}\leq C_{n,j}[u_k+c_k-u]_{\dot{H}^s}=C_{n,j}[u_k-u]_{\dot{H}^s}.
\end{equation}
By (\ref{eq_proof_lemma_appendix_convergence_Sobolev_seminorm}) we obtain
\begin{equation}
\lim_{k\to\infty}\lVert u_k+c_k-u\rVert_{L^2(B_{2^j}(0))}=0.
\end{equation}
\end{proof}

\begin{lem}\label{lem_Hs_cap_Linfty_is_an_algebra}
Let $n\in\mathbb{N}_{>0}$ and let $s\in (0,1)$. Then $\dot{W}^{s,2}\cap L^\infty(\mathbb{R}^n)$ and $H^s\cap L^\infty(\mathbb{R}^n)$ are closed under multiplication of functions. In particular, $H^s\cap L^\infty(\mathbb{R}^n)$ is a Banach algebra. 
\end{lem}
\begin{proof}
Let $u, v\in \dot{H}^s\cap L^\infty(\mathbb{R}^n)$. Then
\begin{equation}
\begin{split}
[uv]_{\dot{W}^{s,2}}^2=&\int_{\mathbb{R}^n}\int_{\mathbb{R}^n}\frac{\lvert uv(x)-uv(y)\rvert^2}{\lvert x-y\rvert^{2s+n}}dxdy\\ \leq&\lVert u\rVert_{L^\infty}\int_{\mathbb{R}^n}\int_{\mathbb{R}^n}\frac{\lvert v(x)-v(y)\rvert^2}{\lvert x-y\rvert^{2s+n}}dxdy\\
&+\lVert v\rVert_{L^\infty}\int_{\mathbb{R}^n}\int_{\mathbb{R}^n}\frac{\lvert u(x)-u(y)\rvert^2}{\lvert x-y\rvert^{2s+n}}dxdy\\
=&\lVert u\rVert_{L^\infty}[v]_{\dot{W}^{s,2}}^2+\lVert v\rVert_{L^\infty}[u]_{\dot{W}^{s,2}}^2.
\end{split}
\end{equation}
Moreover,
\begin{equation}
\lVert uv\rVert_{L^\infty}\leq\lVert u\rVert_{L^\infty}\lVert v\rVert_{L^\infty}.
\end{equation}
Therefore $uv\in \dot{W}^{s,2}\cap L^\infty(\mathbb{R}^n)$.
\\Finally, if $u,v \in H^s\cap L^\infty(\mathbb{R}^n)$, by the previous argument $uv\in \dot{W}^{s,2}\cap L^\infty(\mathbb{R}^n)$. Moreover, there holds 
\begin{equation}
\lVert uv\rVert_{L^2}\leq \lVert u\rVert_{L^2}\lVert v\rVert_{L^\infty},
\end{equation}
and therefore $uv\in H^s\cap L^\infty(\mathbb{R}^n)$.
\end{proof}

\subsection{Results for function on $\mathbb{S}^1$}
\begin{lem}\label{lem_appendix_circle_eqivalent_Sobolev_seminorm}
For any measurable function $u\in \mathscr{D}'(\mathbb{S}^1)$, $u\in \dot{H}^\frac{1}{2}(\mathbb{S}^1)$ if and only if
\begin{equation}\label{eq_lem_equivalent_seminrom_H12}
\int_{\mathbb{S}^1}\int_{\mathbb{S}^1}\frac{\lvert u(x)-u(y)\rvert^2}{\sin^2\left(\frac{1}{2}(x-y)\right)}dxdy<\infty,
\end{equation}
and if (\ref{eq_lem_equivalent_seminrom_H12}) holds, then
\begin{equation}
[u]_{\dot{H}^\frac{1}{2}}^2=\frac{1}{4(2\pi)^2}\int_{\mathbb{S}^1}\int_{\mathbb{S}^1}\frac{\lvert u(x)-u(y)\rvert^2}{\sin^2\left(\frac{1}{2}(x-y)\right)}dxdy.
\end{equation}
\end{lem}
\begin{proof}
Let $u$ be a measurable function in $\mathscr{D}'(\mathbb{S}^1)$. Then we compute
\begin{equation}\label{eq_proof_lemma_equivalent_seminorm_H12}
\begin{split}
&\int_{\mathbb{S}^1}\int_{\mathbb{S}^1}\frac{\lvert u(x)-u(y)\rvert^2}{\sin^2\left(\frac{1}{2}(x-y)\right)}dxdy=\int_{\mathbb{S}^1}\int_{\mathbb{S}^1}\frac{\lvert u(z+y)-u(y)\rvert^2}{\sin^2\left(\frac{z}{2}\right)}dzdy\\
=&\int_{\mathbb{S}^1} \frac{\lVert u(\cdot+z)-u(\cdot)\rVert_{L^2}^2}{\sin^2\left(\frac{z}{2}\right)}dz
=2\pi\int_{\mathbb{S}^1}\frac{1}{\sin^2\left(\frac{z}{2}\right)}\sum_{k\in\mathbb{Z}}\lvert e^{ik\cdot z}-1\rvert^2\lvert \widehat{u}(k)\rvert^2dz,
\end{split}
\end{equation}
where in the last step we used Plancherel's Identity. Now we observe that if $k\in\mathbb{Z}_{\geq 0}$, if $w=e^{iz}$,
\begin{equation}\label{eq_proof_lemma_computation_residue}
\begin{split}
\frac{\lvert e^{ikz}-1\rvert^2}{\lvert e^{\frac{i}{2}z}-e^{-\frac{i}{2}z}\rvert^2}=&\frac{2-e^{ikz}-e^{-ikz}}{2-e^{iz}-e^{-iz}}
=\frac{2e^{ikz}-e^{i2kz}-1}{(2e^{iz}-e^{i2z}-1)e^{i(k-1)z}}\\
=&\frac{2w^k-w^{2k}-1}{(2w-w^2-1)w^{k-1}}
=\frac{(w^k-1)^2}{(w-1)^2w^{k-1}}
=\frac{(\sum_{l=0}^{k-1}w^l)^2}{w^{k-1}}.
\end{split}
\end{equation}
Expanding the product in the numerator, we observe that the coefficient corresponding to $w^{k-1}$ is $k$. Thus
\begin{equation}
Res\left[\frac{(\sum_{l=0}^{k-1}w^l)^2}{w^{k}}\right](0)=k.
\end{equation}
Then, by (\ref{eq_proof_lemma_computation_residue}) and the Residue Theorem,
\begin{equation}
\int_{\mathbb{S}^1} \frac{\lvert e^{ikz}-1\rvert^2}{\sin^2(\frac{z}{2})}dx=\frac{4}{i}\int_{\partial D^2}\frac{(\sum_{l=0}^{k-1}w^l)^2}{w^{k}} dz=8\pi k.
\end{equation}
Moreover, if $k\in\mathbb{Z}_{<0}$, $\lvert e^{ikx}-1\rvert^2=\lvert e^{-ikx}-1\rvert^2$, thus for any $k\in \mathbb{Z}$
\begin{equation}\label{eq_proof_lemma_result_computation_residues}
\int_{\mathbb{S}^1} \frac{\lvert e^{ikx}-1\rvert^2}{\sin^2(\frac{x}{2})}dx=8\pi \lvert k\rvert.
\end{equation}
Plugging (\ref{eq_proof_lemma_result_computation_residues}) in (\ref{eq_proof_lemma_equivalent_seminorm_H12}) and applying Tonelli's theorem, we obtain
\begin{equation}
\int_{\mathbb{S}^1}\int_{\mathbb{S}^1}\frac{\lvert u(x)-u(y)\rvert^2}{\sin^2\left(\frac{1}{2}(x-y)\right)}dxdy=4\sum_{k\in\mathbb{Z}}(2\pi)^2\lvert k\rvert\lvert \widehat{u}(k)\rvert^2=4(2\pi)^2[u]_{\dot{H}^\frac{1}{2}}^2.
\end{equation}
This concludes the proof.
\end{proof}

\begin{lem}\label{lem_multiplication_with_test_function_continuous}
Let $u\in H^\frac{1}{2}(\mathbb{S}^1)$, $\psi\in C^3(\mathbb{S}^1)$. Then $\psi u\in H^\frac{1}{2}(\mathbb{S}^1)$,
\begin{equation}
[ u\psi]_{\dot{H}^\frac{1}{2}}\leq C [ u]_{H^\frac{1}{2}}\lVert \psi\rVert_{C^3},
\end{equation}
and
\begin{equation}
\lVert u\psi\rVert_{H^\frac{1}{2}}\leq C \lVert u\rVert_{H^\frac{1}{2}}\lVert \psi\rVert_{C^3},
\end{equation}
for some independent constant $C$.
\end{lem}
\begin{proof}
We compute
\begin{equation}\label{eq_proof_lemma_multiplication_with_test_function_1}
\begin{split}
[ u\psi]_{\dot{H}^\frac{1}{2}}^2
=&\sum_{n\in\mathbb{Z}}\widehat{\psi u}(n)^2\lvert n\rvert=\sum_{n\in\mathbb{Z}}\left(\sum_{r\in\mathbb{Z}}\hat{\psi}(r)\hat{u}(n-r)\right)^2\lvert n\rvert
\\\leq &2\sum_{n\in\mathbb{Z}}\lvert n\rvert\left(\left(\sum_{\lvert r\rvert\leq \frac{\lvert n\rvert}{2}}\widehat{\psi}(r)\widehat{u}(n-r)\right)^2+\left(\sum_{\lvert r\rvert> \frac{\lvert n\rvert}{2}}\widehat{\psi}(r)\widehat{u}(n-r)\right)^2\right),
\end{split}
\end{equation}
We recall that for any $n, k\in\mathbb{N}$,
\begin{equation}
\widehat{\psi^{(k)}}(n)=(in)^k\widehat{\psi}(n).
\end{equation}
Moreover we observe that for any $n\in \mathbb{Z}\smallsetminus\lbrace 0\rbrace$
\begin{equation}
\lvert \hat{u}(n)\rvert^2\leq \frac{[u]_{\dot{H}^\frac{1}{2}}^2}{\lvert n\rvert}\leq [ u]_{\dot{H}^\frac{1}{2}}^2,
\end{equation}
thus for any $n\in\mathbb{Z}$ $\lvert \hat{u}(n)\rvert\leq [ u]_{\dot{H}^\frac{1}{2}}$.\\
Now we observe that if $\lvert r\rvert\leq \frac{\lvert n\rvert}{2}$, then $\lvert n\rvert\leq 2\lvert n-r\rvert$.
Therefore
\begin{equation}\label{eq_proof_lemma_multiplication_with_test_function_2}
\begin{split}
&\sum_{n\in\mathbb{Z}}\lvert n\rvert \left(\sum_{\lvert r\rvert\leq\frac{\lvert n\rvert}{2}}\widehat{\psi}(r)\widehat{u}(n-r)\right)^2\\
\leq& \sum_{n\in\mathbb{Z}} 2\left(\sum_{0<\lvert r\rvert\leq\frac{\lvert n\rvert}{2}}2^\frac{1}{2}\lvert n-r\rvert^\frac{1}{2}\left\lvert\frac{\widehat{\psi^{(2)}}(r)}{(i r)^2}\right\rvert\widehat{u}(n-r)\right)^2+2\sum_{n\in\mathbb{Z}}\lvert n\rvert\left(\widehat{\psi}(0)\widehat{u}(n)\right)^2\\
\leq & 4 \sum_{n\in\mathbb{Z}}\sum_{r\in \mathbb{Z}\smallsetminus\lbrace0\rbrace}\frac{1}{r^2}\sum_{\lvert r\rvert\leq\frac{\lvert n\rvert}{2}}\frac{\lVert \psi^{(2)}\rVert_{L^1}^2}{ r^2}\widehat{u}(n-r)^2\lvert n-r\rvert+2\lVert \psi\rVert_{L^1}^2\sum_{n\in\mathbb{Z}}\hat{u}(n)^2\lvert n\rvert\\
\leq&\frac{2}{3}\pi^2\sum_{j\in\mathbb{Z}}\left(\sum_{\lvert r\rvert\leq \frac{\lvert j+r\rvert}{2}}\frac{\lVert \psi^{(2)}\rVert_{L^1}^2}{r^2}\right)\widehat{u}(j)^2\lvert j\rvert+2\lVert \psi\rVert_{L^1}^2[ u]_{\dot{H}^\frac{1}{2}}^2\leq C_0\lVert \psi\rVert_{C^2}^2[ u]_{\dot{H}^\frac{1}{2}}^2
\end{split}
\end{equation}
for some independent constant $C_0$. Here in the second step we used the Cauchy Schwarz inequality. On the other hand
\begin{equation}\label{eq_proof_lemma_multiplication_with_test_function_3}
\begin{split}
&\sum_{n\in\mathbb{Z}}\lvert n\rvert \left(\sum_{0<\frac{\lvert n\rvert}{2}<\lvert r\rvert }\widehat{\psi}(r)\widehat{u}(n-r)\right)^2
\leq \sum_{n\in \mathbb{Z}\smallsetminus\lbrace 0\rbrace}\lvert n\rvert [u]_{\dot{H}^\frac{1}{2}}^2\left(\sum_{\frac{\lvert n\rvert}{2}<\lvert r\rvert} \left\lvert\frac{\widehat{\psi^{(3)}}(r)}{( i r)^3}\right\rvert \right)^2\\
\leq & [u]_{\dot{H}^\frac{1}{2}}^2\lVert \psi^{(3)}\rVert_{L^1}^2\sum_{n\in \mathbb{Z}\smallsetminus\lbrace 0\rbrace}\lvert n\rvert\left( \sum_{\frac{\lvert n\rvert}{2}<\lvert r\rvert}\frac{1}{\lvert r\rvert^3}\right)^2\\
\leq & [u]_{\dot{H}^\frac{1}{2}}^2\lVert \psi^{(3)}\rVert_{L^1}^2\sum_{n\in\mathbb{Z}\smallsetminus\lbrace 0\rbrace} \lvert n\rvert\left(C_1\frac{1}{\lvert n\rvert^2}\right)^2\leq C_2 [u]_{\dot{H}^\frac{1}{2}}^2\lVert \psi^{(3)}\rVert_{L^1}^2
\end{split}
\end{equation}
for some independent constants $C_1$, $C_2$.\\
Combining (\ref{eq_proof_lemma_multiplication_with_test_function_1}), (\ref{eq_proof_lemma_multiplication_with_test_function_2}), (\ref{eq_proof_lemma_multiplication_with_test_function_3}) we obtain
\begin{equation}
[u\psi]_{\dot{H}^\frac{1}{2}}\leq C[u]_{\dot{H}^\frac{1}{2}}\lVert \psi\rVert_{C^3}.
\end{equation}
for some independent constant $C$.
Moreover
\begin{equation}
\lVert u\psi\rVert_{L^2}\leq \lVert \psi\rVert_{L^\infty}\lVert u\rVert_{L^2}.
\end{equation}
Therefore, by (\ref{eq_discussion_preliminaries_circle_bound_Sobolev_norm_through_seminorm}),
\begin{equation}
\lVert u\psi\rVert_{H^\frac{1}{2}}\leq C \lVert u\rVert_{H^\frac{1}{2}}
\end{equation}
for some independent constant $C$.
\end{proof}

\begin{lem}\label{lem_approximation_in_H_12}
Let $u\in H^\frac{1}{2}(\mathbb{S}^1)$. Let $\phi\in C^\infty(\mathbb{S}^1)$ be a non-negative function such that $\int_{\mathbb{S}^1}\phi(x)dx=1$.
For any $\varepsilon>0$, let $\phi_\varepsilon(x):=\frac{1}{\varepsilon}\phi\left(\frac{x}{\varepsilon}\right)$ for all $x\in \mathbb{S}^1$ and let $u_\varepsilon:=\phi_\varepsilon\ast u$. Then $u_\varepsilon \to u$ in $H^\frac{1}{2}(\mathbb{S}^1)$ as $\varepsilon\to 0$, and almost everywhere along a subsequence.
\end{lem}
\begin{proof}
For any $n\in\mathbb{Z}$, $\varepsilon>0$ we compute
\begin{equation}\label{eq_proof_lemma_appendix_approximation_Fcoeff_mollifier}
\widehat{\phi_\varepsilon}(n)=\int_0^{2\pi}e^{-inx}\frac{1}{\varepsilon}\phi\left(\frac{x}{\varepsilon}\right)dx=\int_0^\frac{2\pi}{\varepsilon}e^{-iny\varepsilon}\phi(y)dy,
\end{equation}
where $\phi$ is regarded as a function supported on $[0,1]$.
Therefore $\lvert\widehat{\phi_\varepsilon}(n)\rvert\leq 1$ (as $\lVert\phi\rVert_{L^1}=1$), and since the integrand in the right expression of (\ref{eq_proof_lemma_appendix_approximation_Fcoeff_mollifier}) is bounded by $\phi$, by Lebesgue's Dominated convergence Theorem $\widehat{\phi_\varepsilon}(n)\to 1$ as $\varepsilon\to 0$ for any $n\in \mathbb{Z}$.
Now
\begin{equation}\label{eq_proof_lemma_appendix_convergence_1}
[u_\varepsilon-u]_{\dot{H}^\frac{1}{2}}^2=\sum_{n\in\mathbb{Z}}(\widehat{\phi_\varepsilon\ast u}(n)-\widehat{u}(n))^2\lvert n\rvert=\sum_{n\in\mathbb{Z}}(\widehat{\phi_\varepsilon}(n)\widehat{u}(n)-\widehat{u}(n))^2\lvert n\rvert,
\end{equation}
Therefore, by dominated convergence, we conclude that $[u_\varepsilon-u]_{\dot{H}^\frac{1}{2}}^2\to 0$.\\
Similarly,
\begin{equation}\label{eq_proof_lemma_appendix_convergence_2}
\lVert u-u_\varepsilon\rVert_{L^2}^2=\sum_{n\in\mathbb{Z}}(\widehat{\phi_\varepsilon\ast u}(n)-\widehat{u}(n))^2=\sum_{n\in\mathbb{Z}}\left(\widehat{\phi_\varepsilon}(n)\hat{u}(n)-\hat{u}(n)\right)^2.
\end{equation}
Again, by dominated convergence, we conclude that $\lVert u_n-u\rVert_{L^2}^2\to 0$. By (\ref{eq_discussion_preliminaries_circle_bound_Sobolev_norm_through_seminorm}), (\ref{eq_proof_lemma_appendix_convergence_1}) and (\ref{eq_proof_lemma_appendix_convergence_2}) yield
\begin{equation}
\lim_{\varepsilon\to 0^+}\lVert u_\varepsilon-u\rVert_{H^\frac{1}{2}}=0.
\end{equation}
Finally we observe that since $u_\varepsilon\to u$ in $L^2$ as $\varepsilon\to 0$, the convergence also takes place in measure, and therefore there exists a subsequence $(u_{\varepsilon_n})_n$ so that
$\lim_{n\to \infty} u_{\varepsilon_n}=u$ almost everywhere.
\end{proof}

\begin{lem}\label{lem_trace_operator}
There exists a continuous trace operator $Tr: H^1(D^2)\to H^\frac{1}{2}(\mathbb{S}^1)$ that, restricted to $C^1(D^2)\cap C^0(\overline{D^2})$ coincides with the restriction of a function to $\partial D^2$. Moreover, $Tr$ has a continuous right inverse, given by the harmonic extension operator.
\end{lem}
\begin{proof}
First we claim that there exists a continuous operator $Tr: H^1(D^2)\to L^2(\mathbb{S}^1)$ that coincides with the restriction to $\partial D^2$ on $C^1(D^2)\cap C^0(\overline{D^2})$.
We define, for any $u\in H^1(D^2)$, $\theta\in \mathbb{S}^1$,
\begin{equation}\label{eq_proof_lemma_def_trace}
Tr(u)(e^{i\theta}):=\overline{u}(\theta):=\int_0^1\partial_r [u(\theta,r)r^2]dr,
\end{equation}
where the argument of $u$ is given in polar coordinates.
Since $u\in H^1(D^2)$, $\overline{u}(\theta)$ is well defined for almost all $\theta\in \mathbb{S}^1$. Then for any $k\in\mathbb{Z}$,
\begin{equation}
\begin{split}
&\widehat{\overline{u}}(k)=\int_{\mathbb{S}^1}e^{-ik\theta}\int_0^1\partial_r [u(\theta,r)r^2]drd\theta=\int_0^1\int_{\mathbb{S}^1}e^{-ik\theta}\partial_r [u(\theta,r)r^2]\\
=&\int_0^1\int_{\mathbb{S}^1}e^{-ik\theta}(\partial_r u(\theta,r) r^2+2r u(\theta, r)d\theta dr=\int_0^1r^2\widehat{\partial_r u (\cdot, r)}(k)+2r\widehat{u(\cdot, r)}(k)dr,
\end{split}
\end{equation}
where we used Fubini's Theorem (since $u\in H^1(D^2)$).
Therefore
\begin{equation}
\begin{split}\label{eq_proof_lemma_L2_estimate_for_trace}
\lVert \overline{u}\rVert_{L^2(\mathbb{S}^1)}^2=&2\pi\sum_{k\in\mathbb{Z}}\widehat{\overline{u}}(k)^2
=2\pi\sum_{k\in\mathbb{Z}}\left\lvert \int_0^1 \left( r^2\widehat{\partial_r u(\cdot, r)}(k)+2\widehat{u(\cdot, r)}(k)r\right)dr\right\rvert^2\\
\leq& 4\pi\sum_{k\in \mathbb{Z}}\left(\int_0^1 r^4\left( \widehat{\partial_ru(\cdot, r)}(k) \right)^2dr+4\int_0^1\left( \widehat{u(\cdot, r)}(k) \right)^2r^2dr\right)\\
=&4\pi\int_0^1r^4\lVert \partial_r u(\cdot, r)\rVert_{L^2(\mathbb{S}^1)}^2dr+16\pi\int_0^1\lVert u(\cdot, r)\rVert_{L^2(\mathbb{S}^1)}r^2dr\\
\leq & 16\pi\left(\lVert \partial_r u\rVert_{L^2(D^2)}+\lVert u\rVert_{L^2(D^2)}\right)\leq 32\pi \lVert u\rVert_{H^1(D^2)}.
\end{split}
\end{equation}
where we used Plancherel's Identity, Jensen's Inequality and Tonelli's Theorem.
Moreover it is clear from the definition (\ref{eq_proof_lemma_def_trace}) and the fundamental theorem of calculus that if $u\in C^1(D^2)\cap C^0(\overline{D^2})$, then $\overline{u}=u\vert_{\partial D^2}$. This concludes the proof of the first claim.\\
Now we claim that the composition
\begin{equation}\label{eq_proof_lemma_continuity_trace_from_trace3}
H^\frac{1}{2}(\mathbb{S}^1)\to H^1(D^2)\to L^2, \quad u\mapsto \tilde{u}\mapsto Tr(\tilde{u})
\end{equation}
corresponds to the identity embedding of $H^\frac{1}{2}(\mathbb{S}^1)$ in $L^2(\mathbb{S}^1)$.
We have already seen that the right member of the composition is continuous. By Lemma \ref{lem_continuity_harmonic_extension_operator}, the left member of the composition is well defined and continuous.
Moreover, for any $u\in C^0(\partial D^2)$, the harmonic extension is a function continuous up to the boundary coinciding with $u$ on $\partial D^2$(see \cite{Schlag}, Prop. 1.5 and Ex. 2.3). As $C^0(\partial D^2)$ is dense in $H^\frac{1}{2}(\partial D^2)$, we conclude that the composition in (\ref{eq_proof_lemma_continuity_trace_from_trace3}) is the identity. This concludes the proof of the claim. It follows that for any $v\in H^\frac{1}{2}(\mathbb{S}^1)$, there exists $u\in H^1(D^2)$ such that $Tr(u)=v$.
Now we observe that for any $v\in L^2(\mathbb{S}^1)$ the set
\begin{equation}
\mathscr{C}_v:=\left\lbrace u\in H^1(D^2) \bigg\vert Tr(u)=v\right\rbrace
\end{equation}
is convex (possibly empty), and the functional
\begin{equation}
E_{D^2}(u)=\int_{D^2}\lvert \nabla u(x)\rvert^2dx,
\end{equation}
defined for $u\in H^1(D^2)$ is convex on $\mathscr{C}_v$. Therefore, if $\mathscr{C}_v$ is non-empty, there exists a unique minimizer of $E_{D^2}$ in $\mathscr{C}_v$. One also observe that the unique minimizer satisfies $\Delta u=0$ in $D^2$ in the weak sense. By Lemma \ref{lem_continuity_harmonic_extension_operator}, $u$ is given by $\tilde{\overline{u}}$, the harmonic extension of $\overline{u}$. Therefore, for a.e. $r\in (0,1)$, $\theta\in \mathbb{S}^1$
\begin{equation}
u(\theta, r)=P_r\ast \overline{u} (\theta).
\end{equation}
Then for any $u\in H^1(D^2)$
\begin{equation}
\begin{split}
&\lVert \nabla u\rVert_{L^2(D^2)}^2=E_{D^2}(u)\geq E_{D^2}(\tilde{\overline{u}})=\lVert \nabla \tilde{\overline{u}}\rVert_{L^2(D^2)}\geq \lVert \frac{1}{r}\partial_\theta \tilde{\overline{ u}}\rVert_{L^2(D^2)}^2\\
=&\int_0^1\lVert \partial_\theta \tilde{\overline{u}}(\cdot, r)\rVert_{L^2}dr=\int_0^12\pi\sum_{k\in\mathbb{Z}}\left\lvert \mathscr{F}(\partial_\theta P_r\ast v)(k) \right\rvert^2dr\\
=&2\pi\sum_{k\in\mathbb{Z}}\int_0^1\lvert k\rvert^2r^{2\lvert k\rvert}\lvert \widehat{v}(k)\rvert^2dr
=2\pi\sum_{k\in\mathbb{Z}\smallsetminus\lbrace0\rbrace}\frac{\lvert k\rvert^2}{2\lvert k\rvert+1}\lvert \widehat{v}(k)\rvert^2\\
\geq & 2\pi\sum_{k\in\mathbb{Z}\smallsetminus\lbrace0\rbrace}\frac{\lvert k\rvert}{3}\lvert \widehat{v}(k)\rvert^2=\frac{2\pi}{3}[v]_{\dot{H}^\frac{1}{2}}^2.
\end{split}
\end{equation}
This, together with (\ref{eq_proof_lemma_L2_estimate_for_trace}) and (\ref{eq_discussion_preliminaries_circle_bound_Sobolev_norm_through_seminorm}), shows that $Tr$ actually defines a continuous operator from $H^1(D^2)\to H^\frac{1}{2}(\partial D^2)$ with the required property.
Finally, the fact that the harmonic extension is a right inverse to $Tr$ follows from the fact that, as we have seen, the composition defined in (\ref{eq_proof_lemma_continuity_trace_from_trace3}) is equal to the identity.
\end{proof}

\begin{lem}\label{lem_continuity_harmonic_extension_operator}
The operator
\begin{equation}
F: H^\frac{1}{2}(\partial D^2)\to H^1(D^2),\quad u\mapsto \tilde{u}=P_r\ast u
\end{equation}
is well defined, linear and continuous, i.e. there exists a constant $C$ so that for any $u\in H^\frac{1}{2}(\partial D^2)$,
\begin{equation}\label{eq_lem_continuity_harmonic_extension_operator_estimate}
\lVert \tilde{u}\rVert_{H^1(D^2)}\leq C\lVert  u\rVert_{H^\frac{1}{2}(\partial D^2)}.
\end{equation}
Moreover, for any $u\in H^\frac{1}{2}(\partial D^2)$, $F(u)$ is smooth in $D^2$ and it is the unique solution in $H^1(D^2)$ of
\begin{equation}\label{lem_continuity_harmonic_extension_operator_Dirichlet_problem}
\begin{cases}
\Delta v=0\text{ in the weak sense in }D^2\\
Tr(v)=u,
\end{cases}
\end{equation}
where the trace operator $Tr$ was defined in Lemma \ref{lem_trace_operator}. In particular, $F(u)$ is smooth in $D^2$.
\end{lem}
\begin{proof}
The linearity of $F$ follows directly from the definition, therefore, to show that $F$ is well defined and continuous, it is enough to show that (\ref{eq_lem_continuity_harmonic_extension_operator_estimate}) holds for any $u\in H^\frac{1}{2}(\partial D^2)$ for some independent constant $C$.\\
First we observe that
\begin{equation}
\begin{split}
\lVert \tilde{u}\rVert_{L^2(D^2)}^2&=\int_0^1\int_{\partial D^2}\lvert P_r\ast u(\theta)\rvert^2 r d\theta dr=2\pi\int_0^1\sum_{k\in\mathbb{Z}}\widehat{P_r}(k)^2\widehat{u}(k)^2rdr\\
=&2\pi\sum_{k\in\mathbb{Z}}\int_0^1r^{2\lvert k\rvert}\widehat{u}(k)^2rdr=2\pi\sum_{k\in\mathbb{Z}}\frac{1}{2\lvert k\rvert +2}\widehat{u}(k)^2\leq 2\pi\lVert u\Vert_{L^2(\partial D^2)}.
\end{split}
\end{equation}
Here we made use of Plancherel's identity and Tonelli theorem.
Moreover we observe that
\begin{equation}
[\tilde{u}]_{\dot{H}^1(D^2)}^2=\int_{D^2}\lvert \nabla \tilde{u}(z)\rvert^2dz=\int_0^1\int_{\partial D^2}(\lvert \partial_r\tilde{u}(\theta, r)\rvert^2+\frac{1}{r^2}\lvert \partial_\theta\tilde{u}(\theta, r)\rvert^2) rd\theta dr.
\end{equation}
We compute
\begin{equation}
\begin{split}
&\int_0^1\int_{\partial D^2}\left\lvert \partial_r\tilde{u}(\theta, r)\right\rvert^2rd\theta dr=\int_0^1\int_{\partial D^2}\left\lvert \partial_rP_r\ast u(\theta)\right\rvert^2rd\theta dr\\
=&\int_0^1 2\pi\sum_{k\in\mathbb{Z}}\left\lvert \widehat{\partial_r P_r}(k)\right\rvert^2\lvert\widehat{u}(k)\rvert^2rdr
=\int_0^1 2\pi\sum_{k\in\mathbb{Z}}\left\lvert\partial_r (r^{\lvert k\rvert})\right\rvert^2\lvert\widehat{u}(k)\rvert^2rdr\\
=&2\pi\sum_{k\in\mathbb{Z}}\int_0^1 k^2r^{2\lvert k\rvert-2}\lvert\widehat{u}(k)\rvert^2rdr
=2\pi\sum_{k\in\mathbb{Z}\smallsetminus\lbrace0\rbrace}\frac{\lvert k\rvert}{2}\lvert\widehat{u}(k)\rvert^2
=\pi[u]_{\dot{H}^\frac{1}{2}(\partial D^2)}^2,
\end{split}
\end{equation}
\begin{equation}
\begin{split}
&\int_0^1\int_{\partial D^2}\frac{1}{r^2}\lvert \partial_\theta\tilde{u}\rvert^2rd\theta dr=\int_0^1\int_{\partial D^2}\frac{1}{r^2}\lvert \partial_\theta P_r\ast u\rvert^2rd\theta dr\\
=&\int_0^1 2\pi\sum_{k\in\mathbb{Z}}k^2 r^{2\lvert k\rvert} \widehat{u}(k)^2rd\theta dr
=2\pi\sum_{k\in\mathbb{Z}\smallsetminus\lbrace0\rbrace} k^2\int_0^1 r^{2\lvert k\rvert-1}dr \widehat{u}(k)^2\\
=&2\pi\sum_{k\in\mathbb{Z}\smallsetminus\lbrace0\rbrace}\frac{k^2}{2k}\widehat{u}(k)^2=2\pi[u]_{\dot{H}^\frac{1}{2}(\partial D^2)}^2.
\end{split}
\end{equation}
Again we made use of Parseval's Identity and Tonelli Theorem.
The previous computations, together with (\ref{eq_discussion_preliminaries_circle_bound_Sobolev_norm_through_seminorm}), show that there exists a constant $C$ such that for any $u\in H^\frac{1}{2}(\mathbb{S}^1)$
\begin{equation}
\lVert \tilde{u}\rVert_{H^1(D^2)}\leq C \lVert u\rVert_{H^\frac{1}{2}(\partial D^2)}.
\end{equation}
For the last statement, we already showed in the proof of Lemma \ref{lem_trace_operator} that the problem (\ref{lem_continuity_harmonic_extension_operator_Dirichlet_problem}) has at most one solution in $H^1(D^2)$. The fact that $F(u)$ is smooth follows from the fact that $P_r(\theta)$ is a smooth function in $D^2$. We are left to show that $\Delta F(u)=0$ in $D^2$ for any $u\in H^\frac{1}{2}(\mathbb{S}^1)$. For this we observe that, in $D^2$, $\Delta u=(\Delta P_r)\ast \overline{u}$ (where $\Delta$ acts on $P_r$ as a function on $D^2$, but the convolution takes place in $\mathbb{S}^1$ for any $r\in (0,1)$). We compute, for $r\in (0,1)$, $\theta\in \mathbb{S}^1$, 
\begin{equation}
\begin{split}
\Delta P_r(\theta)=&\frac{1}{r}\partial_r\left(r\partial_r \left(\sum_{k\in\mathbb{Z}}r^{\lvert k\rvert}e^{ik\theta}\right)\right)+\frac{1}{r^2}\partial_\theta^2 \left(\sum_{k\in\mathbb{Z}}r^{\lvert k\rvert}e^{ik\theta}\right)\\
=&\sum_{k\in\mathbb{Z}}\lvert k\rvert^2r^{\lvert k\rvert-2}e^{ik\theta}-\sum_{k\in\mathbb{Z}}\lvert k\rvert^2r^{\lvert k\rvert-2}e^{ik\theta}=0.
\end{split}
\end{equation}
Here we used that the decay of the coefficients allows to invert summations and derivatives. This completes the proof of the Lemma.
\end{proof}

\bibliographystyle{plain}
\bibliography{biblio/biblio_sum}

\end{document}